\documentclass[11pt]{amsart}

\usepackage{amsmath}
\usepackage{amsfonts}
\usepackage{amssymb}
\usepackage{amsthm}
\usepackage{mathtools}
\usepackage{caption}
\usepackage{subcaption}
\usepackage{bbm}
\usepackage[export]{adjustbox}
\usepackage{erewhon}
\usepackage{stmaryrd}
\usepackage[english]{babel}
\usepackage{overpic}
\usepackage[x11names]{xcolor}
\usepackage{bm}
\usepackage{comment}
\usepackage{pifont}
\usepackage{pgfplots}
\usepackage{accents}
\usepackage{tensor}
\usepackage[shortlabels]{enumitem}
\usepackage[h]{esvect}

\usepackage{calrsfs}
\DeclareMathAlphabet{\pazocal}{OMS}{zplm}{m}{n}
\DeclareMathAlphabet{\mathpzc}{OT1}{pzc}{m}{it}

\usepackage[all]{xy}

\usepackage{tikz-cd}
\usetikzlibrary{matrix}
\usepackage{graphicx} 
\usepackage{epstopdf}

\usepackage[linktocpage]{hyperref}
\hypersetup{
    colorlinks=true,
    linkcolor=blue,
    citecolor=blue,      
    urlcolor=blue,
}

\newtheorem{theorem}{Theorem}
\newtheorem{definition}[theorem]{Definition}
\newtheorem{proposition}[theorem]{Proposition}
\newtheorem{lemma}[theorem]{Lemma}
\newtheorem{claim}[theorem]{Claim}
\newtheorem{corollary}[theorem]{Corollary}

\theoremstyle{remark}
\newtheorem{example}[theorem]{Example}
\newtheorem{remark}[theorem]{Remark}

\newtheorem{question}[theorem]{Question}
\newtheorem{notation}[theorem]{Notation}

\newcommand{\bs}{\boldsymbol}
\newcommand{\op}{\operatorname}
\newcommand{\s}{\vskip.15in}
\newcommand{\n}{\noindent}

\newcommand{\bdry}{\partial}

\newcommand{\R}{\mathbb{R}}
\newcommand{\C}{\mathbb{C}}
\newcommand{\Z}{\mathbb{Z}}
\newcommand{\Q}{\mathbb{Q}}
\renewcommand{\H}{\mathbb{H}}
\newcommand{\be}{\begin{enumerate}}
\newcommand{\ee}{\end{enumerate}}

\numberwithin{equation}{section}
\numberwithin{theorem}{section}
\numberwithin{figure}{section}

\usepackage{hyphenat}

\usepackage[backend=bibtex,style=alphabetic,maxalphanames=4,maxnames=4]{biblatex}

\renewbibmacro{in:}{}

\DeclareDelimFormat[bib,biblist]{nametitledelim}{\addcomma\space}

\DeclareFieldFormat*{title}{\mkbibitalic{#1}\addcomma}
\DeclareFieldFormat*{journaltitle}{#1}
\DeclareFieldFormat*{volume}{\mkbibbold{#1}}
\DeclareFieldFormat{pages}{#1}
\DeclareFieldFormat[misc]{date}{preprint {#1}}
\DeclareFieldFormat{mr}{%
  MR\addcolon\space
  \ifhyperref
    {\href{http://www.ams.org/mathscinet-getitem?mr=MR#1}{\nolinkurl{#1}}}
    {\nolinkurl{#1}}}
    
\AtEveryBibitem{
  \clearfield{url}
  \clearfield{number}
  \clearfield{doi}
  \clearfield{issn}
  \clearfield{isbn}
  \clearfield{eprintclass}
}
\AtEveryBibitem{\ifentrytype{book}{\clearfield{pages}}{}}

\bibliography{hecke}

\title[Morse theory of loop spaces and Hecke algebras]{Morse theory of loop spaces and Hecke algebras}

\author{Ko Honda}
\address{University of California, Los Angeles, Los Angeles, CA 90095}
\email{honda@math.ucla.edu} \urladdr{http://www.math.ucla.edu/\char126 honda}

\author{Roman Krutowski}
\address{University of California, Los Angeles, Los Angeles, CA 90095}
\email{romankrut@ucla.edu} \urladdr{http://romakrut.com}

\author{Yin Tian}
\address{School of Mathematical Sciences, Beijing Normal University; 
Laboratory of Mathematics and Complex Systems, Ministry of Education, Beijing 100875, China}
\email{yintian@bnu.edu.cn} \urladdr{}

\author{Tianyu Yuan}
\address{School of Mathematical Sciences, Eastern Institute of Technology, Ningbo, Zhejiang, 315200, China}
\email{tyyuan@eitech.edu.cn} \urladdr{}

\date{\today}

\keywords{Higher-dimensional Heegaard Floer homology, braid group, Hecke algebra, loop space, string bracket, string topology}

\subjclass[2010]{Primary 53D40; Secondary 55P50, 57K31.}

\thanks{KH supported by NSF Grant DMS-2003483. YT supported by NSFC Grant No. 11971256 and 12471064.}

\begin{document}

\maketitle

\begin{abstract}
Given a smooth closed $n$-manifold $M$ and a $\kappa$-tuple of basepoints $\bm q\subset M$, we define a Morse-type $A_\infty$-algebra $CM_{-*}(\Omega(M,\bm q))$, called the {\em based multiloop $A_\infty$-algebra}, as a graded generalization of the braid skein algebra due to \cite{morton2021dahas}. For example, when $M=T^2$ the braid skein algebra is the Type A double affine Hecke algebra (DAHA).  The $A_\infty$-operations couple Morse gradient trees on a based loop space with Chas-Sullivan type string operations~\cite{chas-sullivan1999}. We show that, after a certain ``base change'', $CM_{-*}(\Omega(M,\bm q))$ is $A_\infty$-equivalent to the wrapped higher-di\-men\-sional Heegaard Floer $A_\infty$-algebra of $\kappa$ disjoint cotangent fibers which was studied in \cite{honda2022higher}.  We also compute the based multiloop $A_\infty$-al\-gebra for $M=S^2$, which we can regard as a derived Hecke algebra of the $2$-sphere.
\end{abstract}

\tableofcontents

\section{Introduction}\label{sec-intro}

We continue the longstanding study of the relationship between the topology of a smooth manifold and the symplectic geometry of its cotangent bundle. Let $M$ be a smooth closed $n$-manifold with $n\geq 2$ and let $\bm q=(q_1,\dots,q_\kappa)$, $q_i\in M$ for $i=1,\dots,\kappa$, be an ordered disjoint $\kappa$-tuple of basepoints. Motivated by the braid skein algebra of Morton-Samuelson \cite{morton2021dahas}, we consider chains of the based loop space of the unordered configuration space $\mathrm{UConf}_\kappa(M,\bm q)$, modulo certain skein relations. 

The case of $\kappa=1$ was studied by Abbondandolo-Schwarz \cite{abbondandolo2010floer} and Abou\-zaid \cite{abouzaid2012wrapped}, who showed that the wrapped Fukaya category of a cotangent fiber is $A_\infty$-equivalent to the differential graded algebra (dga) of chains of the based loop space. 

When $\kappa\geq1$, a natural generalization of the wrapped Fukaya category is the wrapped higher-dimen\-sional Heegaard Floer homology (HDHF) of Colin, Honda and Tian \cite{colin2020applications}. When $M$ is a closed oriented surface of genus greater than $0$, Honda, Tian and Yuan \cite{honda2022higher} showed that the wrapped HDHF $A_\infty$-algebra of $\kappa$ disjoint cotangent fibers of $T^*M$ is isomorphic to the Hecke algebra  --- also known as the braid skein algebra --- of $\kappa$ strands associated to $M$. (Analogous results also hold for punctured surfaces, but in this paper we assume $M$ is closed for the sake of simplicity.)
When $M=S^2$, however, the above isomorphism does not hold.  Indeed, the homology of the HDHF $A_\infty$-algebra is not supported in degree zero, whereas the braid skein algebra is supported in degree zero and is just an algebra.  

This paper is comprised of the following three parts, with an eye towards understanding the $M=S^2$ case:

\vskip.15in
\noindent
(1) In Section~\ref{section: Morse complex} we define a Morse $A_\infty$-algebra $CM_{-*}(\Omega(M, \bm q))$, which we call the {\em based multiloop $A_\infty$-algebra of the manifold $M$} and which can be viewed as a graded generalization of the braid skein algebra. The based multiloop $A_\infty$-algebra a priori depends on some auxiliary choices including a Riemannian metric on $M$, a Lagrangian action functional $\mathcal{A}_V$, and a pseudogradient vector field $X$ of $\mathcal{A}_V$. Note that instead of the space of chains, we work with a smaller set of generators, i.e.,  the Morse chain complex. The generators are based loops of $\mathrm{UConf}_\kappa(M,\bm q)$ which are critical points of the Lagrangian action functional $\mathcal{A}_V$ and the $A_\infty$-structure is defined by counting certain Morse gradient trees with respect to pseudogradient vector fields, coupled with a secondary type string operation~\cite{chas-sullivan1999}, as was introduced to Lagrangian Floer theory in unpublished work of Fukaya. 

Corollary~\ref{cor: invariance} below allows us to conclude that $CM_{-*}(\Omega(M, \bm q))$ is an invariant of the smooth manifold $M$ by appealing to Floer theory. This can be shown by solely working on the Morse-theoretic side by an argument similar to the one presented in \cite{mazuir2021I}, where the invariance of Morse homology (and the associated $A_\infty$-algebra) of a finite-dimensional manifold is established without invoking the isomorphism with singular homology. This corollary also suggests that the based multiloop homology is of its own interest from the perspective of string topology. As Corollary~\ref{cor: equivalent to braid skein algebra} and Theorem~\ref{thm: multiloop algebra for S2 is Hn} show, the based multiloop homology constitutes a non-trivial (at least in these cases) deformation of a variant of the homology of $\kappa$-th product of the based loop space with itself. We anticipate that an interpretation of this deformation in terms of singular chains, as in e.g., \cite{chataur2005, irie2018, laudenbach2011}, on the $\kappa$-th power of the based loop space $\Omega(M, \bm q)$ may provide tools for further computations.

We additionally point out that the construction of the differential in the Morse model is akin to the Morse-theoretic construction of the secondary loop coproduct of Sullivan \cite{sullivan2004open-closed} and Goresky-Hingston \cite{GH2009loop-product, HW2023product-coproduct}, as explained in \cite{CHO2023coproduct}. Specifically, the specialization $\hbar^2=0$, which requires counting Morse flow lines with up to one switching, should be closely related to the loop coproduct. Furthermore, we anticipate that string topology operations and the whole TQFT structure \cite{CG2004polarized, godin2007higher, tamanoi2010tqft} extend to the based (free) multiloop homology, though we do not pursue this direction here.

\vskip.15in
\noindent
(2) In Section~\ref{section: isomorphism with HDHF} we construct an $A_\infty$-morphism $\mathcal{F}$ from $CW^*(\sqcup_{i=1}^\kappa T_{q_i}^*M)$, the wrap\-ped HDHF chain group of $\kappa$ disjoint cotangent fibers, to the based multiloop $A_\infty$-algebra $CM_{-*}(\Omega(M, \bm q))$, by counting rigid elements in a mixed moduli space which combines pseudoholomorphic curves and Morse gradient trees.  

\begin{remark}[Remark on ground rings] \label{rmk: coefficient rings}
    The ground ring for the based multiloop $A_\infty$-algebra is $R=\Z$.  The ground ring for the wrapped HDHF group $CW^*(\sqcup_{i=1}^\kappa T_{q_i}^*M)$ and the map $\mathcal{F}$ is $R=\Z$ when $\kappa=1$ or $n=2$, and $R=\Q$ when $\kappa>1$ and $n\geq 3$.  The use of $R=\Q$ is due to the use of Kuranishi replacements (see Section~\ref{subsubsection: ghost bubbles}), which require branched manifolds.
\end{remark}

\begin{remark} [Deformation parameter $\hbar$ and base change] \label{rmk: deformation parameter}
    There is a deformation parameter $\hbar$ of degree $|\hbar|=2-n$ that we informally refer to as the ``Planck constant''.  When $n>2$, $CM_{-*}(\Omega(M, \bm q))$ and $CW^*(\sqcup_{i=1}^\kappa T_{q_i}^*M)$ have the form $CM_{0,-*}(\Omega(M, \bm q))\otimes_R R[\hbar]$ and $CW^*_0(\sqcup_{i=1}^\kappa T_{q_i}^*M)\otimes_R R[\hbar]$, and when $n=2$,  $CM_{-*}(\Omega(M, \bm q))$ and $CW^*(\sqcup_{i=1}^\kappa T_{q_i}^*M)$ have the form $CM_{0,-*}(\Omega(M, \bm q))\otimes_R R\llbracket\hbar\rrbracket$ and $CW^*_0(\sqcup_{i=1}^\kappa T_{q_i}^*M)\otimes_R R\llbracket\hbar\rrbracket$. The precise definition of $CM_{0,-*}(\Omega(M, \bm q))$ is given in Section~\ref{subsection: based multiloop complex} and Definition~\ref{defn: based multiloop} and that of $CW^*_0(\sqcup_{i=1}^\kappa T_{q_i}^*M)$ is given in Section~\ref{subsection: review of HDHF}. When $n>2$, we define \emph{$\hbar$-adic completions}
    \begin{gather} \label{eqn: tensoring1}
    CM_{-*}(\Omega(M, \bm q))\llbracket\hbar \rrbracket= CM_{0,-*}(\Omega(M, \bm q))\widehat{\otimes}_{R} R\llbracket\hbar \rrbracket \coloneqq  \prod_{j=0}^\infty CM_{0,-*+(n-2)j}(\Omega(M, \bm q)) \hbar^j,\\
    \label{eqn: tensoring2}
    CW^*(\sqcup_{i=1}^\kappa T_{q_i}^* M)\llbracket\hbar \rrbracket= CW^*_0(\sqcup_{i=1}^\kappa T_{q_i}^* M)\widehat{\otimes}_{R} R\llbracket\hbar \rrbracket\coloneqq  \prod_{j=0}^\infty CW^{*+(n-2)j}_0(\sqcup_{i=1}^\kappa T_{q_i}^* M) \hbar^j, 
    \end{gather}
    The definition above is motivated by the fact that we are considering deformations of initial chain complexes ($A_\infty$-algebras) by a (possibly) non-zero graded formal variable $\hbar$; see also \cite[Section 2.3]{ganatra2023cyclic}.
\end{remark}

We show that:

\begin{theorem}
\label{thm-main}
    $\mathcal{F}$ is an $A_\infty$-equivalence when $n=2$. When $n>2$ it induces an $A_\infty$-equivalence of $\hbar$-adic completions.
\end{theorem}

The theorem has the following corollaries:

\begin{corollary} \label{cor: invariance}
The $A_\infty$-equivalence class of $CM_{-*}(\Omega(M, \bm q))$ when $n=2$ (resp.\ $CM_{-*}(\Omega(M, \bm q))\widehat{\otimes}_{R[\hbar]} \Q\llbracket \hbar\rrbracket$ when $n>2$) is an invariant of the smooth manifold $M$ and the positive integer $\kappa$.
In particular, $CM_{-*}(\Omega(M, \bm q))$ when $n=2$ (resp.\ $CM_{-*}(\Omega(M, \bm q)) \widehat{ \otimes}_{R[\hbar]} \Q\llbracket \hbar\rrbracket$ when $n>2$) does not depend on the auxiliary choices used in the definition.
\end{corollary}

\begin{proof}
This follows from the fact that the wrapped HDHF chain group is an invariant of the smooth manifold $M$ and the positive integer $\kappa$.
\end{proof}

\begin{corollary} \label{cor: equivalent to braid skein algebra}
When $M$ is a closed surface $\not=S^2$, $CM_{-*}(\Omega(M, \bm q))$ is quasi-equi\-valent to the braid skein algebra $\mathrm{BSk}_\kappa(M,\bm q)\otimes_{R[\hbar]} R\llbracket \hbar\rrbracket$ of Morton-Samuelson~\cite{morton2021dahas}.
\end{corollary} 

\begin{proof} 
Combine Theorem~\ref{thm-main} and the main result of \cite{honda2022higher}.
\end{proof}   

\begin{remark}
A curious observation is that, while the HOMFLY-PT relations were algebraically imposed as local relations for the braid skein algebra, in the case of the multiloop $A_\infty$-algebra $CM_{-*}(\Omega(M, \bm q))$, they are {\em not algebraically imposed} and {\em appear naturally} in the computation of $\mu^2$. 
\end{remark}

\noindent
(3) In Section~\ref{section: algebraic model} we prove the following theorem:

\begin{theorem} \label{thm: multiloop algebra for S2 is Hn} 
The based multiloop $A_\infty$-algebra $CM_{-*}(\Omega(S^2, \bm q))$ is quasi-equivalent to the dga $H_\kappa$, described below, for a particular choice of ${\bm q}$ and auxiliary data.
\end{theorem}

\begin{definition} \label{defn: Hn}
    Let $H_\kappa$ be the unital dga over $\mathbb{Z}[\hbar]$, with generators $T_1, \dots, T_{\kappa-1}$ and $x_1$, and relations
    \begin{align}
      \label{first}  T_i^2=1+\hbar T_i, & \quad 1\le i \le \kappa-1, \\
      \label{second}  T_iT_{i+1}T_i=T_{i+1}T_iT_{i+1}, & \quad 1 \le i \le \kappa-2, \\
      \label{third}  T_iT_j=T_jT_i, & \quad 1 \le i,j \le \kappa-1, |i-j|>1, \\
        x_1T_j=T_jx_1, & \quad 2\le j \le \kappa-1, \label{eq Hn x1Tj} \\
       T_1^{-1}x_1T_1^{-1}x_1+x_1T_1^{-1}x_1T_1=0. & \label{eq Hn x1x2}
    \end{align}
    The grading on the generators is given by $|T_i|=0$ and $|x_1|=-1$.
    The differential on the generators is given by $d T_i=0$ and 
    \begin{align} \label{sixth}
        dx_1=T_1T_2\cdots T_{\kappa-2}T_{\kappa-1}^2T_{\kappa-2}\cdots T_2T_1-1,
    \end{align}
    and extends to $H_\kappa$ by the Leibniz rule.
\end{definition}
\begin{remark}
    The free loop space version of $CM_{-*}(\Omega(M,{\bm q}))$ is the {\em free multiloop $A_\infty$-algebra} which is discussed in \cite{KY}.
\end{remark}

Comparing Theorem~\ref{thm: multiloop algebra for S2 is Hn} and Corollary~\ref{cor: equivalent to braid skein algebra} suggests that the DGA $H_\kappa$ should admit a natural interpretation in terms of quantum topology. As it is of derived nature, it is necessary to develop a notion of a derived HOMFLYPT skein algebra of a surface. The natural guess here is that the approach of factorization homology developed in \cite{BBJ2018, GJS2023} should provide the necessary tool for such computations. Therefore we raise the following:

\begin{question}
What is the precise relation between the DGA $H_\kappa$ and the appropriate notion of a derived HOMFLYPT skein algebra of $S^2$?
\end{question}

\s\n
{\em Acknowledgments.}  KH thanks Peter Samuelson for helpful conversations. He is grateful to Hirofumi Sasahira and Kyushu University for their hospitality during his sabbatical in AY 2022--2023. RK thanks Kai Cieliebak, Sam Gunningham, Marco Mazzucchelli, Alexander Petrov and David Popović for stimulating conversations.  Finally, we are very grateful to Alberto Abbondandolo, who taught us the main ideas of the proof of Theorem~\ref{thm: regularity for pseudogradient} in Appendix~\ref{appendix: regularity for unstable submanifolds}.

\section{The based multiloop $A_\infty$-algebra} \label{section: Morse complex}

As in the introduction, $M$ is a smooth closed oriented $n$-manifold with $n\geq 2$ and ${\bm q}$ is an ordered disjoint $\kappa$-tuple of points on $M$.  
Let $g$ be a generic Riemannian metric on $M$ and $|\cdot|$ be the induced norm on $T^*M$.

\subsection{The Lagrangian action functional and pseudogradient vector fields} \label{subsection: Lagrangian action functional}

We follow the setup from \cite{abbondandolo2010floer}; also see \cite{AS2009}.

Consider the space of time-$1$ continuous paths:
\begin{equation}
    \Omega(M,q,q')=\{\gamma\in C^0([0,1],M)\,|\,\gamma(0)=q,\gamma(1)=q'\}.
\end{equation}
We denote $\Omega(M,q)=\Omega(M,q,q)$ and denote the subset of $\Omega(M,q,q')$ in the class $W^{1,2}$ by $\Omega^{1,2}(M,q,q')$. 

Throughout the text we consider Lagrangians $L \colon [0,1] \times TM \to \R$ satisfying the conditions given in \cite[Conditions (L1) and (L2)]{AS2009}. For our purposes it suffices to restrict our attention to the Lagrangians associated with a potential, i.e., let $V\colon[0,1]\times M\to\mathbb{R}$ be a function with small $W^{1,2}$-norm and let $L\colon [0,1]\times TM\to\mathbb{R}$ be a smooth Lagrangian function satisfying
\begin{enumerate}[(L1)]
    \item \label{L1-condition} There is a continuous function $\ell_1$ on $M$ such that for every $(t, x, v) \in [0,1] \times T M$ with $x\in M$ and $v\in T_xM$,
$$
\begin{aligned}
\|\nabla_{v v} L(t, x, v)\| & \leq \ell_1(x), \\
\|\nabla_{v x} L(t, x, v)\| & \leq \ell_1(x)(1+|v|), \\
\|\nabla_{x x} L(t, x, v)\| & \leq \ell_1(x)(1+|v|^2),
\end{aligned}
$$
where $|v|=g(v,v)^{1/2}$.
\item \label{L2-condition} There is a continuous positive function $\ell_2$ on $M$ such that $\nabla_{v v} L(t, x, v) \geq$ $\ell_2(x) I$, for every $(t, x, v) \in[0,1] \times T M$.
\end{enumerate}
To any smooth (potential) function $V \in C^\infty([0,1] \times M)$ we may associate Lagrangian
\begin{equation}
\label{eq-lagrangian}
    L_V(t,x,v)=\tfrac{1}{2}|v|^2-V(t,x),
\end{equation}
satisfying (L1) and (L2). The metric $g$ on $M$ induces a metric $\tilde{g}$ on $\Omega^{1,2}(M,q)$.

Given a Lagrangian $L \colon [0,1] \times TM \to \R$, we define the Lagrangian action functional
\begin{gather*}
    \mathcal{A}_L\colon\Omega^{1,2}(M,q,q')\to\mathbb{R},\\
    \mathcal{A}_L(\gamma)=\int_0^1 L(t,\gamma(t),\dot{\gamma}(t))\,dt.
\end{gather*}
We write $\mathcal{A}=\mathcal{A}_L$ when the Lagrangian $L$ is understood from the context.
When $L=L_V$ for some potential $V$ as above, the action functional $\mathcal{A}$ is twice Fr\'echet differentiable since for each $(t,x)$ the map $v\mapsto L(t,x,v)$ is a polynomial of degree at most $2$; see \cite[p.1583]{abbondandolo2010floer} and \cite{AS2009}.
\begin{remark} \label{rmk: V=0 case}
For a sufficiently generic Riemannian metric $g$ on $M$, 
\be
\item[($\dagger$)] all the critical points $\gamma$ of $\mathcal{A}_{L_V}$ with $V=0$ and $q=q'$ (i.e., $g$-geodesics of $\gamma:[0,1]\to M$ such that $g(0)=g(1)=q$) are nondegenerate and do not pass through $q$ except at their endpoints.
\ee
The nondegeneracy/Morse condition is required in the definition of the differential $d$ and the condition of not passing through $q$ except at their endpoints ensures that $\Omega^{1,2}(M,q)$ (together with the choice of $X\in \mathcal{G}$) behaves well when taking products and higher $A_\infty$-opera\-tions.  ($\dagger$) can be realized by ordering the geodesics $\gamma_1,\gamma_2,\dots$ by nondecreasing action and inductively perturbing $g$ on successively smaller neighborhoods $N(x_i)$ of $x_i\in \op{Im}\gamma_i$.  
\end{remark}

A vector field $X$ on a Hilbert manifold $\mathcal{M}$ is a  {\em (negative) pseudogradient} vector field of $f$ if it satisfies the following:
\be
\item[(PG1)] $f$ is a Lyapunov function for $X$ (this means a pseudogradient roughly points in the {\em negative} gradient direction of $f$);
\item[(PG2)] $X$ is a Morse vector field whose singular points (= critical points of $f$) have finite Morse index;
\item[(PG3)] the pair $(f,X)$ satisfies the \emph{Palais-Smale condition};  
\item[(PG4)] $X$ is forward-complete, i.e., the flow $\Phi(t,x)$ is defined for all $t>0$ and $x\in \mathcal{M}$.
\ee

In addition to the above properties we usually require:
\be
\item[(PG5)] $X$ satisfies Morse-Smale property.
\ee
The following result summarizes properties of $\mathcal{A}$ and, in particular, shows that there are plenty of (Morse-Smale) pseudogradient vector fields  (see \cite[Appendix A.1]{abbondandolo2010floer} and \cite{abbondandolo-majer2006}):

\begin{theorem} \label{thm: summary of morse in infinite dimensions} $\mbox{}$
For smooth Lagrangian $L \colon [0,1] \times TM \to \R$ satisfying (L1) and (L2) with all critical points of $\mathcal{A}_L$ nondegenerate the following hold:
\be
\item[(i)] The action functional $\mathcal{A}_L$ is bounded below. 

\item[(ii)] If all the critical points of $\mathcal{A}_L$ are nondegenerate, then there exists a \emph{smooth} vector field $X$ satisfying (PG1)--(PG4). Moreover, for $L=L_V$ the gradient $-\nabla_{\tilde g} \mathcal{A}$ is a smooth pseudogradient.

\item[(iii)] Any sufficiently small perturbation of $X$ which coincides with $X$ on a neighborhood of the critical points of $\mathcal{A}$
is still a pseudogradient (i.e., satisfies (PG1)--(PG4)).   
\ee
Let $\mathcal{K}$ be the set of small perturbations of $X$ as above which satisfy (iii) and hence are pseudogradients.
\be
\item[(iv)] A generic $\tilde{X}\in \mathcal{K}$ satisfies the {\em Morse-Smale condition} (PG5).
\ee
Let  $\, \mathcal{G}\subset \mathcal{K}$ be the dense subset of Morse-Smale vector fields. 
\end{theorem}

The generalization to $\kappa$-strand multipaths on $M$ is straightforward:
Let $\bm q=(q_1,\dots,q_\kappa)$ and $\bm q'=(q'_1,\dots,q'_\kappa)$ be (ordered, disjoint) $\kappa$-tuples of points on $M$ and ${\sigma}\in S_\kappa$ a permutation. 
We define
\begin{gather*}
    \Omega_{\sigma}(M,\bm q,\bm q')=\prod_{i=1}^\kappa\Omega(M,q_i,q'_{\sigma(i)}),\quad
    \Omega^{1,2}_\sigma(M,\bm q,\bm q')=\prod_{i=1}^\kappa\Omega^{1,2}(M,q_i,q'_{\sigma(i)}),\\
    \Omega^{1,2}(M,\bm q,\bm q')=\displaystyle \bigsqcup_{\sigma\in S_\kappa}\Omega^{1,2}_{\sigma}(M,\bm q,\bm q'),\\
	\Omega^{1,2}_\sigma(M,\bm q)=\Omega^{1,2}_\sigma(M,\bm q,\bm q),\quad  \Omega^{1,2}(M,\bm q)=\Omega^{1,2}(M,\bm q,\bm q).
\end{gather*}
Given the $\kappa$-strand multipath $\boldsymbol{\gamma}=(\gamma_1,\dots,\gamma_\kappa)\in\Omega^{1,2}(M,\bm q,\bm q')$, the Lagran\-gian action functional $\mathcal{A}_L\colon\Omega^{1,2}(M,\bm q,\bm q')\to\mathbb{R}$,
also denoted $\mathcal{A}_V$ or $\mathcal{A}_{L_V}$, is given by:
\begin{equation}
    \mathcal{A}_L(\boldsymbol{\gamma})=\sum_i \mathcal{A}_L(\gamma_i).
\end{equation}
If $\boldsymbol{\gamma}\in\Omega^{1,2}_\sigma(M,\bm q,\bm q')$, then $\sigma$ is called the {\em permutation type of $\boldsymbol{\gamma}$.}

\subsection{The switching map}
\label{section: switching map}

In this section, we introduce \emph{the switching map} which will play a crucial role in the definition of the $A_\infty$-operations of the based multiloop $A_\infty$-algebra. Let $\Omega^{m,2}(M, \bm q, \bm q')$ for some integer $m \ge 1$ be the subset of $\Omega(M, \bm q, \bm q')$ of elements of class $W^{m,2}$ and let $\Delta_M$ be the diagonal in $M\times M$. For $1\leq i<j\leq \kappa$, define the {\em evaluation map} 
\begin{gather*}
ev^{ij}_I: [0,1]\times \Omega(M, \bm q, \bm q') \to M\times M,\\
ev^{ij}_I(\theta, \bm \gamma)=(\gamma_i(\theta),\gamma_j(\theta)),
\end{gather*}
where the subscript $I$ stands for $[0,1]$.  We then set
$$\Delta^{ij}_I =(ev^{ij}_I)^{-1}(\Delta_M).$$ 

The restriction of each $ev^{ij}_I$ to $[0,1] \times \Omega^{m,2}(M, \bm q, \bm q')$ is clearly a $C^{m-1}$-submersion and therefore each $\Delta^{ij}_I$ restricted to $[0,1] \times \Omega^{m,2}(M, \bm q, \bm q')$ is a closed $C^{m-1}$-smooth Banach submanifold of $[0,1] \times \Omega^{m,2}(M,\bm q, \bm q')$ of codimension $n$. Note that the loss of regularity is due to taking derivatives in the $\theta$-direction. By abuse of notation, the restriction of $\Delta^{ij}_I$ to $[0,1] \times \Omega^{m,2}(M, \bm q, \bm q')$ will also be denoted $\Delta^{ij}_I$, and the ambient space will either be specified or be clear from the context. 

Let $\Delta_I$ be the union $\bigcup_{1 \le i<j \le \kappa} \Delta^{ij}_I$. As before, the subset $\Delta_I$ (as well as $\Omega^{m,2}(M, \bm q, \bm q')$) decomposes into $\kappa!$ connected components corresponding to the permutation types:
$$\Delta_I=\bigsqcup_{\sigma \in S_\kappa}\Delta_{I, \sigma}.$$

In the case $\kappa=2$, the naive definition of the \emph{switching map} from $\Delta_I$ to itself would be:
\begin{equation}
(\theta, \gamma_1, \gamma_2) \mapsto (\theta, \gamma_1|_{[0, \theta]} \#\gamma_2|_{[\theta, 1]},\gamma_2|_{[0, \theta]} \#\gamma_1|_{[\theta, 1]}),
\end{equation}
where $\#$ is concatenation (see Figure~\ref{fig:switching-map}). This naive switching map, when restricted to $[0,1] \times \Omega^{m,2}(M, \bm q, \bm q')$, unfortunately does not preserve the $C^{m-1}$-smoothness of both concatenated paths at $\theta$ and hence cannot have image in $\Omega^{m,2}(M, \bm q, \bm q')$. As a remedy for this, we will apply a certain reparametrization of $[0,1]$ which guarantees the $C^{m-1}$-smoothness at $\theta$, i.e., before the switching we replace $(\theta, \gamma_1, \gamma_2)$ by $(\theta, \gamma_1 \circ \nu, \gamma_2 \circ \nu)$, where $\nu \colon [0,1] \to [0,1]$ is a smooth bijection with all partial derivatives vanishing at $\theta$ and satisfying $\nu(t)=t$ on most of $[0,1]$. This is a common trick in string topology and appears for example in \cite{cohenjones2002, irie2018}.  

The actual implementation of the reparametrization will occupy the next two subsections. 

\subsubsection{Multidiagonals}

For strata consistency, we need to define switching maps on multidiagonals.

We first consider the simplest case which is the space
\begin{equation}
    \overline{\Delta}_{I_1} \cup \overline{\Delta}_{I_2} \subset [0,1]^2 \times \Omega(M, \bm q, \bm q'),
\end{equation}
where $\Delta_{I_1} \subset \op{Conf}_2([0,1]) \times \Omega(M, \bm q, \bm q')$ is the preimage of the diagonal $\Delta_M$ under the evaluations with respect to the first $[0,1]$-coordinate, $\Delta_{I_2}$ is the preimage under the evaluations with respect to the second $[0,1]$-coordinate, $\overline{\Delta}_{I_1}$ and $\overline{\Delta}_{I_2}$ are their closures, and $\op{Conf}_2([0,1])$ is the {\em ordered} configuration space of points in $[0,1]$. Note that $\Delta_{I_{i}}$ and $\overline{\Delta}_{I_i}$ are $C^{m-1}$-smooth Banach submanifolds as before when considered inside $[0,1]^2 \times \Omega^{m,2}(M, \bm q, \bm q')$.

We also define subsets $\Delta^{ij, kl}_{I_1, I_2}\subset \Delta^{ij}_{I_1} \cup \Delta^{kl}_{I_2}$ (here  $\{i, j\}= \{k, l\}$ is allowed), given by 
\begin{gather}
    \Delta^{ij, kl}_{I_1, I_2}=(ev^{ij, kl}_{I_1,I_2})^{-1}(\Delta_M \times \Delta_M),
\end{gather}
where
\begin{gather}
    \label{eq: ev-map-double} 
    ev^{ij, kl}_{I_1, I_2} \colon \op{Conf}_2([0,1]) \times \Omega(M, \bm q, \bm q') \to M^{4}, \\
    (\theta_1, \theta_2, \bm \gamma) \mapsto (\gamma_i(\theta_1), \gamma_j(\theta_1), \gamma_k(\theta_2), \gamma_l(\theta_2)). \nonumber
\end{gather}
We write
\begin{equation}
\overline{\Delta}_{I_1, I_2}= \bigcup_{i<j; k<l}\overline{\Delta}^{ij, kl}_{I_1, I_2} \subset [0,1]^2 \times \Omega(M, \bm q, \bm q').    
\end{equation}
The restrictions of the spaces above to $[0,1]^2 \times \Omega^{m,2}(M, \bm q, \bm q')$ are also denoted in the same way by abuse of notation.

\begin{lemma}\label{lemma: two-intersection-moments} $\mbox{}$
\begin{enumerate}
    \item The subset $\Delta^{ij, ij}_{I_1,I_2}$ is a not necessarily closed $C^{m-1}$-submanifold of $[0,1]^{ 2} \times \Omega^{m,2}(M, \bm q, \bm q')$ of codimension $2n$. The boundary of its closure $\overline{\Delta}^{ij,ij}_{I_1,I_2}$ is a $C^{m-2}$-submanifold of codimension $2n+1$ for $m \ge 3$. 
    \item For $\{i, j\} \neq \{k, l\}$ the subset $\Delta^{ij, kl}_{I_1, I_2}$ and its closure $\overline{\Delta}^{ij, kl}_{I_1, I_2}$ in $[0,1]^{ 2} \times \Omega^{m,2}(M, \bm q, \bm q')$ are $C^{m-1}$-submanifolds of $[0,1]^{ 2} \times \Omega^{m,2}(M, \bm q, \bm q')$.
\end{enumerate}
\end{lemma}

\begin{proof}
(1) It is clear that $ev^{ij, ij}_{I_1,I_2}$ is transverse to $\Delta_M \times \Delta_M$ since the strands $\gamma_i$ and $\gamma_j$ can be moved independently at $\theta_1$ and $\theta_2$. Hence $\Delta^{ij, ij}_{I_1,I_2}$ is a codimension $2n$ submanifold of $\op{Conf}_2([0,1]) \times \Omega^{m,2}(M, \bm q, \bm q')$.  The boundary of $\Delta^{ij, ij}_{I_1, I_2} \subset [0,1]^{ 2} \times \Omega^{m,2}(M, \bm q, \bm q')$ is a $C^{m-2}$-submanifold of codimension $2n+1$, as it is equal to the  
preimage of $\Delta_{TM}$ under the submersion from $\Delta_{[0,1]} \times \Omega^{m,2}(M, \bm q, \bm q')$ to $TM$ given by
$$(\theta, \theta, \bm \gamma) \mapsto (\gamma_i(\theta), \dot{\gamma}_i(\theta), \gamma_j(\theta), \dot{\gamma}_j(\theta)).$$

(2) Similarly, $ev^{ij, kl}_{I_1,I_2}$ is transverse to $\Delta_M \times \Delta_M$ and the extension of $ev^{ij, kl}_{I_1,I_2}$ to $[0,1]^{ 2} \times \Omega^{m,2}(M, \bm q, \bm q')$ is also transverse to $\Delta_M \times \Delta_M$, since one of the strands, say $\gamma_l$, can be moved independently of $\gamma_i$ and $\gamma_j$ at the common value $\theta$.
\end{proof}

In the general case, given an integer $\ell \ge 2$, a nonempty subset 
$$A=\{i_1, \dots, i_r\} \subseteq [\ell] \coloneqq \{1, \dots, \ell\},$$ 
and an $r$-tuple of two-element subsets 
$$\bm \tau=(\tau_1=\{j^1_1, j^2_1\}, \dots, \tau_r=\{j^1_r, j^2_r\}),$$
of $[\kappa]$, we define the multi-evaluation map
\begin{equation}
    ev^{\bm \tau}_{I_A} = ev^{\tau_1, \dots, \tau_r}_{I_{i_1}, \dots, I_{i_r}}\colon \op{Conf}_\ell([0,1]) \times \Omega(M, \bm q, \bm q') \to M^{2r},
\end{equation}
obtained by composing the projection 
\begin{gather*}
    \op{Conf}_\ell([0,1])\times \Omega(M, \bm q, \bm q') \to \op{Conf}_r([0,1])\times \Omega(M, \bm q, \bm q'),\\
    (\theta_1,\dots,\theta_\ell,\bm\gamma)\mapsto (\theta_{i_1},\dots,\theta_{i_r},\bm\gamma),
\end{gather*}
and the evaluation map
\begin{gather*}
\op{Conf}_r([0,1])\times \Omega(M, \bm q, \bm q')\to M^{2r},\\
(\theta_1,\dots,\theta_r,\bm\gamma)\mapsto (\gamma_{j_1^1}(\theta_1), \gamma_{j_1^2}(\theta_1),\dots, \gamma_{j_r^1}(\theta_r), \gamma_{j_r^2}(\theta_r)),
\end{gather*}
by analogy with~\eqref{eq: ev-map-double}. Then the multidiagonal
\begin{equation}\label{eq: diagonal-with-many-intersections}
  \Delta_{I_A}^{\bm \tau}=\Delta_{I_{i_1}, \dots, I_{i_r}}^{\bm \tau}\coloneqq (ev^{\tau_1, \dots, \tau_r}_{I_{i_1}, \dots, I_{i_r}})^{-1}\Delta_M^{\times r},\footnote{We sometimes write e.g., $\times \ell$ to emphasize that is is a product of $\ell$ copies.}
\end{equation} 
viewed as a subset of $\op{Conf}_\ell([0,1]) \times \Omega(M, \bm q, \bm q')$, is a $C^{m-1}$-submanifold of codimension $\ell n$.

By taking unions of such submanifolds and their closures over various choices of $\bm \tau$ we define subsets $\Delta_{I_A}=\Delta_{I_{i_1}, \dots, I_{i_r}}, \overline{\Delta}_{I_A}=\overline{\Delta}_{I_{i_1}, \dots, I_{i_r}} \subset [0,1]^{\ell} \times \Omega(M, \bm q, \bm q')$. {\em Note that when we write $\Delta_{I_{i_1}, \dots, I_{i_r}}$, the set $[\ell]$ is implicit.}

\subsubsection{Construction of reparametrizations}

We start by considering the $2$-variable case.  In this case we pick a smooth family of smooth bijective functions $f_{\theta_1,\theta_2} \colon [0,1] \to [0,1]$ parametrized by $(\theta_1,\theta_2) \in \op{Conf}_2((0,1))$ 
such that for all $(\theta_1,\theta_2) \in \op{Conf}_2((0,1))$:
\begin{gather}
f_{\theta_1, \theta_2}(\theta_1)=\theta_1, \quad f_{\theta_1, \theta_2}(\theta_2)=\theta_2; \label{eq: family-property1}   \\
    f_{\theta_1,\theta_2}^{(k)}(\theta_1)=f_{\theta_1,\theta_2}^{(k)}(\theta_2)=0 \quad \text{ for all }k\ge 1;  \\
    f_{\theta_1, \theta_2} 
   \text{ continuously extends to the diagonal } \theta_1=\theta_2 . \label{eq: family-property3}
\end{gather}

Here we present an explicit construction of such a family:
Let $\varphi \colon [0,1] \to [0,1]$ be a smooth monotonically increasing function such that $\varphi(0)=0$, $\varphi(1)=1$, $\varphi(x)=x$ for $x \le 1-\varepsilon$ with $\varepsilon >0$ small, and all the derivatives of $\varphi$ vanish at $x=1$.  We also set $\varphi^\circ(x):=1-\varphi(1-x)$.  Additionally, let $\psi \colon [0,1] \to [0,1]$ be a smooth step function such that $\psi(0)=0$, $\psi(1)=1$, and all the derivatives of $\psi$ vanish at the endpoints $x=0,1$ (one can take the composition of $\varphi$ and $\varphi^\circ$, for instance). 
We then define $f_{\theta_1,\theta_2}$ on $\{\theta_1 \le \theta_2\}$ as follows:

\begin{equation}\label{eq: formula-for-reparametrization}
    f_{\theta_1,\theta_2}(x)= \left\{
    \begin{array}{ll}
       \varphi(\frac{x}{\theta_1})\cdot \theta_1,  &  \mbox{ for } x \le \theta_1, \\
        \psi(\frac{x-\theta_1}{\theta_2-\theta_1})\cdot (\theta_2-\theta_1)+\theta_1, & \mbox{ for } \theta_1 < x \le \theta_2, \\
        \varphi^\circ(\frac{x-\theta_2}{1-\theta_2})\cdot(1-\theta_2)+\theta_2, & \mbox{ for } \theta_2 < x \le 1.
    \end{array} \right.
\end{equation}
We then extend this function to all of $(0,1)^{2}$ by an analogous formula for $\theta_1>\theta_2$, switching the roles of $\theta_1$ and $\theta_2$.
This family is clearly smooth on $\op{Conf}_2((0,1))$ and extends continuously to all of $(0,1)^2$, and hence satisfies 
\eqref{eq: family-property1}-\eqref{eq: family-property3}. We also set $f_\theta:=f_{\theta,\theta}$, which is a smooth family of functions parametrized by the interval $(0,1)$.

We can generalize the above construction to the $r$-variable case, i.e., to a continuous family of smooth monotonically increasing functions $f_{\theta_1, \dots, \theta_r} \colon [0,1] \to [0,1]$ parametrized by $(0,1)^r$, whose restriction to $\op{Conf}_r((0,1))$ is smooth and satisfies following:
\begin{gather}
f_{\theta_1, \dots, \theta_r}(\theta_i)=\theta_i \quad \text{ for all $i$ such that  $1\le i \le r$;} \label{eq: family-property1 for r}   \\
    f_{\theta_1,\dots, \theta_r}^{(k)}(\theta_i)=0 \quad \text{ for all $k\geq 1$ and $i$ such that $1\leq i\leq r$}.
\end{gather}
Such a family can be constructed using a formula similar to \eqref{eq: formula-for-reparametrization}. 

We also choose a function $\varepsilon_{[r]}:[0,1]^r\to(0,\frac{1}{4}]$, $\bm\theta=(\theta_1, \dots, \theta_r) \mapsto \varepsilon_{[r]}(\bm\theta)$, such that $f_{\theta_1, \dots, \theta_r}(x)=x$ on $\varepsilon_{[r]}(\bm\theta)$-neighborhoods of $x=0$ and $x=1$. The existence of such a function is immediate from the definition of $f_{\theta_1, \dots, \theta_r}$. 

Given a subset $A=\{i_1, \dots, i_r\} \subset [\ell]$ as before, we construct a continuous family of smooth monotone functions $\nu_A=\nu_{i_1, \dots, i_r}:[0,1]\to [0,1]$ parametrized by $\overline{\Delta}_{i_1} \cup \dots \cup \overline{\Delta}_{i_r}$: First choose a cover $\{U_S\}_{S \subset A}$ of $\overline{\Delta}_{i_1} \cup \dots \cup \overline{\Delta}_{i_r}$ 
such that:
\be
\item[(C1)] $\overline{\Delta}_{I_S} \subset \bigcup_{S' \subseteq S} U_{S'}$  for all $S \subset A$;
\item[(C2)] $U_{S'} \cap \overline{\Delta}_{S}=\emptyset$  for all pairs $S,S' \subset A$ such that $S' \subsetneq S$.
\ee
Fix a partition of unity $\{\psi_S\}_{S \subset A}$ subordinate to $\{U_S\}_{S \subset A}$. Given $S=\{j_1, \dots, j_{r'}\} \subset A$ we define $$f_S:U_S\times[0,1]\to [0,1], \quad ((\theta_1,\dots,\theta_\ell,\bm\gamma),x)\mapsto f_{\theta_{j_1},\dots,\theta_{j_{r'}}}(x).$$ 
We then set
\begin{equation}
    \nu_A=\sum_{S \subset A}\psi_Sf_S.
\end{equation}

By (C1) and (C2), the restriction of $\nu_A$ to $\overline{\Delta}_S$ is equal to $f_S$ on the complement 
of the union $\bigcup_{S' \subsetneq S} U_{S'}$.  
The function $\nu_{[\ell]}$ can then be extended to all of $[0,1]^\ell \times \Omega(M, \bm q, \bm q')$ by adding to the cover $\{U_S\}_{S \subset [\ell]}$ an open set $U_{\emptyset}$ which does not intersect any of the $\overline{\Delta}_S$.

We also choose a smooth function 
$$\varepsilon_{[\ell]}: [0,1]^\ell \times \Omega(M, \bm q, \bm q')\to  (0, \tfrac{1}{4}]$$ 
such that $\nu_{[\ell]}(\theta_1, \dots, \theta_\ell, \bm \gamma)=\op{id}$  (whenever it is defined) in the $\varepsilon_{[\ell]}(\theta_1, \dots, \theta_\ell, \bm \gamma)$-neighborhoods of $0$ and $1$. Such a function can be constructed using functions $\varepsilon_{\{j_1, \dots, j_{r'}\}}$ and we leave the details to the reader.

\subsubsection{Switching}

\begin{definition}\label{defn: switching-map}
    For $\overline{\Delta}^{ij}_{I_k} \subset [0,1]^{\times \ell} \times \Omega(M, \bm q, \bm q')$ we define the switching map $sw_{I_k}^{ij} \colon \overline{\Delta}^{ij}_{I_k} \to \overline{\Delta}^{ij}_{I_k}$ by the formula
    \begin{gather}
        sw_{I_k}^{ij}(\theta_1, \dots, \theta_\ell, \bm \gamma)=(\theta_1, \dots, \theta_\ell,\gamma_1(\nu_{[\ell]}(\cdot)), \dots, \gamma_{i-1}(\nu_{[\ell]}(\cdot)), \gamma_i(\nu_{[\ell]}(\cdot))|_{[0, \theta_k]} \# \gamma_j(\nu_{[\ell]}(\cdot))|_{[\theta_k, 1]}, \\
         \gamma_{i+1}(\nu_{[\ell]}(\cdot)), \dots,\gamma_{j-1}(\nu_{[\ell]}(\cdot)), \gamma_j(\nu_{[\ell]}(\cdot))|_{[0, \theta_k]} \#\gamma_i(\nu_{[\ell]}(\cdot))|_{[\theta_k, 1]}, \gamma_{j+1}(\nu_{[\ell]}(\cdot)), \dots, \gamma_\kappa(\nu_{[\ell]}(\cdot))). \nonumber
    \end{gather}
\end{definition}

Here $\#$ stands for the standard concatenation. The maps $sw_{I_k}^{ij}$ are $C^{m-1}$-smooth when restricted to $\Delta_{I_k}\subset [0,1] \times \Omega^{m,2}(M, \bm q, \bm q')$. They are not $C^{m-1}$-smooth on all of $\overline{\Delta}_{I_k}$ for the same reason that $f_{\theta_1, \dots, \theta_\ell}$ is not smooth along diagonals $\theta_i=\theta_j$. Nevertheless, there are some additional loci inside $\overline{\Delta}_{I_k}$ where the map still has a certain level of regularity.

Explicitly, let $S_1 \sqcup \dots \sqcup S_r =[\ell]$ be a partition of $[\ell]$ and associate to it a subset $\Delta_{S_1, \dots,  S_r}([0,1]^\ell)$ of $[0,1]^\ell$ consisting of points $(\theta_1, \dots, \theta_\ell)$ such that $\theta_i=\theta_j$ if and only if $i$ and $j$ belong to the same element of the partition.
Then it is immediate from the definition of $\nu_{[\ell]}$ that each map $sw_{I_k}^{ij}$ is $C^{m-1}$ smooth when restricted to the intersection of $\overline{\Delta}_{I_k}$ with $\Delta_{S_1, \dots, S_r}([0,1]^\ell) \times \Omega^{m,2}(M, \bm q, \bm q')$ for any partition $S_1 \sqcup \dots \sqcup S_r$. The main case of interest for us is the one where there are exactly $\ell-1$ elements in the partition, i.e., exactly two coordinates are allowed to be equal.

\begin{figure}[ht]
    \centering
    \includegraphics[width=0.8\linewidth]{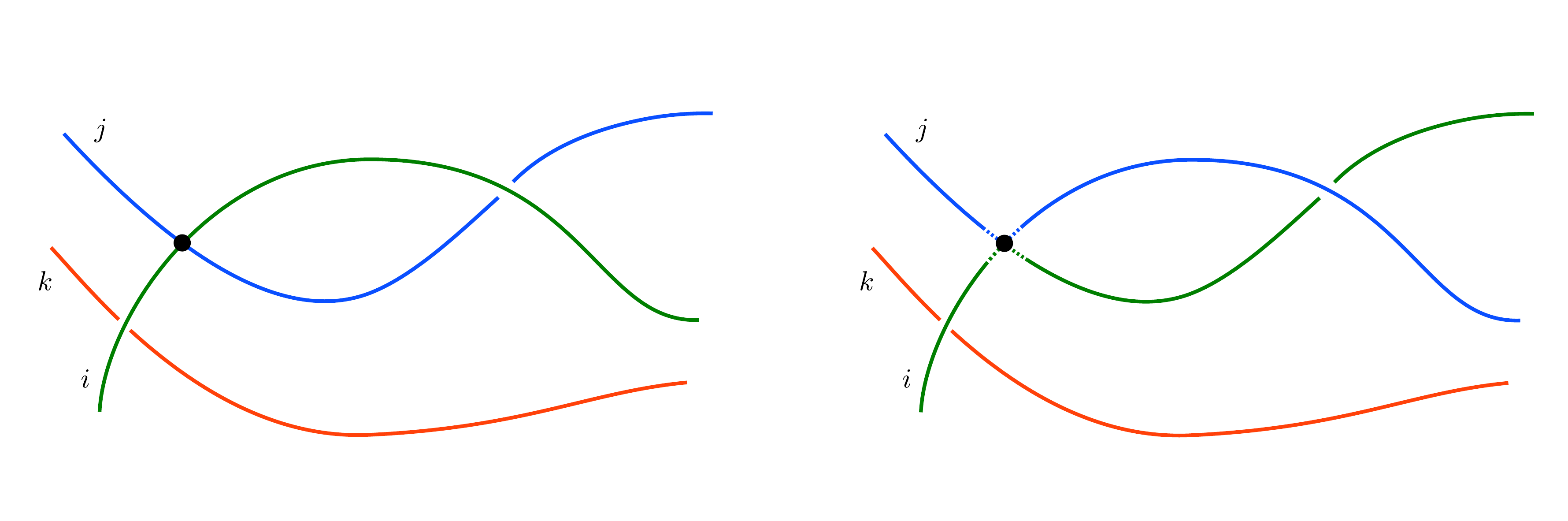}
    \caption{Here we schematically illustrate the effect of the switching map $sw^{ij}_{I_1}$ on the left tuple of paths, where we have locally drawn $3$ out of $\kappa$ paths. All paths except $\gamma_i$ and $\gamma_j$ are unaffected, and the portion of $\gamma_i$ after the intersection point at $\theta$ gets replaced with the reparametrized portion of $\gamma_j$ and vice versa. Dashed lines depict the portions of $\gamma_i$ and $\gamma_j$ affected by the smoothing reparametrization.}
    \label{fig:switching-map}
\end{figure}

\begin{remark}
    The naive switching map differs from the above switching map only in the absence of reparametrizations governed by $\nu$, so the images of both maps are identical. The following proposition justifies the introduction of the reparametrizations. 
    The reason for the complicated construction involving a cover of $\overline{\Delta}_{I_{\ell}}$ satisfying (C1) and (C2) will become clear in the proof of Theorem~\ref{thm: compactness-for-MFLS} and one should understand that on most of $\overline{\Delta}_{I_k}$ the reparametrization $\nu$ is controlled by the function $f_{\theta_k}$.
\end{remark}

\begin{proposition}\label{prop: switching-properties} $\mbox{}$  
For arbitrary indices $p_1, p_2 \in [\ell]$:
\be
\item the map $sw^{ij}_{I_{p_1}}$ is a $C^{m-1}$-immersion when restricted to $\Delta_{S_1, \dots, S_r} \times \Omega^{m,2}(M, \bm q, \bm q')$ for an arbitrary partition $S_1 \sqcup \dots \sqcup S_r=[\ell]$;
\item the following commutativity relations hold:
\begin{gather}
    sw^{ij}_{I_{p_2}}(sw^{ij}_{I_{p_1}}(\bm \theta, \bm \gamma))=sw^{ij}_{I_{p_1}}(sw^{ij}_{I_{p_2}}(\bm \theta, \bm \gamma) ), \text{ if }(\bm \theta, \bm \gamma) \in \Delta^{ij, ij}_{I_{p_1},I_{p_2}}, \label{eq : commutativity-switching}\\
    sw^{jk}_{I_{p_2}}(sw^{ij}_{I_{p_1}}(\bm \theta, \bm \gamma))=sw^{ij}_{I_{p_1}}(sw^{ik}_{I_{p_2}}(\bm \theta, \bm \gamma)), \text{ if }(\bm \theta, \bm \gamma) \in \Delta^{ij, ik}_{I_{p_1},I_{p_2}} \text{ and }\theta_{p_1}\le\theta_{p_2}, \\
    sw^{ik}_{I_{p_2}}(sw^{ij}_{I_{p_1}}(\bm \theta, \bm \gamma))=sw^{kj}_{I_{p_1}}(sw^{ik}_{I_{p_2}}(\bm \theta, \bm \gamma)), \text{ if }(\bm \theta, \bm \gamma) \in \Delta^{ij, ik}_{I_{p_1},I_{p_2}} \text{ and }\theta_{p_1} \ge \theta_{p_2},  \\
    sw_{I_{p_1}}^{ij}(sw^{kl}_{I_{p_2}}(\bm \theta, \bm \gamma))=sw_{I_{p_2}}^{ij}(sw_{I_{p_1}}^{kl}(\bm \theta, \bm \gamma)) \text{ if } (\bm \theta, \bm \gamma) \in \Delta^{ij, kl}_{I_{p_1},I_{p_2}},
\end{gather}
where $i, j, k, l$ are assumed to be distinct and $\bm \theta=(\theta_1, \dots, \theta_\ell)$.
\ee
\end{proposition}

\begin{proof}
    (1) is straightforward from the previous discussion.  
    
    (2) The first relation~\eqref{eq : commutativity-switching} follows from observing that for any $(\bm \theta, \bm \gamma) \in \Delta^{ij, ij}_{I_{p_1},I_{p_2}}$ the reparametrizations $\nu_{\bm \theta, \bm \gamma}$ involved in the definition of both $sw_{I_{p_1}}^{ij}$ and $sw_{I_{p_2}}^{ij}$ are the same. The other relations are similar and left to the reader. 
\end{proof}

We additionally point out that the closure of $\Delta_{I_1, \dots, I_\ell}^{\bm \tau}$ inside $[0,1]^{\ell} \times \Omega^{m,2}(M, \bm q, \bm q')$ is a subset of codimension at least $\ell n$ stratified by finitely many $C^{m-\ell}$-submanifolds (most of them have higher regularity). 

We consider the case where all the $\tau_s$ agree with a fixed $\{i,j\}$ in more detail. The main observation here is that, when $k$ intersection points of $\gamma_i$ and $\gamma_j$ collide at one point, all first $k$ partial derivatives of these two paths agree at the given point. This condition cuts out a $C^{m-k}$-submanifold by considering a natural evaluation map which takes values in the $k$th jet bundle of $TM$. Each partition $(k_1,\dots,k_r)$ of $\ell$ corresponds to a stratum where $k_s$ intersection points of $\gamma_i$ and $\gamma_j$ collide for all $s=1,\dots,r$.  There are $B_\ell$ such strata in this case, where $B_\ell$ is the $\ell$-th Bell number. For a general $\bm \tau$, $B_\ell$ gives an upper bound for the number of strata.

\subsection{The based multiloop complex}
\label{subsection: based multiloop complex}

In this paper, we use \emph{cohomologically graded} Morse chain complexes.

Let $\mathcal{P}$ be set of the critical points of $\mathcal{A}=\mathcal{A}_{L}$, which are given by the Euler-Lagrange equation
\begin{equation}
    \frac{d}{dt}\partial_v L(t,\gamma(t),\dot{\gamma}(t))=\partial_x L(t,\gamma(t),\dot{\gamma}(t)).
\end{equation}
To each critical point $\gamma \in \mathcal{P}$ we associate a $\Z$-graded rank $1$ free abelian group $o_\gamma$, called \emph{orientation line}, whose $\Z$-grading is equal to that of $\gamma$ (with negative Morse grading).  Then for $\kappa=1$ the graded $\Z$-module $CM_{-*}(\Omega^{1,2}(M,q))$ is the direct sum  $\bigoplus_{\gamma \in \mathcal{P}}o_\gamma$. 
The Morse differential is given by counting index $1$ {\em perturbed negative gradient trajectories}
\begin{equation}
    \frac{\partial\Gamma}{\partial s}=X(\Gamma),
\end{equation}
where we are parametrizing paths of $\Omega^{1,2}(M,q)$ by $s$ and $X\in \mathcal{G}$.

Let $\widetilde{\mathcal{M}}(\gamma,\gamma';X)$ be the moduli space of perturbed negative gradient trajectories from $\gamma$ to $\gamma'$ with respect to $X\in \mathcal{G}$ and $\mathcal{M}(\gamma,\gamma';X)=\widetilde{\mathcal{M}}(\gamma,\gamma';X)/\mathbb{R}$ be the quotient by $\mathbb{R}$-translation. To each $\Gamma \in \mathcal{M}(\gamma,\gamma';X)$ there is an associated map between orientation lines
$$\partial_\Gamma \colon o_\gamma \to o_\gamma',$$
whenever the moduli space of such trajectories is $0$-dimensional.
Then we define
\begin{gather}
    d\colon CM_{-*}(\Omega^{1,2}(M,q))\to CM_{-*}(\Omega^{1,2}(M,q))\\
\nonumber    d[\gamma]=\sum_{\substack{\bm \gamma' \in \mathcal{P}, \Gamma \in\mathcal{M}^{\op{ind}=1}(\gamma,\gamma';X)}} \partial_\Gamma([\gamma]),\nonumber
\end{gather}
where the superscript $\op{ind}=1$ refers to the virtual dimension of $\widetilde{\mathcal{M}}(\gamma,\gamma';X)$ and $[\gamma]$ is some generator of $o_\gamma$. The vector field $X$ will be suppressed when it is understood or is unimportant.

For $\kappa\geq1$ and $\boldsymbol{\gamma}\in\Omega^{1,2}_\sigma(M,\bm q)$, there is a straightforward definition of a perturbed negative gradient trajectory on $\Omega_\sigma^{1,2}(M,\bm q)$ given by
\begin{equation}\label{eqn: gradient trajectory}
    \frac{\partial\bm\Gamma}{\partial s}={\bm X}(\bm\Gamma)\coloneqq\prod_{i=1}^\kappa {X^i}(\Gamma_i),
\end{equation}
where ${\bm X}=(X^1,\dots, X^\kappa)$ and $X^i\in \mathcal{G}$ for $\Omega^{1,2}(M,q_i, q_{\sigma(i)})$. By abuse of notation we write ${\bm X}=(X^1,\dots, X^\kappa)\in \mathcal{G}^\kappa$ with the understanding that each $X^i$ depends on $q_i$ and $\sigma$; if $\bm\gamma\in \Omega^{1,2}(M,{\bm q})$ where $\sigma$ is unspecified, then we assume that $\bm X$ is with respect to the appropriate $\sigma$.
The Morse cohomology whose differential is defined by counting gradient multi-trajectories between critical points as above gives nothing but $\kappa$ copies of the Morse (co)homology $HM_{-*}(\Omega^{1,2}(M,q))$ of the based loop space of $M$ tensored with the symmetric group $S_\kappa$.
In order to include a graded version of the HOMFLY skein relation, we need to modify the definition of the Morse differential so that, given a gradient trajectory $\bm\Gamma(s)$ starting at $\bs{\gamma}$ (i.e., satisfying Equation~\eqref{eqn: gradient trajectory}), whenever there exist $s,t$ such that $\bm\Gamma(s)(t)$ crosses the big diagonal of $\mathrm{Sym}^\kappa(M)$, additional trajectories that bifurcate from the original one are created. 

Specifically, denote the set of critical points of $\mathcal{A}$ on $\Omega_\sigma^{1,2}(M,\bm q)$ by $\mathcal{P}_\sigma$ and (abusing notation) the collection of all such multiloops by $$\mathcal{P}=\bigsqcup_{\sigma \in S_\kappa} \mathcal{P}_\sigma.$$ 
To $\bm \gamma=(\gamma_1, \ldots, \gamma_\kappa) \in \mathcal{P}$ we associate the orientation line
$$o_{\bm \gamma}=o_{\gamma_1} \otimes \dots \otimes o_{\gamma_\kappa}.$$

\begin{definition} \label{defn: based multiloop}
Let $CM_{0,-*}(\Omega^{1,2}(M,\bm q))$ be the $R=\Z$-module generated by the orientation lines associated with elements of $\mathcal{P}$. Then the {\em based multiloop complex $CM_{-*}(\Omega^{1,2}(M,\bm q))$ of $M$} is a cochain complex whose underlying module is $CM_{0,-*}(\Omega^{1,2}(M,\bm q))\llbracket \hbar\rrbracket$ when $n=2$ and $CM_{0,-*}(\Omega^{1,2}(M,\bm q)) \otimes_{R}R[\hbar]$ when $n>2$.  The differential is given by Equation~\eqref{eqn: differential of multiloop complex} below.
\end{definition}

In Appendix~\ref{subappendix: mfls} we introduce spaces of perturbations of vector fields $\mathfrak{X}_-,\mathfrak{X}_+, \mathfrak{X}_0$ for any positive integer $m$. Let us fix $m$ sufficiently large ($m\geq 3$ suffices).  The elements of these spaces are certain maps to the space of vector fields on $\Omega^{1,2}(M, \bm q, \bm q')$ from domains $(-\infty,0], [0, +\infty)$ and $[0, +\infty)^2$, respectively. 

Let $\ell\geq 0$ be an integer, $\bm \tau=(\tau_1, \dots, \tau_\ell)$ an $\ell$-tuple of pairs $\tau_1=\{i_1, j_1\}, \dots, \tau_\ell=\{i_\ell, j_\ell \}$ of elements of $[\kappa]$, and 
$$ \bm Y= (Y_0, Y_1, \dots, Y_{\ell-1}, Y_{\ell}) \in \mathfrak{X}_- \times (\mathfrak{X}_0)^{\ell-1} \times \mathfrak{X}_+$$ 
an $(\ell+1)$-tuple of perturbations. We will also need an ordering/permutation $\bm s=(s_1, \dots, s_\ell)$ of elements of $[\ell]$. (In the following definition, the first switch that occurs is labeled $s_1$, the second switch is labeled $s_2$, etc.) 

\begin{definition} \label{def-diff}
Given two critical multiloops $\bm \gamma\in \mathcal{P}_\sigma$, $\bm \gamma' \in \mathcal{P}_{\sigma'}$, a {\em multiloop flow line with switchings  (abbreviated MFLS) from $\bm\gamma$ to $\bm\gamma'$ with respect to data $\ell, \bm Y, \bm\tau$, $\bm s$},\footnote{More precisely, it's a ``multiloop flow line with possible switchings'', but that would be too cumbersome.}  
is a tuple
\begin{equation*}
        \bm \Gamma=\left((\theta_{s_1}, \Gamma_0), (\theta_{s_2}, \Gamma_1, l_1), \dots, (\theta_{s_\ell}, \Gamma_{\ell-1}, l_{\ell-1}), \Gamma_\ell \right), 
\end{equation*}
where $\bm\theta=(\theta_1, \dots, \theta_{\ell}) \in \op{Conf}_\ell([0,1])$, $l_1,\dots,l_{\ell-1}\in [0; +\infty)$ and the following hold:
\begin{enumerate}
        \item  The maps $$\left\{\begin{array}{l} \Gamma_0 \colon (-\infty, 0]\to \Omega^{1,2}(M,\bm q);\\
        \Gamma_i \colon [0,l_i]\to \Omega^{1,2}(M,\bm q) \quad \text{ for }\quad  i=1,\dots,\ell-1;\\
        \Gamma_\ell \colon [0, +\infty) \to \Omega^{1,2}(M,\bm q)
        \end{array}\right.$$
        are continuously differentiable.
        \item \label{item: gradient equation} $\Gamma_i'(s)={(\bm X+Y_i(l_i, s))}_{\Gamma_i(s)}$ for $\, i=0,\dots,\ell$ (by definition, for $i=0, \ell$, $Y_i(l_i, s)$ is simply $Y_i(s)$).
        \item $\Gamma_0(-\infty)=\boldsymbol{\gamma}$, $\Gamma_\ell(+\infty)=\boldsymbol{\gamma}'$.
        \item \label{item: switch-condition} For $k=1, \dots, \ell$, $(\bm \theta, \Gamma_{k-1}(l_{k-1})) \in \Delta^{\tau_k}_{I_{s_k}}$ (here we are taking $l_0=0$) and
        $$sw^{i_kj_k}_{I_{s_k}}(\bm \theta, \Gamma_k(l_k))=(\bm \theta, \Gamma_{k+1}(0)).$$
\end{enumerate}
\end{definition}

Let $\mathcal{M}_\ell(\bm \gamma,\bm \gamma';\bm Y, \bm \tau, \bm s)$ be the moduli space of all MFLS $\bm\gamma$ from $\bm\gamma$ to $\bm\gamma'$ with respect to $\ell, \bm Y, \bm\tau, \bm s$.

\begin{notation} \label{notation: switching} 
We often write $\bm \Gamma=(\Gamma_0,\dots, \Gamma_\ell)$, suppressing the $\theta_i$ and $l_i$ from the notation.  We also often omit the subscript $\ell$ from $\mathcal{M}_\ell(\bm \gamma,\bm \gamma';\bm Y, \bm \tau, \bm s)$ since $|\bm\tau|=\ell$.
\end{notation}
\begin{figure}[h]
    \centering
    \includegraphics[width=3.5cm]{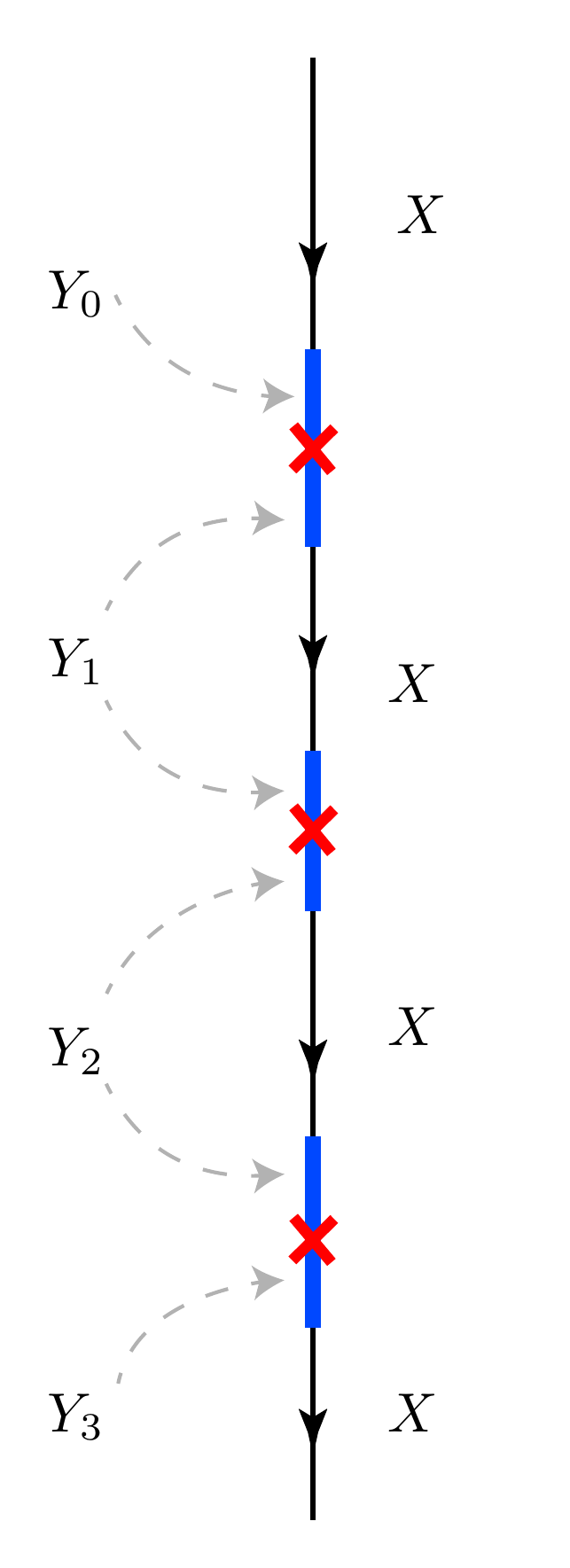}
    \caption{Here we depict the perturbation scheme for an MFLS with $3$ switches, marked by red crosses. Away from the location of switching markers the trajectories are the flow lines of the fixed pseudogradient vector field $X$. Near the switches they are flow lines of the perturbed vector fields $X+Y_i$. Note that $Y_i$ for $i=1, \dots, \ell-1$ is nonzero near the $i$th and $(i+1)$st switching markers as depicted.}
    \label{fig: MFLS-perturbation}
\end{figure}
We write $\op{ind}(\bm\gamma,\bm\gamma',\ell)$ for the virtual dimension of $\mathcal{M}(\boldsymbol{\gamma},\boldsymbol{\gamma}'; \bm Y, \bm \tau, \bm s)$ (or $\op{ind}(\bm\Gamma)$ if $\bm\Gamma\in \mathcal{M}(\boldsymbol{\gamma},\boldsymbol{\gamma}'; \bm Y, \bm \tau, \bm s)$), which is equal to $-(n-2)\ell+|\bm\gamma'|-|\bm\gamma|$ by Theorem~\ref{theorem: transversality-for-MFLS}. The grading of $\boldsymbol{\gamma}=(\gamma_1,\dots,\gamma_\kappa)\in \mathcal{P}$ is {\em cohomological}, i.e., $|\bm \gamma|=|\gamma_1|+\dots+|\gamma_\kappa|$, where $|\gamma_i|=-\op{ind}(\gamma_i)$ and $\op{ind}$ is the Morse index. The moduli spaces $\widetilde{\mathcal{M}}(\bm \gamma,\bm \gamma';\bm Y, \bm \tau, \bm s)$ are oriented as in \cite{KY}.

\begin{lemma}\label{lemma-d-regular}
For generic $\bm Y \in \mathfrak{X}_- \times (\mathfrak{X}_0)^{\ell-1} \times \mathfrak{X}_+$, given $\bm \gamma, \bm \gamma' \in \mathcal{P}$, $\ell$, $\bm\tau$ and $\bm s$:
\be
\item $\mathcal{M}(\bm \gamma,\bm \gamma'; \bm Y, \bm \tau, \bm s)$ is a smooth manifold of dimension
$$\op{dim}\mathcal{M}(\bm \gamma,\bm \gamma'; \bm Y, \bm \tau, \bm s)=\op{ind}(\bm \gamma,\bm \gamma',\ell)=-(n-2)\ell+|\bm\gamma'|-|\bm\gamma|;$$
in particular, if $\op{dim}\mathcal{M}(\bm \gamma,\bm \gamma'; \bm Y, \bm \tau, \bm s)$ is negative, then $\mathcal{M}(\bm \gamma,\bm \gamma'; \bm Y, \bm \tau, \bm s)=\varnothing$.
\item If $\op{dim}\mathcal{M}(\bm \gamma,\bm \gamma'; \bm Y, \bm \tau, \bm s)=0$, then $\mathcal{M}(\bm \gamma,\bm \gamma'; \bm Y, \bm \tau, \bm s)$ is finite.
\item There exists a \emph{consistent choice of perturbation data} $\{\bm Y^{\ell}_{\bm \tau,\bm s}\}_{\ell \ge 0, |\bm \tau|=\ell,\bm s}$, where 
$$\bm Y^{\ell}_{\tau, \bm s}=(Y^{\ell}_{\tau, \bm s,0}, \dots, Y^{\ell}_{\tau, \bm s,\ell}) \in \mathfrak{X}_- \times (\mathfrak{X}_0)^{\ell-1} \times \mathfrak{X}_+,$$ 
such that 
$$\bigcup_{\bm s \in S_\ell, \bm \tau \colon |\bm \tau| =\ell} \mathcal{M}^{\op{ind}=1}(\bm \gamma, \bm \gamma'; \bm Y^{\ell}_{\bm \tau, \bm s}, \bm \tau, \bm s)$$
admits a compactification with boundary strata of the form
\begin{equation}\label{eqn: boundary}
\mathcal{M}^{\op{ind}=0}(\bm\gamma,\bm\gamma''; \bm Y^{\ell'}_{\bm \tau', \bm s'}, \bm \tau', \bm s') \times \mathcal{M}^{\op{ind}=0}(\bm\gamma'',\bm\gamma'; \bm Y^{\ell-\ell'}_{\bm \tau'', \bm s''}, \bm \tau'', \bm s''),
\end{equation}
where $\bm \tau= (\bm \tau', \bm \tau'')$, $|\bm \tau'|=\ell' \ge 0$, $|\bm \tau''|=\ell-\ell' \ge 0$, and $\bm s'$ is the ordering of the elements of $[\ell']$ by considering the order on the first $\ell'$ elements of $\bm s$, while $\bm s''$ is a similarly defined ordering on $[\ell-\ell']$.
\item Fixing a choice of orientations on unstable manifolds of $\bm \gamma$ and $\bm \gamma'$ provides a natural orientation of $\mathcal{M}(\bm \gamma, \bm \gamma'; \bm Y, \bm \tau, \bm s)$,  which depends on the orientation of $M$.
\ee
\end{lemma}

\begin{proof}
    (1) follows from regularity (Theorem~\ref{theorem: transversality-for-MFLS}).
    
    (2) follows from (1) and compactness (Appendix~\ref{section: appendix-compactness}).

    (3) is a consequence of Appendix~\ref{section: appendix-compactness}, regularity, and gluing (which is not proven here). We note that the statement here is slightly inaccurate and we refer to Theorem~\ref{thm: compactness-for-MFLS} for a precise statement.

    (4) is analogous to \cite[Proposition 3.4]{KY} and will be omitted.
\end{proof}

\begin{remark}
    We may alternatively use $CM_{-*}(\Omega^{1,2}(M,\bm q,\bm q'))$ instead of $CM_{-*}(\Omega^{1,2}(M,\bm q))$,  where ${\bm q}'$ is a perturbation  of ${\bm q}$.  The proof of Lemma~\ref{lemma-d-regular} remains the same.
\end{remark}

\begin{figure}[h]
    \centering
    \includegraphics[width=12cm]{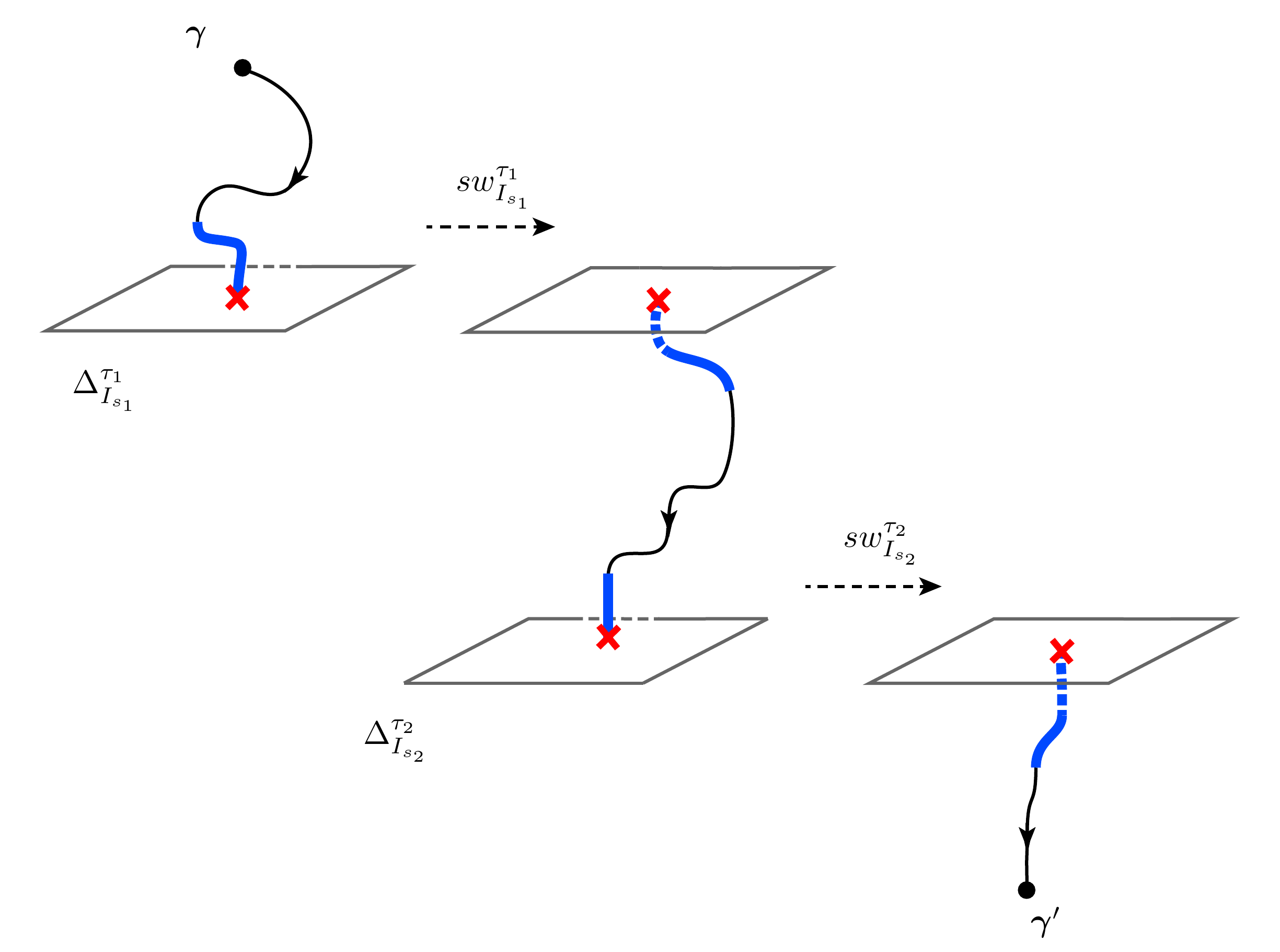}
    \caption{A schematic depiction of an MFLS with $2$ switchings. As before the red crosses denote switches and the thick blue segments denote the perturbed parts of the trajectories.}
    \label{fig: MFLS}
\end{figure}

By Lemma~\ref{lemma-d-regular}(4), to any MFLS $\bm \Gamma \in \mathcal{M}(\bm \gamma,\bm \gamma'; \bm Y, \bm \tau, \bm s)$ with $\op{ind}(\bm\Gamma)=0$ there is an associated map of orientation lines
$$\partial_{\bm \Gamma} \colon o_{\bm \gamma} \to o_{\bm \gamma'}[(n-2)\ell].$$
For a generator $[\bm \gamma'] \in o_{\bm \gamma}$ we will denote by $[\bm \gamma']\hbar^\ell$ the corresponding generator of $o_{\bm \gamma'}[(n-2)\ell]$, where $\hbar$ is our Planck constant with grading $|\hbar|=2-n$. We use the convention 
\begin{gather}
    o[-k] \cong  o\otimes \zeta^{k}, \, k >0; \\
    o[k] \cong \zeta^{-k} \otimes o, \, k>0.
\end{gather}
where $o$ is any orientation line and $\zeta$ is the trivial orientation line of grading $1$.

For a universal (i.e., consistent) choice of perturbation data $\{\bm Y^{\ell}_{\bm \tau,\bm s}\}_{\ell \ge 0, |\bm \tau|=\ell,\bm s}$ as in Lemma~\ref{lemma-d-regular}, the based multiloop differential is defined as follows:
\begin{align}
    \label{eqn: differential of multiloop complex}
    \mu^1_M([\bm \gamma])=\sum_{\substack{\bm \gamma',\,\ell\geq0, \bm \tau, \bm s \\ \bm \Gamma \in \mathcal{M}^{\op{ind}(\bm \gamma,\bm \gamma',\ell)=0}(\bm \gamma,\bm \gamma';\bm Y^\ell_{\bm \tau,\bm s}, \bm \tau, \bm s)}}   (-1)^{|\bm \gamma|}\frac{\partial_{\bm \Gamma}([\bm \gamma])}{\ell!},
\end{align}
where $[\bm \gamma]$ is a generator of $o_{\bm \gamma}$.

\begin{remark}
    Note that the count above is integer despite dividing by $\ell!$, because we can also make sure that the perturbation data is symmetric; see Appendix~\ref{section: appendix-compactness}.
\end{remark}

\begin{lemma}\label{lemma-d2}
   For a choice of universal perturbation data $\{\bm Y^{\ell}_{\bm \tau,\bm s}\}_{\ell \ge 0, |\bm \tau|=\ell,\bm s}$ as in Lemma~\ref{lemma-d-regular} the differential $\mu^1_M$ is well-defined and $\mu^1_M \circ \mu^1_M =0$.
\end{lemma}

\begin{proof}
    This is an immediate consequence of Lemma \ref{lemma-d-regular}. 
\end{proof}

\begin{example} \label{ex: S2}
    Let $M=S^2$ be the unit $2$-sphere with the round metric induced from $\mathbb{R}^3$.
    Pick basepoints $\bm q=(q_1,q_2)$ where $\{q_1,q_2\}\subset S^2$ and $q_1\neq q_2$. 
    Let $x_1$ be the tuple of geodesics of total degree $-1$ which consists of an equator $\gamma_1$ of degree $-1$ based at $q_1$ and the constant loop $\gamma_2$ at $q_2$. Our convention is such that the orientation of the unstable manifold of a geodesic path of degree $-1$ coincides with the orientation of the trajectory from such a path to $1$ contained in the \emph{right} hemisphere with respect to the path. With respect to this convention $[x_1] \in CM_{-*}(S^2, \bm q)$ is the corresponding generator and by abuse of notation we also denote it by $x_1=[x_1]$.
    
    There are two gradient trajectories from $\gamma_1$ to the constant loop at $q_1$. One has no intersection with $q_2$ while the other one crosses the graph of the constant loop of $q_2$ transversely. Note that explicit calculation in the closed string case is given in \cite[Appendix B]{ACF2016} and the open string case is analogous.
    Therefore, there is an MFLS with $1$ switching which occurs when the trajectory above reaches $q_2$.
    This is illustrated in Figure \ref{fig-dx1}.
    Hence,
    \begin{equation*}
        \mu^1_M(x_1)=\hbar\,T_1+o(\hbar),
    \end{equation*}
    where $T_1=\{\omega_1,\omega_2\}$, $\omega_1$ is the shortest geodesic from $q_1$ to $q_2$, and $\omega_2$ is the shortest geodesic from $q_2$ to $q_1$; $o(\hbar)$ denotes the higher-order terms in $\hbar$. Here again we abuse notation and denote the generator $[T_1]$ simply by $T_1$. Note that the sign in front of the first term is implied by our orientation convention and the additional sign twist coming from~\eqref{eqn: differential of multiloop complex}.
\end{example}

\begin{figure}[ht]
    \centering
    \includegraphics[width=6cm]{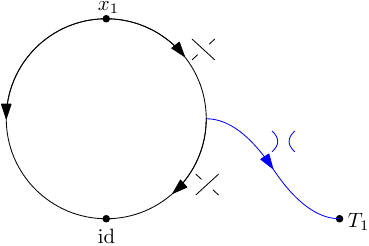}
    \caption{}
    \label{fig-dx1}
\end{figure}

\subsection{The $A_\infty$-structure} \label{subsection: the A infty structure}

In this subsection we describe the $A_\infty$-structure on the based multiloop complex $CM_{-*}(\Omega^{1,2}(M,\bm q))$. The $A_\infty$-operations
are given by a signed count of pseudogradient flow trees of the action functional with $\ell$ switchings. 
See Figure \ref{fig-a_infty} for an illustration, where the edges are flow lines with switchings and the vertices represent concatenations.

\begin{figure}[ht]
    \centering
    \includegraphics[width=9cm]{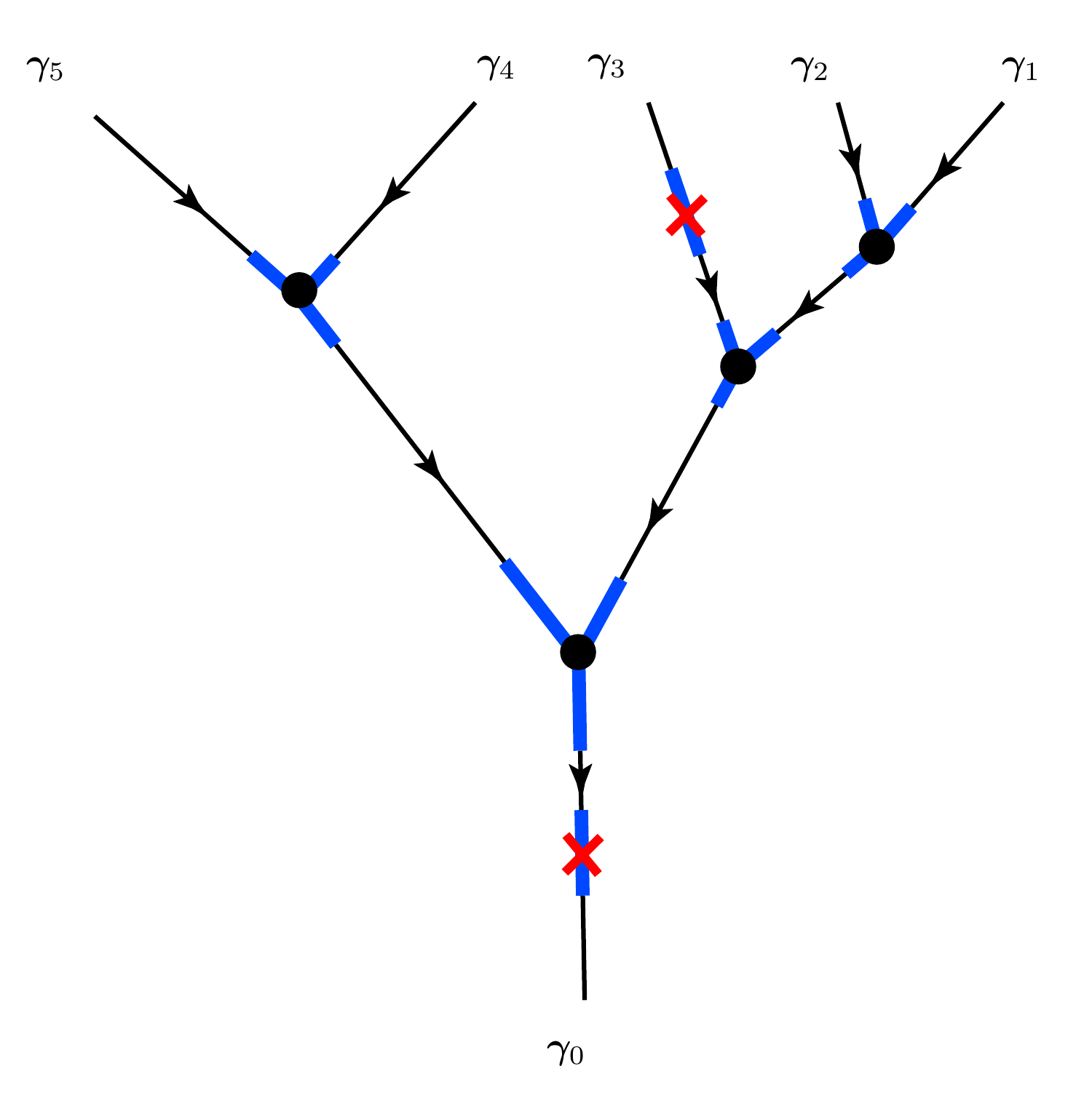}
    \caption{Schematic depiction of an MFTS with $2$ switchings contributing to $[\bm \gamma_0]\hbar^2$ in $\mu^5_M([\bm \gamma_5],[\bm \gamma_4],[\bm \gamma_3], [\bm \gamma_2], [\bm \gamma_1])$. As before we depict in blue the areas where additional perturbations are turned on.}
    \label{fig-a_infty}
\end{figure}

We now give a more precise definition of the $A_\infty$-operations.  We first make some preliminary definitions. 

Let $h_\varepsilon:[0,1]\to[0,1]$, $\varepsilon \in (0, \frac{1}{4}]$, be a family of smooth step functions such that:
\begin{gather}
    h_\varepsilon(0)=0, \quad h_\varepsilon(1)=1, \quad  h_\varepsilon(x)=x \text{~~ for ~~ $\varepsilon\leq x \leq 1-\varepsilon$},\\
    h'_\varepsilon(x)>0 \text{~~ for ~~$x\in (0,1)$}, \quad h_\varepsilon^{(k)}(0)= h_\varepsilon^{(k)}(1)=0 ~~ \forall k>0.
\end{gather}

The family $\{h_\varepsilon\}$ is used to smooth out points of concatenation in the following definition:

\begin{definition}[Weighted concatenation]
    Let $\bm a=(a_1,\dots, a_{d})$ be a vector of positive real numbers and $\bm q_0, \dots, \bm q_d$ be $\kappa$-tuples of points in $M$. Writing $a=a_1+\dots + a_{d}$, the {\em weighted concatenation map} is given by:
\begin{gather*}
    c^{\bm a}: [0,1]^{\ell} \times \Omega^{m,2}(M,\bm q_{d-1}, \bm q_d)\times \dots\times\Omega^{m,2}(M,\bm q_0, \bm q_1)\to[0,1]^{\ell} \times \Omega^{m,2}(M,\bm q_0, \bm q_d),\\
    (\bm \theta, \bm\gamma_d, \dots, \bm\gamma_1)\mapsto \left(\bm \theta,
        \left\{
            \begin{array}{ll}
                \text{$\bm\gamma_1(h_{\varepsilon}(\tfrac{a}{a_1} t))$, \quad $t\in[0,\tfrac{a_1}{a}]$,}\\
				\text{$\bm\gamma_2(h_{\varepsilon}(\tfrac{a}{a_2} t - \tfrac{a_1}{a_2}))$, \quad $t\in [\tfrac{a_1}{a}, \tfrac{a_1+a_2}{a}]$,}\\
				\qquad\vdots \\
                \text{$\bm \gamma_d(h_{\varepsilon}(\tfrac{a}{a_d}t - \tfrac{a_1+\dots+a_{d-1}}{a_d}))$,\quad $t\in[\tfrac{a_1+\dots+a_{d-1}}{a},\tfrac{a_1+\dots+a_d}{a}]$}
            \end{array}
        \right.\right),
\end{gather*}
where $\varepsilon>0$ is a small constant.  
\end{definition}

Observe that precomposing with reparametrizations of type $h_\varepsilon$ guarantees that the image is indeed in $\Omega^{m,2}(M, \bm q_0, \bm q_d)$.

\begin{lemma}\label{lemma: concatenation-commutativity}
Let $(\bm \theta, \bm \gamma_p) \in \overline{\Delta}_{I_k}^{ij}$, where $1 \le p \le d$, $i\not=j \in [\kappa]$, and $k \in [\ell]$. Then 
    \begin{gather}
        c^{\bm a}(\bm \theta,\bm \gamma_d, \dots, \bm \gamma_{p+1}, sw^{ij}_{I_k}(\bm \theta, \bm \gamma_p), \bm \gamma_{p-1}, \dots, \bm \gamma_1)= 
         sw^{\tilde{i}\tilde{j}}_{I_{k+\ell_d+\dots+\ell_{p+1}}}(c^{\bm a}(\bm \theta, \bm \gamma_d, \dots, \bm \gamma_1) ),
    \end{gather}
where $\tilde{i}$ and $\tilde{j}$ are obtained by applying $\sigma^{-1}(\bm \gamma_1)\circ \dots \circ \sigma^{-1}(\bm \gamma_{p-1})$ to $i$ and $j$, respectively, and $\sigma(\bm \gamma)$ is the permutation type of the multipath $\bm \gamma$.
\end{lemma}

\begin{proof}
    This follows from the definitions of $h_{\epsilon}$ and $\nu_\ell$. 
\end{proof}

\begin{definition}\label{defn: metric ribbon tree}
   A \emph{ribbon tree with $d$ inputs and $1$ output} is a directed planar tree $T=(V,E)$ such that:
\begin{enumerate}
    \item $V=V_{\op{int}}\sqcup V_{\op{ext}}$ is the set of vertices and $E=E_{\op{int}} \sqcup E_{\op{ext}}$ is the set of (directed) edges, where $V_{\op{int}}$ (resp.\ $V_{\op{ext}}$) is the set of interior (resp.\ exterior) vertices; $V_{\op{ext}}=\{v_0,v_1,\dots,v_d\}$ and $E_{\op{ext}}=\{e_0, e_1, \dots, e_d\}$; $v_1,\dots,v_d$ are the \emph{incoming vertices} which have $1$ outgoing edge each; $v_0$ is the \emph{outgoing vertex} which has $1$ incoming edge; $e_1, \dots, e_d$ are the outgoing edges from $v_1,\dots, v_d$; and $e_0$ is the incoming edge to $v_0$;
    \item every vertex $v \in V_{\op{int}}$ has exactly $1$ outgoing edge $e^v_0$; the incoming edges $e^v_1, \dots, e^v_{|v|-1}$ are cyclically ordered  (here $|v|$ denotes the valence of $v$) around the vertex in a manner consistent with the planar embedding of $T$ (counterclockwise around $v)$; and the exterior vertices $v_0,\dots,v_d$ are cyclically ordered in a similar manner.
\end{enumerate} 
\end{definition}

Given $e\in E$, let us denote its {\em source} and {\em target vertices} by $s(e), t(e)\in V$.

Let $\mathcal{T}_d$ denote the set of isomorphism classes of ribbon trees with $d$ inputs and $1$ output; we will abuse notation and not distinguish between a ribbon tree and its isomorphism class. For $e \in E_{\op{int}}(T)$ we denote by $T/e$ the element of $\mathcal{T}_d$ obtained from $T$ by contracting the edge $e$. 

We denote by $\mathfrak{T}_d$ the space of Stasheff (= metric) ribbon trees with $d$ inputs and $1$ output and refer the reader to \cite[Section 4]{abouzaid2009tropical} for a more in-depth discussion of the topology of these spaces.  Here a metric ribbon tree will have infinite length along all the exterior edges and finite length along each interior edge.

\begin{definition} \label{defn: switching data tau}
    A {\em switching data $\bm\tau$} on a tree $T\in \mathcal{T}_d$ assigns: 
    \be
    \item to each $e \in E(T)$ of an $\ell_e$-tuple $\bm \tau_e$ of pairs of distinct elements of $[\kappa]$; and 
    \item to each $v \in V_{\op{int}}(T)$ and an incoming edge $e^v_i$, a tuple $\bm \tau_{v,i}$ of pairs of distinct elements of $[\kappa]$, such that $\bm\tau_v= (\bm\tau_{v,1},\dots,\bm\tau_{v, |v|-1})$ has cardinality $\ell_v$.
    \ee
    The integers $\ell_e$ and $\ell_v$ are the {\em switching numbers} on $e$ and $v$.
\end{definition}

We denote by $\mathcal{T}_{d}^{\ell, \ell'}$ the set of $(T,\bm\tau)$, where $T\in \mathcal{T}_d$ and $\bm\tau$ is a switching data on $T$ with $\ell$ edge switches and $\ell'$ vertex switches, i.e., 
\begin{equation}
    \sum_{e \in E(T)}\ell_e=\ell, \sum_{v\in V_{\op{int}}(T)}\ell_v=\ell'.
\end{equation}
Given $(T,\bm\tau) \in \mathcal{T}_d^{\ell, \ell'}$ and $e \in E_{\op{int}}(T)$, we obtain $(T/e,\bm\tau') \in \mathcal{T}_d^{\ell-\ell_e, \ell'+\ell_e}$, where $T/e$ is the contracted ribbon tree and the switching data $\bm\tau'$ is defined as follows: Suppose $e$ is the edge $e^{t(e)}_i$, i.e., the $i$th incoming edge of $t(e)$.  Let $v$ be the vertex of $T/e$ obtained by contracting $e$.  Then 
$$\bm\tau'_v = (\bm\tau_{t(e),1},\dots,\bm\tau_{t(e), i-1},(\bm \tau_{s(e)}, \bm \tau_{e^{t(e)}_i},\bm\tau_{t(e), i}), \bm\tau_{t(e),i+1},\dots,\bm\tau_{t(e), |t(e)|-1}).$$

We are mainly interested in $\mathcal{T}_{d}^{\ell}\coloneqq \mathcal{T}_{d}^{\ell, 0}$. Let $\accentset{\circ}{\mathcal{T}}_d^\ell$ be the subset of  $\mathcal{T}_d^\ell$ consisting of $(T,\bm\tau)$ such that $|E_{\op{int}}(T)|=d-2$, i.e., $T$ is a \emph{binary ribbon tree}. 
To each $(T,\bm\tau) \in \mathcal{T}^\ell_d$ we assign
\begin{equation}\label{eq: tree-perturbations}
    \mathfrak{X}(T,\bm \tau)=\mathfrak{X}_-^{d} \times \mathfrak{X}_0^{|E_{\operatorname{int}}(T)|+\ell}\times {\mathfrak{X}_+}.
\end{equation}
More specifically, to $e\in \{e_1,\dots,e_d\} \subset  E_{\op{ext}}(T)$ with $\ell_e$ edge switches we associate $(Y^e_{0}, \dots, Y^e_{\ell_{e}}) \in \mathfrak{X}_- \times  \mathfrak{X}_0^{\ell_e}$; to $e \in E_{\op{int}}(T)$ we associate $(Y^e_{0}, \dots, Y^e_{\ell_e}) \in \mathfrak{X}_0^{\ell_e+1}$; and to $e=e_0\in E_{\op{ext}}(T)$ we associate $(Y^e_{0}, \dots, Y^e_{\ell_e}) \in \mathfrak{X}_0^{\ell_e} \times \mathfrak{X}_+ $. 

Given $(T,\bm\tau) \in \mathcal{T}^\ell_d$, a \emph{total ordering} $\bm s=(s_1, \dots, s_{\ell})$ assigns to each edge or vertex switch a unique number from $[\ell]$.  Let $S(T,\bm\tau)$ be the set of all total orderings on $(T,\bm\tau)$.   Given $\bm s\in S(T,\bm\tau)$, let $\bm s^e=(s^e_0, \dots, s^e_{\ell_e-1})$ be the restriction of $\bm s$ to the edge $e$.

Let $\bm q_0, \dots, \bm q_d$ be $\kappa$-tuples of points in $M$, let $\vv{\bm \gamma}=(\bm \gamma_1, \dots, \bm \gamma_d)$ such that $\bm \gamma_i \in \Omega^{m,2}(M, \bm q^{i-1}, \bm q^{i})$ for $i=1,\dots,d$, and let $\bm \gamma_0 \in \Omega^{m,2}(M, \bm q^0, \bm q^d)$; all the $\bm\gamma_i$ are a critical points of the corresponding action functionals. 

\begin{definition}\label{def: MFTS}
Given $(T,\bm\tau) \in \mathcal{T}^\ell_d$, $\vv{\bm \gamma}$ and $\bm\gamma_0$ as above, $\bm Y \in \mathfrak{X}(T,\bm\tau)$, and $\bm s\in S(T,\bm\tau)$, an element
$\bm \Gamma_T \in \mathcal{M}_T(\vv{\bm \gamma}, \bm \gamma_0; \bm Y,\bm \tau, \bm s)$ is a tuple
$$\bm \Gamma_T=\left(  \bm \Gamma^{e_1}, \dots, \bm \Gamma^{e_d}, (\bm \Gamma^e)_{e \in E_{\op{int}}(T)}, \bm \Gamma^{e_0} \right),$$
where 
    \begin{enumerate}
        \item $\bm \Gamma^{e_i}=((\theta_{s^{e_i}_0}^{e_i}, \Gamma_0^{e_i}),(\theta^{e_i}_{s^{e_i}_1}, \Gamma_1^{e_i}, l^{e_i}_1) \dots, (\theta_{s^{e_i}_{\ell_{e_i}-1}}^{e_i}, \Gamma^{e_i}_{\ell_{e_i}-1}, l^{e_i}_{\ell_{e_i}-1}), (\Gamma^{e_i}_{\ell_{e_i}}, l^{e_i}_{\ell_{e_i}}))$ (here all $\theta^*_*\in[0,1]$ and $l^*_*\in [0,\infty)$) with 
        \begin{gather*}
            \Gamma_0^{e_i} \colon (-\infty, 0] \to \Omega^{1,2}(M, \bm q^{i-1}, \bm q^{i}), \\
            \Gamma_j^{e_i} \colon [0, l^{e_i}_{j}] \to \Omega^{1,2}(M, \bm q^{i-1}, \bm q^{i}),
        \end{gather*}
        for each $i=1, \dots, d$ and $j=1, \dots, \ell_{e_i}$; we impose the switching conditions of Definition~\ref{def-diff}\eqref{item: switch-condition} dictated by $\bm \tau_{e_i}$;
        \item $\bm \Gamma^{e_0}=((\theta_{s^{e_0}_0}^{e_0}, \Gamma_0^{e_0}, l_0^{e_0}),(\theta^{e_0}_{s^{e_0}_1}, \Gamma_1^{e_0}, l^{e_0}_1) \dots, (\theta_{s^{e_0}_{\ell_{e_0}-1}}^{e_0}, \Gamma^{e_0}_{\ell_{e_0}-1}, l^{e_0}_{\ell_{e_0}-1}), \Gamma^{e_0}_{\ell_{e_0}})$, with 
        \begin{gather*}
            \Gamma_j^{e_0} \colon [0, l^{e_0}_{j}] \to \Omega^{1,2}(M, \bm q^{0}, \bm q^{d}),\\
            \Gamma_{\ell_{e_0}}^{e_0} \colon [0, +\infty) \to \Omega^{1,2}(M, \bm q^{0}, \bm q^{d}),
        \end{gather*}
        for  $j=0, \dots, \ell_{e_i}-1$; we impose the switching conditions dictated by $\bm \tau_{e_0}$;
        \item $\bm \Gamma^e=((\theta_{s^e_0}^{e}, \Gamma_0^{e}, l_0^{e}),(\theta^{e}_{s^e_1}, \Gamma_1^{e}, l^{e}_1) \dots, (\theta_{s^e_{\ell_{e}-1}}^{e}, \Gamma^{e}_{\ell_{e}-1}, l^{e}_{\ell_{e}-1}), (\Gamma^{e}_{\ell_{e}}, l^{e}_{\ell_{e}}))$ for $e \in E_{\op{int}}(T)$ with
        $$\Gamma_j^{e} \colon [0, l^{e}_{j}] \to \Omega^{1,2}(M, \bm q^{R(s(e))-1}, \bm q^{L(s(e))}),$$
        for each $j=0, \dots, \ell_e$, where $v_{R(v)},\dots, v_{L(v)}$ are the exterior vertices with a directed path to $v\in V(T)$ (and the subscripts are in increasing order); we impose the switching conditions dictated by $\bm \tau_e$; 
    \item $(\Gamma^e_j)'(s)=(\bm X+Y^{e}(l^e_j,s))_{\Gamma^e_j(s)}$ for each $e \in E(T)$ and $j=0, \dots, \ell_e$ (we use the convention of Definition~\ref{def-diff}\eqref{item: gradient equation});
    \item $\Gamma^{e_i}_0(-\infty)=\bm \gamma_i$ for $i=1, \dots, d$ and $\Gamma^{e_0}_{\ell_{e_0}}(+\infty)=\bm \gamma_0$;
    \item (Weights) We assign a weight $w_e$ to each $e\in E(T)$: set $w_{e_i}=1$ for $i=1, \dots, d$, and for each $v \in V_{\op{int}}(T)$ with incoming edges $e^v_1, \dots, e^v_{|v|-1}$ and outgoing edge $e^v_0$ we inductively define 
    $$w_{e_0^v}=\op{min}(l_{e_0^v},1)+(1-\op{min}(l_{e_0^v},1))(w_{e^v_1}+\dots+w_{e^v_{|v|-1}}),$$
    where $l_e$ is the total length of $e$, i.e., the sum of all the $l^e_i$;
    \item (Concatenation at interior vertices)  for each $v \in V_{\op{int}}(T)$ we have
    \begin{equation}\label{eqn: concatenation at interior vertices}
        (\bm \theta, \Gamma^{e_0^v}_0(0)) = c^{{\bm w}_v} \left(\bm \theta, \Gamma^{e^v_{|v|-1}}_{\ell_{e^v_{|v|-1}}}(l^{e^v_{|v|-1}}_{\ell_{e_{|v|-1}^v}}),\dots,\Gamma^{e^v_1}_{\ell_{e^v_1}}(l^{e^v_1}_{\ell_{e_1^v}}) \right),
    \end{equation}
    where $\bm w_v= (w_{e^v_1},\dots, w_{e^v_{|v|-1}})$ and $c^{{\bm w}_v}$ is the weighted concatenation map with respect to $\bm w_v$.
\end{enumerate}
\end{definition}

We refer to an element $\bm\Gamma_T\in \mathcal{M}_T(\vv{\bm \gamma}, \bm \gamma_0;\bm Y, \bm \tau, \bm s)$ as a {\em multiloop flow tree with switchings (abbreviated MFTS) from $\vv{\bm \gamma}$ to $\bm \gamma_0$.}

\begin{remark}\label{remark: how-to-define-vertex-deep-strata}
    Given $(T,\bm\tau) \in \mathcal{T}^{\ell, \ell'}_d$, we may also define $\bm \Gamma_T\in \mathcal{M}_T(\vv{\bm \gamma}, \bm \gamma_0; \bm Y, \bm \tau, \bm s)$ by applying multiple switchings to $\Gamma^{e_i^v}_{\ell_i^v}(l^{e^v_i}_{\ell_{e_i^v}})$ prescribed by $\bm \tau_{v,i}$, prior to applying the weighted concatenation in $(7)$.
\end{remark}

\begin{lemma} \label{lemma-ainfty-regular} $\mbox{}$
\be
    \item[(a)] For generic $\bm Y \in \mathfrak{X}(T,\bm\tau)$, the space $\mathcal{M}_T(\vv{\bm \gamma}, \bm \gamma_0;\bm Y, \bm \tau, \bm s)$ is a $C^{m-1}$-manifold of dimension
    \begin{equation*}
       {\op{ind}(\vv{\bm \gamma},\bm \gamma_0,T,\bm s)}= -(n-2)\ell+|\bm \gamma_0|-|\bm \gamma_1|-\dots-|\bm \gamma_d|+|E_{\op{int}}(T)|,
    \end{equation*}
    and admits a compactification.

    \item[(b)] There is a canonical isomorphism
    \begin{equation}\label{eq: MFTS-orientation-iso}
        |\mathcal{M}_T(\vv{\bm \gamma}, \bm \gamma_0;\bm Y, \bm \tau, \bm s)| \otimes o_{\bm \gamma_d} \otimes \dots \otimes o_{\bm \gamma_1} \cong |\mathfrak{T}_d| \otimes o_{\bm \gamma_0}[(n-2)\ell].
    \end{equation}
    Hence to each element $\bm \Gamma_T \in \mathcal{M}_T(\vv{\bm \gamma}, \bm \gamma_0;\bm Y, \bm \tau, \bm s)$ with $\op{ind}(\bm\Gamma_T)=0$ there is an induced natural map
    \begin{equation}
        \mu^d_{\bm \Gamma_T} \colon o_{\bm \gamma_d} \otimes \dots \otimes o_{\bm \gamma_1} \to o_{\bm \gamma_0}[(n-2)\ell].
    \end{equation}
\ee
\end{lemma}

\begin{proof}
(a) follows from Theorems~\ref{theorem: transversality for MFTS} and ~\ref{thm: MFTS-compactness}.

(b) Here we only sketch the argument. Let $\mathfrak{T}_d^\ell$ be the space of \emph{thickened} metric ribbon trees with $\ell$ ordered marked points on them and an assignment of a $2$-element subset of $[\kappa]$ to each marked point; see \cite{FO97zero-loop} for more details on the thickening of a metric ribbon tree, and also Figure~\ref{fig: tree-thickening} for an example of such tree. We denote by $\mathfrak{T}_d^\ell(T,\bm \tau, \bm s) \subset \mathfrak{T}_d^\ell$ the subspace of marked thickened metric ribbon trees whose underlying combinatorial type is that of a tree $(T, \bm\tau, \bm s)$ with $(T, \bm\tau) \in \mathcal{T}^\ell_d$ and $\bm s \in S(T, \bm \tau)$, i.e., the order of marked points agrees with $\bm s$ and the marked points are enhanced with additional data of $2$-element subsets of $[\kappa]$ dictated by $\bm \tau$. 

The manifold $\mathcal{M}_T(\vv{\bm \gamma}, \bm \gamma_0;\bm Y, \bm \tau, \bm s)$ can be presented as the fiber product of 
$$\mathfrak{T}_d^\ell(T,\bm \tau, \bm s) \times W^u(\bm \gamma_d) \times \dots \times W^u(\bm \gamma_1) \quad \mbox{ and } \quad W^s(\bm \gamma_0) \times \displaystyle\prod_{s \in \bm s, \tau_s \in \bm \tau} \Delta^{\tau_s}_{I_s}$$
by the maps
$$\Psi \colon W^s(\bm \gamma_0) \times \displaystyle\prod_{s \in \bm s, \tau_s \in \bm \tau} \Delta^{\tau_s}_{I_s} \hookrightarrow \Omega^{1,2}(M, \bm q)^{\times (\ell+1)},$$
$$\Phi \colon \mathfrak{T}_d^\ell(T,\bm\tau,\bm s) \times W^u(\bm \gamma_d) \times \dots \times W^u(\bm \gamma_1) \to  \Omega^{1,2}(M, \bm q)^{\times (\ell+1)},$$
where $\Psi$ is the natural embedding; $\Phi$ is a complicated map involving various evaluation maps, switching maps and flows along vector fields from $\bm Y$ dictated by a point of $\mathfrak{T}_d^\ell(T, \bm \tau, \bm s)$; and we are assuming that all the $\bm q^i=\bm q$. We note that the data of the $\theta^e_i$- and $l^e_j$-coordinates allows us to recover an element of $\mathfrak{T}_d^\ell(T, \bm \tau, \bm s)$. We refer the reader to \cite[Section 7]{abouzaid2011plumbings}, where a similar map is explained in the context of gradient Morse flow trees inside a finite-dimensional manifold.
Such a presentation evidently provides a natural isomorphism
\begin{equation}\label{eq: tree-orientation}
    |\mathcal{M}_T(\vv{\bm \gamma}, \bm \gamma_0;\bm Y, \bm \tau, \bm s)| \otimes o_{\bm \gamma_d} \otimes \dots \otimes o_{\bm \gamma_1} \cong |\mathfrak{T}_d^\ell| \otimes \zeta^{-n\ell} \otimes o_{\gamma_0},
\end{equation}
where use coorientations of $\Delta_{I_s}^{\tau_s}$ induced by orientation of $M$ as in \cite[Proposition 3.4]{KY}. We clearly have an isomorphism $|\mathfrak{T}^\ell_d| \cong |\mathfrak{T}_d| \otimes \zeta^{2\ell}$, induced by choosing local orientations at marked points via orientation of the page. Hence we obtain~\eqref{eq: MFTS-orientation-iso}.
\end{proof}

\begin{figure}[h]
    \centering
    \includegraphics[width=10cm]{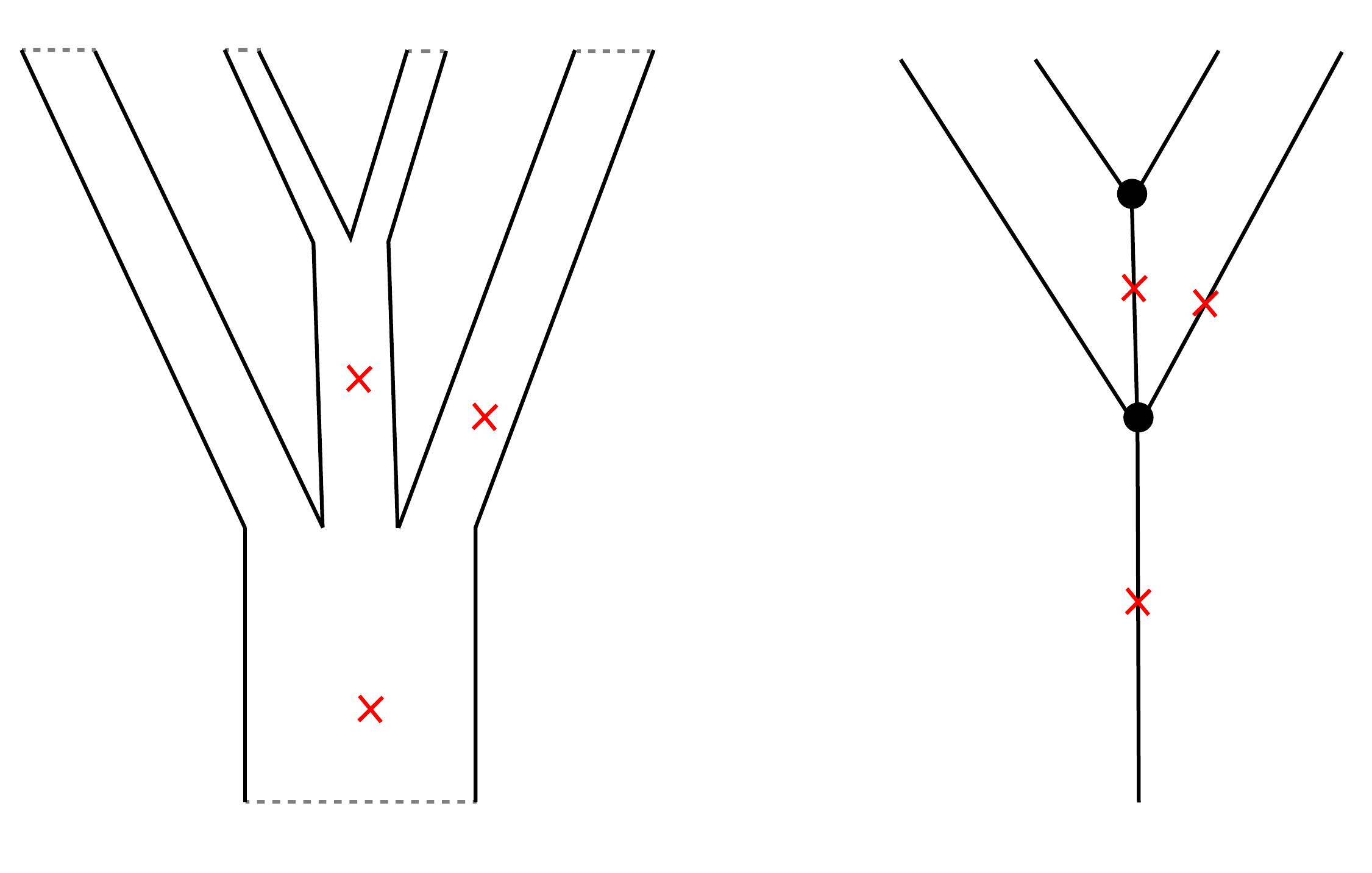}
    \caption{Here we depict an example of a thickening of a tree $(T, \bm \tau) \in \mathcal{T}^3_4$.}
    \label{fig: tree-thickening}
\end{figure}

Fixing a $\kappa$-tuple of points $\bm q$ in $M$ and a choice of universal perturbation data $\bm Y=\{\bm Y_{T}\}$, we now define the $A_\infty$-operations on $CM_{-*}(\Omega^{1,2}(M, \bm q))$. For $\vv{\bm \gamma}=(\bm \gamma_1, \dots, \bm \gamma_d)$ a tuple of critical points and generators $[\bm \gamma_1] \in o_{\bm \gamma_1}, \ldots, [\bm \gamma_d]\in o_{\bm \gamma_d}$ we set:
\begin{gather}\label{eq: A-infty-maps}
    \mu^d_M\colon CM_{-*}(\Omega^{1,2}(M,\bm q))\otimes\dots\otimes CM_{-*}(\Omega^{1,2}(M,\bm q))\to CM_{-*}(\Omega^{1,2}(M,\bm q)),\\
    [\bm {\gamma}_d] \otimes \dots \otimes [\bm \gamma_1]\mapsto\sum_{\substack{\ell\geq0, \, T\in \accentset{\circ}{\mathcal{T}}_d^\ell, \bm s \in S(T), \, \bm \gamma_0 \in \mathcal{P},  \\ \op{ind}(\vv{\bm \gamma}, \bm \gamma_0; T,\bm s)=0, \, \bm \Gamma_T \in \mathcal{M}_T(\vv{\bm \gamma}, \bm \gamma_0;\bm Y, \bm \tau, \bm s)}}\frac{(-1)^{\dagger_d}}{\ell!}\mu^d_{\bm \Gamma_T}([\bm {\gamma}_d] \otimes \dots \otimes [\bm \gamma_1]),\nonumber
\end{gather}
where 
\begin{equation}\label{eq-trees-sign}
    \dagger_d=\sum_{j=1}^dj\cdot|\gamma_{j}|.
\end{equation}
Note that for the tuple of inputs $([\bm \gamma_d]\hbar^{i_d}, \ldots, [\bm \gamma_1]\hbar^{i_1})$, the sign $\dagger_d$ is computed using the corresponding gradings of these generators, i.e.,
$$\dagger_d=\sum_{j=1}^d j \cdot (|\bm \gamma_j|+(2-n)i_j).$$
The operation $\mu^d_M$ is then extended to tuples of generators $([\bm \gamma_d]\hbar^{j_d}, \ldots, [\bm \gamma_1]\hbar^{j_1})$ for all $j_1, \ldots, j_d \in \Z$, in the same way as for the wrapped HDHF Fukaya category and we refer the reader to the discussion following~\eqref{eqn: second partial u}.

Note that $\mu^1_M$ coincides with the differential (\ref{eqn: differential of multiloop complex}).

\begin{theorem}
    There exists a universal data $\bm Y$ such that the collection $\{\mu^d_M,\,d\geq1\}$ defines the structure of an $A_\infty$-algebra on $CM_{-*}(\Omega^{1,2}(M,\bm q))$, i.e., the following equations are satisfied
    \begin{equation}\label{eq: A-infty relation}
        \sum_{k,d'} (-1)^{\text{\ding{64}}_k} \, \mu_M^{d-k+1} (\bm \gamma_d, \ldots, \bm \gamma_{k+d'+1}, \mu_M^{d'}(\bm \gamma_{k+d'}, \ldots, \bm \gamma_{k+1}), \bm \gamma_k, \ldots, \bm \gamma_1) = 0,
    \end{equation}
    where 
    \begin{equation}\label{eq: A-infty relation-sign}
    \text{\ding{64}}_k=|\bm \gamma_1|+\dots+|\bm \gamma_k|-k.    
    \end{equation}
\end{theorem}

\begin{proof}
    The universal data $\bm Y$ is given by Theorem~\ref{thm: MFTS-compactness}. The terms in the sum~\eqref{eq: A-infty-maps} are integers as the data $\bm Y$ that we choose is {\em symmetric}, i.e., does not depend on the choice of ordering $\bm s\in S(T,\bm\tau)$; see~\eqref{eq: MFTS0data-symmetric}. 
    
    Verifying the $A_\infty$-relations boils down to the analysis of $\bdry \mathcal{M}_T(\vv{\bm \gamma}, \bm \gamma_0;\bm Y, \bm \tau, \bm s)$ for $\op{ind}(\vv{\bm \gamma}, \bm \gamma_0, T)=1$, as given by Theorem~\ref{thm: MFTS-compactness}. The contributions corresponding to~\eqref{eq: MFTS-boundary-internal-segment} with $i>0$ cancel out as in Lemma~\ref{lemma-d-regular} and Corollary~\ref{cor: extend the moduli space}. The contributions corresponding to \eqref{eq: MFTS-boundary-switch-to-vertex} consist of interior points and hence cancel out, as in the proof of Theorem~\ref{thm: MFTS-compactness}. The local behavior is illustrated in Figure~\ref{fig-trivalent}.

    Next we analyze the contributions from edge contractions~\eqref{eq: MFTS-boundary-edge-contraction}. Such a boundary point $\bm\Gamma_T$ is an element of $\mathcal{M}_{T'}(\vv{\bm \gamma}, \bm \gamma_0;\bm Y, \bm \tau', \bm s)$ with $(T'=T/e,\bm\tau')$. Since $T$ is binary, there is a unique binary tree $\tilde{T}\not=T$ and a unique edge $\tilde{e} \in \tilde{T}$ such that $\tilde{T}/\tilde{e}=T'$. Hence $\bm \Gamma_T \in \partial \mathcal{M}_{\tilde{T}}(\vv{\bm \gamma}, \bm \gamma_0;\bm Y, \tilde{\bm \tau}, \bm s)$ for some $\tilde{\bm\tau}$ and the contributions from \eqref{eq: MFTS-boundary-edge-contraction} also cancel out. See \cite[Lemma 5.19, Corollary 5.20]{Mescher2018} for more details in the finite-dimensional case.

    The remaining boundary points~\eqref{eq: MFTS-boundary-breaking} correspond precisely to the compositions $\mu_M^{d-d'}\circ \mu_M^{d'}$, implying the theorem. The signs induced by the boundary orientation and by the product orientation agree up to an overall sign dictated by~\eqref{eq: A-infty relation}; this is verified in the same way as in Section~\ref{subsection: review of HDHF} for the Floer-theoretic moduli spaces. The main observation here is that the space $\mathfrak{T}_d$ is diffeomorphic to the space of $(d+1)$-punctured disks $\mathcal{R}^d$ (see the discussion below) \cite{fukaya1997zero}, and the orientation conventions for $\mathfrak{T}_d$ (see \cite[Section 8]{abouzaid2011plumbings}) coincide with those for $\mathcal{R}^d$.
\end{proof}

\begin{figure}[ht]
    \centering
    \includegraphics[width=3cm]{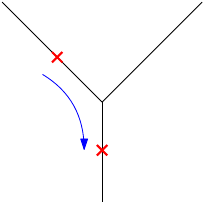}
    \caption{The crosses denote switchings.}
    \label{fig-trivalent}
\end{figure}

\section{The isomorphism with HDHF} \label{section: isomorphism with HDHF}

In this section we define an $A_\infty$-morphism from the wrapped HDHF $A_\infty$-algebra of $\kappa$ cotangent fibers of $T^*M$ to the based multiloop complex
\begin{equation*}
 \mathcal{F}= ( \mathcal{F}^d\colon CW^*(\sqcup_{i=1}^\kappa T_{q_i}^*M)\to CM^{*+1-d}(\Omega^{1,2}(M,\bm q)))_{d\geq 1},
\end{equation*}
by counting rigid elements in a mixed moduli space which combines holomorphic curves and Morse gradient trees and
show that $\mathcal{F}$ is an $A_\infty$-equivalence.

\subsection{Review of HDHF} \label{subsection: review of HDHF}

We briefly sketch the construction of $CW^*(\sqcup_{i=1}^\kappa T_{q_i}^*M)$, referring the reader to \cite[Sections 2.1 and 2.2]{honda2022higher} for more details.

Let $X=T^*M$ be the cotangent bundle of $M$ with canonical Liouville form $\lambda=p\,dq$.  As before, let $g$ be a Riemannian metric on $M$ and $|\cdot|$ be the induced norm on $T^*M$. Choose the Hamiltonian 
\begin{gather}
\label{eq-H} H\colon T^*M\to\mathbb{R},\quad  (q,p)\mapsto \tfrac{1}{2}|p|^2,
\end{gather}
where $q\in M$ and $p\in T^*_q M$. The Hamiltonian vector field $X_{H}$ with respect to the canonical symplectic form $\omega=dq\wedge dp$ is then given by $i_{X_{H}}\omega=dH$. Let $\phi^t\coloneqq \phi^t_{H}$ be the time-$t$ flow of $X_{H}$.

The objects of the wrapped HDHF Fukaya category that we consider are {\em ordered} $\kappa$-tuples of disjoint exact Lagrangians in $(X,\omega=d\lambda)$ that are cylindrical at infinity. In particular we consider $\sqcup_{i=1}^\kappa T_{q_i}^* M$, where $\bm q=(q_1,\dots,q_\kappa)$ is an ordered $\kappa$-tuple of distinct points on a small disk $U\subset{M}$.  
By choosing $g$, we can guarantee that all Hamiltonian chords of $\phi^1$ between the cotangent fibers $\{T_{q_1}^*{M},\dots,T_{q_\kappa}^*{M}\}$ are nondegenerate. We assume that $c_1(T^*M)=0$ and all cotangent fibers are relatively spin. Then there exists a well-defined $\Z$-grading on the Hamiltonian chords $y$ connecting $T_{q_{i}}^*M$ and $T_{q_{j}}^*M$ for any such pair and to each $y$ one may associate a $\Z$-graded rank $1$ free abelian group $o_y$ called \emph{orientation line}, whose $\Z$-grading is equal to that of the chord $y$.  As a graded $\Z$-module the chain complex $CW^*(T^*_{q_i}M, T^*_{q_j}M)$ is a direct sum of such orientation lines over all Hamiltonian chords $y$ from $T^*_{q_i}M$ to $T^*_{q_j}M$. See \cite{seidel2008fukaya, abouzaid2010geometric, ganatra2013thesis} for more details.

We set
$$CW^*_0(\sqcup_{i=1}^\kappa T_{q_i}^*M)=\bigoplus_{\sigma \in S_\kappa} \otimes_{i=1}^\kappa CW^*(T_{q_i}^*M, T_{q_{\sigma(i)}}^*M)$$
To each tuple of chords $\bm y=(y_1, \dots, y_\kappa)$ we associate an orientation line 
$$o_{\bm y}=o_{y_1}\otimes\dots\otimes o_{y_\kappa}$$
with $\Z$-grading given by
$$|\bm y|=|y_1|+\dots+|y_\kappa|.$$
Such orientation lines are the direct summands of $CW^*_0(\sqcup_{i=1}^\kappa T_{q_i}^*M)$.

We then define the (module underlying the) {\em wrapped HDHF chain complex $CW^*(\sqcup_{i=1}^\kappa T_{q_i}^*M)$} to be $CW^*_0(\sqcup_{i=1}^\kappa T_{q_i}^*M)\otimes_R R\llbracket \hbar\rrbracket$ for $n=2$ and $CW^*_0(\sqcup_{i=1}^\kappa T_{q_i}^*M)\otimes_R R[\hbar]$ for $n>2$.

Let $\mathcal{R}^d$ be the moduli space of unit disks $S$ with $d+1$ boundary punctures, where $d$ of them are marked as incoming and $1$ as outgoing, modulo automorphisms. The punctures of $S$ are $p_0, \dots, p_d$, arranged in a counterclockwise manner, where $p_0$ is the unique outgoing puncture. Let $\partial_i S$ be the boundary arc from $p_i$ to $p_{i+1}$. We choose representatives $S$ of equivalence classes of $\mathcal{R}^d$ in a smooth manner (e.g., by setting $p_0=-i$ and $p_1=i$) and abuse notation by writing $S\in \mathcal{R}^d$.
We call $S$ the ``$A_\infty$-base direction''.  
We endow each $S$ with a collection of strip-like ends $[0,+\infty)_{s_i}\times[0,1]_{t_i}$ about $p_i$ for $i=1,\dots,d$, and $(-\infty,0]_{s_0}\times[0,1]_{t_0}$ about $p_0$, and we make such a choice in a universal and consistent manner as explained in \cite[Section (9g)]{seidel2008fukaya}.

We denote by $\mathcal{R}^d_{\kappa, \chi}$ the moduli space of simple (= with simple branch points) $\kappa$-fold branched covers $\pi: \dot F \to S$, where $\dot F= F \setminus \cup_{i=0}^d {\bf p_i}$ is a compact Riemann surface with boundary and boundary punctures and each $\bf p_i$ is the $\kappa$-tuple of points in $\partial F$ mapping to the $i$th boundary puncture of $S$ under the extension of $\pi$ to $F\to S$. Here $S$ ranges over all of $\mathcal{R}^d$, $\chi(\dot F)=\chi$, and all the branch points are assumed to lie in $\operatorname{int}(S)$. Note that for each such $(\dot F, \pi) \in \mathcal{R}^d_{\kappa, \chi}$ there are exactly $\kappa-\chi$ branch points.

We make a universal choice of conformally consistent Floer data $(J^{d, \kappa, \chi}, H^{d, \kappa, \chi})$ which assigns to each $(\dot F, \pi) \in \mathcal{R}^d_{\kappa, \chi}$ a family of almost complex structures $J^{d, \kappa, \chi}_{\dot F, \pi} \colon \dot F \to \mathcal{J}(T^*M)$ and a family of Hamiltonian functions $H^{d, \kappa, \chi}_{\dot F, \pi} \colon \dot F \to \mathcal{H}(T^*M)$. Here $\mathcal{J}(T^*M)$ is the set of compatible almost complex structures of contact type and $\mathcal{H}(T^*M)$ is the set of quadratic-at-infinity Hamiltonian functions. We refer to \cite{honda2022higher, KY, seidel2008fukaya} for more details on constructing such consistent data. 

Referring to \cite{abouzaid2010geometric} for more details, we choose a function $\rho_{S} \colon \partial \overline{S} \to [1, +\infty)$ and a $1$-form $\alpha_{S}$ on each $S \in \mathcal{R}^d$ such that $\alpha_{S}|_{\bdry D}=0$, $d\alpha_{S}\leq 0$, and for the choice as above it is required that
\begin{gather}
    X_{H^{d, \kappa, \chi}_{\dot F, \pi}} \otimes \pi^*(\alpha_{S})=X_{\frac{H}{\rho_{S}(p_i)} \circ \psi^{\rho_{S}(p_i)}} \otimes dt,\\
    J^{d, \kappa, \chi}_{\dot F, \pi} =  (\psi^{\rho_{S}(p_i)})^*J^{d, \kappa, \chi}_{\dot F, \pi}=J_t
\end{gather}
on each strip-like end, where $J_t$, $t\in[0,1]$, is a fixed path in $\mathcal{J}(T^*M)$.

Given $\vv{\bm y}=(\bm y_d, \dots, \bm y_1)$ and $\bm y_0$, where each $\bm y_i \in CW^*(\sqcup_{j=1}^\kappa T^*_{q_j}M)$, we denote by $\mathcal{M}_{\kappa, \chi}(\vv{\bm y};\bm y_0)$ the moduli space of maps
\begin{equation}
    u= (\pi, v) \colon (\dot F,j)\to(S\times T^*M,J_{S}),
\end{equation}
where $\pi: (\dot F,j) \to S$ represents a point in $\mathcal{R}^d_{\kappa, \chi}$ and $\pi_{T^*M}\circ u=v$ satisfies:
\begin{align}
    \label{floer-condition}
    \left\{
        \begin{array}{ll}
            \text{$(dv-X_{H^{d, \kappa, \chi}_{\dot F, \pi}} \otimes \pi^*(\alpha_{S}))^{0,1}=0$;}\\
            \text{$v(z) \in \sqcup_i T^*_{q_i}M$ for each $z \in \partial F$;}\\
            \text{$v$ limits to $\psi^{\rho_{S}(p_i)}(\bm y_i)$ as $s_i\to+\infty$ for $i=1,\dots,d$;}\\
            \text{$v$ limits to $\psi^{\rho_{S}(p_0)}(\bm y_0)$ as $s_0\to-\infty$.}
        \end{array}
    \right.
\end{align}

The Fredholm index of a curve $u \in \mathcal{M}_{\kappa, \chi}(\vv{\bm y}; \bm y_0)$ is given by the formula
\begin{equation}
    \operatorname{ind}(u)=(n-2)(\chi-\kappa)+|\bm y_0|-|\bm y_1|-\dots-|\bm y_d|+d-2.
\end{equation}

 There is a natural isomorphism 
\begin{equation}\label{eq-orientation-lines-mu-d}
|\mathcal{M}_{\kappa, \chi}(\vv{\bm y};\bm y_0)| \otimes o_{\vv{{\bm y}}} \cong \otimes | \mathcal{R}^d_{\kappa, \chi}| \otimes o_{\bm y_0} \cong  \zeta^{\otimes n(\chi-\kappa)} \otimes |\mathcal{R}^d| \otimes \zeta^{2(\kappa-\chi)} \otimes o_{\bm y_0} \cong |\mathcal{R}^d| \otimes o_{\bm y_0}[(n-2)(\kappa-\chi)]
\end{equation}
as shown, for instance, in \cite[Lemma B.1]{ganatra2013thesis}. If $u \in \mathcal{M}_{\kappa, \chi}(\vv{\bm y};\bm y_0)$ and $\operatorname{dim}\mathcal{M}_{\kappa, \chi}(\vv{\bm y};\bm y_0)=0$, then \eqref{eq-orientation-lines-mu-d} gives a natural map
\begin{equation} \label{eqn: partial u}
    \partial_u \colon o_{\vv{\bm y}} \to o_{\bm y_0}[(n-2)(\kappa-\chi)].
\end{equation}
Note that above we used the isomorphism $| \mathcal{R}^d_{\kappa, \chi}| \cong |\mathcal{R}^d| \otimes \zeta^{2(\kappa-\chi)}$, which relies on the following. Let $\mathcal{R}^{d, \kappa-\chi}$ be the moduli space of disks $S$ with $\kappa-\chi$ interior marked points, for example studied in \cite[Section 4]{siegel2021} to define \emph{bulk deformations}. There is a covering map $\mathcal{R}^d_{\kappa, \chi} \to \mathcal{R}^{d, \kappa-\chi}$ and a fibration $\mathcal{R}^{d, \kappa-\chi} \to \mathcal{R}^d$. We then choose the standard orientation conventions for $\mathcal{R}^d$ as in \cite{seidel2008fukaya} and orient the fibers using the orientation on $S$ at each marked point (note that the marked points are ordered). This orientation on $\mathcal{R}^{d, \kappa-\chi}$ clearly induces an orientation $\mathcal{R}^d_{\kappa, \chi}$.

Now we explain how to extend $\partial_u$ to a map 
\begin{equation} \label{eqn: second partial u}
    \partial_u \colon o_{\bm y_d}[(n-2)j_d] \otimes \dots \otimes o_{\bm y_1}[(n-2)j_1] \to o_{\bm y_0}[(n-2)(j_1+\dots+j_d+(\kappa-\chi))].
\end{equation}
For simplicity suppose only one $j_i \neq 0$.  Then \eqref{eq-orientation-lines-mu-d} induces the isomorphism 
\begin{equation}
    \zeta^{(2-n)j_i} \otimes |\mathcal{M}_{\kappa, \chi}(\vv{\bm y};\bm y_0)| \otimes o_{\bm y_d} \otimes \dots \otimes o_{\bm y_1} \cong \zeta^{(2-n)j_i} \otimes |\mathcal{R}^d_{\kappa, \chi}| \otimes o_{\bm y_0}[(n-2)(\kappa-\chi)],
\end{equation}
where $\zeta$ is $\Z$ in degree $1$. This is equivalent to 
\begin{gather}\label{eq-mu-m-sign-shift}
|\mathcal{M}_{\kappa, \chi}(\vv{\bm y};\bm y_0)| \otimes o_{\bm y_d} \otimes \dots \otimes \zeta^{(2-n)j_i} \otimes o_{\bm y_i} \otimes \dots \otimes o_{\bm y_1} \cong |\mathcal{R}^d_{\kappa, \chi}| \otimes  \zeta^{(2-n)j_i}\otimes o_{\bm y_0}[(n-2)(\kappa-\chi)],
\end{gather}
with Koszul sign difference $(-1)^{(2-n)j_i(|\bm y_d|+\dots+|\bm y_{i+1}|+d)}.$
For $u \in \mathcal{M}_{\kappa, \chi}(\vv{\bm y};\bm y_0)$ and $\operatorname{dim}\mathcal{M}_{\kappa, \chi}(\vv{\bm y};\bm y_0)=0$ we then get $\partial_u$ in~\eqref{eqn: second partial u}.
The $\mu^d$-composition map is then defined as
\begin{equation}
\label{eq-m_2}
    \mu^d_F([\bm y_d]\hbar^{j_d},\dots,[\bm y_1]\hbar^{j_1})=\sum_{\substack{\bm y_0,\chi\leq\kappa \\ u \in \mathcal{M}_{\kappa, \chi}^{\operatorname{ind}=0}(\vv{\bm y}; \bm y_0)}}(-1)^{\dagger_d}\partial_u([\bm y_d]\hbar^{j_d}, \dots, [\bm y_1]\hbar^{j_1})\cdot\hbar^{\kappa-\chi},
\end{equation}
where $[\bm y_d]\hbar^{j_d} \in o_{\bm y_d}[(n-2)j_d], \dots, [\bm y_1]\hbar^{j_1} \in o_{\bm y_1}[(n-2)j_1]$ and
\begin{equation}\label{eq: dagger-Floer}
    \dagger_d=|\bm y_1|+2|\bm y_2|+\dots+d|\bm y_d|.
\end{equation}

The proof of the $A_\infty$-relations (see~\eqref{eq: A-infty relation}) up to signs is outlined in \cite{CHT, honda2022higher} and here we only provide a discussion regarding signs. As pointed out in \cite{CHT}, $\bdry \overline{\mathcal{M}}^{\operatorname{ind}=1}_{\kappa, \chi}(\vv{\bm y}; \bm y_0)$ is covered by
\begin{equation}\label{eq-mu-m-boundary}
    \bigsqcup_{\substack{\tilde{\bm y}; \, \chi_1, \chi_2 \le \kappa \\ \chi_1+\chi_2=\kappa+\chi}} \mathcal{M}^{\operatorname{ind}=0}_{\kappa, \chi_1}(\bm y_{p+d_1}, \dots, \bm y_{p+1}; \tilde{\bm y}) \times \mathcal{M}^{\operatorname{ind}=0}_{\kappa, \chi_2}( \bm y_d, \dots, \bm y_{p+d_1+1}, \tilde{\bm y}, \bm y_p, \dots, \bm y_1; \bm y_0).
\end{equation}
It remains to analyze the difference between the boundary orientation and the chosen product orientation as in \cite[(12f), (12g)]{seidel2008fukaya}. We claim that the analysis is essentially the same and follows from a commutative diagram analogous to the one presented in \cite[(12f)]{seidel2008fukaya}. The key observation here is that 
$$\operatorname{dim} \mathcal{R}^d= \operatorname{dim } \mathcal{R}^d_{\kappa, \chi} \, \operatorname{mod} \, 2,$$
and the fact that the Koszul sign in \eqref{eq-mu-m-sign-shift} naturally incorporates into such diagram.

\begin{remark}
    We note that it is expected that the HDHF as described above arises from the appropriate Floer theory on the orbifold $\op{Sym}^{\kappa}(T^*M)$. We do not pursue this approach here and refer the reader to \cite{MSS2025orbifold} for the development of this perspective.
\end{remark}

\subsection{Half-disk Hurwitz spaces}

We denote by $\mathcal{H}^d$ the moduli space of unit disks with $d+2$ boundary punctures of which $d$ are incoming and appear consecutively along the boundary and $2$ are outgoing, modulo automorphisms. We refer to \cite[Section 4.4]{abouzaid2012wrapped} for a detailed treatment of these spaces. We denote by $p_{-1}, p_0, \dots, p_d$ the punctures of $T \in \mathcal{H}^d$ appearing in counterclockwise order, where $p_{-1},p_0$ are the outgoing punctures. The portion of the boundary between $p_{i}$ and $p_{i+1}$ (we assume $p_{d+1}=p_{-1}$) will be denoted $\bdry_i T$ as before and the portion between $p_{-1}$ and $p_0$, denoted $\partial_{\op{out}}T$, will be called the \emph{outgoing segment}.  Since $\bdry_{\op{out}}T$ is a preferred segment, we will often refer to $T\in \mathcal{H}^d$ as a ``half-disk''. 

Now observe that $\mathcal{H}^d \cong \mathcal{R}^{d+1}$. This identification provides us with a smooth choice of representatives of equivalence classes of $\mathcal{H}^d$ and a universal choice of strip-like ends as in Section~\ref{subsection: review of HDHF}. We also notice that $\mathcal{H}^1$ consists of just one point $T$, corresponding to the twice-punctured half-strip $[0, +\infty) \times [0,1]$. As before we have the natural projections 
$$\pi_{T^*M}: T\times {T^*M}\to {T^*M}\quad \mbox{and} \quad \pi_{T}:  T\times {T^*M}\to T.$$  

We further introduce the moduli space $\mathcal{H}^d_{\kappa, \chi}$ of $\kappa$-fold simple branched covers $\pi \colon \dot{F} \to T$, with $T \in \mathcal{H}^d$, $\chi(\dot{F})=\chi$, and $\dot{F}=F \setminus \cup_{i=-1}^d \bm p_i$, where each $\bm p_i$ is a $\kappa$-tuple of boundary punctures of $F$, given by the $\kappa$ preimages of the puncture $p_i$ of $T$. All the branched points are assumed to lie in  $\op{int}(T)$.

Again, we make a universal and conformally consistent choice of Floer data $(J^{d, \kappa, \chi}, H^{d, \kappa, \chi}, \rho^d, \alpha^d)$ for all tuples $(d, \kappa, \chi)$, which assigns to each $(\dot{F}, \pi) \in \mathcal{H}^d_{\kappa, \chi}$ a family of almost complex structures $J^{d, \kappa, \chi}_{\dot F, \pi} \colon \dot F \to \mathcal{J}(T^*M)$ and a family of Hamiltonian functions $H^{d, \kappa, \chi}_{\dot F, \pi} \colon \dot F \to \mathcal{H}(T^*M)$. Additionally, there is a choice of a function $\rho^d_{T} \colon \partial \overline{T} \to [1, +\infty)$ and a $1$-form $\alpha^d_{T}$ on each $T \in \mathcal{H}^d$ such that $\alpha^d_{T}|_{\bdry D}=0$, $d\alpha_{T}\leq 0$, and for the choice as above we require
\begin{gather}
    X_{H^{d, \kappa, \chi}_{\dot F, \pi}} \otimes \pi^*(\alpha_{T})=X_{\frac{H}{\rho_{T}(p_i)} \circ \psi^{\rho_{T}(p_i)}} \otimes dt \text{ for } 1 \le i \le d,\\
    J^{d, \kappa, \chi}_{\dot F, \pi} =  (\psi^{\rho_{T}(p_i)})^*J^{d, \kappa, \chi}_{\dot F, \pi}=J_t \text{ for } -1\le i \le d,
\end{gather}
on each strip-like end as above, where $J_t$, $t\in[0,1]$, is a fixed path in $\mathcal{J}(T^*M)$. In addition, we require the restrictions of $H^{d, \kappa, \chi}_{\dot{F}, \pi}$ to the preimages of $\partial_{\op{out}}T$ to vanish. 
We also assume that $\rho^d_{T} \equiv 1$ near $p_{-1}$ and $p_0$.
We refer to \cite[Section 4.5]{abouzaid2012wrapped} for additional details.

Given a tuple $\vv{\bm y}=(\bm y_d, \dots, \bm y_1)$ with each $\bm y_i \in CW^*(\sqcup_{j=1}^\kappa T^*_{q_j}M)$, let $\mathcal{H}_{\kappa, \chi}(\vv{\bm y})$ be the moduli space of holomorphic maps
\begin{gather}
    u=(\pi, v) \colon (\dot{F}, j) \to (T \times T^*M, J_T) \nonumber\\
    (dv -X_{H^{d, \kappa, \chi}_{\dot{F}, \pi}}\otimes \pi^*(\alpha_T))^{0,1}=0 
\end{gather}
such that $((\dot F,j),\pi)\in \mathcal{H}^d_{\kappa, \chi}$ and the following boundary conditions hold:
\begin{align}
    \label{boundary-half-strip-condition}
    \left\{
        \begin{array}{ll}
            \text{$v(z) \in \sqcup_j T^*_{q_j}M$ for each $z \in \pi^{-1}(\partial_i T)$ for $i=0,1,\dots, d$;} \\
            \text{$v(z) \in M$ for each $z \in \pi^{-1}(\partial_{\op{out}} T)$;} \\
            \text{$v$ limits to $\psi^{\rho_{T}(p_i)}(\bm y_i)$ as $s_i\to+\infty$ for $i=0, 1,\dots,d$;}\\
            \text{$v$ limits to $\bm q$ as $s_i\to-\infty$ for $i=-1$.}
        \end{array}
    \right.
\end{align}

The choice of Floer data as above guarantees that each $\mathcal{H}_{\kappa, \chi}(\vv{\bm y})$ is a smooth manifold of dimension
\begin{equation}\label{eq: dimension-for-covers-of-half-disks}
     (n-2)(\kappa-\chi)-|\bm y_1|-\dots-|\bm y_d|+d-1.
\end{equation}
This index computation is analogous to the one from \cite[Lemma 4.3]{KY}.

As in~\eqref{eq-orientation-lines-mu-d} we have a natural isomorphism of orientation line bundles
\begin{equation}\label{eq: orientation-half-disks}
    |\mathcal{H}_{\kappa, \chi}(\vv{\bm y})| \cong |\mathcal{H}^{d}_{\kappa, \chi}| \otimes o_{\bm y_1}^{-1} \otimes \dots \otimes o_{\bm y_d}^{-1}.
\end{equation}

\subsection{The linear map $\mathcal{F}^1$} \label{subsection: the linear map}

In this subsection we define the linear component $\mathcal{F}^1$ of the morphism $\mathcal{F}$ between the two $A_\infty$-algebras of interest by counting elements of \emph{mixed moduli spaces}. Given $\bm y\in CW^*(\sqcup_i T_{q_i}^*M)$, we define 
\begin{align}
\label{eq-F1}
    \mathcal{F}^1(\bm y)=\sum_{{\substack{\bm \gamma,\,r\geq0 \\ (u, \bm \Gamma) \in \mathcal{P}_r(\bm y, \bm \Gamma)}}}(-1)^{|\bm y|}\mathcal{F}^1_{(u,\bm \Gamma)}([\bm y]),
\end{align}
where $\boldsymbol{\gamma}\in CM_{-*}(\Omega^{1,2}(M,\bm q))$, 
the elements in $\mathcal{P}_r(\bm y,\boldsymbol{\gamma})$ are schematically given by Figure \ref{fig-F1}, and 
\begin{equation} \label{eqn: dim for mathcal P}
    {\op{ind}({\bf y}, \bm\gamma, r)\coloneqq  -(n-2)r+|\bm\gamma| - |{\bf y}|}
\end{equation}
refers to the virtual dimension of $\mathcal{P}_r(\bm y,\boldsymbol{\gamma})$.
\begin{figure}[ht]
    \centering
    \includegraphics[width=3cm]{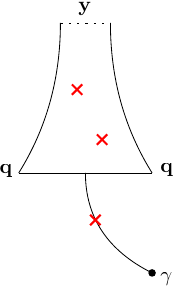}
    \caption{A schematic description of an element in $\mathcal{P}(\bm y,\boldsymbol{\gamma})$. The upper half denotes the base direction of an HDHF holomorphic branched cover of a half-disk in $T^*M$, where the bottom edge lies on the zero section $M$, viewed as an element $\boldsymbol{\gamma}'\in\Omega^{1,2}(M,\bm q)$; the lower arc is an MFLS from $\boldsymbol{\gamma}'$ to $\boldsymbol{\gamma}$; the sum of the number of branch points on the holomorphic curve (or more precisely the branching number) and the number of switching points of the MFLS is $r$}.
    \label{fig-F1}
\end{figure}

\subsubsection{Mixed moduli spaces $\mathcal{P}_r(\bm y,\bm \gamma)$}

Each element of $\mathcal{P}_r(\bm y,\bm \gamma)$ is a tuple $((\dot{F}, \pi),u),\bm \Gamma, \bm s)$ (we also write $(u,\bm \Gamma)$ if $\bm s$ is understood or unimportant), where:

\s\n
(A)  The tuple $\bm s$ is a total ordering of $[r]$. The curve $((\dot{F}, \pi), u)$ is an element of the space $\mathcal{H}_{\kappa, \chi}(\bm y)$ for some $\kappa, \chi \in \Z$ such that $b=\kappa-\chi \ge 0$. In particular, $\pi \colon \dot{F} \to T$ is a branched cover of the half-strip $T=[0,+\infty)\times [0,1] \setminus \{p_{-1}=(0,1), p_0=(0,0)\}$, where $p_1=+\infty$.  We require that the restriction of $\bm s$ to the first $b$ elements agrees with the ordering of the branch points of $u$. Then $(\theta_{s_1}, \dots, \theta_{s_{b}})$ are the $[0,1]$-coordinates of the corresponding branch points of $(\dot{F}, \pi) \in \mathcal{H}^1_{\kappa, \chi}$.

\s\n
(B) Before defining $\bm \Gamma$ we introduce the evaluation map
\begin{gather}
    \mathcal{E}:\mathcal{H}_{\kappa, \chi}(\bm y)\to \Omega^{1,2}(M,{\bm q}),\\
    (u,\bm \Gamma)\mapsto \mathcal{E}(u), \nonumber
\end{gather} 
as follows: If $u$ does not have any branch points along $\bdry_{\op{out}} T$, then we let $\mathcal{E}(u)$ be the restriction of $u$ to $\bm \gamma'=\pi^{-1}(\partial_{\op{out}} T)$, precomposed with $\nu_{[r]}^{\circ b}$ evaluated at $(\theta_{s_1}, \dots, \theta_{s_r}, \bm \gamma')$, where the notation $f^{\circ k}$ means the composition of $f$ with itself $k$ times.
If $u$ has branch points along $\bdry_{\op{out}} T$, then let 
$$\bm \gamma' = \lim_{\delta\to 0^+} v|_{\pi^{-1}(\{\delta\} \times [0,1])},$$ 
and we set 
\begin{equation}\label{eq: evaluation map}
\mathcal{E}(u)=(\gamma'_1\circ\nu_{[r]}^{\circ b}, \dots, \gamma'_\kappa \circ \nu_{[r]}^{\circ b}).
\end{equation}

Observe that $\mathcal{E}(u) \in \Omega^{m,2}(M, \bm q)$ for arbitrary $m \ge 1$.

\s\n
(C) $\bm \Gamma$ is an MFLS which ``starts'' at $\mathcal{E}(u)$ (below we set $\ell=r-b$)
\begin{equation*}
    \bm \Gamma=\left((\theta_{s_{b+1}}, \Gamma_0, l_0), (\theta_{s_{b+2}}, \Gamma_1, l_1), \dots, (\theta_{s_r}, \Gamma_{\ell-1}, l_{\ell-1}), \Gamma_\ell \right), 
\end{equation*}
where $\Gamma_0 \colon [0, l_0] \to \Omega^{1,2}(M, \bm q)$ is a continuously differentiable map satisfying $\Gamma_0(0)=\mathcal{E}(u)$; $\Gamma_\ell(+\infty)=\bm \gamma$; and the remaining conditions are similar to the ones in Definition~\ref{def-diff} and we omit the details.

\s
Note that $\mathcal{P}_r(\bm y, \bm \gamma)$ depends on the choice of Floer data $(J^{1, \kappa, \chi}, H^{1, \kappa, \chi})$ and perturbation data $\bm Y^\ell$. It admits a natural decomposition $\mathcal{P}_r(\bm y, \bm \gamma)=\bigsqcup_{b+\ell=r}\mathcal{P}_{b, \ell}(\bm y, \bm \gamma)$, where $b$ is the number of branch points and $\ell$ is the number of switches.

\begin{lemma}
    \label{lemma-f1-regular}
    There exists a universal Floer data $(J^{1, \kappa, \chi}, H^{1, \kappa, \chi})$ and a universal perturbation data $\bm Y^\ell$ for all tuples $(\kappa, \chi, \ell)$ such that all spaces $\mathcal{P}_r(\bm y, \bm \gamma)$ are smooth manifolds of dimension $\op{ind}({\bf y}, \bm\gamma, \ell)$.
    Moreover, if $\op{ind}({\bf y}, \bm\gamma, \ell)\leq 1$, then $\mathcal{P}_r(\bm y,\bm \gamma)$ admits a compactification.
\end{lemma}

\begin{proof}
To show that $\mathcal{P}_{b, \ell}(\bm y, \bm \gamma)$ is a smooth manifold we use the fact that there is a choice of universal Floer data for which the maps
$$\mathcal{E} \colon \mathcal{H}_{\kappa, \chi}(\bm y) \to \Omega^{1,2}(M, \bm q)$$
are immersions. We refer the reader to \cite[Section 5]{abbondandolo2010floer} and \cite[Section 3.4]{McDuff-Salamon-2012} for additional details. We can then show that for a generic choice of $\bm Y^\ell$ the moduli space $\mathcal{P}_{b, \ell}(\bm y, \bm \gamma)$ is smooth by applying Lemma~\ref{lemma: finite-trajectories-two-immersions} as in the proof of Theorem~\ref{theorem: transversality-for-MFLS}. Here we omit the precise definition of the universality of the perturbation data $\bm Y^\ell$ and simply refer to the analogous discussion in Appendix~\ref{section: appendix-compactness}.
The universality of the Floer data ensures the existence of a compactification and we delay a more careful discussion of the boundary strata until the next section.
\end{proof}

Given an element $(\bm u, \bm \Gamma) \in \mathcal{P}_r(\bm y, \bm \gamma)$ we have a canonical isomorphism
     \begin{equation}
        |\mathcal{P}_r(\bm y, \bm \gamma)| \cong o_{\bm \gamma} \otimes |\mathcal{H}_{\kappa, \chi}(\bm y)|
    \end{equation}

Whenever the dimension of $\mathcal{P}_r(\bm y, \bm \gamma)$ the above together with~\eqref{eq: orientation-half-disks} provides a natural map
\begin{equation}
    \mathcal{F}^1_{(\bm u, \bm \Gamma)} \colon o_{\bm y} \to o_{\bm \gamma}[(n-2)r],
\end{equation}
that we use to define the chain map~\eqref{eq-F1}.

\subsubsection{Ghost bubbles and Kuranishi replacements}  \label{subsubsection: ghost bubbles}

As discussed in Section 3.7 and Lemma 3.9.3 of \cite{CHT}, a technical difficulty that we face is the formation of ghost bubbles.  More precisely, a sequence of curves $(u_i,\bm \Gamma_{(i)})$, $i=1,2,\dots$, in $\mathcal{P}_r(\bm y, \bm \gamma)$, after passing to a subsequence, can limit to 
$$(u_\infty= u_{\infty,1}\cup u_{\infty,2}, \bm \Gamma_{(\infty)}),$$ 
where:
\begin{itemize}
\item[(i)] the {\em main part} $u_{\infty,1}$ is the union of components that are not ghost bubbles; each component of the $u_{\infty,1}$ is somewhere injective; and 
\item[(ii)] $u_{\infty,2}$ is a union of {\em ghost bubbles}, i.e., locally constant maps where the domain is a possibly disconnected compact Riemann surface (possibly with boundary) and $u_{\infty,2}$ maps to points on $\op{Im}(u_{\infty,1})$.
\end{itemize}  
Let $F$ be a connected compact oriented surface with boundary and let $\mathcal{M}^{F}$ be the moduli space of holomorphic maps $u:(F,j)\to T\times T^*M$ that are constant and such that $u$ maps $\partial F$ to $\bdry_{\op{out}} T$ times the zero section $M$.  Each component of $u_{\infty,2}$ belongs to some $\mathcal{M}^F$.

When $n=2$, we can eliminate ghost bubbles using Ekholm-Shende~\cite[Theorem 1.1]{ES} (which in turn builds on work of Doan-Walpuski~\cite{DW}) as in \cite[Lemma 3.9.3]{CHT}.  Note that the Ekholm-Shende result also applies to the case of boundary nodes along Lagrangians.

On the other hand, when $n\geq 3$, we excise the portion $\mathcal{B}$ of $\mathcal{P}_r(\bm y,\bm \gamma)$ that is close to breaking into $u_\infty= u_{\infty,1}\cup u_{\infty,2}$ with $u_{\infty,2}\not=\varnothing$ and replace it with a Kuranishi model $\mathcal{B}'$ of a neighborhood of all possible $u_{\infty,1}\cup u_{\infty,2}$ satisfying (i) and (ii).  
This gives us the {\em Kuranishi replacements} 
$$\mathcal{P}^\sharp_r(\bm y,\bm \gamma),\quad \mathcal{M}^{\sharp}_{\kappa, \chi}(\bm y,\bm y'),\quad \mbox{and} \quad \mathcal{M}^{F,\sharp}$$ 
of $\mathcal{P}_r(\bm y,\bm \gamma)$, $\mathcal{M}_{\kappa,\chi}(\bm y,\bm y')$, and $\mathcal{M}^{F}$ for all $\chi$, ${\bf y}$, ${\bf y}'$, $\bm\gamma$, and $F$ such that:
\be
\item[(KR1)] $\mathcal{P}^\sharp_r(\bm y,\bm \gamma)$ and $\mathcal{M}^{\sharp}_{\kappa, \chi}(\bm y,\bm y')$ are transversely cut out branched manifolds; 
\item[(KR2)] $\mathcal{M}^{F,\sharp}$ is a transversely cut out branched manifold which consists of perturbed holomorphic maps from connected compact Riemann surfaces $(F,j)$ which are homologous to a constant map; and
\item[(KR3)] all the strata of $\bdry \mathcal{P}^\sharp_r(\bm y,\bm \gamma)$ are finite fiber products of moduli spaces of the form $\mathcal{P}^\sharp_{r'}(\bm y',\bm \gamma')$, $\mathcal{M}^{\sharp}_{\kappa, \chi'}(\bm y',\bm y'')$, and $\mathcal{M}^{F,\sharp}$.
\ee

The Kuranishi replacement can be constructed using the work of Fukaya-Oh-Ohta-Ono \cite{FOOO,FOOO2,FOOO3} or its variant due to McDuff-Wehrheim \cite{MW,McD}. (Also see virtual fundamental cycle techniques of Hofer-Wysocki-Zehnder \cite{HWZ2} and Pardon~\cite{Par}.)

\s
{\em In what follows we assume that for $n\geq 3$ we pass to the relevant Kuranishi replacement as necessary and omit $\sharp$ from the notation. Kuranishi replacements will also be made without mention for analogous moduli spaces.  Note that the use of Kuranishi replacements forces us to use $\Q$-coefficients; see Remark~\ref{rmk: coefficient rings}.}

\subsubsection{Chain map}

\begin{theorem} \label{thm: chain map}
    $\mathcal{F}^1$ is a chain map.
\end{theorem}

\begin{proof}[Sketch of proof.]
We analyze the codimension $1$ boundary of 
$$\mathcal{P}_r\coloneqq \mathcal{P}_r^{\mathrm{ind}=1}(\bm y,\bm\gamma).$$
A new codimension $1$ phenomenon is the nodal degeneration along the zero section $M$, that is, some branch point or switch point may approach the bottom boundary of the $A_\infty$-base disk. We give a sketch of the proof that this degeneration admits a continuation inside $\overline{\mathcal{P}}_r$.

\s\n
{\em Step 1 (Elimination of some nodal degenerations)  }
Consider $(u_*,\bm \Gamma)\in \overline{\mathcal{P}}_r$, where $u_* \in \overline{\mathcal{H}}_{\kappa, \chi}(\bm y)$ is a nodal curve with $k$ boundary nodes along $\pi^{-1}(\bdry_{\op{out}} T)$ in addition to $b=\kappa-\chi$ branch points (here the nodes are counted with multiplicity). 
The projection of each node to $T$ is assigned a $[0,1]$-coordinate; we then assign a $k$-tuple $\bm \theta_k$ to the $k$ boundary nodes and choose some ordering on it. Then we may evaluate $\mathcal{E}(u_*)$ as in~\eqref{eq: evaluation map}, after complementing with an additional $(r-b-k)$-tuple of $\theta$-coordinates. Observe that $\mathcal{E}(u_*)$ belongs to a deeper stratum where pairs of strands intersect at each of the $\theta$-coordinates in $\bm \theta_k$. We then denote by $sw(\mathcal{E}(u_*))$ the multiloop obtained by applying $k$ corresponding switching maps in the order previously chosen.

Figure~\ref{fig-chain-map} gives a schematic description of a neighborhood of $\mathcal{E}(u_*)\in \Omega^{1,2}(M,{\bm q})$ when $k=1$; in the figure $\mathcal{E}(u_*)$ is identified with $sw(\mathcal{E}(u_*))$. 
\begin{figure}[ht]
    \centering
    \includegraphics[width=5cm]{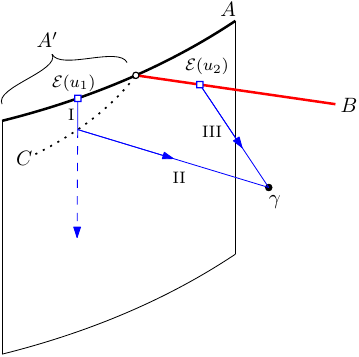}
    \caption{The neighborhood of $\mathcal{E}(u_*)$ in $\Omega^{1,2}(M,{\bm q})$, where $u_*$ is a nodal curve and $\mathcal{E}(u_*)$ is identified with $sw(\mathcal{E}(u_*))$. We are considering the $n=2$ case.  The $1$-dimensional manifolds $A$ and $B$ are given by thick black and red lines, $A'$ is one-half of $A$, and the intersection of $A$ and $B$ represents $sw(\mathcal{E}(u_*))\equiv \mathcal{E}(u_*)$. 
    The blue arcs denote MFLS: the pair (I,II) denotes the MFLS from $\mathcal{E}(u_1)$ to $\boldsymbol{\gamma}$ whilst the blue arc III denotes the MFLS from $\mathcal{E}(u_2)$ to $\boldsymbol{\gamma}$. The dotted arc $C$ indicates where switchings along gradient trajectories take place. The tuples $(u_1,(\mathrm{I,II}))$ and $(u_2,\mathrm{III})$ are viewed as elements of $\mathcal{P}^{\mathrm{ind}=1}(\bm y,\bm \gamma)$.
    Here $\mathcal{E}(u_2)$ approaches $\mathcal{E}(u_*)$ and then continues to $\mathcal{E}(u_1)$.}
    \label{fig-chain-map}
\end{figure}  

Let
\be
\item[(i)] $B$ be the set of $\mathcal{E}(u_2)$, $u_2\in \overline{\mathcal{H}}_{\kappa, \chi-k}(\bm y)$ (i.e., $u_2$ has $b+k$ branch points), such that there is an MFLS $\bm \Gamma_2$ from $\mathcal{E}(u_2)$ to $\bm\gamma$ with $\ell=r-b-k$ switches as in (C) in the definition of $\mathcal{P}_r$;
\item[(ii)] $A$ be the set of $\mathcal{E}(u_1)$, $u_1\in \overline{\mathcal{H}}_{\kappa, \chi}(\bm y)$; and
\item[(iii)] $A'\subset A$ be the subset consisting of $\mathcal{E}(u_1)$, $u_1\in \overline{\mathcal{H}}_{\kappa, \chi}(\bm y)$, for which there is an MFLS $\bm \Gamma_1$ from $sw(\mathcal{E}(u_1))$ to $\bm\gamma$ with $\ell-k$ additional switchings.

\ee   
We have $u_*\in \overline{\mathcal{H}}_{\kappa, \chi}(\bm y)$, $\mathcal{E}(u_*)\in A$, and $sw(\mathcal{E}(u_*))\in \bdry B$.   
If $\dim \mathcal{H}_{\kappa, \chi-k}(\bm y)= a+1$, then
\begin{equation}\dim \mathcal{H}_{\kappa, \chi}(\bm y)= a+1+k(n-2),\end{equation}
by the Fredholm index formula; each boundary node contributes $n-2$ to the index. 

On the other hand, the appearance of a boundary node along the Lagrangian $\bdry_2 T_1\times M$ is a codimension $(n+1)-2\cdot 1=n-1$ condition, since the Lagrangian is $(n+1)$-dimensional and the boundary of domain is mapped to a real $1$-dimensional curve. The incidence condition for $k$ boundary nodes is $k(n-1)$-dimensional.  In total, there is a dimension drop of 
$$k(n-1) -k (n-2)=k.$$ 
This in turn forces $k=1$ since $\mathcal{P}_r$ has $\op{ind}=1$ by assumption.  

Hence $u_*$ has a single boundary node and either (i) has no almost constant bubbles or (ii) has an almost constant bubble with one boundary component which is attached to the main part along the single boundary node.  Note that every boundary component of an almost constant bubble must have a boundary node by examining the projection $\pi$ to $T$.  Finally, the almost constant bubble in (ii) is not a disk and hence has negative Euler characteristic and negative Fredholm index; it therefore does not exist by (KR1).  We are now in Case (i).

\s\n
{\em Step 2 (Description of $\overline{\mathcal{H}}_{\kappa, \chi}(\bm y)$ and $\overline{\mathcal{H}}_{\kappa,\chi-1}(\bm y)$ near $u_*$)}  
 The following theorem is a gluing result which rephrases the theorem of Hirschi and Hugtenberg \cite{HH25openDM} in a slightly different setting. We note that their result generalizes the gluing result from Ekholm and Shende \cite{ekholm2021skeins}, which in turn is a special case of ideas of Fukaya (unpublished) relating higher-genus curve counting in Lagrangian Floer theory and string topology. Let $U(u_*)$ be a neighborhood of $u_*$ in $\overline{\mathcal{H}}_{\kappa, \chi}(\bm y)$ and let 
$$Z=\{u_1\in \overline{\mathcal{H}}_{\kappa, \chi}(\bm y)\mid   \mbox{$u_1$ has a node along $\bdry_{\op{out}} T$}\}.$$

\begin{theorem}\label{thm: Fukaya}
There exists a neighborhood of $Z\cap U(u_*)$ in $\overline{\mathcal{H}}_{\kappa, \chi}(\bm y)$ of the form $D^{a}\times D^{n-1}$, such that $Z\cap U(u_*)=D^a\times\{0\}$.  Moreover, $Z\cap U(u_*)$ can be viewed as an open subset of $\bdry \overline{\mathcal{H}}_{\kappa, \chi-1}(\bm y)$. 
\end{theorem}  

\begin{proof}[Sketch of proof of Theorem~\ref{thm: Fukaya}] 
We convert the desired perturbations of a nodal curve into a smooth one into a commonly used gluing result (see for example \cite{FO3lagrangian} which treats holomorphic curves before/after Lagrangian surgery). 

Let us consider the following local model near a nodal curve: 
\be
\item the ambient symplectic/almost complex manifold is $\C^2$ with the standard symplectic form, 
\item the Lagrangian boundary condition is the standard $\R^2\subset \C^2$, and 
\item the punctures of the finite energy holomorphic curves under consideration limit to chords (in a Morse-Bott family) from $S^1$ to itself, where $S^1\subset S^3$ is the Legendrian boundary at infinity of $\R^1$.
\ee  
We glue two half-planes $\H=\{z\mid   \op{Im}(z)\leq 0\}\subset \C$ along a node using the model 
$$C_c\coloneqq \{(z_1,z_2)\in \C^2\mid  z_1\in \H, z_2^2-z_1^2=c\}_{c\in \R_{\geq 0}}.$$ 
Since $C_c$ maps to $\H$ under the projection to the $z_1$-coordinate, we may view $C_c$ as a subset of $\H\times \C$, where $\H$ is viewed as the $A_\infty$-base direction.

More generally, for $\R^{n+1}\subset \C^{n+1}_{z_1,\dots, z_{n+1}}$ with $n\geq 1$, the pairs of half-planes are of the form:
\begin{gather*}
P_1(z_3,\dots,z_{n+1})\coloneqq \{(z_1,\dots, z_{n+1})\in \C^{n+1}\mid   z_2=z_1, z_1\in \H,  z_3,\dots, z_{n+1}\in \R\}, \\
P_2\coloneqq \{(z_1,\dots, z_{n+1})\in \C^{n+1} \mid   z_2=-z_1, z_1\in \H, z_3=\dots=z_{n+1}=0\},
\end{gather*}
and the gluing of $P_1(z_3,\dots,z_{n+1})$ and $P_2$ into $C_c$ occurs when $z_3=\dots= z_{n+1}=0$ (when $n=1$ this just means there is no extra condition). The half-planes are Morse-Bott regular since they can be doubled into standard complex lines in $\C^{n+1}$ and a curve is regular if and only if its double is; the same also holds for $C_c$.  Again we may view $P_1(z_3,\dots,z_{n+1})$, $P_2$, and $C_c$ as subsets of $\H\times \C^n$.

Next, given a nodal curve $u_*$, we ``zoom in'' near the intersection point $x\in T\times T^*M$ and view it as a $2$-level SFT building, where the top level maps to $(T\times T^*M) \setminus \{x\}$ and the bottom level maps to the local model just treated, with $\op{dim} M=n$. The curves on the bottom level corresponding to $\overline{\mathcal{H}}_{\kappa, \chi}(\bm y)$ are pairs $(P_1(z_3,\dots,z_{n+1}), P_2)$ (parametrized by $D^{n-1}$) and those corresponding to $\overline{\mathcal{H}}_{\kappa, \chi-1}(\bm y)$ are the $C_c$ with $z_3=\dots=z_{n+1}=0$ (parametrized by $c\in \R_{\geq 0}$).

Finally, it remains to verify the necessary gluing result for finite energy holomorphic curves with Lagrangian boundary in this case. Such gluing analysis is performed in \cite[Appendix A]{HH25openDM} and we refer the reader to the proof of \cite[Theorem A.4]{HH25openDM} for further details. The theorem then follows.
\end{proof}

\n
{\em Step 3 (Descriptions of $A$ and $B$)}  Now suppose that having an MFLS that goes to $\bm\gamma$ is a codimension $a$ condition for some $a\geq 0$. Recalling the definitions of $A$ and $B$ from earlier in the proof,  $\dim B=1$.  Theorem~\ref{thm: Fukaya}, and the fact that the evaluation map $\mathcal{E}$ is an immersion give: 

\begin{claim} \label{claim: description of A and B}
There exists an identification $A\cap \mathcal{E}(U(u_*))\simeq D^a\times  D^{n-1}$ with respect to which $\mathcal{E}(u_*)=(0,0)$, $\mathcal{E}(Z\cap U(u_*))= D^a\times\{0\}$, and $sw(\mathcal{E}(u_*))$ is the unique point on $\bdry B$ on a neighborhood of $sw(\mathcal{E}(u_*))$.
\end{claim}

The fact that $\mathcal{E}(u_*)$ for $u_*$, regarded as an element of $\overline{\mathcal{H}}_{\kappa, \chi-1}(\bm y)$, coincides with $sw(\mathcal{E}(u_*))$ for $u_*$, regarded as an element of $\overline{\mathcal{H}}_{\kappa, \chi}(\bm y)$, follows from the fact that in the first case we apply $\nu_{[r]}$ a total of $(b+1)$ times and in the second case only $b$ times when computing the evaluation map; see~\eqref{eq: evaluation map}.

\s\n
{\em Step 4 (Description of $A'$)}  A standard gluing result for MFLS gives the following:

\begin{claim} \label{claim: description of A'}
Under the identification from Claim~\ref{claim: description of A and B}, $A'\cap  \mathcal{E}(U(u_*))= \{0\}\times r$, where $r$ is a radial ray in $D^{n-1}$ emanating from $0$.
\end{claim}

Finally note that the number of switchings along (I, II) is $1$ larger than the number along III, while $\chi(u_1)=\chi(u_2)+1$. 
Therefore the continuation between $(u_1,(\mathrm{I, II}))$ and $(u_2,\mathrm{III})$ does not change the exponent of $\hbar$ in Equation~(\ref{eq-F1}).
    
The remaining codimension $1$ degenerations correspond to breakings of $u$ or $\bm \Gamma$ and are treated elsewhere in the paper. This yields $d\circ\mathcal{F}^1=\mathcal{F}^1\circ d$.
\end{proof}

\subsubsection{Quasi-isomorphism}

\begin{proposition}
\label{prop-F1}
    $\mathcal{F}^1$ is a quasi-isomorphism, after base changing to $R\llbracket \hbar\rrbracket$.
\end{proposition}

\begin{proof}
    Observe that $\mathcal{F}^1$ induces a morphism between quotients $\hbar=0$:
    \begin{equation}
        \mathcal{F}^1|_{\hbar=0}\colon HF^*(\sqcup_i T_{q_i}^*M)|_{\hbar=0}\to HM^*(\Omega^{1,2}(M,\bm q))|_{\hbar=0},
    \end{equation}
    which is clearly a quasi-isomorphism, as it essentially coincides with the multi-loop version of the quasi-isomorphism introduced in \cite[Section 4]{abouzaid2012wrapped}; see \cite[Lemma 6.5]{honda2022higher} and \cite[Lemma 4.18]{KY} for more details.

    Next, since $\mathcal{F}^1$ is a deformation of $\mathcal{F}^1|_{\hbar=0}$ and the latter is a quasi-isomorphism, it is a standard fact of homological algebra that $\mathcal{F}^1$ is also a quasi-isomorphism; see e.g., \cite[Lemma 4.19]{KY} for a detailed proof.
\end{proof}

\begin{remark}
    We refer the reader to \cite{Hersovich2017spectral, Keller2006dgc, vfKreeke2023deformations, WvdB2015curvature} for a more in-depth discussion of the theory of deformations of $A_\infty$-algebras and $A_\infty$-categories. 
\end{remark}

\subsection{The higher-order maps $\mathcal{F}^d$} \label{subsection: higher-order maps}

Given $\vv{\bm y}=(\bm y_d,\dots,\bm y_1)\in CW^*(\sqcup_i T_{q_i}^*M)$, $d\geq 1$, we define
\begin{align}\label{eq: functor-equation}
    \mathcal{F}^d([\bm y_d],\dots,[\bm y_1])=\sum_{\substack{\bm \gamma,\,r\geq0 \\ (\bm u, \bm \Gamma) \in \mathcal{P}_r(\vv{y}, \bm \Gamma)}}(-1)^{\dagger(\bm u)+\dagger(\vv{\bm y})+\text{\ding{70}}(\bm u)}\mathcal{F}^d_{(\bm u, \bm \Gamma)}([\bm y_d], \ldots, [\bm y_1]), 
\end{align}
where $\boldsymbol{\gamma}\in CM_{-*}(\Omega^{1,2}(M,\bm q))$ and the terms on the right-hand side are explained below. We use the notation
$$\op{ind}(\vv{\bm y}, \bm \gamma, r) \coloneqq-(n-2)r+|\bm \gamma|-|\bm y_1|-\dots-|\bm y_d|+(d-1)$$
to denote the virtual dimension of the \emph{singular mixed moduli space} $\mathcal{P}_r(\vv{\bm y},\bm \gamma)$.
An element of $\mathcal{P}_r(\vv{\bm y},\bm \gamma)$, roughly speaking, counts holomorphic curves followed by Morse gradient trees, and the case of $m=6$ is schematically given in Figure \ref{fig-F}. In \cite{abouzaid2011plumbings} elements of such moduli spaces are called \emph{mushrooms}.
\begin{figure}[ht]
    \centering
    \includegraphics[width=6cm]{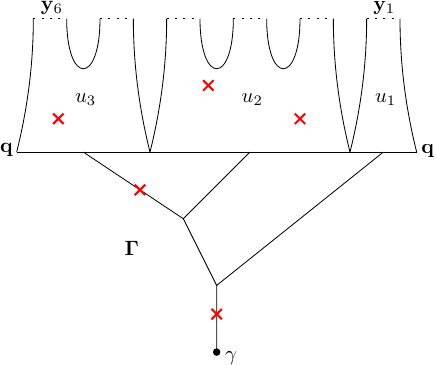}
    \caption{A schematic description of an element of $\mathcal{P}_5(\bm y_6,\dots,\bm y_1,\bm \gamma)$, represented by the tuple $(u_1, u_2, u_3, \bm \Gamma)$, where $u_1,u_2,u_3$ are holomorphic curves with a total of $3$ branch points and $\bm \Gamma$ denotes a Morse flow shrub with $2$ switchings.}
    \label{fig-F}
\end{figure}
An element of $\mathcal{P}_r(\vv{\bm y},\bm \gamma)$ is a tuple 
\begin{equation}
    (\bm u=(u_1, \dots, u_w), \bm \Gamma, \bm s)
\end{equation}
defined as follows:

\s\n
(A) $\bm s$ is a total ordering of the branch points and the switching markers that appear below.

\s\n
(B) Given a partition $d_1+\dots+d_w=d$, we partition $\vv{\bm y}=(\vv*{\bm y}{w}, \dots, \vv*{\bm y}{1})$ into $w$ vectors. Then $\bm u=(u_1, \ldots, u_w)$ is a $w$-tuple of holomorphic maps with
$$u_1 \in \mathcal{H}_{\kappa, \chi_1}(\vv*{\bm y}{1}), \dots, u_w \in \mathcal{H}_{\kappa, \chi_w}(\vv*{\bm y}{w}).$$

Let $b=(\kappa-\chi_1)+\dots+(\kappa-\chi_w)$ be the total number of the branch points of these curves. We require the restriction of $\bm s$ to the set of branch points of each $u_i$ to coincide with the one existing on it.

Letting $\op{ind}(u_i)$ be the virtual dimension of $\mathcal{H}_{\kappa, \chi_i}(\vv*{\bm y}{i})$, we define 
\begin{gather}
    \label{eq: dagger-curves-functor}
    \dagger(\bm u)=\sum_{i=1}^{w}i\cdot\op{ind}(u_i),\\
    \label{eq: dagger-inputs-functor}
    \dagger(\vv{\bm y})=\dagger(\vv*{\bm y}{1})+\ldots+\dagger(\vv*{\bm y}{w}),\\
    \label{eq: sign-correction-functor}
    \text{\ding{70}}(\bm u)= \sum_{i=1}^wd_i \operatorname{ind}(u_i).
\end{gather} 
Here \eqref{eq: dagger-inputs-functor} is analogous to ~\eqref{eq: dagger-Floer}.

Summing up the terms from~\eqref{eq: dagger-curves-functor}, \eqref{eq: dagger-inputs-functor} and \eqref{eq: sign-correction-functor} yields the sign twist in~\eqref{eq: functor-equation}.

\s\n
(C) The tuple $\bm \Gamma$ is a Morse flow \emph{shrub} with $\ell=r-b$ switches. Explicitly, there is a tree $(T, \bm \tau) \in \mathcal{T}^\ell_w$ and 
$$\bm \Gamma=(\bm \Gamma^{e_1}, \dots, \bm \Gamma^{e_w}, (\bm \Gamma^e)_{e \in E_{\op{int}}(T)}, \bm \Gamma^{e_0} )$$ 
is defined as in Definition~\ref{def: MFTS}, with only (1) altered and we explain this alteration below:

We have
    \begin{gather*}
    \bm \Gamma^{e_i}=((\theta_{s^{e_i}_0}^{e_i}, \Gamma_0^{e_i}, l_0^{e_i}),(\theta^{e_i}_{s^{e_i}_1}, \Gamma_1^{e_i}, l^{e_i}_1) \dots, (\theta_{s^{e_i}_{\ell_{e_i}-1}}^{e_i}, \Gamma^{e_i}_{\ell_{e_i}-1}, l^{e_i}_{\ell_{e_i}-1}), (\Gamma^{e_i}_{\ell_{e_i}}, l^{e_i}_{\ell_{e_i}})),\\ 
    \Gamma_j^{e_i} \colon [0, l^{e_i}_{j}] \to \Omega^{1,2}(M, \bm q^{i-1}, \bm q^{i}),
    \end{gather*}
for each $i=1, \dots, w$ and $j=0, \dots, \ell_{e_i}$, and we impose the switching conditions of Definition~\ref{def-diff}\eqref{item: switch-condition} dictated by $\bm \tau_{e_i}$. 
        
For each $i=1, \dots, d$, we have $\Gamma_0^{e_i}(0)=\mathcal{E}(u_i)$.

Finally, the distance $d(v_i, v)$ from each incoming vertex $v_i$ of the tree to the incoming vertex $v$ of the outgoing edge $e_0$ is independent of $i$. Here $d(v_i,v)$ is obtained by adding up all the lengths $l^e_j$ that appear along the shortest path from $v_i$ to $v$. See \cite[Definition 3.1]{abouzaid2011plumbings} for more details on the description of the moduli space of shrubs $S^{\ell}_w$ (with $\ell$ marked points).

\s
Given an element $(\bm u, \bm \Gamma, \bm s) \in \mathcal{P}_r(\vv{\bm y}, \bm \gamma)$ we have a canonical isomorphism
     \begin{equation}
        |\mathcal{P}_r(\vv{\bm y}, \bm \gamma)| \cong |S^{\ell}_w| \otimes o_{\bm \gamma} \otimes |\mathcal{H}_{\kappa, \chi_1}(\vv*{\bm y}{1})| \otimes \dots \otimes |\mathcal{H}_{\kappa, \chi_w}(\vv*{\bm y}{w})|
    \end{equation}

Combining it with isomorphisms of the form~\eqref{eq: orientation-half-disks} we obtain a natural map
\begin{equation}
    \mathcal{F}^d_{(\bm u, \bm \Gamma)} \colon o_{\bm y_1} \otimes \dots \otimes o_{\bm y_d} \to o_{\bm \gamma}[(n-2)r].
\end{equation}

\begin{lemma}
    There is a choice of universal Floer data $(J^{d, \kappa, \chi}, H^{d, \kappa, \chi})$ and universal tree perturbation data $\bm Y$ such that $\mathcal{P}_r(\vv{\bm y},\bm \gamma)$ is a smooth manifold of dimension $\op{ind}(\vv{\bm y},\bm\gamma,r)$. Moreover, when $\op{ind}(\vv{\bm y},\bm\gamma,r) \le 1$, it admits a compactification.
\end{lemma}

\begin{proof}
    The proof is essentially a combination of the proofs of Lemmas \ref{lemma-ainfty-regular} and \ref{lemma-f1-regular} and is omitted.
\end{proof}

\begin{proposition}
\label{prop-a-infty-map}
    $\mathcal{F}$ is an $A_\infty$-morphism, i.e., it satisfies
\begin{gather}\label{eq: A-infty-functor}
      \sum_{w} \sum_{s_1, \dots, s_w} \mu^w_{M}(\mathcal{F}^{s_w}(\bm y_d, \dots, \bm y_{d-s_w+1}), \dots, \mathcal{F}^{s_1}(\bm y_{s_1}, \dots, \bm y_1))= \\
     \sum_{k, d'}(-1)^{\text{\ding{64}}_n} \mathcal{F}^{d-d'+1}(\bm y_d, \dots, \bm y_{k+d'+1}, \mu^{d'}_F(\bm y_{k+d'}, \dots, \bm y_{k+1}), \bm y_k, \dots, \bm y_1). \nonumber
\end{gather}
\end{proposition}

\begin{proof}
We consider the moduli space $\mathcal{P}_r(\vv{\bm y},\bm \gamma)$ of dimension $1$ and study its boundary degenerations.

\s\n
(1) There is a portion of the boundary that has the form
$$\mathcal{P}_{r_w}(\vv*{\bm y}{w},\bm \gamma'_w) \times \dots \times \mathcal{P}_{r_1}(\vv*{\bm y}{1},\bm \gamma'_1) \times \mathcal{M}_{T}(\bm \gamma'_w, \dots, \bm \gamma'_1, \bm \gamma; \bm \tau'),$$
where $\bm \tau'$ has $\ell'$ elements and $r=r_1+\dots+r_w+\ell'$. This corresponds to a breaking on the shrub side, i.e., to the limit when the distance from the interior vertices to the unique interior vertex adjacent to $e_0$ goes to $+\infty$.
This corresponds to the left-hand side of Equation~\eqref{eq: A-infty-functor}. The example of $\mu^2(\mathcal{F}^5(\bm y_6,\dots,\bm y_2),\mathcal{F}^1(\bm y_1))$ is illustrated on the right-hand side of Figure \ref{fig-ainfty-relation}. We also refer the reader to \cite[Lemma 4.8]{abouzaid2011plumbings}.

\s\n
(2)
Another possibility corresponds to breaking of a target disk of some $u_i \in \mathcal{H}_{\kappa, \chi_i}(\vv*{\bm y}{i})$ for $i$ such that $1 \le i \le w$. Such strata have the form
$$\mathcal{M}_{\kappa, \chi'}(\bm y_{k+d'}, \dots, \bm y_{k+1}; \bm y') \times \mathcal{P}_{r'}(\bm y_d, \dots, \bm y_{k+d'+1}, \bm y', \bm y_k, \dots, \bm y_1, \bm \gamma)$$
and the elements of such strata are counted by the right-hand side of Equation~\eqref{eq: A-infty-functor}. The example of $\mathcal{F}^5(\bm y_6,\dots,\mu^2(\bm y_4,\bm y_3),\dots,\bm y_1)$ is illustrated on the left-hand side of Figure \ref{fig-ainfty-relation}.
\begin{figure}[ht]
    \centering
    \includegraphics[width=12cm]{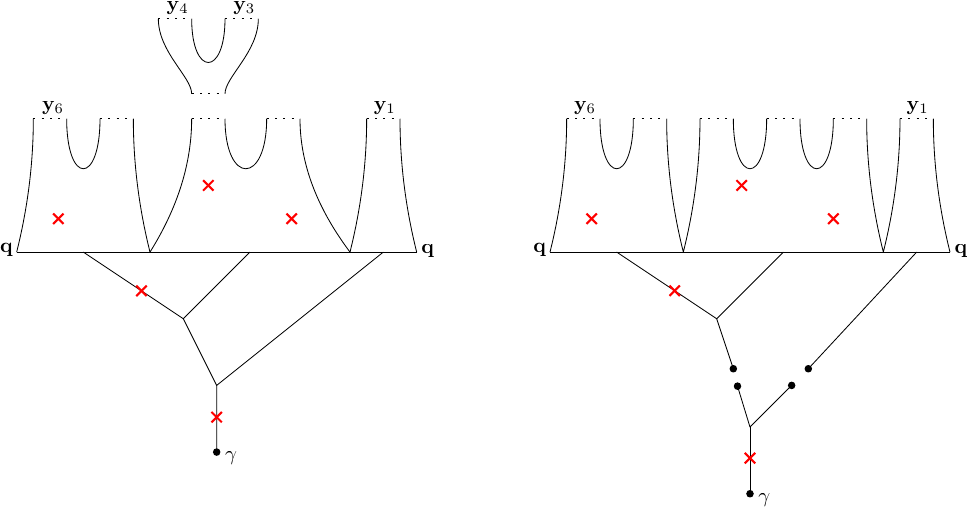}
    \caption{Some boundary strata of $\mathcal{P}_r(\vec{\bm y}, \bm \gamma)$.}
    \label{fig-ainfty-relation}
\end{figure}

\s\n
(3) Going from the left to the middle of Figure \ref{fig-ainfty-relation-2} corresponds to a boundary degeneration which is already seen on the level of $\mathcal{H}^d$, corresponding to the pinching along an arc connecting $\partial_{\op{out}}T$ to some $\partial_iT$.  Going from the right to the middle corresponds to the contraction of two incoming edges with a common outgoing vertex. Both of these gluing results are described in detail in a similar situation for $\kappa=1$ in \cite[Section 6.4]{abbondandolo2010floer} and we refer the reader to it for a detailed analysis.
\begin{figure}[ht]
    \centering
    \includegraphics[width=10cm]{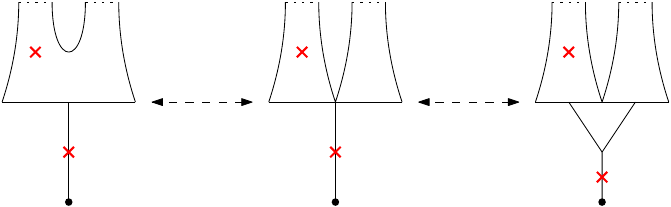}
    \caption{Breaking along the zero section.}
    \label{fig-ainfty-relation-2}
\end{figure}

\s\n
(4) The last possible scenario is a branch point on the $u_i$ side (for some $1 \le i \le w$) becoming a switching point on the shrub side (and vice versa). This is treated in Theorem~\ref{thm: chain map}.
    
The remainder of the proof consists of a very tedious sign verification that we omit here. We note that correctness of the sign $\dagger(\bm u)$ can be checked solely by analyzing breakings of type (2) and in an analogous situation it is carried out in \cite[Section 8.3]{abouzaid2011plumbings} (note that symbolically it is the same computation). The verification of the sign $\dagger(\vv{\bm y})+\text{\ding{70}}(u)$ amounts to the analysis of breakings of type (1) and is discussed in \cite[Section 8.4]{abouzaid2011plumbings}, although in loc.\ cit.\ the sign provided is just $\dagger(\vv{\bm y})$ and should be replaced with by $\dagger(\vv{\bm y})+\text{\ding{70}}(u)$.
\end{proof}

\begin{theorem}\label{thm: quasi-equivalence-Floer-to-Morse}
    $\mathcal{F}$ induces a quasi-equivalence of $\hbar$-adic completions.
\end{theorem}

\begin{proof}
    This follows from Propositions \ref{prop-F1} and \ref{prop-a-infty-map}. 
\end{proof}

\section{The based multiloop $A_\infty$-algebra of $S^2$} \label{section: algebraic model}

The goal of this section is to prove basic properties of the dga $H_\kappa$ from Definition~\ref{defn: Hn} in Section~\ref{subsection: basic properties of Hkappa} and to prove Theorem~\ref{thm: multiloop algebra for S2 is Hn} in Sections~\ref{subsection: verification}--\ref{subsection: direct limit argument}. 

We point out that given an $A_\infty$-algebra $(A, \{\mu^d\}_{d \ge 1})$ with vanishing $\mu^d$ for $d \ge 3$ there is a naturally associated dga $(A, d, \cdot)$ with differential and multiplication given via $$d(a)=\mu^1(a), \quad a*b=(-1)^{|a|}\mu^2(b,a).$$ In this section we call a tuple $(A, \mu^1, \mu^2)$ arising from an $A_\infty$-algebra $(A, \{\mu^d\}_{d \ge 1})$ as above a dga by abuse of notation, while understanding that to obtain an actual dga one needs to pass through this assignment. In particular, the dga $H_\kappa$ is obtained via such an assignment from $A_\infty$-algebra to which as we claim the $A_\infty$-algebra $CM_{-*}(S^2, \bm q)$ is quasi-equivalent. 

\subsection{Basic properties of $H_\kappa$} \label{subsection: basic properties of Hkappa}

The subalgebra $H_\kappa^{\operatorname{fin}}$ of $H_\kappa$ generated by the $T_i$ is isomorphic to the finite Hecke algebra associated to the symmetric group $S_\kappa$. 
In particular, $T_i^{-1}=T_i-\hbar 1$. The generator $x_1$ gives a homotopy between 
$$T_1T_2\cdots T_{\kappa-2}T_{\kappa-1}^2T_{\kappa-2}\cdots T_2T_1 \quad \mbox{and}\quad  1.$$  

\begin{remark}
    There is a generalization $H_\kappa(c)$ with a central parameter $c$, where the only difference is that the differential becomes 
\begin{equation} 
dx_1=T_1T_2\cdots T_{\kappa-2}T_{\kappa-1}^2T_{\kappa-2}\cdots T_2T_1-c.
\end{equation}
\end{remark}

\begin{lemma}
    The dga $H_\kappa$ is well-defined. 
\end{lemma}

\begin{proof}
    It suffices to verify that the differential $d$ preserves the defining relations. 
    This is obvious for the first three relations \eqref{first}, \eqref{second}, \eqref{third} since $dT_i=0$. 
    
    For the fourth relation \eqref{eq Hn x1Tj}, we have
    \begin{align*}
        d(x_1T_j)=&T_1T_2\cdots T_{\kappa-2}T_{\kappa-1}^2T_{\kappa-2}\cdots T_2T_1T_j-T_j \\
        =&T_1\cdots T_{\kappa-2}T_{\kappa-1}^2T_{\kappa-2}\cdots T_jT_{j-1}T_jT_{j-2} \cdots T_1- T_j \\
        =&T_1\cdots T_{\kappa-2}T_{\kappa-1}^2T_{\kappa-2}\cdots T_{j-1}T_{j}T_{j-1}T_{j-2} \cdots T_1- T_j \\
        =&T_1\cdots T_{j-2}T_{j-1}T_{j}T_{j-1} \cdots T_{\kappa-2}T_{\kappa-1}^2T_{\kappa-2}\cdots T_1- T_j \\
        =&T_1\cdots T_{j-2}T_{j}T_{j-1}T_{j} \cdots T_{\kappa-2}T_{\kappa-1}^2T_{\kappa-2}\cdots T_1- T_j \\
        =&T_jT_1\cdots\cdots T_{\kappa-2}T_{\kappa-1}^2T_{\kappa-2}\cdots T_1- T_j = d(T_jx_1).
    \end{align*}
    For the last relation \eqref{eq Hn x1x2}, let $T'=T_2\cdots T_{\kappa-2}T_{\kappa-1}^2T_{\kappa-2}\cdots T_2$ which commutes with $x_1$. We have
    \begin{align*}
        &d(T_1^{-1}x_1T_1^{-1}x_1+x_1T_1^{-1}x_1T_1)\\
        =&T_1^{-1}(T_1T'T_1-1)T_1^{-1}x_1-T_1^{-1}x_1T_1^{-1}(T_1T'T_1-1)\\ &+(T_1T'T_1-1)T_1^{-1}x_1T_1-x_1T_1^{-1}(T_1T'T_1-1)T_1 \\
        =&T'x_1-T_1^{-2}x_1-T_1^{-1}x_1T'T_1+T_1^{-1}x_1T_1^{-1}+T_1T'x_1T_1-T_1^{-1}x_1T_1\\
		 & \qquad \qquad \qquad -x_1T'T_1^2+x_1 \\
        =&(x_1-T_1^{-2}x_1)+(T_1^{-1}x_1T_1^{-1}-T_1^{-1}x_1T_1)+(T'x_1-x_1T'T_1^2)\\
 &  \qquad \qquad \qquad +(-T_1^{-1}x_1T'T_1+T_1T'x_1T_1) \\
        =&\hbar T_1^{-1}x_1+(-T_1^{-1}x_1 \hbar)+(-x_1T'\hbar T_1)+\hbar x_1T'T_1 = 0,
    \end{align*}
    where the last equality uses the quadratic relation $T_1^2-1=\hbar T_1, T_1-T_1^{-1}=\hbar$.
\end{proof}

There is an alternate formulation of $H_\kappa$ which is obtained by inductively introducing generators $x_2, \dots, x_\kappa$ as follows:
\begin{equation} \label{eq xk+1}
    x_{k+1}=T_k^{-1}x_kT_k^{-1}, \quad 1\le k \le \kappa-1.
\end{equation}
We have
\begin{gather} 
    T_kx_k=T_k^{-1}x_k+\hbar x_k=x_{k+1}T_{k}+\hbar x_k, \label{eq Hn Tkxk1} \\
    T_kx_{k+1}=x_kT_k^{-1}=x_kT_k-\hbar x_k, \label{eq Hn Tkxk2}\\
    T_kx_j=x_jT_k,\quad j\neq k,k+1. \label{eq Hn Tkxk3}
\end{gather}
The anticommutator relation (\ref{eq Hn x1x2}) becomes $x_2x_1+x_1x_2T_1^2=0$ and a direct computation shows that
\begin{align} \label{eq Hn xjxk}
    x_kx_j+x_jx_k(T_{j-1}\cdots T_1)(T_{k-1}\cdots T_2)T_1^2(T_2^{-1}\cdots T_{k-1}^{-1})(T_1^{-1}\cdots T_{j-1}^{-1})=0,
\end{align}
holds for $1 \le j < k \le \kappa$.

We now state and prove some properties of $H_\kappa$.  The first concerns the specialization of $H_\kappa$ at $\hbar=0$. 

\begin{lemma} \label{lemma: Hecke hbar is zero}
$H_\kappa|_{\hbar=0} \cong \mathbb{Z}[x_1,\dots,x_\kappa] \rtimes S_\kappa$,  
where $\mathbb{Z}[x_1,\dots,x_\kappa]$ denotes the graded polynomial ring, i.e., $x_kx_j+x_jx_k=0, j\neq k,$ and $|x_i|=-1$.
\end{lemma}

\begin{proof}
This follows from noting that (i) the subalgebra $H_\kappa^{\operatorname{fin}}|_{\hbar=0}$ is the group ring $\mathbb{Z}[S_\kappa]$; (ii) the differential becomes trivial since $dx_1=0$; and (iii) the anticommutator relation (\ref{eq Hn xjxk}) reduces to $x_kx_j+x_jx_k=0$ for $1\le j<k\le \kappa$.
\end{proof}

The second concerns the zeroth cohomology $H^0(H_\kappa)$. 

\begin{lemma}
There is an isomorphism of rings:
\begin{equation}H^0(H_\kappa) \cong \mathbb{Z}[\operatorname{Br}_\kappa(S^2)] / \{T_i^2=1+\hbar T_i\}.\end{equation}
Hence $H^0(H_\kappa)$ is isomorphic to the braid skein algebra of $S^2$. 
\end{lemma}

\begin{proof}
It is easy to see that $H^0(H_\kappa)$ is a quotient of the finite Hecke algebra $H_\kappa^{\operatorname{fin}}$ by the relation 
\begin{equation}T_1T_2\cdots T_{\kappa-2}T_{\kappa-1}^2T_{\kappa-2}\cdots T_2T_1-1=0.\end{equation}
By replacing $T_i$ with the half twist along the $i$th and $(i+1)$st points on the sphere $S^2$, the corresponding relation is indeed that of the defining relation of the braid group $\operatorname{Br}_\kappa(S^2)$ of $S^2$.
\end{proof}

Finally we present a linear basis of $H_\kappa$ of PBW type.

\begin{lemma} \label{lemma: PBW}
    The algebra $H_\kappa$ has a free basis $x^{\alpha}T_{w}$ over $\mathbb{Z}[\hbar]$, where:
\begin{itemize}
\item $x^{\alpha}=x_1^{\alpha_1}\cdots x_\kappa^{\alpha_\kappa}$, for $\alpha=(\alpha_1,\dots, \alpha_\kappa), \alpha_i \in \mathbb{Z}_{\geq 0}$, and
\item $T_w, w \in S_\kappa$ is a free $\mathbb{Z}$-basis of $H_\kappa^{\operatorname{fin}}$.
\end{itemize}  
\end{lemma}

\begin{proof}
    The relations \eqref{eq Hn Tkxk1}, \eqref{eq Hn Tkxk2}, and \eqref{eq Hn Tkxk3} imply that one can move $T_i$ from the left of $x_j$ to the right. 
    By the relation (\ref{eq Hn xjxk}) one can move $x_j$ in front of $x_k$ for all $j<k$ in a product of $x_i$'s. 
    Hence the collection $x^{\alpha}T_{w}$ spans $H_\kappa$. 
    
    The linear independence follows from observing that the collection $x^{\alpha}T_{w}|_{\hbar=0}$ is a free $\mathbb{Z}$-basis of $H_\kappa|_{\hbar=0}$.
\end{proof}

\subsection{Verification of \eqref{first}--\eqref{sixth}} \label{subsection: verification}

We begin the proof of Theorem~\ref{thm: multiloop algebra for S2 is Hn} by verifying \eqref{first}--\eqref{sixth} of Definition~\ref{defn: Hn} for the HDHF $A_\infty$-algebra for the standard round metric on $S^2$ and certain positions of endpoints. 

{\em We will use the fact that $\mu^3_M$ vanishes for certain tuples of total degree greater than $-3$.}  
In Section~\ref{subsection: the Katok examples} we show the vanishing of $\mu_M^d$, $d\geq 3$, for tuples of total degree greater than $-N\ll 0$, for a suitable Finsler modification of the metric, and a similar argument shows the vanishing of $\mu^3_M$ for the round metric on $S^2$ in the cases we need.  Then in Sections~\ref{subsection: completion} and \ref{subsection: direct limit argument} we construct a quasi-isomorphism of the HDHF $A_\infty$-algebra with the dga $H_\kappa$ by a direct limit argument.

Let $d$ be the standard round metric on $S^2$. For all the computations in Section~\ref{subsection: verification} we will use the standard energy functional
$$\mathcal{A}(\gamma)=\int_0^1\frac{1}{2}|\dot{\gamma}(t)|^2 dt.$$
The standard action functional with the round metric behaves poorly with products. Instead of perturbing $d$, we can get products to behave generically by taking disjoint copies ${\bm q}, {\bm q}', {\bm q}''$ of the basepoints while requiring that no $3$ points lie on a geodesic. In what follows, the points ${\bm q}, {\bm q}', {\bm q}''$ clearly violate this, but it is to be understood that we further make a small perturbation of ${\bm q}, {\bm q}', {\bm q}''$ so that no $3$ points are collinear.

\s\n
{\em Description of the basepoints.} Fix $\varepsilon>0$ small. Given points $a\not= b\in S^2$ that are not antipodal, let $g(a,b)$ be the shortest geodesic connecting $a$ and $b$. Pick basepoints ${\bm q}=(q_1,\dots,q_\kappa)\subset S^2$, arranged in order on a short geodesic $g$, and their pushoffs ${\bm q}'=(q_1',\dots, q_\kappa')$, ${\bm q}''=(q_1'',\dots, q_\kappa'')\subset S^2$, such that the following hold:
\begin{enumerate}[(P1)]
    \item \label{item: positions-p1} There exists $\tilde q\in S^2$ near $q_1$ such that $\tilde q, q_i, q_i', q_i''$ lie on a geodesic $g_i$ in that order for $i=1,\dots, \kappa$, moreover $g_1$ is orthogonal to $g$; in the actual perturbation we push each $q_i''$ to the right of the corresponding geodesic $g_i$. 
    \item Let $s_i$ be the point of intersection of $g(q_i, q_{i+1}')$ and $g(q_{i+1}, q_i')$. Then $d(q_i, s_i)>d(q_{i+1}, s_i)$ and $d(q_i, q_{i+1}')< d(q_{i+1}, q_{i}')$ for $i=1, \dots, \kappa-1$. An analogous condition holds for points of intersection of $g(q_{i}', q_{i+1}'')$ and $g(q_{i+1}', q_i'')$.
    \item \label{item: positions-p3} For $i=1, \dots , \kappa$ we have $d(q_i, q_i'), d(q_i, q_i'')< \varepsilon$, and for $i=1,\dots,\kappa-1$ we have $d(q_i, q_i')>d(q_{i+1}, q_{i+1}')$, $d(q_i, q_i'')>d(q_{i+1}, q_{i+1}'')$, 
    $\tfrac{d(\tilde{q}, q_{i+1})}{2\pi-d(q_{i+1}, q_{i+1}')} > \tfrac{d(\tilde{q}, q_i)}{2\pi-d(q_i, q_i')}$,
    $\tfrac{d(\tilde{q}, q_{i+1})}{2\pi-d(q_{i+1}, q_{i+1}'')} > \tfrac{d(\tilde{q}, q_i)}{2\pi-d(q_i, q_i'')}$. (Here $2\pi$ is the length of the equator and $2\pi-d(q_i, q_i')$ is the length of the second shortest geodesic from $q_i$ to $q_i'$.) 
\end{enumerate}
See Figure~\ref{fig: points}
\begin{figure}[ht]
	\begin{overpic}[scale=0.7]{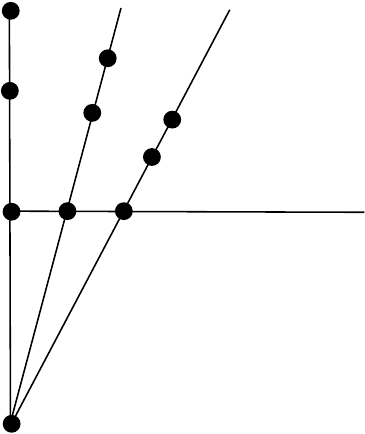}
	\put(-4,3) {\tiny $\tilde q$}
    \put(5,54.5){\tiny $q_1$}\put(5,82.3){\tiny $q_1'$} \put(5,100){\tiny $q_1''$}
    \put(19,54){\tiny $q_2$} \put(24.7,72){\tiny $q_2'$} \put(28,87){\tiny $q_2''$}
    \put(32.5,53.5){\tiny $q_3$} \put(38.5,63){\tiny $q_3'$} \put(43,72){\tiny $q_3''$}
    \end{overpic}
	\caption{The points of ${\bm q}$, ${\bm q}'$, and ${\bm q}''$.}
	\label{fig: points}
\end{figure}

We point out that we actually compute 
$$\mu_M^1 \colon CM_{-*}(\Omega(S^2, {\bm q}, {\bm q'})) \to CM_{-*}(\Omega(S^2, {\bm q}, {\bm q'})), $$
$$ \mu_M^2 \colon CM_{-*}(\Omega(S^2, {\bm q}, {\bm q'})) \otimes CM_{-*}(\Omega(S^2, {\bm q'}, {\bm q''}) )\to CM_{-*}(\Omega(S^2, {\bm q}, {\bm q''})).$$
The result is quasi-equivalent to the computations in the $A_\infty$-algebra $CM_{-*}(\Omega(S^2, {\bm q}))$ itself with the $A_\infty$-morphism induced by isotoping tuples of points ${\bm q'}$ and ${\bm q''}$ to ${\bm q}$. The elements of $CM_{-*}(\Omega(S^2, {\bm q}, {\bm q'}))$, $CM_{-*}(\Omega(S^2, {\bm q'}, {\bm q''}))$, and $CM_{-*}(\Omega(S^2, {\bm q'}, {\bm q''}))$ are formally identified via the obvious bijections. 

\s\n 
{\em Description of generators of $CM_{-*}(\Omega(S^2, \bm q))$ as a $\Z[\hbar]$-module.} We are viewing $CM_{-*}(\Omega(S^2, \bm q))$ as being generated by $\kappa$-tuples of geodesics from ${\bm q}$ to ${\bm q}'$.  

Let $[T_i]=T_i$ (by abuse of notation) be the generator of degree $0$ consisting of shortest geodesics $g(q_j, q_j')$ for $j\not=i, i+1$ and shortest geodesics $g(q_i, q_{i+1}')$ and $g(q_{i+1}, q_{i}')$ (note that since both of these geodesics are of index $0$ the orientation line $o_{T_i}$ is naturally trivialized). We have:
\be
\item[(P4)] When viewed as a braid in $S^2\times[0,1]$, $T_i$ is a positive half twist; see Figure~\ref{fig: braidT1}.
\ee 
This is due to the fact that all the geodesics are constant speed time-$1$ geodesics, and hence $g(q_i, q_{i+1}')$ passes through $s_i$ after $g(q_{i+1}, q_i')$ does by (P2);
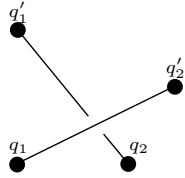
\begin{figure}[ht]
	\begin{overpic}[scale=0.7]{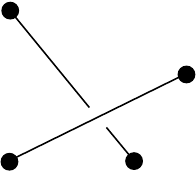}
    \put(0,13){\tiny $q_1$}\put(0,89){\tiny $q_1'$} 
    \put(69,13){\tiny $q_2$} \put(90,57){\tiny $q_2'$} 
    \end{overpic}
	\caption{$T_1$ viewed as a braid in $S^2\times [0,1]$.  The $[0,1]$-direction points out of the page.}
	\label{fig: braidT1}
\end{figure}

The degree $0$ summand $(CM_{-*}(\Omega(S^2, {\bm q}, {\bm q'})))_0$ of $CM_{-*}(\Omega(S^2, {\bm q}, {\bm q'}))$ is spanned by $\kappa$-tuples $T_\sigma$ of shortest geodesics from ${\bm q}$ to ${\bm q}'$ corresponding to $\sigma\in S_\kappa$.  The crossings of $T_\sigma$ are all positive by the same reason as above and hence the $T_\sigma$ form a standard basis for the finite Hecke algebra on $\kappa$ strands, consisting of shortest words in $T_i$ but not their inverses $T_i^{-1}$.  Once \eqref{first}--\eqref{third} are verified and we show that there are no other relations, it follows that $(CM_{-*}(\Omega(S^2, {\bm q}, {\bm q'})))_0$ is the finite Hecke algebra.  

Let $[x_i]=x_i$ be the generator of degree $-1$ consisting of the geodesic from $q_i$ to $q_i'$ which is contained in the equator through $q_i$ to $q_i'$ and almost makes a complete loop, together with $g(q_j,q_j')$ for $j\not=i$. Here we use the same orientation convention as in Example~\ref{ex: S2} and we denote positively oriented generator by the same letter as the path itself in what follows.

According to our sign conventions, for any $[\bm y]=[y_1] \otimes \dots \otimes [y_{\kappa}] \in CM_{-*}(\Omega(S^2, {\bm q}, {\bm q'}))$ and $[\bm y']=[y_1'] \otimes \dots \otimes [y_{\kappa}'] \in CM_{-*}(\Omega(S^2, {\bm q'}, {\bm q''}))$, the sign in the computation of $\mu_M^2([\bm y'], [\bm y])$ is equal to $(-1)^{\dagger}$, where
\begin{equation}\label{product-sign-convention}
\dagger=|\bm y|.
\end{equation}

\s\n 
{\em Verification of \eqref{first}--\eqref{third}.}  We verify \eqref{first}; the proofs of \eqref{second} and \eqref{third} are analogous but easier. We use the notation $\bm\gamma_0\sharp\bm\gamma_1$ 
to denote the concatenation of multipaths $\bm\gamma_0$ and $\bm\gamma_1$. Notice that since all the involved generators in~\eqref{first}--\eqref{third} are of degree $0$ the product associated with $\mu_M^2$ as described in the beginning of this section has the same signs.
First observe that since $T_i^2$ has degree $0$, the MFLS starting from $T_i\sharp T_i$ without switches eventually arrives at a collection of shortest geodesics corresponding to $1$. The case of an MFLS with a single switch would then contribute $\hbar T_i$. 

In general, let $\bm\gamma$ be a non-geodesic multipath for which there is a rigid (=isolated) MFLS from  $T_i\sharp T_i$ to $\bm\gamma$ satisfying the conditions of Definition~\ref{def-diff}, except that $\Gamma_\ell: [0,l_{\ell}]\to \Omega^{1,2}(S^2,\bm q,\bm q'')$ is a finite-length trajectory and $\Gamma_\ell(l_\ell)=\bm\gamma$.  Also let $\op{CR}(\bm\gamma)$ (short for {\em complete resolution} of $\bm\gamma$) be given by the analog of the right-hand side of Equation~\eqref{eqn: differential of multiloop complex}, where we are now counting (with $\hbar$-terms) rigid MFLS from $\bm\gamma$ satisfying the conditions of Definition~\ref{def-diff}, except that $\Gamma_0:[l_0,0]\to \Omega^{1,2}(S^2,\bm q,\bm q'')$ is a finite-length trajectory and $\Gamma_{0}(l_0)=\bm\gamma$. Observe that $\op{CR}(\bm\gamma)$ can be written uniquely as a $\Z\llbracket\hbar\rrbracket$-linear combination of $T_\sigma$, $\sigma\in S_\kappa$. 

Now suppose $\bm\gamma$ occurs before a crossing and $\bm\gamma'$ (resp.\ $\bm\gamma''$) is the multipath obtained by flowing without any crossings until one is reached, passing through the crossing (resp.\ resolving at the crossing), and flowing without any further crossings. Note that $\bm\gamma$, $\bm\gamma'$, and $\bm\gamma''$ are not necessarily generators of $CM_{-*}(\Omega(S^2, \bm q))$.  Let $[\cdot]_H$ denote the equivalence class in $\Z[\hbar,\hbar^{-1}]\otimes _{\Z} \Z B_\kappa$ modulo isotopy and the HOMFLY skein relations.  Then:
\be
\item[(*)] $[\bm\gamma]_H =[\bm\gamma']_H\pm \hbar [\bm\gamma'']_H$ and $\op{CR}(\bm\gamma)=\op{CR}(\bm\gamma') \pm \hbar\op{CR}(\bm\gamma'')$. 
\ee
Starting ``at the end'' with $\bm\gamma$, $\bm\gamma'$, and $\bm\gamma''$ such that $\bm\gamma'$ and $\bm\gamma''$ are geodesic, we have $\op{CR}(\bm\gamma')=\bm\gamma'$ and $\op{CR}(\bm\gamma'')=\bm\gamma''$ and hence $[\op{CR}(\bm\gamma)]_H=[\bm\gamma]_H$ by (*). Working backwards, we inductively establish that $[\op{CR}(\bm\gamma)]_H=[\bm\gamma]_H$ for any $\bm\gamma$. Finally, specializing to the case where there is an MFLS from $T_i\sharp T_i$ to $\bm\gamma$ with no crossings or switches, $T_i^2= [\op{CR}(\bm\gamma)]_H=[\bm\gamma]_H$, which implies \eqref{first}. 
\s\n
{\em Verification of \eqref{sixth}.}  As in Example~\ref{ex: S2}, in the calculation of $\mu^1_M(x_1)$ there are two MFLS from $x_1$ to $1$: one where the corresponding $1$-parameter family of paths from $q_1$ to $q_1'$ does not intersect any $q_i$ for $i>1$ (i.e., it is to the right of the geodesic of degree $-1$ from $q_1$ to $q_1'$) and another one that flows through a multipath $\bm\gamma$ consisting of a small loop  from $q_1$ to $q_1'$ (after isotoping $q_1'$ to $q_1$) which encircles all the $q_i$, $i>1$, together with short geodesics from $q_i$ to $q_i'$ for $i>1$; see Figure~\ref{fig: unstable-right-left}. The multipath $\bm\gamma$ is homotopic as a braid to $$T_1T_2\dots T_{\kappa-2}T^2_{\kappa-1}T_{\kappa-2}\dots T_2 T_1$$ 
and hence in view of (*) we obtain
$$\mu^1_M(x_1)=T_1T_2\dots T_{\kappa-2}T^2_{\kappa-1}T_{\kappa-2}\dots T_2 T_1-1.$$
This implies~\eqref{sixth}. 
\begin{figure}[ht]
    \centering
    \includegraphics[width=5cm]{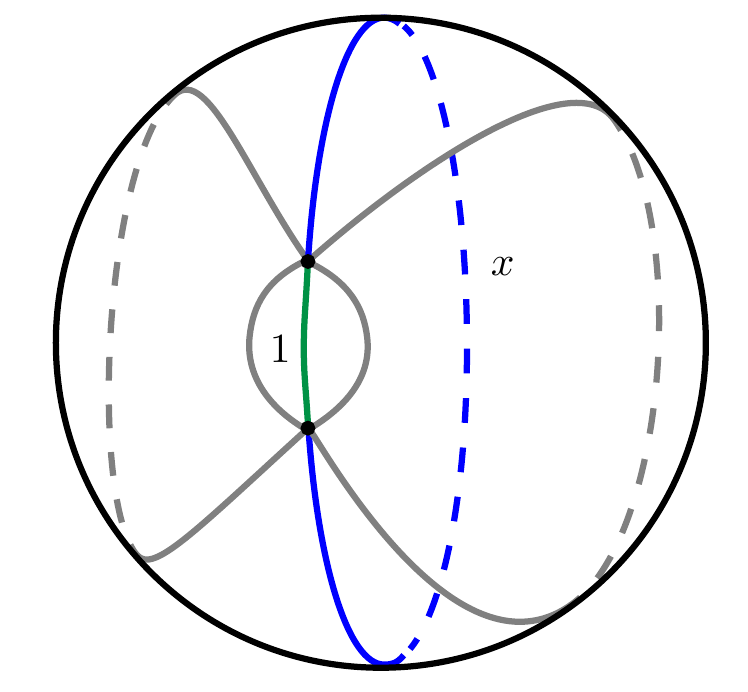}
    \caption{The generator $1$ is in green and the generator $x$ is in blue. The curves in the unstable manifold of $x$ are in gray. Observe that the gray curves to the right of $x$ are at the same time to the left of $1$ and are on the left side of this figure.}
    \label{fig: unstable-right-left}
\end{figure}

The relation $dT_i=0$ follows immediately by noting that $|T_i|=0$ and has the highest degree. 

\s\n
{\em Verification of \eqref{eq Hn x1Tj}.}  This follows from an explicit proof of the fact that $\mu^2_M(x, 1)=-\mu^2_M(1,x)=x$ in the usual $CM_{-*}(\Omega(S^2,q))$ with $\kappa=1$, where $x$ is a degree $-1$ generator and the sign change comes from~\eqref{product-sign-convention}: Referring to Figure~\ref{fig: unstable-right-left}, the compactification $\overline{W}_u(x)$ of the unstable manifold of $x$ is homeomorphic to $S^1$ and $\overline{W}_u(x)\setminus \{1,x\}$ consists of the ``left'' and ``right'' portions $W^L(x)$ and $W^R(x)$, depending on whether a given curve is contained in the left or right hemisphere with respect to the geodesic path $x$. Then the flow from $\gamma_L\in W^L(x)$ approaches $1$ from the right, while the flow from $\gamma_R\in W^R$ approaches it from the left. Now consider the flow starting from the concatenation $\gamma_L\sharp 1$. If $\gamma_L$ is sufficiently to the left, then the flow approaches $1$ from the right. Likewise, if $\gamma_R$ is sufficiently to the right, then the flow from $\gamma_R\sharp 1$  approaches $1$ from the left. Therefore, by continuity the algebraic count of trajectories from $\tilde{\gamma}\sharp 1$, $\tilde\gamma\in \overline{W}_u(x)\setminus \{1,x\}$, to $x$ is $1$ (for example, if there was no such flow this would contradict the continuity of the side from which the flow is approaching).

The flow tree starting from $x_1=x \otimes 1 \otimes \dots \otimes 1$ and $T_j=1 \otimes \dots \otimes 1 \otimes T^{\kappa=2}_j  \otimes 1 \otimes \dots \otimes 1$ decomposes into the flow tree for concatenating $1$ and $x$, viewed as generators for $\kappa=1$, and the flow tree corresponding to the concatenation of $1 \otimes 1$ and $T^{\kappa=2}_j$, viewed as generators for $\kappa=2$. The same can be said about flow tree with order of the above generators reversed, hence $\mu^2_M(x_1, T_j)=-\mu^2_M(T_j, x_1)$ implying~\eqref{eq Hn x1Tj} for the associated multiplication.

\s\n
{\em Verification of \eqref{eq xk+1}.} We show that $\mu^2_M(T_1, \mu^2_M(x_2,T_1))= -x_1$; the situation for larger $k$ is analogous. 
In the composition $\mu^2_M(x_2,T_1)$ there is a single concatenation of a geodesic of type $x$ with a short geodesic; in view of the multiplication in the usual $CM_{-*}(\Omega(S^2,q))$ with $\kappa=1$, the flow tree from the pair consisting of the type $x$ geodesic and the short geodesic limits to a geodesic of type $x$ from $q_2$ to $q_1'$.  There are no crossings in this case and hence no $\hbar$ term.  The same consideration also holds for $\mu_M^2(T_1, \mu^2_M(x_2,T_1))$, and hence $\mu^2_M(T_1, \mu^2_M(x_2,T_1))=- x_1$. This, given the associativity, implies~\eqref{eq xk+1} for the associated multiplication.  

\s\n
{\em Verification of \eqref{eq Hn x1x2}.} We claim that  
\begin{gather} \label{eqn: Hn x1x2 rephrased}
    x_2x_1+x_1x_2 T_1^2=x_2x_1 + x_1x_2 + \hbar x_1x_2 T_1=0,  \text{ or equivalently}  \\
    \mu^2_M(x_1, x_2)+\mu^2_M(T_1^2, \mu^2_M(x_2, x_1))=\mu_M^2(x_1,x_2)+\mu^2_M(1+\hbar T_1, \mu_M^2(x_2,x_1))=0. \nonumber
\end{gather}
Let $\bm\gamma$ be the multipath consisting of geodesics of type $x$ from $q_1$ to $q_1''$ and $q_2$ to $q_2''$, denoted $x_{11}$ and $x_{22}$, together with short geodesics from $q_i$ to $q_i''$ for $i\geq 2$. We fix $o_{\bm \gamma} \cong o_{x_{11}} \otimes o_{x_{22}}$ and the generator $[\bm \gamma]=\bm \gamma$ is identified with $[x_{11}] \otimes [x_{22}]$. 

Let $r_1=\tilde q$ and $r_2$ be the first and second intersection points of $x_{11}$ and $x_{22}$ viewed in the projection to $S^2$.
By the positioning of the points ${\bm q}$ and ${\bm q}''$ given by (P3) one show that (verification left to the reader):
\be
\item[(P5)] If we view $\bm\gamma$ as a braid in $S^2\times[0,1]$, then $x_{11}$ lies below $x_{22}$ at both $r_1$ and $r_2$. 
\ee 
See Figure~\ref{fig: x1x2}, top right.
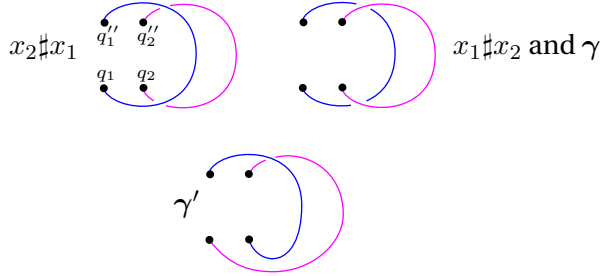
\begin{figure}[ht]
	\begin{overpic}[scale=0.35]{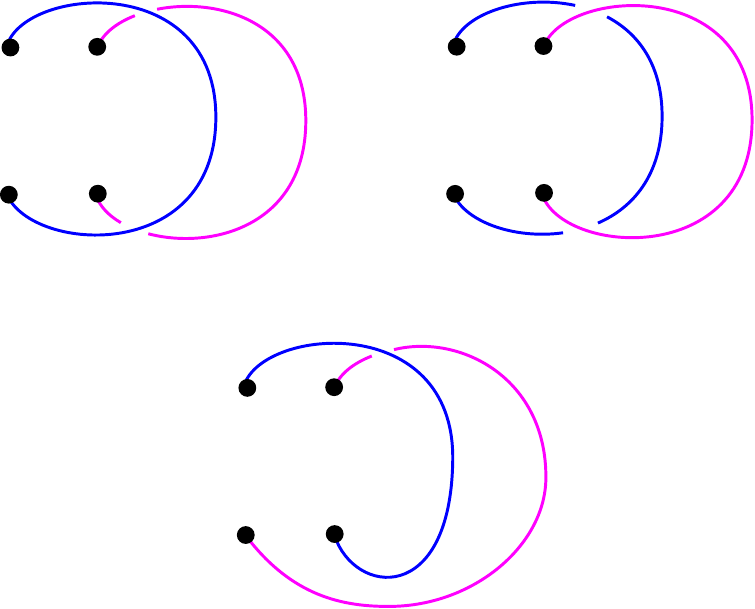}
	\put(-27,65){$x_2\sharp x_1$} \put(105,65){{$x_1\sharp x_2$} and $\bm\gamma$}  \put(22,18){$\bm\gamma'$}
    \put(-1,58){\tiny $q_1$}\put(-1,69){\tiny $q_1''$} 
    \put(11,58){\tiny $q_2$} \put(11,69){\tiny $q_2''$} 
    \end{overpic}
	\caption{A schematic representation of the concatenations of $x_1$ and $x_2$ (top left) and $x_2$ and $x_1$ (top right), and the multipath $\bm\gamma'$ (bottom) in the projection of $S^2\times[0,1]$ to $S^2$. Here we show $S^2-\{pt\}=\R^2$. The blue and purple are geodesics of type $x$. The $[0,1]$-direction points out of the page.} 
	\label{fig: x1x2}
\end{figure}
Similarly, let $\bm\gamma'$ be the multipath consisting of geodesics of type $x$ from $q_1$ to $q_2''$ and $q_2$ to $q_1''$, denoted $x_{12}$ and $x_{21}$, together with short geodesics from $q_i$ to $q_i''$ for $i\geq 2$. By the positioning of the points ${\bm q}$ and ${\bm q}''$ given by (P2), $x_{12}$ lies below $x_{21}$.

We have $\mu^2_M(x_2,x_1)=-\bm\gamma$, since the braid types of $\bm \gamma$ and $x_1 \sharp x_2$ are the same and hence no switchings occur on the MFTS from the pair $(x_1,x_2)$ to $\bm\gamma$. Observe that when one evaluates $\mu^2_M(x_2, x_1)$ the concatenation may happen after applying a negative gradient from $x_{11}$ and $x_{22}$ for a certain amount of time, but this does not change the result of our computation. The sign comes from~\eqref{product-sign-convention}.

The proof of $\mu_M^2(x_1,x_2)=\bm \gamma+\hbar \bm\gamma'$ is involved and will occupy the next several pages. Assuming this, we proceed as follows:
We claim that $\mu^2_M(T_1, \bm\gamma)= \bm\gamma'$ and $\mu^2_M(T_1, \bm\gamma')=\bm \gamma+ \hbar \bm\gamma'$. The first claim is straightforward and the second is computed as follows, assuming associativity:
$$\mu^2_M(T_1, \bm\gamma')=\mu^2_M(T_1, \mu^2_M(T_1, \bm \gamma))=\mu^2_M(1+\hbar T_1, \bm\gamma)=\bm\gamma+\hbar \bm\gamma'.$$
Combining all of the above, we obtain:
$$\mu^2_M(x_1,x_2) + \mu^2_M(1+\hbar T_1,\mu^2_M(x_2,x_1))=(\bm\gamma+ \hbar \bm \gamma')+\mu_M^2(1+\hbar T_1,-\bm\gamma) =(\bm \gamma+\hbar \bm \gamma')-\bm \gamma-\hbar \bm \gamma'=0.$$

\s\n
{\em Verification of $\mu_M^2(x_1,x_2)=\bm \gamma+\hbar \bm\gamma'$.}  Clearly there is an MFTS from $(x_1,x_2)$ without switchings to $\bm\gamma$. The sign is positive now because in the canonical isomorphism~\eqref{eq: tree-orientation} for this MFTS the order of $o_{x_1}$ and $o_{x_2}$ is switched as opposed to the one associated with the only MFTS contributing to $\mu^2_M(x_2, x_1)$. 

We now show that there is an MFTS with $1$ switching from $(x_1,x_2)$ to $\bm\gamma'$. Let $\tilde{\gamma}_2$ be a curve in the unstable manifold of $x_{22}$ (connecting $q_2$ and $q_2'$) such that there is a flow line from $\tilde{\gamma}_2\sharp1$ to $x_{22}$ (here we are abusing notation to denote the analogous curve connecting $q_2$ and $q_2''$).  We consider the portion of the unstable manifold of $x_{22}$ which is to the left of $\tilde{\gamma}_2$ and parametrize it by $[0,1]_{z_2}$, where $z_2=0$ corresponds to $\tilde{\gamma}_2$ and $z_2=1$ corresponds to the index $0$ generator $1$ connecting $q_2$ and $q_2''$. Similarly, we parametrize by $[0,1]_{z_1}$ the portion of the unstable manifold of $x_{11}$ (connecting $q_1'$ and $q_1''$) which is to the right of the curve $\tilde{\gamma}_1$, such that there is a flow line from $1\sharp\tilde{\gamma}_1$, with $z_1=0$ corresponding to $\tilde{\gamma}_1$. We will denote by $\gamma_1(z_1)$ and $\gamma_2(z_2)$ the paths corresponding to the values $z_1$ and $z_2$ with respect to these parametrizations.

\begin{figure}[ht]
    \centering
    \includegraphics[width=5.5cm]{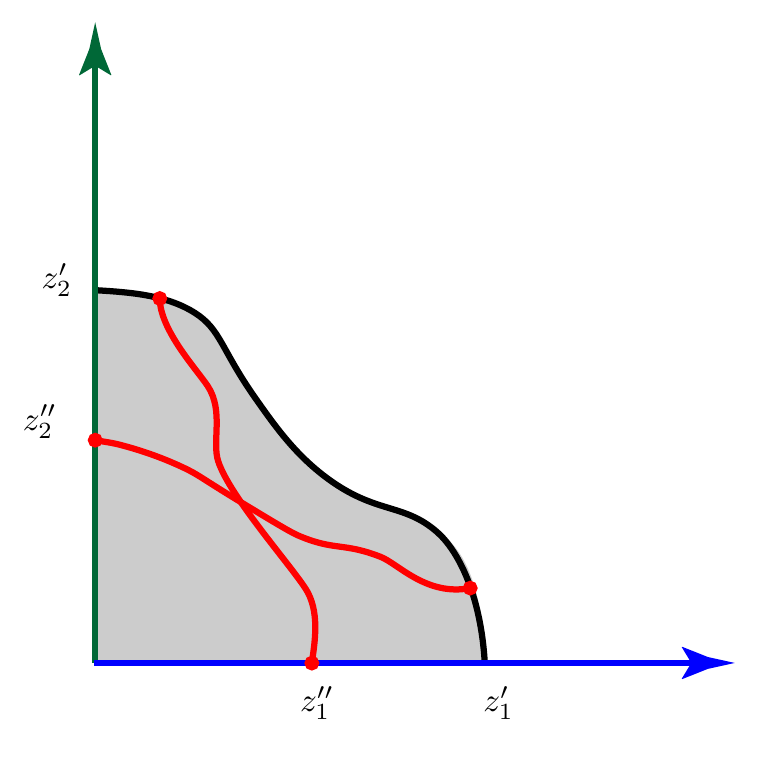}
    \caption{The bottom (blue) side is the $z_1$-coordinate axis and the vertical (green) side is the $z_2$-coordinate axis. The shaded gray region is $U_{\tilde{q}}$ and the red curves depict the sets $S_1$ and $S_2$.}
    \label{U_q}
\end{figure}

Let $\tilde{q}$ denote the intersection point of the images of $\tilde{\gamma}_1$ and $\tilde{\gamma}_2$ that is closer to $q_1$ and $q_2$ (see Figure~\ref{lr-gradient}, where $\tilde \gamma_i'$ is replaced by $\tilde \gamma_i$ and $q_i''$ is replaced by $q_i'$). 

\begin{claim} \label{claim: z1 and z2}
For sufficiently small $z_1$ and $z_2$, there exists a moment $s=s_0$ on the Morse flow line $\bm\Gamma: [0,\infty)\to \Omega(S^2,(q_1,q_2),(q_1'',q_2''))$ starting at $(1, \gamma_2(z_2))\sharp(\gamma_1(z_1), 1)$, when a crossing of the two strands occurs.  Moreover this crossing is {\em derived from} the initial intersection point $\tilde{q}$.
\end{claim}

By {\em derived from} $\tilde{q}$ we mean that there is a continuous family of intersection points of $\op{Im}\bm\Gamma(s)$, $[0,s_0]$, which connects $\tilde{q}$ and such a crossing. 

The key point is that $\tilde\gamma_1,\tilde\gamma_2$ intersect {\em asynchronously} at $\tilde q\in \op{Im}(\tilde \gamma_1)\cap \op{Im}(\tilde\gamma_2)$, i.e., $t_1\not=t_2$ where $\tilde\gamma_1(t_1)=\tilde\gamma_2(t_1)=a$, whereas the claim guarantees the existence of some $s=s_0$ where the intersection of the two paths of $\bm\Gamma(s_0)$ is {\em synchronous}.

\begin{proof} 
Observe that the flow line $\bm\Gamma$ ends at $\bm \gamma$. Along $\bm\Gamma$ the braid representative goes from the one on the top left to the one on the top right of Figure~\ref{fig: x1x2}, implying a crossing. 
This crossing is derived from $\tilde q$ and hence the claim for $(z_1,z_2)=(0,0) \in [0,1]_{z_1} \times [0,1]_{x_2}$ holds. By continuity there exists a crossing derived from $\tilde{q}$ for each $(z_1,z_2)$ in a small neighborhood of $(0,0) \in [0,1] \times [0,1]$.
\end{proof} 

Let $U_{\tilde{q}}$ be the connected subset consisting of all $(z_1,z_2)\in [0,1]\times [0,1]$ satisfying Claim~\ref{claim: z1 and z2}. For each $(z_1, z_2) \in U_{\tilde{q}}$ there exists an MFTS with $1$ switching derived from $\tilde{q}$. Let $\bm\Gamma(z_1,z_2)$ be the flow line starting at $(z_1,z_2)$ and let $s=s_0$ be the moment given by Claim~\ref{claim: z1 and z2} when the crossing derived from $\tilde{q}$ occurs.  We write $\bm\Gamma(z_1,z_2)(s_0)=(\tilde{\gamma}'_1(z_1,z_2), \tilde{\gamma}'_2(z_1,z_2))$.  Let $(\tilde{\gamma}''_1(z_1,z_2), \tilde{\gamma}''_2(z_1,z_2))$ be the result of applying the switching map; then $\tilde{\gamma}''_1(z_1,z_2)$ (resp.\ $\tilde{\gamma}''_2(z_1,z_2) $) is approximately the concatenation $a_1(z_1,z_2)\sharp b_2(z_1,z_2)$ (resp.\ $a_2(z_1,z_2)\sharp b_1(z_1,z_2)$), where we write $\tilde\gamma'_i(z_1,z_2)= a_i(z_1,z_2)\sharp b_i(z_1,z_2)$.
\begin{figure}[ht]
    \centering
    \includegraphics[width=6cm]{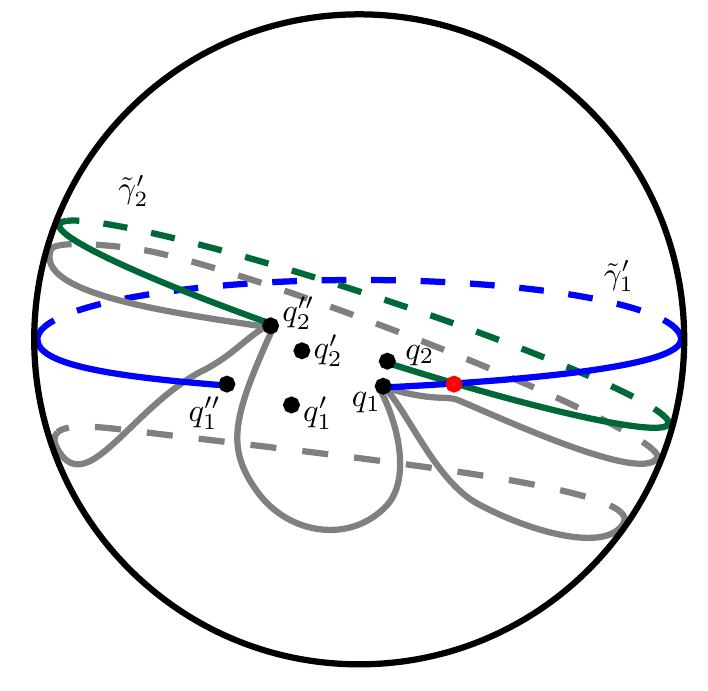}
    \caption{The curves $\tilde{\gamma}'_1$ (in blue) and $\tilde{\gamma}'_2$ (in green) intersect at the red dot, which represents the crossing derived from $\tilde q$. The curves in gray show various stages of the flow line starting at $\tilde{\gamma}''_1$.}
    \label{lr-gradient}
\end{figure}

We also claim that there exists a unique value $z_1'$ such that for $z_1>z_1'$ there is no such MFTS with $1$ switching starting at $(z_1, 0) \in [0,1] \times [0,1]$. Similarly there is a unique $z_2'$ with analogous property for the second coordinate. Moreover, on each ray starting from $(0,0)$ there is a unique point satisfying analogous property and $U_{\tilde{q}}$ has the triangular shape as in Figure~\ref{U_q}. The sides of this triangle along the $z_1$- and $z_2$-axes are denoted $I_1$ and $I_2$ and the hypotenuse is denoted $I_3$.

Following the proof strategy of $\mu^2_M(x,1)=x$ (see the paragraph on the verification of \eqref{eq Hn x1Tj} above), we now analyze how the various MFTS approach the generator $(1,1)$, represented by geodesics $g(q_1, q_2'')$ and $g(q_2, q_1'')$. For each of these generators there are two options: either a flow line approaches it from the right, or from the left. Hence there are $4$ possible types of flow lines from $(x_1,x_2)$ which we denote by $rr, rl, lr, ll$, where the first (resp.\ second) letter corresponds to the geodesic $g(q_1, q_2'')$ (resp.\ $g(q_2, q_1'')$). 

\begin{claim} \label{claim: lr}
    The pair $((1, \tilde{\gamma}_2), (\tilde{\gamma}_1, 1))$ corresponding to $(z_1,z_2)=(0,0)$ is of type $lr$.
\end{claim} 

\begin{proof}
    Let $(\tilde{\gamma}'_1, \tilde{\gamma}'_2)=(\tilde{\gamma}'_1(0,0), \tilde{\gamma}'_2(0,0))$; see Figure~\ref{lr-gradient}. Note that $\tilde{\gamma}'_1$ and $\tilde{\gamma}'_2$ are very close to the generators $x_1$ and $x_2$.  Applying the switching map, $\tilde{\gamma}_1''=\tilde{\gamma}''_1(0,0)$ and $\tilde\gamma_2''=\tilde{\gamma}''_2(0,0)$ are approximately concatenations $a_1(0,0)\sharp b_2(0,0)$, $a_2(0,0)\sharp b_1(0,0)$ which are very close to two geodesics that intersect at an angle slightly less than $\pi$. The gradient vector at $\tilde{\gamma}_1'$, regarded as a vector field along $\tilde{\gamma}_1'$, has a relatively small norm away from this angle, whereas near the angle its norm is significantly greater than $0$ and the vector points into the interior of this angle; see \cite[Chapter 12]{milnor2016morse} for a more precise description. Hence $\tilde{\gamma}_1''$ is pushed away to the right along the flow line as shown in Figure~\ref{lr-gradient} and the flow line approaches $g(q_1, q_2'')$ from the left. Similarly, the flow line starting at $\tilde{\gamma}_2''$ approaches $g(q_2, q_1'')$ from the right.
\end{proof}

\begin{claim} \label{claim: ll and rr}
    The points on $I_1$ near $(z_1',0)$ are of type $ll$ and the points on $I_2$ near $(0, z_2')$ are of type $rr$.
\end{claim}

\begin{proof}
    We prove the first assertion; the second is similar. Let $(\tilde{\gamma}'_1,\tilde{\gamma}'_2)=(\tilde{\gamma}'_1(z_1,0),\tilde{\gamma}'_2(z_1,0))$. The curve $\tilde{\gamma}_2''=\tilde{\gamma}_2''(z_1,0)$ is approximately the concatenation $a_2(z_1,0)\sharp b_1(z_1,0)$. Since the curve $\tilde{\gamma}'_1$ is ``sufficiently'' far away from the equator, the pseudogradient vector field has ``sufficiently'' large norm away from the ends, so even though $\tilde{\gamma}_2''$ has acceleration near the crossing pointing into the angle, it is not enough to pull the curve in that direction. One may also see this by arranging that the curve $\tilde{\gamma}_2''$ is contained entirely in a hemisphere cut out by the equator passing through $q_2$ and $q_1''$. Figure~\ref{fig: ll_gradient} gives the flow line for this case. 
\end{proof}

\begin{figure}[h]
    \centering
    \includegraphics[width=6cm]{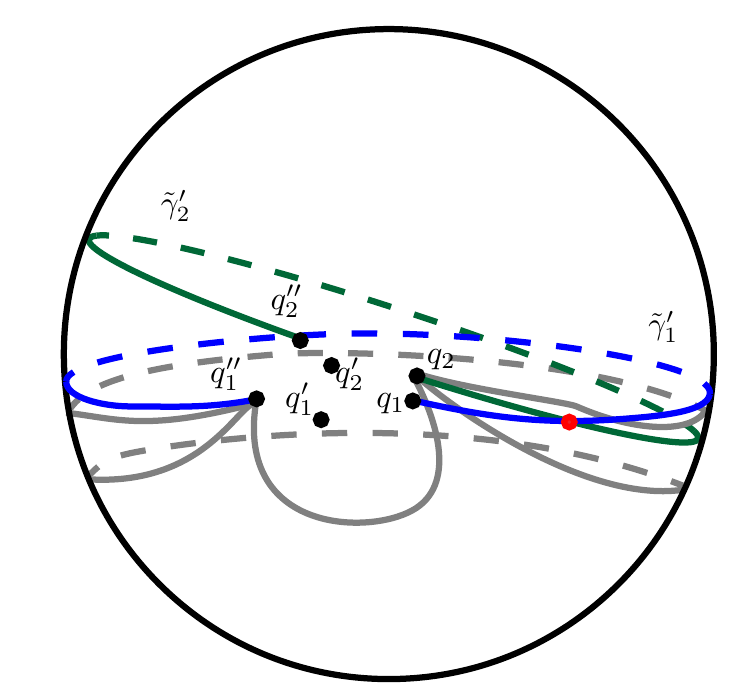}
    \caption{We use the same conventions as in Figure~\ref{lr-gradient}. The curves in gray show the various stages of the flow line starting at $\tilde{\gamma}_2''$ obtained from $(\tilde\gamma_1',\tilde{\gamma}_2')$ via the switching map.}
    \label{fig: ll_gradient}
\end{figure}

By Claims~\ref{claim: lr} and \ref{claim: ll and rr}, along $I_1$ (resp.\ $I_2$) there is a unique point $(z_1'', 0)$ (resp.\ $(0, z_2'')$) such that corresponding MFTS flows to the generator $x_2'=(1, x)$ (resp.\ $x_1'=(x, 1)$) of index $(-1)$ and permutation type $(12)$. Let $S_1$ (resp.\ $S_2$) be the subset of $U_{\tilde{q}}$ consisting of points $(z_1, z_2)$ such that corresponding MFTS approaches either $x_2'$ (resp.\ $x_1'$) or $\bm \gamma'$. Then $(z_1'',0) \in S_1$ and $(0,z_2'') \in S_2$. 

Now we analyze the third side $I_3$, which is given by the graph of a continuous monotone function. The defining condition for a point $(z_1, z_2)$ to belong to $I_3$ is that at the switching moment the corresponding curves $\tilde{\gamma}'_1$ and $\tilde{\gamma}'_2$ are tangent at the crossing derived from $\tilde{q}$.  Figure~\ref{fig: tangency} shows the boundary points $(z_1',0)$ and $(0, z_2')$ of $I_3$, which are of types $ll$ and $rr$ by Claim~\ref{claim: ll and rr}. 

\begin{claim}
    Along $I_3$ there is a unique point $P_1\in S_1$ and a unique point $P_2\in S_2$ such that $z_1(P_1)<z_1(P_2)$.
\end{claim} 

\begin{proof}
    For each $(z_1,z_2)\in I_3$, the point of tangency of $\tilde{\gamma}'_1(z_1,z_2)$ and $\tilde{\gamma}'_2(z_1,z_2)$ is closer to $q_1$ and $q_2$ than to $q_1''$ and $q_2''$. For this reason, there exists $(z_1,z_2)\in I_3$ where the average acceleration (i.e., the average norm of the pseudogradient vector) of $\tilde{\gamma}_1''(z_1,z_2)$ along $a_1(z_1,z_2)$ is smaller than that of $\tilde{\gamma}_1''(z_1,z_2)$ along $b_2(z_1,z_2)$, but the flow approaches $g(q_1, q_2'')$ from the right. Here one should think of the acceleration vector along $a_1(z_1,z_2)$ as pointing downwards and $b_2(z_1,z_2)$ as pointing upwards in Figure~\ref{fig: tangency}. Similarly, the flow from $\tilde{\gamma}_2''(z_1,z_2)$ approaches $g(q_2, q_1'')$ from the left. Hence a segment of $I_3$ consists of points of $rl$ type, which implies the claim.
\end{proof}

Finally we claim that each horizontal line below $z_2=z_2(P_1)$ intersects $S_1$ once and each vertical line to the left of $z_1=z_1(P_1)$ intersects $S_2$ once inside $U_{\tilde{q}}$, as shown in Figure~\ref{U_q}. This can be shown by an argument similar to our analysis of MFTS corresponding to $(z_1',0)$ and we omit it here. Hence $S_1$ and $S_2$ are as shown in Figure~\ref{U_q} and intersect at a unique point. This point corresponds to an MFTS with $1$ switch converging to $\bm \gamma'$. By a similar analysis it is the only MFTS with $1$ switch starting at $(x_2, x_1)$ and hence $\mu^2_M(x_1,x_2)=\bm\gamma+\hbar \bm \gamma'$.

\begin{figure}[ht]
    \centering
    \includegraphics[width=12cm]{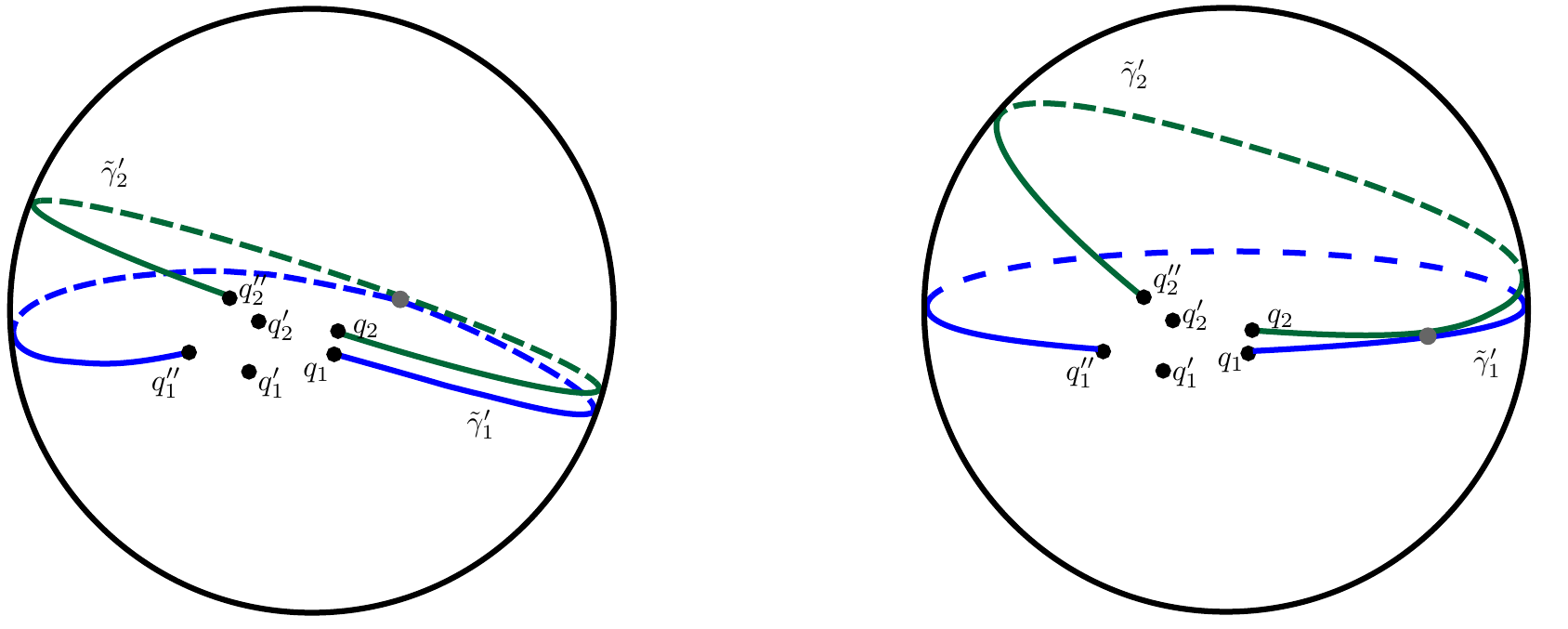}
    \caption{The pair $(\tilde{\gamma}_1',\tilde{\gamma}_2')$ at switching moments on the MFTS corresponding to $(z_1', 0)$ (on the left) and to $(0, z_2')$ (on the right). The gray dots are points of tangency.}
    \label{fig: tangency}
\end{figure}

\subsection{The Katok examples and the vanishing of the higher $A_\infty$-operations} \label{subsection: the Katok examples}

Continuing with the proof of Theorem~\ref{thm: multiloop algebra for S2 is Hn}, in order to show the vanishing of the higher $A_\infty$-operations, we now switch from the standard metric on $S^2$ to a certain family $F_\lambda$, $\lambda\in[0,1)$, of Finsler metrics on $S^2$, often called the ``Katok examples'' (see \cite{Katok-examples}). The explicit formula for $F_\lambda$, $\lambda \in [0,1)$, is given by
\begin{equation}
    F_{\lambda}=\frac{\sqrt{\left(1-\lambda^2 \sin ^2 r\right) d r^2+\sin ^2r d \phi^2}-\lambda \sin ^2 r d \phi}{1-\lambda^2 \sin ^2 r},
\end{equation}
in the standard polar geodesic coordinates on $S^2$ with $(r, \phi) \in(0, \pi) \times[0,2 \pi]$. Let $p=(\tfrac{\pi}{2}, 0)$ and $q_\varepsilon=(\tfrac{\pi}{2}, 2\pi \varepsilon)$, with $\varepsilon>0$ small, be points on the equator $r=\tfrac{\pi}{2}$. 

The following {\em approximate monotonicity} property holds, which in turn will enable us to show that $\mu^d=0$ for $d\geq 3$.

\begin{lemma}[Approximate monotonicity]\label{lemma: approximate monotonicity}
    Given a sequence $N_i\to \infty$, there exist sequences $\lambda_i\to 1$ with $\lambda_i$ irrational, and $\varepsilon_i\to 0$ such that:
    \be
    \item for all $F_{\lambda_i}$-geodesics $\gamma$ from $p$ to $q_{\varepsilon_i}$ of Morse index $\leq N_i$, the Morse index $\op{ind}(\gamma)$ is approximately proportional to the $F_{\lambda_i}$-length $\mathcal{L}_{F_\lambda}(\gamma)$, i.e., there exists $C_i$ close to $\pi$ and $\varepsilon'>0$ small such that $|C_i\cdot \op{ind}(\gamma)- \mathcal{L}_{F_\lambda}(\gamma)|< \varepsilon'$; and
    \item all other $F_{\lambda_i}$-geodesics have $F_{\lambda_i}$-length $> C_i N_i$.
    \ee
\end{lemma}

\begin{proof}
Let $\gamma_-$ and $\gamma_+$ be the two simple $F_\lambda$-geodesics from $p$ to $q_{\varepsilon}$ along the equator and $\gamma_{\pm}^m$ be the $F_\lambda$-geodesic from $p$ to $q_\varepsilon$ in the direction of $\gamma_{\pm}$ which passes $m$ times through the antipodal point of $p$, such that $\gamma_{-}=\gamma_{-}^0$ and $\gamma_{+}=\gamma_{+}^1$.
The $F_\lambda$-lengths of the geodesics are given by 
\begin{equation}\label{Finsler-action}
    \mathcal{L}_{F_{\lambda}}(\gamma_-^m)=\frac{2\pi(m+\varepsilon)}{1+\lambda}, \quad \mathcal{L}_{F_{\lambda}}(\gamma_+^m)=\frac{2\pi(m-\varepsilon)}{1-\lambda}.
\end{equation}

The Morse index of $\gamma_\pm^m$ is given by the number of points that are conjugate to $p$ along $\gamma_\pm^m$. The $\phi$-angle distance of $\gamma_-^m$ is equal to $2\pi(m+\varepsilon)$ and by \cite[p.145, l.-11]{Ziller-Katok-examples} the $\phi$-angle distance between two consecutive conjugate points along $\gamma_-^m$ is equal to $\pi (1+\lambda)$. Therefore,
\begin{equation}\label{eq: ind-small_m}
    \op{ind}(\gamma_-^m)=\left \lfloor\tfrac{2(m+\varepsilon)}{1+\lambda} \right\rfloor.
\end{equation}
Similarly we have:
\begin{equation}\label{eq: ind-big_m}
    \op{ind}(\gamma_+^m)=\left \lfloor\tfrac{2(m-\varepsilon)}{1-\lambda} \right \rfloor.
\end{equation}

Now \cite[Theorem 2 and Proposition 6]{robles2007} with $W=\bdry_\phi$ and $\sigma=0$ (using the notation there), supplies a correspondence between geodesics and conjugate points on geodesics starting at $p$ with respect to the round metric and those with respect to $F_\lambda$ (also see \cite[p.1642,l.-3]{robles2007} for an explicit description of the geodesics). Using this correspondence one can show that the points where geodesics starting at $p$ intersect the equator $r=\tfrac{\pi}{2}$ are exactly the conjugate points along $\gamma_\pm^m$, which is equal to:
\begin{equation}
    \op{Conj}_p=\left\{(\tfrac{\pi}{2}, \varphi^k_\pm) \in S^2 \mid   \varphi^k_\pm = \pi(1\pm \lambda) k \op{  mod } 2\pi, \, k=1,2,3,\dots\right\}
\end{equation}
Given an irrational $\lambda$ we may pick $\varepsilon\ll \lambda$ such that $q_\varepsilon$ is not a conjugate point of $p$ and hence all geodesic paths from $p$ to $q_\varepsilon$ are contained in the equator.

For an irrational $\lambda$ it is straightforward to show that the sequence of indices of $\gamma_{\pm}^m$ coincides with the sequence of all nonnegative integers after rearranging (with no repetitions). Given $N\gg 0$, if we take $\lambda > 1-\frac{2}{N+1}$, then $\op{ind}(\gamma_+^1) =\lfloor(N+1)(1-\varepsilon) \rfloor>N$ for $\varepsilon>0$ small by Equation~\eqref{eq: ind-big_m}. Hence the $F_\lambda$-geodesics of index $\leq N$ are $\gamma_-^m$, $m=1,\dots, N$. (1) and (2) then follow from Equations~\eqref{Finsler-action} and \eqref{eq: ind-small_m}.
\end{proof}

\n
{\em Katok examples and Morse theory.}
At this point we recall that $L=F_\lambda^2$ is not a smooth Lagrangian function on $TM$ and in fact is only $C^1$ along the zero section.  Hence we cannot directly apply Theorem~\ref{thm: summary of morse in infinite dimensions}. There are various remedies in the literature allowing one to develop Morse theory in the Finsler setting. We adopt the approach of \cite{lu2015} and \cite{gongxue2020} and pick a smooth approximation of $L$, i.e., by \cite[Lemma 5.3]{gongxue2020} there exists a family of modifications $L_\eta$ of $L_0=L$ parametrized by $\eta \in [0,1]$, such that each $L_\eta$ is a smooth nonnegative Lagrangian satisfying (L1), (L2) and the following on $TM$:

(a) $L_0(x, v)-\eta \leq L_\eta(x, v) \leq L_0(x, v)$ for all $(x, v) \in T M$,

(b) $L_\eta(x, v)=L_0(x, v)$ for $L_0(x, v)\geq \eta$.

\begin{claim}
For $\eta$ sufficiently small, all the critical points of $\mathcal{A}_{L_\eta}$ coincide with those of $\mathcal{A}_L$.
\end{claim}

\begin{proof}
Let $\gamma$ be a critical point of $\mathcal{A}_{L_\eta}$ with $\eta$ small.
There exists $\eta'=\eta'(\eta)$ such that $\lim_{\eta\to 0} \eta'(\eta)=0$ and if $L(\gamma(t_0), \dot{\gamma}(t_0))> \eta'$ for some $t_0 \in [0,1]$, then the curve $\gamma$ is a Finsler geodesic by (b). On the other hand, if $L(\gamma(t), \dot{\gamma}(t))\leq \eta'$ for all $t \in [0,1]$, then $\mathcal{A}_L(\gamma) \le \eta' < (\frac{2\pi\varepsilon}{1+\lambda})^2$, where $(\frac{2\pi\varepsilon}{1+\lambda})^2$ is the smallest value of $\mathcal{A}_L$ amongst all the paths from $p$ to $q$. 
\end{proof}

Hence, for $\eta$ sufficiently small, $\mathcal{A}_{L_\eta}$ has only Finsler geodesics of $L$ as its critical points and the action $\mathcal{A}_{\sqrt{L_\eta}}$ of each of these geodesics is given by~\eqref{Finsler-action}, while the Morse indices are given by~\eqref{eq: ind-small_m} and~\eqref{eq: ind-big_m}. Additionally, Theorem~\ref{thm: summary of morse in infinite dimensions} is applicable for $L_\eta$ and, therefore, we may use the results from Section~\ref{section: Morse complex} to compute the $A_\infty$-algebra of interest. 

Returning to the calculation of $CM_{-*}(\Omega(S^2, \bm q,\bm q'))$, we put the points $\bm q,\bm q', \bm q''$ etc.\ in position to satisfy~\ref{item: positions-p1}--\ref{item: positions-p3} and the following:
\be
\item $\bm q, \bm q',\bm q''$ etc.\ are contained in a small ball with respect to the standard round metric;
\item the $r$-coordinates of $q_1,q_2,\dots$ are in decreasing order; and
\item the $\phi$-coordinates of $q_1,q_1',q_1'',\dots$ are in decreasing order.
\ee
We claim that computations that we previously performed in the case of round metric yield the same results for the Lagrangian $L_\eta$. The main reason for this is that for generators in degrees $0$ and $1$ the Finsler (pseudo-)gradient trajectories behave similarly to the round metric case. 

Let $A_{\kappa, F}$ denote the HDHF $A_\infty$-algebra of $\kappa$ cotangent fibers and let $A_{\kappa, \varepsilon, \eta}$ denote the Morse multipath $A_\infty$-algebra corresponding to $L_\eta$ as above. We will use the notation $\tau_{\ge -M}(A)$ to denote the $(-M)$-truncated subcomplex obtained from $A$ by setting all elements of degree less than $-M$  equal to $0$. The quasi-equivalence between $A_{\kappa, F}$ and $A_{\kappa, \varepsilon, \eta}$ given by Theorem~\ref{thm: quasi-equivalence-Floer-to-Morse} induces a quasi-isomorphism between the truncations $\tau_{\ge -M}(A_{\kappa, F})$ and $\tau_{\ge -M}(A_{\kappa, \varepsilon, \eta})$.

\begin{claim}\label{claim: truncated-dga}
For $M=\big \lfloor\frac{2(1-\varepsilon)}{1-\lambda} \big \rfloor$ and $d \ge 3$ we have:
$$\mu_M^d(\bm \gamma_d, \ldots, \bm \gamma_1)=0 \text{, whenever } |\gamma_1|+\ldots+|\gamma_d|+(2-d) >- M.$$
In particular, the truncated complex $\tau_{\ge -\lfloor\frac{M}{3} \rfloor+1}(A_{\kappa, \varepsilon, \eta})$ has an induced structure of a dga.
\end{claim}

\begin{proof}
Arguing by contradiction, suppose there exist multipaths $\bm \gamma_1, \dots, \bm \gamma_d$ in $\tau_{\ge -M}(A_{\kappa, \varepsilon, \eta})$ with $d\geq 3$ such that $m_d(\bm \gamma_1, \dots, \bm \gamma_d) \neq 0$. In particular there is an MFTS connecting $(\bm \gamma_1, \dots, \bm \gamma_d)$ to some $\bm\gamma_0$, and we have 
$$\op{ind}(\bm \gamma_0)=\op{ind}(\bm \gamma_1)+\dots+\op{ind}(\bm \gamma_d)+(d-2) < \big \lfloor\tfrac{2(m-\varepsilon)}{1-\lambda} \big \rfloor.$$

By Lemma~\ref{lemma: approximate monotonicity}, we have:
\begin{gather}\label{eq:action-gamma0}
    \mathcal{A}_{\sqrt{L_\eta}}(\bm \gamma_0) \approx \frac{2\pi \op{ind}(\bm \gamma_0)}{1+\lambda}=\frac{2\pi(\op{ind}(\bm \gamma_1)+\dots+\op{ind}(\bm \gamma_d)+(d-2))}{1+\lambda},\\\label{eq:action-gamma_sum}
    \mathcal{A}_{\sqrt{L_\eta}}(\bm \gamma_1)+\dots+\mathcal{A}_{\sqrt{L_\eta}}(\bm \gamma_d)\approx\frac{2\pi(\op{ind}(\bm\gamma_1)+\dots+\op{ind}(\bm\gamma_d))}{1+\lambda},
\end{gather}
where the estimate is up to some terms involving the round metric distances among $\bm q$, $\bm q'$, $\bm q''$ etc., and the terms from Lemma~\ref{lemma: approximate monotonicity}, which can all be taken to be small.
Since the $\mathcal{A}_{\sqrt{L_\eta}}$-action is decreasing along the pseudogradient trajectories for $\mathcal{A}_{L_\eta}$, we obtain a contradiction since $d-2\ge 1$ by comparing~\eqref{eq:action-gamma0} and~\eqref{eq:action-gamma_sum}. The first part of the claim then follows. To show the latter, observe that $\mu^3_M$ vanishes on $\tau_{\ge -\lfloor\frac{M}{3} \rfloor+1}(A_{\kappa, \varepsilon, \eta})$ and therefore the tuple $(\tau_{\ge -\lfloor\frac{M}{3} \rfloor+1}(A_{\kappa, \varepsilon, \eta}), \mu^1_M, \mu^2_M)$ is a dga.
\end{proof}

\begin{claim} \label{claim: generation of algebra}
As an algebra over $\Z\llbracket \hbar\rrbracket$, the dga $\tau_{\ge -\lfloor\frac{M}{3} \rfloor+1}(A_{\kappa, \varepsilon, \eta})$ is generated by the elements $T_1,\dots,T_{\kappa-1}$ and $x_1,\dots, x_\kappa$.
\end{claim}

\begin{proof}[Proof of Claim~\ref{claim: generation of algebra}]
Using the fact that $A_{\kappa, \varepsilon, \eta}$ is generated by $\kappa$-tuples of geodesics from ${\bm q}$ to ${\bm q}'$, the algebra structure of $\tau_{\ge -\lfloor\frac{M}{3} \rfloor+1}(A_{\kappa, \varepsilon, \eta})$ for $\kappa=1$, and the relations \eqref{first}-\eqref{eq Hn x1x2} and \eqref{eq Hn xjxk}, any $\kappa$-tuple $\bm\gamma_0$ can be written as $\pm x_1^{j_1}\dots x_\kappa^{j_\kappa} T_\sigma$ for some $\sigma\in S_\kappa$, modulo $\hbar$ terms.  We then iterate this procedure to eventually write $\bm\gamma_0$ as a possibly infinite linear combination of terms of the form  $\pm \hbar^i x_1^{j_1}\dots x_\kappa^{j_\kappa} T_\sigma$.
\end{proof}

\begin{remark}
    The sums are expected to be finite, but the energy functional does not seem to give us sufficiently good control.
\end{remark}

\subsection{Completion of proof of Theorem~\ref{thm: multiloop algebra for S2 is Hn}}\label{subsection: completion}
We now complete the proof of Theorem~\ref{thm: multiloop algebra for S2 is Hn}.

\s\n
{\em No other relations.} We claim that in each of the algebras $\tau_{\ge -M}(A_{\kappa, \varepsilon, \eta})$ there are no other relations besides those that are in the ideal $\mathcal{I}$ generated by \eqref{first}--\eqref{eq Hn x1x2}. Suppose we have a relation 
\begin{equation} \label{eqn: linear combo}
    R_0\coloneqq w_0 + \hbar w_1 + \dots +\hbar^k w_k=0.
\end{equation} 
of fixed degree; each $w_i$ is a linear combination of finitely many monomials and $w_0\neq 0$. Moreover, by Lemma~\ref{lemma: PBW} we may assume that each $w_i$ is a linear combination of monomials $x^\alpha T_\sigma$ of PBW type. Now each monomial $x^\alpha T_\sigma$ has a representative of the form $\sum_{i\ge 0}(\sum_j a_{ij} \bm \gamma_{ij}) \hbar^i$, where each $\bm \gamma_{ij}$ is a tuple of geodesics and the leading term $\sum_j a_{0j}\bm \gamma_{0j}$ is a tuple of geodesics $x^{\alpha_k}$ of permutation type $\sigma$. After expressing $R_0$ via the above geodesic representatives, the terms with $\hbar^0$ all come from monomials of $w_0$, and all are represented by different tuples $\bm \gamma_j$, unless $w_0=0$. This is a contradiction.

\s
The truncations $\tau_{\ge -\lfloor\frac{M}{3} \rfloor}(A_{\kappa, \varepsilon, \eta})$ and $\tau_{\ge -\lfloor\frac{M}{3} \rfloor}(H_\kappa)$ therefore coincide. As a corollary the two $A_\infty$-algebras have the same cohomology rings. Hence by taking $M \to \infty$ the above argument already shows that the dga $H_\kappa$ has the same cohomology as $A_{\kappa, F}$.

\s\n
{\em Quasi-equivalence between $A_{\kappa, F}$ and $H_\kappa$.} 
Fix sequences $(N_i=i)_{i\in \Z_{\ge 0}}$ and $(\lambda_i)_{i \in \Z_{\ge 0} }$ as in Lemma~\ref{lemma: approximate monotonicity} and denote by $(A_i, \{\mu^d_i\}_{d \ge 1})$ the $A_\infty$-algebra $A_{\kappa, \varepsilon(\lambda_i), \eta(\lambda_i)}$ constructed above. As graded modules we can decompose $A_i=\bigoplus_{s\le 0}A_i(s)$, where $A_i(s)$ is the degree $s$ part. 
Note that each $A_i$ is non-positively graded. Claim~\ref{claim: truncated-dga} then implies that the following property holds. 
\begin{enumerate}[(A)]
    \item For any $s \le 0$, the $A_\infty$-multiplication $\mu_i^d: A_{i}^{\otimes d}(s) \to A_i(s-d+2)$ is zero for $i \ge -s+d-2$, where $d\ge 3$.
\end{enumerate}

There is a homotopy between $L_{\eta_i}$ and $L_{\eta_{i+1}}$ which arises from varying $\lambda$ from $\lambda_i$ to $\lambda_{i+1}$. Counting \emph{two-colored Morse flow trees with switches} corresponding to such homotopy as in \cite{mazuir2021I, mazuir2021II} introduces an $A_\infty$-morphism $\mathcal{F}_i \colon A_i \to A_{i+1}$ for each $i \ge 1$. We use without proof the fact that all the $\mathcal{F}_i$ are quasi-equivalances and refer the reader to \cite[Proposition 20]{mazuir2021I} for a detailed account of this claim in the finite-dimensional setting. A rigorous proof in our setting can be carried out in a similar manner by introducing moduli spaces of two-colored trees with switches.

Note that each $\mathcal{F}_i^1|_{\hbar=0}$ is an identity morphism, since for $\kappa=1$ in each degree both $A_i$ and $A_{i+1}$ have a unique generator (which is also a cycle) and $\mathcal{F}_i^1|_{\hbar=0}$ is a quasi-isomorphism. Since we are working over $\Z\llbracket \hbar \rrbracket$, it follows that $\mathcal{F}_i^1$ is an isomorphism. Therefore,
\begin{enumerate}[(B)]
    \item For any $s \le 0$, the map $\mathcal{F}_i^1: A_{i-1}(s) \to A_i(s)$ is an isomorphism of abelian groups for $i \ge -s$.
\end{enumerate}

Finally, using action arguments as in the proof of Claim~\ref{claim: truncated-dga} one may show that $\mathcal{F}_i^k(\bm \gamma_k, \ldots,\bm \gamma_1)=0$ whenever $k \ge 2$ and $|\bm \gamma_k|+\dots+|\bm \gamma_1| \ge k-1-i$.  Hence the $A_\infty$-morphisms $\mathcal{F}_i$ also satisfy:
\begin{enumerate}[(C)]
\item For any $s \le 0$, the map $\mathcal{F}_i^k: A_{i-1}^{\otimes k}(s) \to A_i(s-k+1)$ is zero for $i \ge -s+k-1$, where $k>1$. 
\end{enumerate}
Roughly speaking, Properties (A)--(C) say that, for a fixed degree $s$, the $A_{\infty}$-morphism $\mathcal{F}_i$ is the identity map for $i$ sufficiently large. 

We use Properties (A)--(C) to show that the direct limit $A$ of the system $(\mathcal{F}_i)_{i,k\in \Z_{\geq 1}}$ exists in Lemma \ref{thm: direct limit}.  The limit $A$ is quasi-isomorphic to each $A_i$.  Finally, we show in Lemma \ref{lemma: direct limit dga} that $A$ is isomorphic to the dga $H_{\kappa}$.  This completes the proof of Theorem \ref{thm: multiloop algebra for S2 is Hn}.

\subsection{A direct limit argument} \label{subsection: direct limit argument}

\begin{lemma}\label{thm: direct limit}
Suppose $(A_i)_{i\in \Z_{\geq 0}}$ and $(\mathcal{F}_i^k: A_{i-1}^{\otimes k}\to A_i)_{i,k\in \Z_{\geq 1}}$ are sequences of non-positively graded $A_\infty$-algebras and $A_\infty$-morphisms such that Properties (A)--(C) above hold.
Then the direct limit of the system $(\mathcal{F}_i)_{i\in \Z_{\geq 1}}$ exists. Moreover, such a limit can be explicitly described from $A_i$ and it is a dga. 
\end{lemma}

\begin{proof}
The proof is carried out in three steps:
\begin{enumerate}
\item Define an $A_{\infty}$-algebra $A=\oplus_sA(s)$.
\item Construct a family of $A_{\infty}$-morphisms $\mathcal{G}_i =(\mathcal{G}_i^k)$, where $\mathcal{G}_i^k: A_i^{\otimes k} \to A$, and check that $\mathcal{G}_{i+1} \circ \mathcal{F}_{i+1}=\mathcal{F}_i$.
\item Verify that $A$ satisfies the universal property of a direct limit.
\end{enumerate} 

\s\n
{\em Step 1.} Define $A=\oplus_{s\le 0}A(s)$ and $A(s)=\varinjlim A_i(s)$, where we consider the directed system of $\Z\llbracket\hbar\rrbracket$-modules associated with $\mathcal{F}_i^1: A_{i-1}(s) \to A_i(s)$.
By Property (B) there exist $\mathcal{G}_{i}^1: A_i(s)\to A(s)$ such that the following diagram commutes:
$$\xymatrix{
A_{i-1}(s) \ar[r]^{\mathcal{F}_i^1} \ar[dr]_{\mathcal{G}_{i-1}^1} & A_i(s) \ar[d]^{\mathcal{G}_{i}^1} \\
& A(s),
}$$ 
where all the maps are isomorphisms for $i \ge -s$.

We now define the $A_{\infty}$-operations $(\mu^d)_{d\geq 1}$ on $A$. For any $s\le 0$, $\mu^1$ is defined by the following diagram:
$$\xymatrix{
A(s) \ar[r]^{\mu^1}  & A(s+1)  \\
A_{i}(s) \ar[r]^{\mu_i^1}  \ar[u]_{\mathcal{G}_i^1} \ar[d]^{\mathcal{F}_{i+1}^1} & A_i(s+1) \ar[u]^{\mathcal{G}_i^1} \ar[d]_{\mathcal{F}_{i+1}^1} \\
A_{i+1}(s) \ar[r]^{\mu_{i+1}^1} \ar@/^3pc/[uu]^{\mathcal{G}_{i+1}^1} & A_{i+1}(s+1) \ar@/_3pc/[uu]_{\mathcal{G}_{i+1}^1} ,
}$$
for $i \ge -s$. Here all vertical maps are isomorphisms. The map $\mu^1$ is independent of the choice of $i$ since the lower square is commutative. 

The map $\mu^2$ is defined via the following diagram:
$$\xymatrix{
A(s) \otimes A (t) \ar[r]^{\mu^2}  & A(s+t)  \\
A_{i}(s) \otimes A_{i}(t) \ar[r]^{\mu^2_i}  \ar[u]_{\mathcal{G}_i^1 \otimes \mathcal{G}_i^1}  & A_i(s+t) \ar[u]^{\mathcal{G}_i^1},
}$$
for $i \ge -s-t$. 
One can similarly check that $\mu^2$ is independent of the choice of $i$ as in the case of $\mu^1$.  
The corresponding square is commutative by Properties (B) and (C). 

The map $\mu^3$ is defined by the following diagram:
$$\xymatrix{
A(s) \otimes A (t) \otimes A(r) \ar[r]^{\mu^3}  & A(s+t+r)  \\
A_{i}(s) \otimes A_{i}(t) \otimes A_i(r) \ar[r]^{\mu_i^3}  \ar[u]_{\mathcal{G}_i^1 \otimes \mathcal{G}_i^1 \otimes \mathcal{G}_i^1}  & A_i(s+t+r-1) \ar[u]^{\mathcal{G}_i^1},
}$$
for $i \ge -s-t-r+1$. It is independent of choices of $i$. Moreover, $\mu_i^3$ is zero by Property (A). It follows that $\mu^3$ is the zero map. 
The maps $\mu^k$, $k \ge 3$, can be defined similarly and are zero maps.

To verify that the collection $(\mu^d)_{d \ge 1}$ satisfies the $A_{\infty}$-relations, observe that for any given generators as inputs, we can choose $i \gg 0$ such that $(\mu^d)_{d \ge 1}$ satisfies the $A_{\infty}$-relations for these generators if and only if $(\mu_i^d)_{d\geq 1}$ does for the corresponding generators in $A_i$. The latter is true since $A_i$ is an $A_{\infty}$-algebra. 

\s\n
{\em Step 2.} We construct a family of $A_{\infty}$-morphisms $\mathcal{G}_i=(\mathcal{G}_i^k)_{k \ge 1}$, where $\mathcal{G}_i^k: A_{i}^{\otimes k} \to A$ is defined by the following diagram:
$$\xymatrix{
A_{i}^{\otimes k}(s) \ar[r]^{\mathcal{G}_i^k} \ar[dr]_{(\mathcal{F}_m \circ \cdots \circ \mathcal{F}_{i+1})^k \quad \,} & A(s-k+1) \\
& A_m(s-k+1) \ar[u]^{\mathcal{G}_{m}^1},
}$$ 
for $m \ge \mbox{max}\{i,-s+k-1\}$.
Here $(\mathcal{F}_m \circ \cdots \circ \mathcal{F}_{i+1})^k$ denotes the degree $1-k$ part of the composition $\mathcal{F}_m \circ \cdots \circ \mathcal{F}_{i+1}$ of $A_{\infty}$-morphisms. 
It is straightforward to verify that $\mathcal{G}_i=(\mathcal{G}_i^k)_{k \ge 1}$ is an $A_{\infty}$-morphism. Roughly speaking, $\mathcal{G}_i=\mathcal{F}_m \circ \cdots \circ \mathcal{F}_{i+1}$ for sufficiently large $m$. 
It follows that 
\begin{equation}\label{eq: functors to direct limit commute}
\mathcal{G}_{i+1} \circ \mathcal{F}_{i+1}=\mathcal{G}_{i}. 
\end{equation}

\s\n{\em Step 3.} 
We verify that $A$ satisfies the universal property of the direct limit. 
Suppose that there is an $A_{\infty}$-algebra $B$ and a family of $A_{\infty}$-morphisms $\mathcal{H}_i=(\mathcal{H}_i^k)_{k\ge 1}$, where $\mathcal{H}_i^k: A_{i}^{\otimes k} \to B$, such that $\mathcal{H}_i=\mathcal{H}_{i+1} \circ \mathcal{F}_{i+1}$.
We construct an $A_{\infty}$-morphism $\mathcal{H}=(\mathcal{H}^k)_{k \ge 1}$, where $\mathcal{H}^k: A^{\otimes k} \to B$ is defined by the following diagram:
$$\xymatrix{
A^{\otimes k}(s) \ar[r]^{\mathcal{H}^k}  & B(s-k+1) \\
A_i^{\otimes k}(s) \ar[u]^{(\mathcal{G}_{i}^1)^{\otimes k}} \ar[ur]_{\mathcal{H}_i^k} &
}$$ 
for $i \ge -s+k-1$. It is easy to verify that $\mathcal{H}$ is an $A_{\infty}$-morphism. 
Moreover, $\mathcal{H} \circ \mathcal{G}_m=\mathcal{H}_m$ for sufficiently large $m$. Hence $\mathcal{H} \circ \mathcal{G}_i=\mathcal{H}_i$ holds for any $i$ as we may proceed by induction on $i$ using~\eqref{eq: functors to direct limit commute}: $$\mathcal{H} \circ \mathcal{G}_i=\mathcal{H} \circ \mathcal{G}_{i+1} \circ \mathcal{F}_{i+1}=\mathcal{H}_{i+1} \circ \mathcal{F}_{i+1}=\mathcal{H}_i.$$
\vskip-.25in
\end{proof}

The limit $A$ constructed above is a dga since $\mu^k=0$ for $k\ge3$. 

\begin{lemma}\label{lemma: direct limit dga}
Specializing to the case of $A_i=A_{\kappa, \varepsilon(\lambda_i), \eta(\lambda_i)}$, the direct limit dga $A$ is isomorphic to $H_{\kappa}$.
\end{lemma}

\begin{proof}
The dga $H_{\kappa}$ is generated by the $T_i$ and $x_1$. 
There are corresponding elements in each $A_i$, hence in $A$.
By abuse of notation, these elements still denoted by $T_i$ and $x_1$. 
We then define a map $\phi: H_{\kappa} \to A$ which sends $T_i$ to $T_i$ and $x_1$ to $x_1$. 
The map $\phi$ extends to a dga morphism since all the relations of $H_{\kappa}$ are preserved by $\phi$. 
By Claim~\ref{claim: generation of algebra}, each $A_i$ above a certain degree is generated by the $T_i$ and $x_1$, which clearly implies that $A$ is generated by the $T_i$ and $x_1$. Hence the morphism $\phi$ is surjective; it is also injective since $\phi|_{\hbar=0}$ is the identity.
\end{proof}

\appendix
\section{Regularity for unstable manifolds} \label{appendix: regularity for unstable submanifolds}

The goal of this appendix is to prove Theorem~\ref{thm: regularity for pseudogradient}.  We learned the main ideas of the proof from Alberto Abbondandolo, to whom we are grateful.

\subsection{Preliminaries on difference quotients}\label{appendix subsection: preliminaries-for-regularity}

Let $S^1=[0,1]/(0\sim 1)$. Given $f: S^1 \to  \R^n$ and a real number $h>0$, the difference quotient $D^h f$ is given by:
$$D^hf(x)=\frac{f(x+h)-f(x)}{h}.$$
The following is immediate from \cite[Theorem A.1.22]{wendl2008}; similar statements are standard and can be found for example in \cite[Section 5.8.2]{evans1998}.

\begin{lemma}\label{lemma: difference-quotients}
  Assume $1 \le p < \infty$ and $k \in \Z_{>0}$ .
  \begin{enumerate}
      \item If $f \in W^{k,p}(S^1, \R^n)$, then $D^hf$ converges to $f'$  in $W^{k-1,p}(S^1, \R^n)$ as $h \to 0$.
      \item Suppose $p >1$. If $f \in W^{k-1,p}(S^1, \R^n)$, and $\|D^hf\|_{L^p} \le C$ for some $C>0$ and all $h>0$ in a neighborhood of $0\in \R$, then $f \in W^{k,p}(S^1, \R^n)$. 
  \end{enumerate}
\end{lemma}

\begin{corollary}
    The following statements are equivalent for $k \in \Z_{>0}$:
\begin{enumerate}
    \item The $L^2$-map $\gamma \colon S^1 \to M$ represents an element in $W^{k,2}(S^1, M)$.
    \item The path $T_\tau (\gamma)_\tau \colon (-\varepsilon, \varepsilon) \to W^{k-1,2}(S^1, M)$, $\tau\mapsto (t\mapsto \gamma(\tau+t))$,
    is differentiable at $\tau=0$ with respect to the smooth structure on $W^{k-1, 2}(S^1, M)$.
\end{enumerate}
\end{corollary}

\begin{proof}
This follows from Lemma~\ref{lemma: difference-quotients} since the derivative of the path $T_\tau(\gamma)$ at $\tau=0$ is equal to the limit of the difference quotients in the appropriate local chart of $W^{k,2}(S^1, \gamma^*TM)$, as in \cite{klingenberg2012lectures}.
\end{proof}

The goal of the remainder of this subappendix is to prove the analogs of the above results in the non-periodic setting using \emph{twisted difference quotients}. Let $T \colon (-\varepsilon_T, \varepsilon_T)_h \times [0,1]_x \to [0,1]$ be a smooth function satisfying the following: 
\begin{gather}\label{eq: deformation-condition}
    T_0 \equiv 0, T_h(0)=0, T_h(1)=0  \mbox{ for all $h$}; \nonumber \\
    (\bdry_h T_h)|_{h=0} (x) > 0 \text{ away from $x=0,1$};\\
    \partial_xT_{h}(x)>-1 \mbox{ for all $h$},
    \nonumber 
\end{gather}
where we are writing $T_h(x)=T(h,x).$

\begin{example}
    \label{example-of-translations}
    Choose $\delta>0$ small and $\varepsilon_T$ satisfying $0< \varepsilon_T<\delta$.  We set $T_h(x)=h$ for $x \in [\delta, 1-\delta]$, $T_h(x)=\frac{h}{\delta}x$ for $x \in [0, \delta]$, and $T_h(x)=\frac{h}{\delta}(1-x)$ for $x \in [1-\delta, 1]$, and then smooth out this family near the corners. 
\end{example}
\begin{figure}[ht]
    \centering
\begin{tikzpicture}
    \begin{axis}[
        axis lines = middle,
        xlabel = $x$,
        ylabel = {$A_h(x)$},
        xmin = 0, xmax = 1.2,
        ymin = 0, ymax = 1.2,
        xtick={0, 0.2, 0.8, 1},
        ytick={0.3, 0.9, 1},
        xticklabels={$0$, $\delta$, $(1-\delta)$, $1$},
        yticklabels={$\delta+h$, $(1-\delta)+h$, $1$},
        domain=0:1,
        samples=200,
        width=8cm,
        height=7cm,
        grid = both,
    ]
    
    \addplot [
        domain=0:0.2,
        samples=100,
        thick,
        black,
    ]
    {0.5*x + x};
    
    \addplot [
        domain=0.2:0.8,
        samples=100,
        thick,
        black,
    ]
    {x + 0.1};
    
    \addplot [
        domain=0.8:1,
        samples=100,
        thick,
        black,
    ]
    {0.5*(1-x) + x};
    
    \end{axis}
\end{tikzpicture}
    \caption{The graph of $A_h(x)=\op{id}+T_h(x)$ before smoothing the corners.}
    \label{fig: example-graph}
\end{figure}
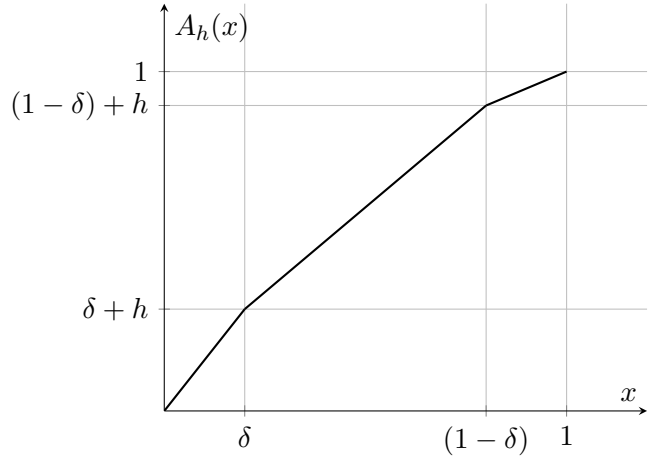
By \eqref{eq: deformation-condition}, $A_h=\op{id}+T_h$, $h\in (-\varepsilon_T, \varepsilon_T)$, is a smooth family of diffeomorphisms $[0,1]\stackrel\sim\to[0,1]$ such that $A_h(0)=0$, $A_h(1)=1$, and $A_0=\op{id}$. 

Given a function $f \colon [0,1] \to \R^n$, its \emph{$T$-twisted difference quotients} are:
\begin{equation}
    D^{T_h}f=\frac{f(x+T_h(x))-f(x)}{h}.
\end{equation}
We now prove the analog of Lemma~\ref{lemma: difference-quotients}, following the proof of \cite[Theorem A.1.22]{wendl2008}.

\begin{lemma}\label{lemma: twisted-difference-quotients}
Assume $1 \le p < \infty$ and $k \in \Z_{>0}$ .
\begin{enumerate}
      \item If $f \in W^{k,p}([0,1], \R^n)$, then $D^{T_h}f$ converges to $(\bdry_h T_h)|_{h=0}  \cdot f'$ 
      in $W^{k-1,p}([0,1], \R^n)$ as $h \to 0$.
      \item Suppose $p >1$, $f \in W^{k-1,p}([0,1], \R^n)$, and $\|D^{T_h}f\|_{W^{k-1,p}} \le C$ for some $C>0$ and all $h>0$ in a neighborhood of $0\in \R$.
      Then:
      \be
      \item $f \in W^{k,p}(\tilde I, \R^n)$ for any closed segment $\tilde I \subset [0,1]$ that does not intersect $\{0,1\}$. 
      \item If the $\|D^{T_h}f\|_{W^{k-1,p}}$ are uniformly bounded over all choices of $T_h$ satisfying~\eqref{eq: deformation-condition}, then $f \in  W^{k,p}([0,1], \R^n)$.
      \ee
\end{enumerate}
\end{lemma}

\begin{proof}[Sketch of proof]
(1) Given $f \in C^1([0,1], \R^n)$, we claim that the sequence $D^{T_h}f$ converges uniformly to $(\bdry_h T_h)|_{h=0}  \cdot f'$.  We compute:
\begin{align*}
    |D^{T_h}f(x)-(\bdry_h T_h)|_{h=0} (x)f'(x)|=|f'(x')&\tfrac{T_h(x)}{h}-(\bdry_h T_h)|_{h=0} (x) f'(x))| \\
    & \le |\tfrac{T_h(x)}{h}(f'(x') -f'(x))|+|f'(x) (\tfrac{T_h(x)}{h}-(\bdry_h T_h)|_{h=0} (x))|,
\end{align*}
where $x'$ is chosen in the $T_h(x)$-neighborhood of $x$ to satisfy the mean value theorem. Applying the uniform continuity of $f'(x)$ and the uniform convergence of $\tfrac{T_h(x)}{h}$ to $(\bdry_h T_h)|_{h=0} $ implies the claim.

Now let $f \in W^{k,p}([0,1], \R^n)$.  We show that for $h>0$ small there exists $C>0$ such that
\begin{equation}\label{eq-dominance}
\left\|D^{T_h}f \right \|_{L^p} \le C \left \|\tfrac{T_h(x)}{h}f' \right\|_{L^p},
\end{equation}
by observing the following sequence of inequalities 
\begin{gather}
    \int_0^1|D^{T_h}f|^p dx= \int_0^1 \left |\tfrac{T_h(x)}{h}\int_0^1 f'(x+\tau T_h(x)) d\tau \right |^pdx \\
    \le \int_0^1 \int_0^1 \left | \tfrac{T_h(x)}{h} f'(x+\tau T_h(x)) \right |^pdx d\tau \leq C^p \int_0^1\left \| \tfrac{T_h(x)}{h} f' \right \|^p_{L^p} d\tau=C^p \left\| \tfrac{T_h(x)}{h} f' \right \|_{L^p}^p. \nonumber
\end{gather}

Finally we claim that, given $\varepsilon>0$, for $h$ sufficiently small we have:
\begin{equation}\label{eq: ineq}
\left \| D^{T_h}f - (\bdry_h T_h)|_{h=0}  f' \right \|_{L^p} <\varepsilon.
\end{equation}
Indeed, pick a smooth function $f_\varepsilon$ which is very close to $f$ in the $L^p$-norm and observe that 
\begin{equation} \label{eqn: comparing with smooth}
    \left \| D^{T_h}f -D^{T_h}f_{\varepsilon} \right \|_{L^p} < \left \| \tfrac{T_h(x)}{h}(f'- f'_{\varepsilon}) \right \|_{L^p}
\end{equation} 
by~\eqref{eq-dominance}. Then \eqref{eq: ineq} follows from \eqref{eqn: comparing with smooth} and the convergence for $f_{\varepsilon}$ shown earlier; see \cite{wendl2008} for more details. The case $k>1$ follows by applying \eqref{eq: ineq} to the derivatives of $f$.

\s
(2)(a) Let $\widetilde{T}_h$ be a family of functions on $[0,1]$ satisfying
$$(\op{id}+\widetilde{T}_{-h})(\op{id}+T_h)=\op{id}.$$
This family also satisfies~\eqref{eq: deformation-condition}.

We will treat the case $k=1$; the case $k>1$ follows by applying the same proof to the derivatives of $f$. Suppose $f \in L^p([0,1], \R^n)$ and $\|D^{T_h}f\|_{L^p} \le C$ for some $C>0$ and all $h>0$ in a neighborhood of $0\in \R$. 
Observe that $D^{T_h}f\cdot \frac{h}{T_h(x)}$ is bounded in the $L^p$-norm over $\tilde{I}$ as $(\bdry_h T_h)|_{h=0}  \ge \kappa>0$ on $\tilde{I}$. Hence, by the Banach-Alaoglu theorem, there is a sequence $h_j \to 0$ such that the corresponding sequence of functions $D^{T_{h_j}}f\cdot \frac{h_j}{T_{h_j}(x)}$ converges in the $L^p$-norm to a function $g$. 

We claim that $g$ is a weak derivative of $f$ along $\tilde{I}$. Indeed, for any $\varphi \in C_0^\infty(\tilde I, \R^n)$ and small $h>0$ we have
\begin{align*}
    \int_{\tilde{I}}g \varphi dx = &\lim_{h_j \to 0}\int_{\tilde{I}}\tfrac{f(x+T_{h_j}(x))-f(x)}{T_{h_j}(x)} \varphi dx \\
    = & \lim_{h_j \to 0}\int_{\tilde{I}} \tfrac{f(u)}{-\widetilde{T}_{-h_j}(u)}\varphi(u+\widetilde{T}_{-h_j}(u))\cdot (1+\bdry_u \widetilde T_{-h_j}(u))du -  \lim_{h_j \to 0}\int_{\tilde{I}} \tfrac{f(u)}{T_{h_j}(u)}\varphi(u) du \\
    = & \lim_{h_j \to 0}\int_{\tilde{I}}f(u) \tfrac{\varphi(u+\widetilde{T}_{-h_j}(u))-\varphi(u)}{-\widetilde{T}_{-h_j}(u)} du \\ 
    =& -\int_{\tilde{I}} f(u) \varphi'(u) du.
\end{align*}
Hence $g$ is a weak derivative of $f$ along $\tilde{I}$, which implies that $f|_{\tilde{I}} \in W^{1,p}(\tilde{I}, \R^n)$.

(2)(b) We treat the $k=1$ case. Under our assumption that $\|D^{T_h}f\|_{L^p}$ are uniformly bounded over all choices of $T_h$ satisfying~\eqref{eq: deformation-condition}, it suffices to find a family of the form $T_h(x)=h\tilde t(x)$ with $\tilde t(x)$ smooth, such that $D^{T_h}f\cdot \frac{h}{T_h(x)}=\tfrac{f(x+T_h(x))-f(x)}{T_h(x)}$ have bounded $L^p$-norms on $[0,1]$. We can take such a function $\tilde t(x)$ to be a smoothing of a staircase function $t(x)$ of the following form:
given a decreasing sequence $b_n \to 0$ there is a choice of a sequence $a_n \to 0$ with $a_n < 1/2$, such that
$t(x)=b_n$ for $x\in[a_{n+1},a_n]$ and $x \in [1-a_n, 1-a_{n+1}]$.  Details on how to use such a $T_h(x)$ are left to the reader. 
\end{proof}

\begin{corollary}\label{corollary: regularity-via-twisted-difference-quotients}
      The following statements are equivalent for $k \in \mathbb{Z}_{>0}$:
\begin{enumerate}
    \item The $L^2$-map $\gamma \colon [0,1] \to M$ is an element of $W^{k,2}([0,1], M)$. 
    \item For any smooth family of functions $T\coloneqq (T_h)_{h\in (-\varepsilon_T,\varepsilon_T)}$ satisfying~\eqref{eq: deformation-condition}, the path 
    \begin{gather*}
        P_{T}(\gamma) \colon (-\varepsilon_T, \varepsilon_T) \to W^{k-1,2}([0,1], M),\quad h\mapsto \gamma\circ A_h,
    \end{gather*} 
    is differentiable at $h=0$ with respect to the smooth structure on $W^{k-1,2}([0,1], M)$. 
\end{enumerate}
\end{corollary}

\begin{proof}
    This is an immediate consequence of Lemma~\ref{lemma: twisted-difference-quotients}.
\end{proof}

\subsection{Regularity for unstable manifolds.}\label{section: regularity}

In Appendix~\ref{section: regularity} we show that the unstable manifolds of a certain smooth pseudogradient vector field of an action functional $\mathcal{A}_L$ on the space 
$$\Omega^{1,2}(M^n; q, q')=\{\gamma \in W^{1,2}([0,1],M^n) \mid   \gamma(0)=q, \, \gamma(1)=q'\}$$
consist of smooth paths.

\begin{lemma}\label{lemma: path-inside-submanifold}
    For each smooth family $T\coloneqq (T_h)_{h\in (-\varepsilon_T,\varepsilon_T)}$ of functions satisfying~\eqref{eq: deformation-condition}, let
    $N_{T}$ be a finite-dimensional manifold and $\iota_T: N_T \to  (-\varepsilon_T, \varepsilon_T)_h \times \Omega^{1,2}(M, q, q')$ a smooth immersion
    such that:
    \be 
    \item the set $\Xi\coloneqq  \iota_T(N_T) \cap \{h=0\}$ is independent of $T$; and
    \item for each $(0,\gamma)\in \Xi$ and $T$ the path 
    $$\widetilde{P}_{T}(\gamma): (-\varepsilon_T, \varepsilon_T)\to (-\varepsilon_T, \varepsilon_T) \times \Omega^{1,2}(M, q, q'), \quad h\mapsto (h,\gamma\circ A_h),$$
    is a subset of $\iota_T(N_T)$.
    \ee
    Then the elements of $\Xi$ are smooth paths.
\end{lemma}

\begin{proof}
    Consider the natural inclusion $\phi_k \colon W^{k,2}([0,1],M) \to W^{k-1,2}([0,1],M)$, $f\mapsto f$. 
    
    We treat the $k=1$ case; the same argument works inductively for $k>1$. The map $(\op{id}\times\phi_1)\circ \iota_T$ is a smooth immersion for each $T$ since $N_T$ is finite-dimensional and hence the inclusion of the tangent space at each point is a closed embedding. 
  
    Given $(0,\gamma) \in \Xi$ and $T$, the path $(\op{id}\times \phi_1)\circ \widetilde{P}_{T}(\gamma)$ (i.e., the path $\widetilde{P}_{T}(\gamma)$ viewed as a path in  $(-\varepsilon_T, \varepsilon_T) \times \Omega^{0,2}(M, q,q')$) is differentiable at $h=0$ by Corollary~\ref{corollary: regularity-via-twisted-difference-quotients}. Since this path has image in $(\op{id}\times\phi_1)\circ \iota_T(N_T)$, it is differentiable with respect to the smooth structure on this immersed submanifold and hence is differentiable as a path in $N_{T}$. Therefore, $\widetilde{P}_{T}(\gamma)$ is a differentiable map to $(-\varepsilon_T,\varepsilon_T)\times \Omega^{1,2}(M, q, q')$ and  $\gamma\in \Omega^{2,2}(M, q, q')$ by applying Corollary~\ref{corollary: regularity-via-twisted-difference-quotients} again.
\end{proof}

We now recall the setup from \cite{AS2009} in more detail. Consider a {\em time-dependent coordinate system,} i.e., a smooth map
$$\varphi \colon [0,1] \times U \to M,$$
where $U$ is an open subset of $\R^n$ containing $0$, $\varphi|_{\{t\}\times U}$ is a diffeomorphism onto the open subset $\varphi(\{t\}\times U)$ for all $t\in[0,1]$, and $\varphi(0,0)=q, \varphi(1,0)=q'$. Then there is an associated chart of $\Omega^{1,2}(M, q, q')$
\begin{gather*}
    \varphi_* \colon \Omega^{1,2}(U; 0, 0) \to \Omega^{1,2}(M, q, q'),\\
    \gamma\mapsto \varphi(\cdot, \gamma(\cdot))).
\end{gather*} 
Here $\varphi(\cdot,\gamma(\cdot))$ is shorthand for a map $t\mapsto \varphi(t, \gamma(t)).$
Following the notation in \cite{AS2009}, we require such charts $(U, \varphi)$ to be \emph{bi-bounded}, meaning that $U$ is bounded and the image $\varphi([0,1]\times U)$ has compact closure, and all derivatives of $\varphi$ and of $(t, q) \mapsto \varphi(t, \cdot)^{-1}(q)$ are bounded. We denote the image of $\varphi_*$ by $\mathcal{U}_{(U, \varphi)}$; it is said to be {\em centered at the path $\varphi(t,0)$}.  Such charts form a smooth atlas of $\Omega^{1,2}(M, q, q')$.

\begin{example}
The standard example of a time-dependent coordinate system is given as follows: Let $\gamma \in \Omega^{1,2}(M, q, q')$ be a smooth path. The pullback bundle $\gamma^*TM$ admits a trivialization $\psi \colon [0,1] \times \R^n \to \gamma^*TM$ and let $U \subset \R^n$ be a small open ball centered at $0$. Then we define $\varphi:[0,1]\times U\to M$ via
$$\varphi(t, \xi)=\exp_{\gamma(t)}(\gamma_*\psi(\xi)),$$
where the exponential map is with respect to some metric $g$. 
\end{example}

Similarly, for $(a,b) \subset (-\varepsilon_T, \varepsilon_T)$, we consider a smooth map
$$\tilde{\varphi} \colon (a,b) \times [0,1] \times U \to M,$$
where $U$ is an open subset of $\R^n$ containing $0$, $\tilde\varphi|_{\{h\}\times\{t\}\times U}$ is a diffeomorphism onto the open subset $\tilde \varphi(\{h\} \times \{t\}\times U)$ for all $(h,t) \in (a,b) \times [0,1]$, and $\tilde{\varphi}(h, 0, 0)=q, \tilde{\varphi}(h, 1, 0)=q'$ for all $h\in (a,b)$. There is a similar notion of bi-boundedness for such a pair $(U, \tilde{\varphi})$. To such a pair we associate a smooth chart $\widetilde{\mathcal{U}}_{(U, \tilde{\varphi})}$ of $(-\varepsilon_T, \varepsilon_T) \times \Omega^{1,2}(M, q, q')$ via
\begin{gather}
\tilde{\varphi}_* \colon (a, b) \times \Omega^{1,2}(U; 0, 0) \to (-\varepsilon_T, \varepsilon_T) \times \Omega^{1,2}(M, q, q'), \\
(h, \gamma)\mapsto (h,  \tilde{\varphi}(h,\cdot, \gamma(\cdot))). \nonumber
\end{gather}
The charts $\widetilde{\mathcal{U}}_{(U, \tilde{\varphi})}$ form a smooth atlas of $(-\varepsilon_T, \varepsilon_T) \times \Omega^{1,2}(M, q, q')$. 

\begin{example}\label{exmpl: parametric chart}
For $(U, \varphi)$ bi-bounded and $A=(A_h=\op{id}+T_h)_{h \in (-\varepsilon_T, \varepsilon_T)}$, we define the map
\begin{gather*}
    \tilde{\varphi}^{A} \colon (-\varepsilon_T, \varepsilon_T) \times [0,1] \times U \to M.\\
    (h,t,x) \mapsto \varphi(A_h(t), x),
\end{gather*}
The map $\tilde{\varphi}^{A}$ is smooth and the pair $(U, \tilde{\varphi}^{A})$ is bi-bounded. The associated chart $\widetilde{\mathcal{U}}_{(U, \tilde{\varphi}^{A})}$ of $(-\varepsilon_T, \varepsilon_T) \times \Omega^{1,2}(M, q, q')$ then has the form
\begin{equation}
\tilde{\varphi}^{A}_*(h, \gamma)=(h, \varphi(A_h(\cdot), \gamma(\cdot))).
\end{equation}
\end{example}

The family $A$ induces a family of smooth maps
\begin{gather}\label{eqn: definition of widetilde A h}
    \widetilde{A}_h: \Omega^{1,2}(M, q, q') \to (-\varepsilon_T, \varepsilon_T) \times \Omega^{1,2}(M, q, q'),\\
    \gamma \mapsto (h,\gamma\circ A_h).\nonumber
\end{gather}

Similarly there is a (local) family of smooth maps
$$\widetilde{A}_h \colon \Omega^{1,2}(U; 0, 0) \to (-\varepsilon_T, \varepsilon_T) \times \Omega^{1,2}(U, 0, 0),$$
$$\gamma \mapsto (h, \gamma\circ A_h),$$
with the same notation by abuse of notation.

We claim that the following commutativity relation is satisfied for any $h\in (-\varepsilon_T, \varepsilon_T)$.
\begin{equation}\label{eq: equivariance-for-charts}
    \tilde{\varphi}^{A}_* \circ \widetilde{A}_{h}=\widetilde{A}_{h} \circ \varphi_*, 
\end{equation}
as it is straightforward to check:
$$\tilde{\varphi}^{A}_*(h, \gamma\circ A_h(\cdot)))=(h, \varphi(A_h(\cdot), \gamma\circ A_h(\cdot))))=(h, \varphi_*(\gamma)\circ A_h(\cdot)).$$

\begin{remark}
    The equivariance~\eqref{eq: equivariance-for-charts} allows us to work with the $\widetilde{A}_h$-action directly in the chart $\widetilde{\mathcal{U}}_{(U, \tilde{\varphi}^{A})}$. 
\end{remark}

\s
Next we define several related Lagrangian functions. Given a Lagrangian $L \colon [0,1] \times TM \to \R$ satisfying the assumptions of \cite[Section 3]{AS2009} and a family $A=(A_h)$, we define 
\begin{gather}\label{eq: lagrangian-change-of-variables}
    L_h: [0,1]\times TM\to \R,\quad    (t, x, v)\mapsto \partial_tA_h(t) \cdot L\left( A_h(t),x, \tfrac{v}{\partial_tA_h(t)}\right),\\
    L_{A} \colon (-\varepsilon_T, \varepsilon_T) \times [0,1] \times TM \to \R, \quad (h,t,x,v)\mapsto L_h(t,x,v). \nonumber
\end{gather}
Given charts $(U,\varphi)$ and  $(U, \tilde{\varphi}^A)$ we define related pullbacks:
\begin{gather*}
    \varphi^*L \colon [0,1] \times U \times \R^n \to \R,\quad   (t,x,v)\mapsto L(t,\varphi(t, x), D_x\varphi(t,x)(v)),\\
    (\varphi^*L)_h \colon  [0,1] \times U \times \R^n \to \R,\quad   (t,x,v)\mapsto \partial_tA_h(t) \cdot \varphi^*L\left( A_h(t), x, \tfrac{v}{\partial_tA(t)} \right),\\
    (\tilde{\varphi}^A)^* L_A \colon (-\varepsilon_T, \varepsilon_T) \times [0,1] \times U \times \R^n \to \R,\quad 
    (h, t, x, v)\mapsto L_A(h, t, \tilde{\varphi}^A(h, t, x), D_x \tilde{\varphi}^A(h,t,x)(v)).
\end{gather*}

One verifies that
\begin{equation}\label{eq: chart-family}
    (\varphi^*L)_h=(\tilde{\varphi}^{A})^*(L_h).
\end{equation}
This equation allows us to work with the family $L_h$ locally in the chart associated with $(U,\tilde{\varphi}^{A})$ and from now on we will allow ourselves to write $L$ instead of $\varphi^*(L)$ and $L_h$ instead of $(\tilde{\varphi}^{A})^*(L_h)$ when a specific chart is understood.

We are now ready to prove the main theorem of Appendix \ref{appendix: regularity for unstable submanifolds}. The proof follows the general outline of \cite[Theorem 4.1]{AS2009} with an extra ingredient: the parameter $h\in (-\varepsilon_T,\varepsilon_T)$ and a corresponding family of diffeomorphisms $(A_h:[0,1]\stackrel\sim\to [0,1])$.

\begin{theorem}\label{thm: regularity for pseudogradient}
     There exists a choice of smooth pseudogradient vector field $X$ for a given action functional $\mathcal{A}_L$ on $\Omega^{1,2}(M, q_0, q_1)$  such that all unstable manifolds consist of smooth paths. 
\end{theorem}

\begin{proof}
Let $T=(T_h)_{h\in (-\varepsilon_T,\varepsilon_T)}$ be a smooth family of functions satisfying~\eqref{eq: deformation-condition} and let $$A=(A_h=\op{id}+T_h)$$ be the associated family of diffeomorphisms. If $\gamma_{\operatorname{crit}}$ is a critical point of $\mathcal{A}_L$, then $P_{h}(\gamma_{\operatorname{crit}})\coloneqq \gamma_{\operatorname{crit}}\circ A_{h}$ is a critical point of $\mathcal{A}_{L_{h}}$. By \cite[Proposition 3.1]{AS2009}, $\gamma_{\operatorname{crit}}$ is smooth, and hence the path 
$$\widetilde{P}_{T}(\gamma_{\operatorname{crit}})=\{(h, P_h(\gamma_{\operatorname{crit}}) )\mid   h \in (-\varepsilon_T, \varepsilon_T)\}$$ is a smooth $1$-dimensional submanifold of $(-\varepsilon_T, \varepsilon_T) \times \Omega^{1,2}(M, q, q')$. 

The goal is to construct a smooth vector field $\widetilde{X}$ on $(-\varepsilon_T, \varepsilon_T) \times \Omega^{1,2}(M, q, q')$ which satisfies the following conditions:
\begin{enumerate}
    \item[(a)] The component of $\widetilde{X}$ in the $\bdry_h$-direction is zero;
    \item[(b)] $X\coloneqq \widetilde{X}|_{\{0\} \times \Omega^{1,2}(M, q, q')}$ is a pseudogradient vector field of $\mathcal{A}_L$ which is independent of $T$;
    \item[(c)] $\widetilde{X}=0$ along all the paths $\widetilde{P}_{T}(\gamma_{\operatorname{crit}})$;
    \item[(d)] $\widetilde{X}_{\widetilde{A}_h(0, \gamma)}=d\widetilde{A}_h(\widetilde{X}_{(0, \gamma)})$ for any $\gamma \in W^u(\gamma_{\operatorname{crit}}; X)$, where $\widetilde{A}_h$ is given by Equation~\eqref{eqn: definition of widetilde A h}. 
\end{enumerate}
We note that (b) and (d) imply
\begin{enumerate}
    \item[(d')] the unstable manifolds $W^u(\widetilde{P}_{T}(\gamma_{\operatorname{crit}});\widetilde X)$ satisfy the conditions of Lemma~\ref{lemma: path-inside-submanifold}.
\end{enumerate}
Indeed, for any $h\in (-\varepsilon_T, \varepsilon_T)$, $\widetilde{A}_h$ takes the trajectory of $\widetilde{X}$ from $(0, \gamma_{\operatorname{crit}})$ to $(0, \gamma)$ to the trajectory of $\widetilde{X}$ from $(h, P_h(\gamma_{\operatorname{crit}}))$ to $(h,P_h(\gamma))$, giving Lemma~\ref{lemma: path-inside-submanifold}(2).  Moreover Lemma~\ref{lemma: path-inside-submanifold}(1) holds because $X$ is independent of $T$.
The existence of such a vector field $\widetilde{X}$ for each $T$ and Lemma~\ref{lemma: path-inside-submanifold} imply the theorem.

Let $\gamma_0 \in \Omega^{1,2}(M, q, q')$ be a smooth path.  We consider a time-dependent coordinate system $\varphi \colon [0,1] \times U \to M$ such that $\varphi(t,0)=\gamma_0(t)$ and the associated charts $\mathcal{U}_{(U, \varphi)}$ and $\widetilde{\mathcal{U}}_{(U, \tilde{\varphi}^{A})}$ centered at $\gamma_0$.  

To construct $\widetilde X$ we construct $\widetilde Y$ on the charts $\widetilde{\mathcal{U}}_{(U, \tilde{\varphi}^{A})}$ centered at $\gamma_0$ and apply a partition of unity argument.

\s\n
{\em $\gamma_0$ non-critical.}  Suppose $\gamma_0$ is a non-critical point of $\mathcal{A}_L$.  Let $Y_0$ be a tangent vector at $0 \in \Omega^{1,2}(U; 0, 0)$ such that $d\varphi(Y_0)$ is a smooth vector field along $\gamma_0$ which is sufficiently close to the gradient vector field $\operatorname{grad}_{\mathcal{A}_L}(\gamma_0)$ with respect to the $W^{1,2}$-norm (note that since we take the gradient with respect to the $W^{1,2}$-metric it is not immediate that such gradient is smooth as in the case of the $L^2$-metric). We extend $Y_0$ to $Y$ on all of $\Omega^{1,2}(U; 0, 0)$ by setting $Y_\gamma=Y_0$ for all $\gamma \in \Omega^{1,2}(U; 0,0 )$. Next, we extend $Y$ to $\widetilde{Y}$ on $(-\varepsilon_T, \varepsilon_T) \times \Omega^{1,2}(U; 0,0)$ by
\begin{equation}\label{eq: gradient-translations}
\widetilde{Y}_{(h, \gamma)}(\cdot)=d\widetilde{A}_h(Y_0)_{\widetilde{A}_h(0)}(\cdot)=Y_0(A_h(\cdot)).
\end{equation}
The term $Y_0(A_h(\cdot))$ does not depend on $\gamma$ but depends on $h$. Since $Y_0$ is a smooth vector field along $\gamma_0$, $\widetilde{Y}$ is smooth in the $h$-direction and hence smooth on $(-\varepsilon_T, \varepsilon_T) \times \Omega^{1,2}(U; 0, 0)$. Furthermore, by definition $\widetilde{Y}$ satisfies the invariance condition
\begin{equation}\label{eq: vector-field-invariance-chart} 
\widetilde{Y}_{\widetilde{A}_h(0, \gamma)}=d\widetilde{A}_h(\widetilde{Y}_{(0, \gamma})),
\end{equation}
for any $\gamma \in \Omega^{1,2}(U; 0, 0)$. Hence, by~\eqref{eq: equivariance-for-charts}, $d\tilde{\varphi}^{A}(\widetilde{Y})$ satisfies (d) over $\widetilde{\mathcal{U}}_{(U, \tilde{\varphi}^{A})}$.

\s\n
{\em $\gamma_0$ critical.} Let $\gamma_0$ be a critical point of $\mathcal{A}_L$. Note that $\gamma_0$ is a smooth path. Recall from \cite[Equations (3.8),(3.10)]{AS2009} that the Gateaux differential $d\mathcal{A}_L$ and the second Gateaux differential $d^2\mathcal{A}_L$ at $\gamma \in \Omega^{1,2}(U; 0, 0)$ have expressions
\begin{equation}\label{eq: differential formula}
    d\mathcal{A}_L(\gamma)[\xi]=\int_0^1(\partial_xL(t, \gamma, \dot{\gamma})\xi+\partial_vL(t, \gamma, \dot{\gamma})\dot{\xi})dt,
\end{equation}
\begin{gather}\label{eq: second differential formula}
    d^2\mathcal{A}_L(\gamma)[\xi, \eta]=\int_0^1\left(\frac{\partial^2 L}{\partial v^2}(t, \gamma, \dot{\gamma}) \dot{\xi} \dot{\eta}+\frac{\partial^2 L}{\partial v \partial x}(t, \gamma, \dot{\gamma}) \dot{\xi}\eta+\right. \\
    \left.+\frac{\partial^2 L}{\partial x\partial v}(t, \gamma, \dot{\gamma}) \xi \dot{\eta}+\frac{\partial^2 L}{\partial x^2}(t, \gamma, \dot{\gamma}) \xi \eta \right) dt, \nonumber
\end{gather}
for any $\xi, \eta \in \Omega^{1,2}(\R^n; 0,0)\subset W^{1,2}([0,1], \R^n)$. There is also an associated linear operator 
\begin{gather*}
    D^2\mathcal{A}_L(\gamma) \colon \Omega^{1,2}(\R^n; 0,0) \to \Omega^{1,2}(\R^n; 0,0)^*,\\
    D^2\mathcal{A}_L(\gamma)(\xi)[\eta]=d^2\mathcal{A}_L(\gamma)[\xi, \eta].
\end{gather*}

The pseudogradient $Y$ constructed in the proof of \cite[Theorem 4.1]{AS2009} satisfies 
$$\langle Y_\gamma, \eta\rangle_0=D^2\mathcal{A}_L(0)(\gamma)[\eta], \quad \gamma \in \Omega^{1,2}(U; 0, 0),$$
where the Hilbert inner product $\langle \cdot, \cdot \rangle_0$ is given by
$$\langle \xi, \eta \rangle_0=\int_0^1\left(\frac{\partial^2L}{\partial v^2}(t,0,0)\dot{\xi}\cdot\dot{\eta}+\xi\cdot\eta\right)dt$$
on $W^{1,2}([0,1], \R^n)$. We remark that $\langle \cdot, \cdot \rangle_0$ is equivalent to $\langle \cdot, \cdot \rangle_{{W^{1,2}}}$. To extend $Y$ to $(-\varepsilon_T, \varepsilon_T) \times \Omega^{1,2}(U; 0, 0)$ let us define a family of Hilbert products on $\Omega^{1,2}(\R^n;0,0)$ varying over $(-\varepsilon_T, \varepsilon_T)$ by the formula
$$\langle \xi, \eta \rangle_h=\int_0^1\left(\frac{\partial^2L_h}{\partial v^2}(t,0,0)\dot{\xi}\cdot\dot{\eta}+(\partial_tA_h)\xi\cdot\eta \right) dt.$$
This family of Hilbert products satisfies
\begin{equation}\label{eq: metric-invariance}
    \widetilde{A}_h^*\langle \cdot , \cdot \rangle_h=\langle \cdot, \cdot \rangle_0,
\end{equation}
which is verified by a direct computation
\begin{align*}
    \langle  d\widetilde{A}_h(\xi), d\widetilde{A}_h(\eta) \rangle_h &=\int_0^1\left(\frac{1}{\partial_tA_h}\frac{\partial^2L}{\partial v^2}(A_h(t),0,0)(\partial_tA_h)^2\dot{\xi}(A_h(t))\cdot \dot{\eta}(A_h(t))+ 
    (\partial_tA_h)\xi(A_h(t))\cdot \eta(A_h(t)\right)dt \\ 
    &=\int_0^1\left(\frac{\partial^2 L}{\partial v^2}(t,0,0)\xi(t) \cdot \eta(t)+\xi(t) \cdot \eta(t)\right)dt=\langle \xi, \eta\rangle_0,
\end{align*}
using the change-of-variables formula. The product $\langle \cdot, \cdot \rangle_h$ varies smoothly with respect to $h$ since $L_h$ and its partial derivatives vary smoothly.

Let $\widetilde{Y}$ be the vector field satisfying
\begin{equation}\label{eq: parametric-vector-field-defn}
    \langle \widetilde{Y}_{(h, \gamma)}, \eta \rangle_h=D^2\mathcal{A}_{L_h}(0)(\gamma)[\eta], 
\end{equation}
for any $(h, \gamma) \in (-\varepsilon_T, \varepsilon_T) \times \Omega^{1,2}(U; 0, 0)$ and $\eta \in \Omega^{1,2}(\R^n; 0,0)$.
Clearly $\widetilde{Y}$ is a smooth vector field since $\langle \cdot, \cdot \rangle_h$ and $D^2\mathcal{A}_{L_h}$ vary smoothly with respect to $h$ and $D^2\mathcal{A}_{L_h}(0)$ is linear for each value of $h \in (-\varepsilon_T, \varepsilon_T)$.

We claim that
\begin{equation}\label{eq: second-differential-invariance}
    D^2\mathcal{A}_{L_h}(0)(P_h(\gamma))[d\widetilde{A}_h(\eta)]=D^2\mathcal{A}_L(0)(\gamma)[\eta]
\end{equation}
for each $\gamma\in \Omega^{1,2}(U; 0,0)$, which is again verified by the direct computation
\begin{align*}
D^2\mathcal{A}_{L_h}(0)&(P_h(\gamma))[d\widetilde{A}_h(\eta)]=\int_0^1\left( \partial_tA_h\frac{\partial^2 L}{\partial v^2}(A_h(t), 0, 0)\dot{\gamma}(A_h(t))\cdot \dot{\eta}(A_h(t))\right.\\ 
& +\partial_tA_h \frac{\partial^2L}{\partial x \partial v}(A_h(t), 0,0)\gamma(A_h(t))\cdot\dot{\eta}(A_h(t))
+\partial_tA_h \frac{\partial^2L}{\partial v \partial x}(A_h(t), 0,0)\dot{\gamma}(A_h(t))\cdot\eta(A_h(t))\\
&\left. +\partial_tA_h \frac{\partial^2L}{\partial x^2}(A_h(t),0,0)\gamma(A_h(t))\cdot\eta(A_h(t))\right) dt= D^2\mathcal{A}_L(0)(\gamma)[\eta],
\end{align*}
using the change-of-variables formula.

Combining~\eqref{eq: metric-invariance} and~\eqref{eq: second-differential-invariance}, for any $\gamma \in \Omega^{1,2}(U; 0,0)$ and $h \in (-\varepsilon_T, \varepsilon_T)$,
\begin{gather*}
\langle d \widetilde{A}_h(\widetilde{Y}_{(0, \gamma)}), d\widetilde{A}_h(\eta) \rangle_h=D^2\mathcal{A}_L(0)(\gamma)[\eta]=
D^2\mathcal{A}_{L_h}(0)(P_h(\gamma))[d\widetilde{A}_h(\eta)]=\langle \widetilde{Y}_{\widetilde{A}_h(0,\gamma)}, d\widetilde{A}_h(\eta) \rangle_h.
\end{gather*}
Hence the invariance condition~\eqref{eq: vector-field-invariance-chart} holds for $\widetilde{Y}$.

\s\n
{\em Partition of unity argument.} 
Pick a cover of $(-\varepsilon_T, \varepsilon_T) \times \Omega^{1,2}(M, q, q')$ by open charts of the form $\widetilde{\mathcal{U}}_{(U, \tilde{\varphi}^{A})}$.  Denote by $\widetilde{\mathcal{U}}_\gamma$ such an open chart centered at the smooth curve $\gamma$ and $\mathcal{U}_\gamma$ the corresponding chart of the form $\mathcal{U}_{(U, \varphi)}$ centered at $\gamma$. By paracompactness, there is a locally finite, countable subcover $\{\mathcal{U}_{\gamma_j}\}_{j\in \mathcal{J}}$ of the cover 
$$\{\mathcal{U}_\gamma \mid   \gamma \in C^\infty([0,1], M) \cap \Omega^{1,2}(M, q, q')\}.$$ 
Writing $\varphi_j$ for the smooth map associated with the chart $\mathcal{U}_{\gamma_j}$, let $\widetilde{\mathcal{U}}_{\gamma_j}$ be the image of $(-\varepsilon_T, \varepsilon_T) \times (\varphi_j)_*^{-1}(\mathcal{U}_{\gamma_j})$ under $(\tilde{\varphi}_j)_*^{A}$. 

Let $\{\psi_j\}_{j \in \mathcal{J}}$ be a partition of unity subordinate to the cover $\{\mathcal{U}_{\gamma_j}\}_{j \in \mathcal{J}}$. We define a smooth function $\widetilde{\psi}_j$ supported on $\widetilde{\mathcal{U}}_{\gamma_j}$ by setting 
$$\widetilde{\psi}_{j}(h, \gamma(\cdot))=\psi_j(\gamma(A_h^{-1}(\cdot))$$ 
for each point $(h, \gamma)$ in $(-\varepsilon_T, \varepsilon_T) \times (\varphi_j)_*^{-1}(\mathcal{U}_{\gamma_j})$. Notice that $\{\widetilde{\psi}_j\}_{j \in \mathcal{J}}$ constitute a partition of unity and each one of them satisfies $\widetilde{\psi}_j(\widetilde{A}_h(0,\gamma))=\widetilde{\psi}_j(0, \gamma)$.
Finally, we set 
\begin{equation}
    \widetilde{X}=\textstyle\sum_{j \in \mathcal{J}}\widetilde{\psi}_j \widetilde Y_{\mathcal{I}(j)},
\end{equation}
where $\widetilde Y_{\mathcal{I}(j)}$ is of one of two types constructed above depending on whether ${\mathcal{I}(j)}$ is a critical point of $\mathcal{A}_L$. Combining the invariance properties of $\widetilde{\psi}_j$ and of $\widetilde Y_{\mathcal{I}(j)}$ shown above we conclude that $\widetilde{X}$ indeed satisfies (d).  Since our construction is compatible with the one in \cite[Theorem 4.1]{AS2009}, the restriction of $\widetilde{X}$ to $\{0\} \times \Omega^{1,2}(M, q, q')$ is guaranteed to be a pseudogradient.
\end{proof}

\section{Transversality}\label{appendix: transversality}

\subsection{Morse flow lines with switchings.}\label{subappendix: mfls} 

As in the proof of Theorem~\ref{thm: regularity for pseudogradient} (now we are in the case of arbitrary $\kappa\geq 1$), there exists a countable locally finite cover $\{\mathcal{U}_{\bm \gamma_j}\}_{j \in \mathcal{J}}$ of $\Omega^{1,2}(M, \bm q, \bm q')$, where $\bm \gamma_j=(\gamma_{j1},\dots, \gamma_{j\kappa})$ is a $\kappa$-tuple of smooth paths and $\mathcal{U}_{\bm \gamma_j}$ is the product of the corresponding $\mathcal{U}_{\gamma_{ji}}$.  We also assume the following: 
\be
\item[(C1)] if some $\gamma\in \mathcal{U}_{\gamma_{ji}}$ is a critical point, then $\gamma$ lies in only one chart $\mathcal{U}_{\gamma_{ji}}$ (called a {\em critical point chart}) and in that case $\gamma=\gamma_{ji}$;
\item [(C2)] critical point charts will be arbitrarily small and will be shrunk as necessary in response to given submanifolds in $[0,1]^\ell\times \Omega^{1,2}(M,\bm q, \bm q').$
\ee

Let $\mathfrak{X}_j$ be the subspace of $T_{\bm \gamma_j}\Omega^{1,2}(M, \bm q, \bm q')$ which is equal to the direct sum 
$$V_{j1} \oplus \dots \oplus V_{j\kappa} \subset T_{\gamma_{j1}}\Omega^{1,2}(M, q_1, q'_{\sigma(1)}) \oplus \dots \oplus T_{\gamma_{j\kappa}}\Omega^{1,2}(M, q_\kappa, q'_{\sigma(\kappa)})$$ 
where $V_{ji}= C_{\bm\varepsilon}(\bm \gamma_{ji}^*TM)$ if $\gamma_{ji}$ is not a critical point and $0$ otherwise. Here we denote by $C_{\bm\varepsilon}(\bm \gamma_{ji}^*TM)$ the separable Banach space of sections $Y$ with finite Floer $C_{\bm \varepsilon}$-norm \cite{Floer-unregularized}: 
$$\|Y \|_{C_{\bm \varepsilon}}=\sum_{k=1}^\infty \varepsilon_k \| Y \|_{C^k} < \infty,$$
where $\bm \varepsilon=(\varepsilon_k)_{k\ge 1}$ is a fixed sequence of sufficiently small positive numbers. We pick the sequence $\bm \varepsilon$ such that for each section $Y' \in W^{1,2}(\bm \gamma_{ji}^*TM)$  and each point $p$ on $\gamma_{ji}$ there is a section $Y \in C_{\bm\varepsilon}(\bm \gamma_{ji}^*TM)$ satisfying $Y(p)=Y'(p)$ and $\|Y-Y' \|_{W^{1,2}}$ is smaller than a given $\delta>0$. The existence of such $\bm \varepsilon$ is shown in \cite[Theorem B.6]{wendl2008}. Moreover we fix a choice of $\bm \varepsilon$ that satisfies the above property for all $\bm \gamma_j$ with $j \in \mathcal{J}$.

We regard an element $Y_j \in \mathfrak{X}_j$ as a vector field over $\mathcal{U}_{\bm \gamma_j}$.  
Consider the direct sum $\bigoplus_{j \in \mathcal{J}} \mathfrak{X}_j$ with supremum norm, i.e., it consists of bounded sequences $\{Y_j\}_{j \in \mathcal{J}}$. Clearly, it is a separable Banach space. We fix a partition of unity $\{\psi_j\}_{j\in \mathcal{J}}$ subordinate to the cover $\{\mathcal{U}_{\bm \gamma_j}\}_{j \in \mathcal{J}}$. We set
\begin{equation} \label{eqn: defn of frak X}
\mathfrak{X}=\left\{X \in \Gamma(T\Omega^{1,2}(M, \bm q, \bm q'))\mid   \, X= \sum_{j\in \mathcal{J}} \psi_jY_j, \, \text{ for some } (Y_j)_j \in \bigoplus_{j \in \mathcal{J}}\mathfrak{X}_j\right\}.
\end{equation}

For a fine enough cover $\{\mathcal{U}_{\bm \gamma_j}\}_{j \in \mathcal{J}}$ this space, which is a Banach subspace of the Banach space of bounded sections of $T\Omega^{1,2}(M, \bm q, \bm q')$, is isomorphic to the direct sum $\bigoplus_{j \in \mathcal{J}} \mathfrak{X}_j$. In particular, it is a separable smooth Banach manifold. 

\begin{lemma}\label{lemma: vector field}
    Given $\bm Z=(Z_1,\dots,Z_\kappa) \in T_{\bm \gamma}\Omega^{1,2}(M, \bm q, \bm q')$ for a smooth tuple of paths $\bm \gamma =(\gamma_1,\dots,\gamma_\kappa)\in \Omega^{1,2}(M, \bm q, \bm q')$ such that $\bm Z$ is of class $C_{\bm\varepsilon}$ along $\bm \gamma$ and $Z_k=0$ whenever $\gamma_k$ is in a critical point chart, there exists a vector field $X \in  \mathfrak{X}$ such that $X_{\bm \gamma}=Z$.
\end{lemma}

\begin{proof} 
Let $\{\psi_j\}_{j\in \mathcal{J}}$ be a partition of unity subordinate to the cover $\{\mathcal{U}_{\gamma_j}\}_{j\in \mathcal{J}}$ of $\Omega^{1,2}(M, \bm q,\bm q')$.
To construct the vector field $X\in \mathfrak{X}$, we set $Y_i=Z/\psi_i(\bm\gamma)$ for some $i$ such that $\bm \gamma \in \mathcal{U}_{\bm\gamma_i}$ and $\psi_i(\bm\gamma)\not=0$, and set $Y_j=0$ on $\mathcal{U}_{\bm\gamma_j}$ for $j \neq i$. Then $(Y_j)_j\in \bigoplus_{j \in \mathcal{J}} \mathfrak{X}_j$ and $X= \sum \psi_jY_j$ is the desired vector field.
\end{proof}

Since $\mathfrak{X}$ is separable, the space of maps $C_{\bm\varepsilon}([0,1], \mathfrak{X})$ is a separable Banach space.  The ``parameter'' spaces that we use are subspaces of spaces of a similar type and hence are also separable, allowing us to use the Sard-Smale theorem. In particular, we set
\begin{equation}
    \mathfrak{X}_-=
    \{X \in C_{\bm\varepsilon}((-\infty, 0], \mathfrak{X}) \, \mid   \, X(s)=0 \, \, \forall s \le -1 \ \},
\end{equation}
\begin{equation}
    \mathfrak{X}_+=
    \{X \in C_{\bm\varepsilon}([0,+\infty), \mathfrak{X}) \, \mid   \, X(s)=0 \, \, \forall s \ge 1 \ \}
\end{equation}
Next, fix a nondecreasing smooth function $g \colon [0, +\infty) \to [0, +\infty)$ such that $g(l)=\frac{l}{3}$ for $l \le 1$ and $g(l)=1$ for $l\ge 3$, and define
\begin{gather}\label{eq: finite-length-parameter-space}
    \mathfrak{X}_0=
    \{X\in C_{\bm\varepsilon}([0,+\infty)^2, \mathfrak{X})~ \mid  ~ \mbox{$X(l,s)=0$ for $s \in  (g(l), \op{max}(g(l),l-1))\cup [l, +\infty)$}; \\
    \lim_{l \to +\infty}\op{split}X(l,\cdot) \text{ exists in } \mathfrak{X}_+ \times \mathfrak{X}_-; \text{ all partial derivatives} \nonumber\\
    \text{ of } X 
    \text{ vanish at } (0,s) \text{ for all } s \in [0,+\infty)\}, \nonumber
\end{gather}
where, writing $X_l(s)\coloneqq X(l,s)$ and $X_l^T(s)\coloneqq  X_l(s+T)$,
$$\op{split}X_l=(X_l|_{[0,1]}\#0|_{[1, +\infty)}, 0|_{(-\infty,-1]}\#X_l^{l+1}|_{[-1,0]}).$$
We refer to \cite[Remark 5.8]{Mescher2018} for an explanation of the technical condition on the vanishing of the partial derivatives in the definition. For our purposes we will only need the restriction $X_l|_{[0,l]}$ and quite often we will regard $X_l$ as a path of vector fields $X_l \in C_{\bm\varepsilon}([0, l], \mathfrak{X})$. 

\s
Fix a pseudogradient vector field $X$ (i.e., satisfying (PG1)--(PG5)) for a given Lagrangian $L \colon [0,1] \times TM \to \R$ on $\Omega^{1,2}(M, \bm q, \bm q')$ satisfying Theorem~\ref{thm: regularity for pseudogradient}. 

Given a critical point $\bm\gamma$, we set
\begin{equation}
\mathcal{P}_-(\bm \gamma)=\{ \Gamma \in C^1((-\infty, 0], \Omega^{1,2}(M, \bm q, \bm q') \mid   \, \lim_{s \to -\infty} \Gamma= \bm \gamma,  \, \lim_{s \to -\infty} \partial_s \Gamma=0\},
\end{equation}
which is a smooth Banach manifold. There is a Banach vector bundle $\mathcal{F}_-$ over $\mathcal{P}_-(\bm\gamma)$ with fibers 
$$\mathcal{F}_-|_\Gamma=C^0_0(\Gamma^*(T\Omega^{1,2}(M, \bm q, \bm q'))),$$
where the subscript $0$ means that the limit $s \to -\infty$ is equal to $0$. Then the unstable manifold of $\bm\gamma$ with respect to $X$ can be identified with the set of zeros of the section 
\begin{gather*} 
    \mathcal{P}_-(\bm \gamma) \to \mathcal{F}_-, \quad \Gamma\mapsto (\partial_s\Gamma+X)|_{\Gamma(s)},
\end{gather*}
via the evaluation map $\Gamma\mapsto\Gamma(0)$.

More generally, there is a vector bundle $\mathcal{F}_-$ over $\mathcal{P}_- (\bm \gamma)\times \mathfrak{X}_-$
with fibers 
$$\mathcal{F}_-|_{(\Gamma,Y)}=C^0_0(\Gamma^*(T\Omega^{1,2}(M, \bm q, \bm q'))),$$
and a section 
\begin{gather*}
    \sigma_- \colon \mathcal{P}_-(\bm \gamma) \times \mathfrak{X}_-\to \mathcal{F}_-, \quad (\Gamma, Y)\mapsto (\partial_s\Gamma+X+Y)|_{\Gamma(s)}.
\end{gather*}
We then define
\begin{equation}
    W^u(\bm \gamma; \mathfrak{X}_-)\coloneqq \sigma_-^{-1}(0), \quad W^u(\bm \gamma; Y)=  W^u(\bm \gamma; \mathfrak{X}_-)\cap (\mathcal{P}_-(\bm \gamma) \times \{Y\}).
\end{equation}
We also have an evaluation map
\begin{equation}
E_- \colon W^u(\bm \gamma; \mathfrak{X}_-) \to \Omega^{1,2}(M, \bm q, \bm q'),\quad 
(\Gamma, Y)\mapsto \Gamma(0).
\end{equation}

\begin{lemma}\label{lemma: manifold of trajectories to submanifold} $\mbox{}$
\be
\item
The space $W^u(\bm \gamma; \mathfrak{X}_-)$ is smooth Banach manifold.
\item The evaluation map $E_-$ has image in the space of smooth paths $C^\infty(M, \bm q, \bm q')$.
\item Let $U$ be an open subset of $W^u(\bm \gamma; \mathfrak{X}_-)$ consisting of $(\Gamma, Y)$ such that $\Gamma(0)$ is contained in a neighborhood $\mathcal{U}_{\bm \gamma_j}=\mathcal{U}_{\gamma_{j1}}\times\dots\times \mathcal{U}_{\gamma_{j\kappa}}$ with all $\gamma_{ji}$ non-critical. 
Then the differential of $E_-$ restricted to $U$ has any $C_{\bm\varepsilon}$-section in its image.
\ee
\end{lemma}

\begin{proof}
    (1) As in \cite[Theorem 12]{Schwarz93} one can show that $\sigma_-$ is a smooth map. Hence it remains to show that $\sigma_-$ is transverse to the $0$-section. We compute:
    $$D\sigma_-(\Gamma, Y)(\Xi, Z)=(\nabla_s\Xi+\nabla_{\Xi}(X+Y)+Z)|_{\Gamma(s)}.$$
    Its surjectivity is standard and shown in \cite[Proposition 1.6]{abbondandolo-majer2006} (in fact the $Z$ term is not necessary for the surjectivity). Hence $W^u(\bm \gamma; \mathfrak{X}_-)$ is a smooth Banach manifold.
    
    (2) One can show as in Appendix~\ref{section: regularity} that for each $Y \in \mathfrak{X}_-$ the unstable manifold $W^u(\bm \gamma; Y)$ consists of smooth paths. Hence the image of $E_-$ is contained in $C^\infty(M, \bm q, \bm q')$.

    (3)  Let $(\Gamma,Y)\in U$. Given $v \in T_{\Gamma(0)}\Omega^{1,2}(M, \bm q, \bm q')$ with finite $C_{\bm\varepsilon}$-norm, we claim that there exists 
    $$(\Xi, Z) \in \operatorname{ker}D\sigma_-(\Gamma,Y)=T_{(\Gamma, Y)}W^u(\bm \gamma; \mathfrak{X}_-)$$ 
    such that $DE_- (\Gamma,Y)(\Xi, Z)=v.$ To this end, let $\Xi=0$ (and hence satisfies 
    $\nabla_s\Xi(s)+\nabla_{\Xi}(X+Y)=0$) on $s\in (-\infty,-1]$ 
    and extend $\Xi$ over $s\in [-1, 0]$ so that $\Xi(0)=v$, $\Xi\in T_\Gamma \mathcal{P}_-(\bm\gamma)$, and 
    $$Z(s)|_{\Gamma(s)}=-\nabla_s\Xi(s)-\nabla_{\Xi(s)}(X_{\Gamma(s)}+Y(s)|_{\Gamma(s)})$$
    is nonzero only near $s=0$ where $\Gamma(s)\in \mathcal{U}_{\bm\gamma_j}$. Moreover we can guarantee that it has finite $C_{\bm\varepsilon}$-norm as a section over $(-\infty, 0]$.
    
    Finally, using the method of proof of Lemma~\ref{lemma: vector field}, there exists an extension of $Z(s)|_{\Gamma(s)}$, $s \in (-\infty, 0]$, to a vector field $Z(s) \in \mathfrak{X}$, subject to $Z$ being a  $C_{\bm\varepsilon}$-section over $(-\infty, 0]$.
    
    Explicitly, let $\{\mathcal{U}_{\bm\gamma_{j_1}}, \dots, \mathcal{U}_{\bm\gamma_{j_k}}\}$ be a finite subcollection of $\{\mathcal{U}_{\bm \gamma_j}\}_{j \in \mathcal{J}}$ covering the image of $[-1, 0]$ under $\Gamma$ and let $\nu_{j_1}, \dots, \nu_{j_k} \colon [-1, 0] \to \R$ be a partition of unity subordinate to the pullback cover of $[-1, 0]$. We then set 
    $$X_i(s)=\frac{Z(s)|_{\Gamma(s)}\nu_i(s)}{\psi_i}, \quad \text{ for  $i=j_1, \dots, j_k$ ~~ and ~~ $\psi_i\not=0$},$$ 
    and set $X_i=0$ otherwise. Then $Z(s)=\sum \psi_iX_i(s)$ is the desired family of vector fields in $\mathfrak{X}_-$.
\end{proof}

\begin{lemma}\label{lemma: unstable-to-diagonal}
Let $N \subset [0,1]^{\ell} \times \Omega^{m,2}(M, \bm q, \bm q')$ be a $C^{m-1}$-submanifold of codimension $n>0$ which is disjoint from 
\begin{equation} \label{eqn: tilde mathcal C}
    \tilde{\mathcal{C}}\coloneqq [0,1]^\ell\times \{ \bm \gamma=(\gamma_1,\dots,\gamma_\kappa)\mid \mbox{some $\gamma_i$ is a critical point }\},
\end{equation}
and, at each point $(\bm \theta, \bm \gamma) \in N$ with $\bm \gamma$ a tuple of smooth paths,
\begin{gather}\label{eq: transversality condition}
    \text{the tangent space } T_{(\bm \theta, \bm \gamma)}N \text{ is transverse to the space of $C_{\bm\varepsilon}$-sections along $\bm \gamma$,  viewed as a subset of}\\
    \{0\} \times T\Omega^{m,2}(M, \bm q, \bm q') \subset \R^{\ell} \times T\Omega^{m,2}(M,\bm q, \bm q').\nonumber
\end{gather}
Then:
\be
\item The space 
$$\tilde W^u(\bm \gamma, N; \mathfrak{X}_-) \coloneqq  \{(\bm\theta, \Gamma) \in [0,1]^{\ell} \times W^u(\bm \gamma; \mathfrak{X}_-) \mid  (\bm \theta, \Gamma(0)) \in N\}$$
is a $C^{m-1}$-smooth Banach manifold.\footnote{``Thickened'' versions of $W^u(\bm\gamma;\mathfrak{X}_-)$ etc.\ obtained by multiplying by $[0,1]^\ell$ will have tildes over them.}
\item For a generic $Y \in \mathfrak{X}_-$, provided $m>\operatorname{ind}(\bm \gamma)-n$,
$$\tilde W^u(\bm \gamma, N; Y) \coloneqq \tilde W^u(\bm \gamma, N; \mathfrak{X}_-) \cap W^u(\bm\gamma; Y)$$ 
is a $C^{m-1}$-smooth manifold of dimension $\operatorname{ind}(\bm \gamma)+\ell-n$ and the evaluation map 
$$\tilde W^u(\bm \gamma, N; Y)\to N, \quad (\bm\theta,\Gamma)\mapsto \Gamma(0),$$ 
is an immersion.
\ee
\end{lemma}

\begin{proof}
(1) Consider the evaluation map (a ``thickening'' of $E_-$) 
\begin{gather} \label{eqn: tilde E minus}
    \tilde E_-\colon [0,1]^{\ell} \times W^u(\bm \gamma; \mathfrak{X}_-) \to [0,1]^{\ell} \times \Omega^{m,2}(M, \bm q, \bm q'),\\
    (\bm\theta, \Gamma) \mapsto (\bm\theta, \Gamma(0)).\nonumber
\end{gather}
There exists an open subset $U\subset W^u(\bm \gamma; \mathfrak{X}_-)$ such that $N\cap \op{Im}(\tilde E_-)\subset \tilde E_-([0,1]^\ell\times U)$ and $U$ satisfies the assumptions of Lemma~\ref{lemma: manifold of trajectories to submanifold}(3).  This is because $N$ is a submanifold and (C2) allows us to shrink critical point charts in response to $N$.  Hence differential of $\tilde E_-$ on $U$ by Lemma~\ref{lemma: manifold of trajectories to submanifold}(3) has any $C_{\bm\varepsilon}$-section in its image and using~\eqref{eq: transversality condition}  we conclude that the preimage $\tilde W^u(\bm \gamma, N; \mathfrak{X}_-)=\tilde E_-^{-1}(N)$ is a $C^{m-1}$-submanifold.

(2) Let $\tilde \sigma_-= \op{id}_{[0,1]^\ell}\times \sigma_-$ and let $D_1\tilde\sigma_-{(\bm\theta,\Gamma, Y)}$ be the component of $D\tilde\sigma_-{(\bm\theta,\Gamma, Y)}$ in the $(\bm\theta,\Gamma)$-direction, i.e.,
\begin{gather*}
   D_1\tilde\sigma_-{(\bm\theta,\Gamma, Y)} \colon \R^{\ell} \times C^1_0(\Gamma^*T\Omega^{m,2}(M, \bm q, \bm q')) \to C^0_0(\Gamma^*T\Omega^{m,2}(M, \bm q, \bm q')),\\
   (\varphi, \Xi)\mapsto \nabla_s\Xi+\nabla_{\Xi}(X+Y);
\end{gather*}
it is Fredholm of index $\operatorname{ind}(\bm \gamma)+\ell$  at any point $(\bm\theta, \Gamma, Y) \in [0,1]^{\ell} \times W^u(\bm \gamma; \mathfrak{X}_-)$. 
Here we have used the fact that $X|_\Gamma$ is (at least) $C^m$. Then $\tilde W^u(\bm \gamma, N; \mathfrak{X}_-)=(\tilde\sigma_-|_{\mathcal{P}_-(\bm\gamma,N)\times \mathfrak{X}_-})^{-1}(0)$, where
$$\mathcal{P}_-(\bm \gamma, N)=\{(\bm\theta, \Gamma) \in [0,1]^{\ell} \times \mathcal{P}_-(\bm \gamma)\mid   (\bm\theta, \Gamma(0)) \in N\}.$$
The tangent space of this manifold at $(\bm\theta, \Gamma)$ is given by those $(\varphi,\Xi)$ such that $(\varphi, \Xi(0)) \in TN$, which is a codimension $n$ linear subspace of $\R^{\ell} \times C^1_0(\Gamma^*T\Omega^{m,2}(M, \bm q, \bm q'))$. The restriction of $D_1\tilde \sigma_-$ to this subspace has index $\operatorname{ind}(D_1\tilde\sigma_-)-n$ and 
the claim now follows from \cite[Lemma A.3.6]{McDuff-Salamon-2012} and the Sard-Smale theorem.
\end{proof}

\begin{remark}
   The manifold $\tilde W^u(\bm \gamma, N; \mathfrak{X}_-)$ is a fiber product $$([0,1]^{\ell}\times W^u(\bm \gamma; \mathfrak{X}_-)) \times_{[0,1]^{\ell} \times \Omega^{m,2}(M, \bm q, \bm q')} N$$ and fits into the following pullback diagram
    \[\begin{tikzcd}
	\tilde W^u(\bm \gamma, N; \mathfrak{X}_-) & {[0,1]^{\ell}\times W^u(\bm \gamma; \mathfrak{X}_-)}\\
	N &{[0,1]^{\ell}\times \Omega^{m,2}(M, \bm q, \bm q').}
	\arrow[hook, from=1-1, to=1-2]
	\arrow["\tilde{E}_-", from=1-1, to=2-1]
	\arrow["\tilde{E}_-", from=1-2, to=2-2]
	\arrow[hook, from=2-1, to=2-2]
\end{tikzcd}\]
\end{remark}

\smallskip
Let $N_1, N_2$ be $C^{m-1}$-smooth manifolds and let $\iota_j \colon N_j \to [0,1]^{\ell} \times \Omega^{m,2}(M, \bm q, \bm q')$ be $C^{m-1}$-smooth maps for $j=1,2$. We introduce the following spaces (note that we often write $N_j$ instead of $(N_j, \iota_j)$, for simplicity): 
\begin{gather}
    \tilde W^{ft}(N_1, N_2; \mathfrak{X}_0)=\{(\bm\theta, l, \Gamma, Y,  x_1, x_2) \mid   \bm\theta\in [0,1]^\ell,  \, l>0,\\
    \Gamma \colon [0, l] \to [0,1]^{\ell} \times \Omega^{m,2}(M, \bm q, \bm q'), Y\in \mathfrak{X}_0, (x_1,x_2) \in N_1 \times N_2,\nonumber \\ 
    \partial_s \Gamma+X+Y=0, (\bm\theta, \Gamma(0))=\iota_1(x_1), (\bm\theta, \Gamma(l))=\iota_2(x_2)\}, \nonumber
\end{gather}
\begin{gather}
    \tilde W^{ft}(N_1, -; \mathfrak{X}_0)=\{(\bm\theta, l, \Gamma, Y, x_1) \mid   \bm\theta\in [0,1]^\ell, \, l>0,\\
    \Gamma \colon [0, l] \to \Omega^{m,2}(M, \bm q, \bm q'), \, Y\in \mathfrak{X}_0, x_1 \in N_1,\nonumber\\ 
    \partial_s \Gamma+X+Y=0, (\bm\theta, \Gamma(0)) = \iota_1(x_1)\}, \nonumber
\end{gather}
\begin{gather}
    \tilde W^{ft}(-, N_2; \mathfrak{X}_0)=\{(\bm\theta, l, \Gamma, Y, x_2) \mid   \bm\theta\in [0,1]^\ell,  \, l>0,\\
    \Gamma \colon [0, l] \to  \Omega^{1,2}(M, \bm q, \bm q'), \, Y\in \mathfrak{X}_0, x_2 \in N_2,\nonumber \\ 
    \partial_s \Gamma+X+Y=0, (\bm\theta, \Gamma(l)) = \iota_2(x_2)\}, \nonumber
\end{gather}
\begin{gather}
    \tilde W^{ft}(-, -; \mathfrak{X}_0)=\{(\bm\theta, l, \Gamma, Y) \mid   \bm\theta\in [0,1]^\ell,  \, l>0,\\ 
    \Gamma \colon [0, l] \to \Omega^{1,2}(M, \bm q, \bm q'), \, Y\in \mathfrak{X}_0, \, \partial_s \Gamma+X+Y=0\}. \nonumber
\end{gather}
The superscript ``$ft$'' stands for ``finite time''. 

The space $\tilde W^{ft}(-, -; \mathfrak{X}_0)$ comes with an evaluation map 
\begin{gather} \label{eqn: tilde E zero}
    \tilde E_0=(\tilde E_0^s, \tilde E_0^t) \colon \tilde W^{ft}(-, -; \mathfrak{X}_0) \to ([0,1]^{\ell} \times \Omega^{1,2}(M, \bm q, \bm q'))^{2},\\
    (\bm\theta, l, \Gamma, Y) \mapsto (\bm\theta, \Gamma(0), \bm\theta, \Gamma(l)).\nonumber
\end{gather}
The evaluation maps of the other spaces are defined analogously.

\begin{lemma}\label{lemma: finite-trajectories-two-immersions} 
Let $N_1$ and $N_2$ be $C^{m-1}$-smooth Banach manifolds and $\iota_j \colon N_j \to [0,1]^{\ell} \times \Omega^{m,2}(M, \bm q, \bm q')$, $j=1,2$, be $C^{m-1}$-immersions such that $\op{dim} N_1$ is finite and $\iota(N_1) \subset [0,1]^{\ell} \times C^\infty(M, \bm q, \bm q')$, $\iota_2$ is Fredholm, $\iota_1$ and $\iota_2$ intersect transversely, and
\be
\item[(*)] at each point $(\bm\theta,\bm\gamma)$ of $\iota_2(N_2)$ which belongs to $ [0,1]^{\ell} \times C^\infty(M, \bm q, \bm q')$, $T_{(\bm\theta,\bm\gamma)}\iota_2(N_2)$ is transverse to the space of $C_{\bm\varepsilon}$-sections as in \eqref{eq: transversality condition};   and
\item[(**)] $\iota_2(N_2)$ is disjoint from an open neighborhood of $\tilde{\mathcal{C}}$ given by \eqref{eqn: tilde mathcal C}.
\ee
Then:
\begin{enumerate}
    \item $\tilde W^{ft}(-, -; \mathfrak{X}_0)$ is a $C^m$-smooth Banach manifold with boundary.
    \item $\tilde W^{ft}(N_1, N_2; \mathfrak{X}_0)$ and $\tilde W^{ft}(N_1, -; \mathfrak{X}_0)$ are $C^{m-1}$-smooth Banach manifolds with boundary.
    \item For a generic $Y\in \mathfrak{X}_0$, $\tilde W^{ft}(N_1, N_2; Y)$ and $\tilde W^{ft}(N_1, -; Y)$ are smooth manifolds with boundary and
    \begin{gather}
    \operatorname{dim}\tilde W^{ft}(N_1, N_2; Y)=\operatorname{dim}N_1+1-\operatorname{codim}\iota_2(N_2), \label{eq: dimension-lemma finite trajectories-2}\\
    \operatorname{dim}\tilde W^{ft}(N_1, -; Y)=\operatorname{dim}N_1+1. \label{eq: dimension-lemma finite trajectories-1}
    \end{gather}
\end{enumerate}
\end{lemma}

\begin{proof}[Sketch of proof]
(1) There is a $C^{m}$-diffeomorphism 
$$\tilde W^{ft}(-, -; \mathfrak{X}_0)\cong  [0,1]^{\ell} \times [0, +\infty) \times \Omega^{1,2}(M, \bm q, \bm q') \times \mathfrak{X}_0,$$ 
which associates to any $(\bm\theta, l, \bm\gamma, Y)$ a unique trajectory of length $l$ starting at $\gamma$ along $Y$. 

(2) This is analogous to the proof of Lemma~\ref{lemma: unstable-to-diagonal}(1) and a detailed account in the finite-dimensional setting can be found in \cite[Theorem 1.21]{Mescher2018}. 

We briefly discuss the case $\tilde W^{ft}(N_1, N_2; \mathfrak{X}_0)$; the case of $\tilde W^{ft}(N_1, -; \mathfrak{X}_0)$ is straightforward.
As in the proof of Lemma~\ref{lemma: unstable-to-diagonal}(1),  Condition (**) and (C2) allow us to shrink critical point charts in response to $\iota_2(N_2)$. By the proof of Lemma~\ref{lemma: manifold of trajectories to submanifold}(3) and Condition (*), the map $\tilde E_0$ is transverse to $(\iota_1, \iota_2)$.  (Roughly speaking, the transversality with $\iota_1$ is freely achieved, and the transversality with $\iota_2$ is a constrained one, analogous to the construction of $(\Xi,Z)$ which are zero away from $s=0$ in the proof of Lemma~\ref{lemma: manifold of trajectories to submanifold}(3).)
(2) then follows since 
$$\tilde W^{ft}(N_1, N_2; \mathfrak{X}_0)=N_1\tensor*[_{\iota_1}]{\times}{_{\tilde E_0^s}} \tilde W^{ft}(-, -; \mathfrak{X}_0) \tensor*[_{\tilde E_0^t}]{\times}{_{\iota_2}}N_2,$$
which is a $C^{m-1}$-smooth Banach manifold.

(3)  We will treat the $\tilde W^{ft}(N_1, N_2; Y)$ case.  Consider the manifold
\begin{align*}
\mathcal{P}(\iota_1, \iota_2) &\coloneqq \{(\bm\theta, l, \Gamma, x_1, x_2)\mid    \bm\theta \in [0,1]^{\ell}, l \in [0;+\infty), \Gamma \colon [0, l] \to \Omega^{m,2}(M, \bm q, \bm q'), (x_1, x_2) \in N_1 \times N_2, \\
& \qquad \qquad \qquad \qquad (\bm\theta, \Gamma(0))=\iota_1(x_1), (\bm\theta, \Gamma(l))=\iota_2(x_2) \}.
\end{align*}
Then $\tilde W^{ft}(N_1, N_2; \mathfrak{X}_0)$ is the zero set of the following map: 
\begin{gather*}
\tilde\sigma_0 \colon \mathcal{P}(\iota_1, \iota_2) \times \mathfrak{X}_0 \to \mathcal{F}_0,\quad 
(\bm\theta,l,\Gamma,x_1, x_2, Y)\mapsto\partial_s\Gamma+X+Y_l,
\end{gather*}
where $\mathcal{F}_0$ is the vector bundle over $\mathcal{P}(\iota_1, \iota_2) \times \mathfrak{X}_0$ with fiber
$$\mathcal{F}_0|_{(\bm\theta,l,\Gamma,x_1, x_2, Y)}=C^0_0(\Gamma^*(T\Omega^{1,2}(M, \bm q, \bm q'))).$$ 
Here the unstable manifold $\tilde W^{ft}(N_1, -; \mathfrak{X}_0)$ consists of (trajectories along) smooth paths, by our choice of $X$ and construction of $\mathfrak{X}_0$, and an argument similar to that of the proof of Theorem~\ref{thm: regularity for pseudogradient}.

As before, we write $D_1\sigma_0$ for the component of $D\sigma_0$ that does not depend on the variation along the parameter space $\mathfrak{X}_0$. Then by the existence and uniqueness of solutions of ODEs we conclude that $\operatorname{dim} \operatorname{ker} D_1\sigma_0=\operatorname{dim}N_1+1$, and by \cite[Lemma 2.21 (iii)]{abbondandolo-majer2006} this operator is surjective. Clearly, the restriction of $D_1\sigma_0$ to $\mathcal{P}(\iota_1, \iota_2) \times \mathfrak{X}_0 $ then has index $\operatorname{dim}N_1+1-\operatorname{codim}_{\iota_2}N_2$ as in the proof of Lemma~\ref{lemma: unstable-to-diagonal}, which concludes the stated index computation.
\end{proof}

We are finally in a position to introduce the cast of characters that appear in the statement and the proof of the main result (Theorem~\ref{theorem: transversality-for-MFLS}) of Appendix~\ref{subappendix: mfls} . 

Let $\bm \gamma, \bm \gamma' \in \Omega^{1,2}(M, \bm q, \bm q')$ be critical points, $\ell$ a positive integer, and $\bm \tau=(\tau_1, \dots, \tau_\ell)$ an $\ell$-tuple of pairs $\tau_1=\{i_1, j_1\}, \dots, \tau_\ell=\{i_\ell, j_\ell\}$ of distinct elements of $\{1, \dots, \kappa\}$.  Recall from Section~\ref{section: switching map} that we have switching maps $sw_k^{ij} \colon \overline{\Delta}_{I_k}^{ij} \to \overline{\Delta}_{I_k}^{ij}$.
Let $\bm s=(s_1, \dots, s_\ell)$ be a total order on $[\ell]$ and let 
$$\bm Y=(Y_-, Y_1, \dots, Y_{\ell-1}, Y_+) \in \mathfrak{X}_- \times (\mathfrak{X}_0)^{\ell-1} \times \mathfrak{X}_+.$$  
Then we set:

\begin{gather}
    \tilde W_\ell(\bm \gamma, \Delta_{I_{s_\ell}}^{i_\ell j_\ell}; \mathfrak{X}_- \times (\mathfrak{X}_0)^{\ell-1}, \bm \tau, \bm s )=\tilde W^u(\bm \gamma, {\Delta}_{I_{s_1}}^{i_1j_1};\mathfrak{X}_-) \tensor*[_{sw^{i_1j_1}_{I_{s_1}} \circ \tilde E_-}]{\times}{_{\tilde E_0^s}}\tilde W^{ft}(-,\Delta^{i_2j_2}_{t_{s_2}}; \mathfrak{X}_0 ) \\ 
    \tensor*[_{sw^{i_2j_2}_{I_{s_2}} \circ \tilde E_0^t}]{\times}{_{\tilde E_0^s}} 
    \dots \tensor*[_{sw^{i_{\ell-1}j_{\ell-1}}_{I_{s_{\ell-1}}} \circ \tilde E_0^t}]{\times}{_{\tilde E_0^s}}\tilde W^{ft}(-,\Delta_{I_{s_\ell}}^{i_\ell j_\ell};\mathfrak{X}_0 ),  \nonumber
\end{gather}
\begin{gather}\label{eq: defn-MFLS-parameters}
    \tilde W_\ell(\bm \gamma, \bm \gamma'; \mathfrak{X}_- \times  (\mathfrak{X}_0)^{\ell-1} \times \mathfrak{X}_+, \bm \tau, \bm s)=\tilde W^u(\bm \gamma, \Delta_{I_{s_1}}^{i_1j_1};\mathfrak{X}_-)\tensor*[_{sw^{i_1j_1}_{I_{s_1}} \circ \tilde E_-}]{\times}{_{\tilde E_0^s}}\tilde W^{ft}(-,\Delta^{i_2j_2}_{t_{s_2}}; \mathfrak{X}_0 )\\
    \qquad\qquad 
    \tensor*[_{sw^{i_2j_2}_{I_{s_2}} \circ \tilde E_0^t}]{\times}{_{\tilde E_0^s}} 
    \dots
    \tensor*[_{sw^{i_{\ell-1}j_{\ell-1}}_{I_{s_{\ell-1}}} \circ \tilde E_0^t}]{\times}{_{\tilde E_0^s}}\tilde W^{ft}(-,\Delta_{I_\ell}^{i_\ell j_\ell};\mathfrak{X}_0 )\tensor*[_{sw^{i_\ell j_\ell}_{I_{s_\ell}} \circ \tilde E_0^t}]{\times}{_{\tilde E_+}} \tilde W^s(\Delta_{I_{s_\ell}}^{i_\ell j_\ell},\bm \gamma'; \mathfrak{X}_+). \nonumber
\end{gather}

The fiber $\mathcal{M}_{\ell}(\bm \gamma, \bm \gamma'; \bm Y, \bm \tau, \bm s)$ over $\bm Y$ of the natural projection 
$$\tilde W_\ell(\bm \gamma, \bm \gamma'; \mathfrak{X}_- \times (\mathfrak{X}_0)^{\ell-1} \times \mathfrak{X}_+, \bm \tau, \bm s) \to \mathfrak{X}_- \times (\mathfrak{X}_0)^{\ell-1} \times \mathfrak{X}_+$$
is the space of Morse flow lines with switchings (corresponding to the permutation $\bm s$) from Definition~\ref{def-diff}. For $i=0, \dots, \ell-1$ there exists an evaluation map 
\begin{equation}
    \tilde E_i \colon \tilde W_\ell(\bm \gamma, \bm \gamma'; \mathfrak{X}_- \times (\mathfrak{X}_0)^{\ell-1} \times \mathfrak{X}_+, \bm \tau, \bm s) \to \op{Conf}_\ell([0,1]) \times \Omega^{1,2}(M, \bm q, \bm q'),
\end{equation} 
given by an intermediate evaluation map of the form $\tilde E_0^t$ in~\eqref{eq: defn-MFLS-parameters}. Its restriction to $\mathcal{M}_{\ell}(\bm \gamma, \bm \gamma'; \bm Y, \bm \tau, \bm s)$, also denoted $\tilde E_i$, sends $\bm \Gamma$ to $(\bm\theta, \Gamma_i(l_i))$, using the notation from Definition~\ref{def-diff}. 

\s
We also introduce some notation to describe higher-codimensional phenomena in the moduli spaces $\tilde W_\ell(\bm \gamma, \bm \gamma'; \mathfrak{X}_- \times  (\mathfrak{X}_0)^{\ell-1} \times \mathfrak{X}_+, \bm \tau, \bm s)$, i.e., when many switches occur at the same moment.

Let $k_1+\dots+k_r=\ell$ be a partition of $\ell$ (with all $k_i$ positive). Write $\bm\tau=(\bm \tau_{k_1}, \dots, \bm \tau_{k_r})$, where $\bm \tau_{k_1}=(\{i_1,j_1\},\dots, \{i_{k_1},j_{k_1}\})$ consists of the first $k_1$ elements of $\bm \tau$, $\bm \tau_{k_2}$ consists of the next $k_2$ elements of $\bm\tau$, etc. We have closed subsets 
\begin{equation}\label{eqn: bunch of Deltas}
    \Delta^{\bm \tau_{k_1}}= \Delta^{\bm \tau_{k_1}}_{I_{s_1},\dots,I_{s_{k_1}}}, \dots, \Delta^{\bm \tau_{k_r}}= \Delta^{\bm \tau_{k_r}}_{I_{s_{\ell-k_r+1}},\dots,I_{s_\ell}} \subset \op{Conf}_\ell([0,1]) \times \Omega^{1,2}(M, \bm q, \bm q'),
\end{equation}
as defined in~\eqref{eq: diagonal-with-many-intersections}, whose restrictions to $\op{Conf}_\ell([0,1]) \times \Omega^{m,2}(M, \bm q, \bm q')$ are $C^{m-1}$ submanifolds (and we denote these restrictions similarly by abuse of notation).

As discussed at the end of Section~\ref{section: switching map}, the closure $\overline{\Delta}^{\bm \tau_{k_j}}$ in $[0,1]^\ell\times \Omega^{m,2}(M, \bm q, \bm q')$ can be stratified by $C^{m-k_j}$ submanifolds of codimension at least $k_jn$ for each $j \colon 1 \le j \le r$.

We use the notation 
\begin{equation}\label{eq: multiple-switching}
    sw^{\bm \tau_{k_1}}=sw^{i_{k_1}j_{k_1}}_{I_{s_{k_1}}} \circ \dots \circ sw^{i_1j_1}_{I_{s_1}} \colon \overline{\Delta}^{\bm \tau_{k_1}} \to \overline{\Delta}^{\bm \tau_{k_1}},
\end{equation}
etc.\ for the composition of switching maps on the locus where it is well-defined. Let $\bm N=(N_1, \dots, N_r)$, where $N_j$ is one of the strata of $\overline{\Delta}^{\bm \tau_{k_j}}$ from the stratification that we mentioned above.
We define 
\begin{align}\label{eq: MFTS-with-deep-strata}
    \tilde W_\ell(\bm \gamma, \bm \gamma', & \bm N; \mathfrak{X}_- \times (\mathfrak{X}_0)^{\ell-1} \times \mathfrak{X}_+, \bm \tau, \bm s)= \\
    & \tilde W^u(\bm \gamma, N_1;\mathfrak{X}_-)\tensor*[_{sw^{\bm \tau_{k_1}} \circ \tilde E_-}]{\times}{_{\tilde E_0^s}} \dots \tensor*[_{sw^{\bm \tau_{k_r}} \circ \tilde E_0^t}]{\times}{_{E_+}} \tilde W^s(\bm \gamma', N_r; \mathfrak{X}_+),\nonumber
\end{align}
by analogy with Equation~\eqref{eq: defn-MFLS-parameters}.  By the discussion right after Definition~\ref{defn: switching-map}, the restrictions of $sw^{\bm \tau_{k_j}}$ to the strata $N_j$ are $C^{m-k_j}$-smooth. We note that there are only $r-1$ spaces of finite-length trajectories in the above definition, although in the notation we declare that we use $\ell-1$ corresponding perturbation spaces. That is because for each $k_j$ we ``collapse'' $k_j-1$ finite-length trajectories and the corresponding perturbations are not used. 
The fiber over $\bm Y \in \mathfrak{X}_- \times (\mathfrak{X}_0)^{\ell-1} \times \mathfrak{X}_+$ is denoted by
\begin{equation}\label{eq-deep-boundary-strata}
    \mathcal{M}_\ell(\bm \gamma, \bm \gamma', \bm N; \bm Y, \bm \tau, \bm s).
\end{equation}

We also write:
\begin{gather*}
\tilde W_{\ell}(\bm\gamma, \bm\gamma', \Delta_{I_{s_i}, I_{s_{i+1}}}^{\tau_i, \tau_{i+1}}; \mathfrak{X}_- \times (\mathfrak{X}_0)^{\ell-1} \times \mathfrak{X}_+, \bm \tau, \bm s)= \tilde W_\ell(\bm \gamma, \bm \gamma', \bm N; \mathfrak{X}_- \times (\mathfrak{X}_0)^{\ell-1} \times \mathfrak{X}_+, \bm \tau, \bm s),\\
\mathcal{M}_{\ell}(\bm\gamma, \bm\gamma', \Delta_{I_{s_i}, I_{s_{i+1}}}^{\tau_i, \tau_{i+1}};\bm Y, \bm \tau, \bm s)= \mathcal{M}_\ell(\bm \gamma, \bm \gamma', \bm N; \bm Y, \bm \tau, \bm s),
\end{gather*}
where $(k_1, \dots, k_{i-1}, k_i, k_{i+1}, \dots, k_{\ell-1})=(1, \dots, 1, 2, 1, \dots, 1)$,  $N_i=\Delta_{I_{s_i}, I_{s_{i+1}}}^{\tau_i, \tau_{i+1}}$, and $N_j= \Delta^{\tau_j}_{I_j}$ for $j\not=i$. 

The following is the main result of Appendix~\ref{subappendix: mfls}:

\begin{theorem}\label{theorem: transversality-for-MFLS} $\mbox{}$
\begin{enumerate}
       \item  For a generic $\bm Y\in \mathfrak{X}_- \times (\mathfrak{X}_0)^{\ell-1} \times\mathfrak{X}_+$, given any $\bm\gamma\in \mathcal{P}_{\sigma}$, $\bm\gamma'\in \mathcal{P}_{\sigma'}$, $\ell$, and $\bm\tau$, the space $\mathcal{M}_{\ell}(\bm \gamma, \bm \gamma'; \bm Y, \bm \tau, \bm s)$ is a smooth manifold with corners of dimension
\begin{equation}
    \operatorname{dim}\mathcal{M}_{\ell}(\bm \gamma, \bm \gamma'; \bm Y, \bm \tau, \bm s)=\operatorname{ind}(\bm \gamma)-\operatorname{ind}(\bm \gamma')-(n-2)\ell-1,
\end{equation}
    \item For a generic $\bm Y$, given any collection of submanifolds $\bm N=(N_1, \dots, N_r)$ associated with a given tuple $\bm \tau$ and a partition $(k_1, \dots, k_r)$ of $\ell$, the space $\mathcal{M}_\ell(\bm \gamma, \bm \gamma', \bm N; \bm Y,\bm \tau,\bm s)$ is a smooth manifold of dimension at most
\begin{equation}
\label{eqn: dimension-deep-strata for flow lines}
    \op{dim}\mathcal{M}_\ell(\bm \gamma, \bm \gamma', \bm N; \bm Y, \bm \tau, \bm s)=\op{ind}(\bm \gamma)-\op{ind}(\bm \gamma')-(n-1)\ell+r-1,
\end{equation}
where $k=\op{max}(k_1, \dots, k_r)$.
\end{enumerate}
\end{theorem}

\begin{proof}
    (1) We fix some $m \gg1$ and inductively prove that $\tilde W_k\coloneqq \tilde W_k(\bm \gamma, \Delta_{I_{s_k}}^{i_kj_k}; \mathfrak{X}_- \times (\mathfrak{X}_0)^{k-1}, \bm \tau|_{k}, \bm s|_k)$ is a $C^{m-1}$-smooth Banach manifold, where $1\leq k\leq \ell$ and $\bm \tau|_k=(\tau_1, \dots, \tau_k)$.  
    
    Lemma \ref{lemma: unstable-to-diagonal} provides the base case $k=1$ of the induction, i.e., for 
    $\tilde W_1=\tilde W^u(\bm \gamma, \Delta_{I_{s_1}}^{i_1j_1}; \mathfrak{X}_-).$ It suffices to verify Condition~\eqref{eq: transversality condition}, which is indeed satisfied since for any $(\bm \theta, \bm \gamma) \in \Delta_{I_{s_1}}^{i_1j_1}$, there exists a $C_{\bm\varepsilon}$-section along $\gamma_{i_1}$ which points at $\theta_{s_1}$ in an arbitrary given direction $X \in T_{\gamma_{i_1}(\theta_{s_1})}M$, by our choice of $\bm \varepsilon$. Such sections clearly form a subspace transverse to the tangent space $T_{(\bm \theta, \bm \gamma)}\Delta_{I_{s_1}}^{i_1j_1}$.
    
    For $k>1$ we have:
    \begin{gather*}
    \tilde W_k= \tilde W_{k-1} \tensor*[_{sw^{i_{k-1}j_{k-1}}_{I_{s_{k-1}}} \circ \tilde E_0^t}]{\times}{_{\tilde E_0^s}}\tilde W^{ft}(-,-;\mathfrak{X}_0 ) \tensor*[_{\tilde E_0^t}]{\times}{} \Delta^{i_kj_k}_{I_{s_k}}.
    \end{gather*}
    As in the $k=1$ case, $\tilde E_0^t \colon \tilde W^{ft}(-,-;\mathfrak{X}_0 ) \to [0,1]^\ell\times\Omega^{1,2}(M, \bm q, \bm q')$ is transverse to $\Delta^{i_kj_k}_{I_{s_k}}$ at any point $(\bm \theta, \bm \gamma) \in \Delta^{i_kj_k}_{I_{s_k}}$ with $\bm \gamma \in C^\infty(M, \bm q, \bm q')$. The transversality of $sw^{i_{k-1}j_{k-1}}_{I_{s_{k-1}}} \circ \tilde E_0^t$ and $\tilde E_0^s$ is achieved freely as in the proof of Lemma~\ref{lemma: finite-trajectories-two-immersions}, and $\tilde W_k$ is a $C^{m-1}$-smooth Banach manifold. Note that  $\tilde W_k$ consists of trajectories of smooth paths by Appendix~\ref{section: regularity}.  Hence $\tilde W_\ell$ is a $C^{m-1}$-smooth Banach manifold.
    
    By a similar argument,
    $$\tilde W_\ell^\dagger\coloneqq \tilde W_\ell(\bm \gamma, \bm \gamma'; \mathfrak{X}_- \times (\mathfrak{X}_0)^{\ell-1} \times \mathfrak{X}_+, \bm \tau, \bm s)=\tilde W_\ell\tensor*[_{sw^{i_\ell j_\ell}_{I_{s_\ell}} \circ \tilde E_0^t}]{\times}{_{\tilde E_+}} \tilde W^s( \Delta_{I_{s_\ell}}^{i_\ell j_\ell},\bm \gamma'; \mathfrak{X}_+)$$
    is a $C^{m-1}$-smooth Banach manifold.

    Finally, by the arguments from Lemmas \ref{lemma: unstable-to-diagonal} and \ref{lemma: finite-trajectories-two-immersions}, the projection $\tilde W_\ell^\dagger\to \mathfrak{X}_- \times (\mathfrak{X}_0)^{\ell-1} \times \mathfrak{X}_+$ is Fredholm of index
    $\operatorname{ind}(\bm \gamma)-\operatorname{ind}(\bm \gamma')-(n-2)\ell-1$. Here we rely on the fact that $sw^{i_{k-1}j_{k-1}} _{I_{s_{k-1}}} \circ \tilde E_0^t$ is an immersion for any $k$ when restricted to a portion of $\tilde W_k$ corresponding to any given choice of perturbation data.
    Then it follows from the Sard-Smale theorem that for a generic $\bm Y$ the space $\mathcal{M}_{\ell}(\bm \gamma, \bm \gamma'; \bm Y, \bm \tau, \bm s)$ is a $C^{m-1}$ manifold of dimension stated in (1). Since $m$ in the above argument is arbitrary we conclude that this space is a smooth manifold.

    (2) The above proof can be repeated by replacing $\Delta^{i_kj_k}_{I_{s_k}}$ by the appropriate $N_k$. We point out again that our definition of switching maps guarantees that the restrictions to the strata $N_k$ regarded as subsets of $[0,1]^\ell \times \Omega^{m,2}(M, \bm q, \bm q')$ are at least $C^{m-k}$-smooth.
\end{proof}
    
\subsection{Morse flow trees with switchings}\label{section: appendix mfts} 

In Appendix~\ref{section: appendix mfts} we prove Theorem~\ref{theorem: transversality for MFTS}, which implies the first half of Lemma~\ref{lemma-ainfty-regular}(a) in Section~\ref{subsection: the A infty structure}. The relevant moduli spaces will be constructed inductively in terms of fiber products, following the approach of Appendix~\ref{subappendix: mfls}. 

Let $\bm q^0, \dots, \bm q^d$ be $\kappa$-tuples of distinct points on $M$, let $\vv{\bm \gamma}=(\bm \gamma_1, \dots, \bm \gamma_d)$, where $\bm \gamma_i \in \Omega^{1,2}(M, \bm q^i, \bm q^{i+1})$ is a $\kappa$-tuple of critical points, and let $\bm \gamma_0 \in \Omega^{1,2}(M, \bm q^0, \bm q^d)$ be a $\kappa$-tuple of critical points.

Let $(T,\bm\tau) \in \mathcal{T}_d^\ell$ be a ribbon tree $T$ with switching data $\bm\tau$; for $\mathcal{T}_d^\ell$, there are no vertex switchings, so $\bm\tau=(\bm\tau_e)_{e\in E(T)}$. The perturbation space $\mathfrak{X}(T,\bm\tau )$ is given by Equation~\eqref{eq: tree-perturbations}. Let $\bm s$ be a total ordering on all the switchings.   

We will inductively associate to each vertex $v \in V(T)$ a certain space of trajectories $\tilde W_v(\vv{\bm \gamma};\mathfrak{X}(T,\bm\tau),\bm s)$ and an evaluation/concatenation map
$$C_v \colon \tilde W_v(\vv{\bm \gamma};\mathfrak{X}(T,\bm\tau),\bm s) \to \Omega^{1,2}(M, \bm q^{R(v)-1}, \bm q^{L(v)}),$$
where $R(v)$ and $L(v)$ are as defined in Definition~\ref{def: MFTS}.

First, for each incoming vertex $v_i$, let
\begin{equation}
    \tilde W_{v_i}(\vv{\bm \gamma};\mathfrak{X}(T,\bm\tau),\bm s)=[0,1]^\ell \times W^u(\bm \gamma_i; (\mathfrak{X}_-)_{i}),
\end{equation}
where $(\mathfrak{X}_-)_{i}$ is the $i$th copy of $\mathfrak{X}_-$ in the product $\mathfrak{X}(T,\bm\tau)$, and let $C_{v_i}$ be the evaluation map $\tilde E_-$ from \eqref{eqn: tilde E minus}.

Now let $v \in V_{\operatorname{int}}(T)$ and suppose that for all of its incoming vertices $v'=v^1, \dots, v^{|v|-1}$ (i.e., the edge $e^v_i$ is from $v^i$ to $v$) the space $\tilde W_{v'}(\vv{\bm \gamma};\mathfrak{X}(T,\bm\tau),\bm s)$ and the map $C_{v'}$ have already been defined. For each $v^i$ we set:
\begin{gather}\label{eqn: MFTS-intermediate-space}
    \tilde W_{v^i, e^v_i}(\vv{\bm \gamma}; \mathfrak{X}(T,\bm\tau),\bm s)=\tilde W_{v^i}(\vv{\bm \gamma};\mathfrak{X}(T,\bm\tau),\bm s)\tensor*[_{C_{v^i}}]{\times}{_{\tilde E_0^s}}\tilde W^{ft}(-,{\Delta}^{\tilde{\tau}_1}_{I_{s(\tilde \tau_1)}};\mathfrak{X}_0) \\
    \qquad\qquad\qquad\qquad\qquad \qquad\qquad\qquad\qquad\qquad \tensor*[_{sw^{\tilde{\tau}_1}_{I_{s(\tilde \tau_1)}} \circ \tilde E_0^t}]{\times}{_{\tilde E_0^s}}\dots \tensor*[_{sw^{\tilde{\tau}_{\ell(e_i^v)}}_{I_{s(\tilde \tau_{\ell(e_i^v)})}} \circ \tilde E_0^t}]{\times}{_{\tilde E_0^s}}\tilde W^{ft}(-,-;\mathfrak{X}_0 ),\nonumber
\end{gather} 
where $\bm \tau_{e^v_i}=(\tilde{\tau}_1, \dots, \tilde{\tau}_{\ell(e_i^v)})$ and $\bm s(\tilde\tau_i)$ is the ordering assigned to the switching data $\tilde\tau_i$. Let 
\begin{equation}
    \tilde E_0^t \colon \tilde  W_{v^i, e^v_i}(\vv{\bm \gamma}; \mathfrak{X}(T,\bm\tau),\bm s) \to \Omega^{1,2}(M, \bm q^{R(v^i)-1}, \bm q^{L(v^i)})
\end{equation}
denote the composition of the projection to the last factor in~\eqref{eqn: MFTS-intermediate-space} and $\tilde E_0^t$ given by \eqref{eqn: tilde E zero}.  We then set: 
\begin{equation}
    \tilde W_v(\vv{\bm \gamma};\mathfrak{X}(T,\bm\tau),\bm s)=\tilde W_{v^1, e^v_1}(\vv{\bm \gamma}; \mathfrak{X}(T,\bm\tau),\bm s)  \tensor*[_{[0,1]^\ell}]{\times}{_{[0,1]^\ell}}\dots \tensor*[_{[0,1]^\ell}]{\times}{_{[0,1]^\ell}}\tilde W_{v^{|v|-1}, e^v_{|v|-1}}(\vv{\bm \gamma}; \mathfrak{X}(T,\bm\tau),\bm s),
\end{equation}
where we are taking fiber products over $[0,1]^\ell$.
Given $\bm \Gamma \in \tilde W_v(\vv{\bm \gamma};\mathfrak{X}(T,\bm\tau),\bm s)$, its component corresponding to the last factor $\tilde W^{ft}(-,-;\mathfrak{X}_0 )$ of $\tilde W_{v^i, e^v_i}(\vv{\bm \gamma}; \mathfrak{X}(T,\bm\tau),\bm s)$ is a trajectory
$$\Gamma^{e^{v}_i}_{\ell(e^v_i)} \colon [0, l^{e^{v}_i}_{\ell(e^v_i)}] \to \Omega^{1,2}(M, \bm q^{R(v^i)-1}, \bm q^{L(v^i)}).$$
Then we set
\begin{equation}
    C_v(\bm \Gamma)=c^{\bm w(\bm \Gamma)} \left(\bm \theta, \Gamma^{e^v_{|v|-1}}_{\ell_{e^v_{|v|-1}}}(l^{e^v_{|v|-1}}_{\ell_{e_{|v|-1}^v}}),\dots,\Gamma^{e^v_1}_{\ell_{e^v_1}}(l^{e^v_1}_{\ell_{e_1^v}}) \right),
\end{equation}
as in Equation~\eqref{eqn: concatenation at interior vertices}. Here $\bm w(\bm \Gamma)=(w_{e^v_1}(\bm \Gamma), \dots, w_{e^v_{|v|-1}}(\bm \Gamma))$, with each $w_{e^v_i}(\bm \Gamma)$ defined inductively as in Definition~\ref{def: MFTS}(6).

Finally, writing $v'_0$ for the unique vertex with an edge to $v_0$, we set:
\begin{gather} \label{eqn: MFTS-final-space}
    \tilde W_T(\vv{\bm \gamma}, \bm \gamma_0;\mathfrak{X}(T,\bm\tau), \bm s)=\tilde W_{v'_0}(\vv{\bm \gamma};\mathfrak{X}(T,\bm\tau),\bm s)\tensor*[_{C_{v'_0}}]{\times}{_{\tilde E_0^s}}\tilde W^{ft}(-,\Delta^{\tilde{\tau}_1}_{I_{s(\tilde\tau_1)}};\mathfrak{X}_0 ) \\
    \qquad\qquad\qquad\qquad\qquad \tensor*[_{sw^{\tilde\tau_1}_{I_{s(\tilde\tau_1)}} \circ \tilde E_0^t}]{\times}{_{E_0^s}}\dots \tensor*[_{sw^{\tilde\tau_{\ell(e_0)}}_{I_{s(\tilde\tau_{\ell(e_0)})}} \circ \tilde E_0^t}]{\times}{_{\tilde E_+}} \tilde W^s(\Delta^{\tilde \tau_{\ell(e_0)}}_{I_{s(\tilde\tau_{\ell(e_0)})}}, \bm \gamma_0; \mathfrak{X}_+).\nonumber 
\end{gather}
where $\bm \tau_{e_0}=(\tilde{\tau}_1, \dots, \tilde{\tau}_{\ell(e_0)})$. 

The fiber $\mathcal{M}_T(\vv{\bm \gamma}, \bm \gamma_0; \bm Y, \bm \tau, \bm s)$ over $\bm Y$ of the natural projection 
$$\tilde W_T( \vv{\bm \gamma}, \bm \gamma_0;\mathfrak{X}(T,\bm\tau), \bm s) \to \mathfrak{X}(T,\bm\tau)$$ 
coincides with the MFTS moduli space introduced in Definition~\ref{def: MFTS}.

We also have analogous spaces $\tilde W_T(\vv{\bm \gamma}, \bm \gamma_0;\mathfrak{X}(T, \bm \tau), \bm s)$ and $\mathcal{M}_T(\vv{\bm \gamma}, \bm \gamma_0; \bm Y, \bm \tau, \bm s)$, $\bm Y\in \mathfrak{X}(T, \bm \tau)$, defined using $C_{\bm\varepsilon}$-spaces. 

As in Appendix~\ref{subappendix: mfls}, given a partition $k_1+\dots+k_{r_e}=\ell_e$ and a decomposition $\bm\tau_e=(\bm\tau_{e,k_1},\dots, \bm\tau_{e,k_{r_e}})$, we may choose a tuple $\bm N_e=(N_1^e, \dots, N_{r_e}^e)$ of strata in $\overline\Delta^{\bm \tau_{e,k_1}}, \dots, \overline\Delta^{\bm \tau_{e,k_{r_e}}}$, defined as in Equation~\eqref{eqn: bunch of Deltas}. Given a collection $(\bm N_e)_{e \in E(T)}$, we may define a spaces
\begin{equation}
     \tilde W_T(\vv{\bm \gamma}, \bm \gamma_0, (\bm N_e)_{e \in E(T)};\mathfrak{X}(T,\bm\tau), \bm s), \quad \mathcal{M}_T(\vv{\bm \gamma}, \bm \gamma_0, (\bm N_e)_{e \in E(T)};\bm Y, \bm \tau, \bm s),
\end{equation}
by analogy with~\eqref{eq: MFTS-with-deep-strata} and \eqref{eq-deep-boundary-strata}. Here $\bm Y\in \mathfrak{X}(T,\bm\tau)$.

Given $(T,\bm\tau) \in \mathcal{T}_d^{\ell, \ell'}$, we consider $\mathfrak{X}(T,(\bm\tau_e)_{e\in E(T)})$ the space obtained by forgetting the vertex switches $(\bm \tau_v)_{v\in V(T)}$. To each vertex $v \in V_{\op{int}}(T)$ we additionally associate $\bm N_v=(N_{v,1}, \dots, N_{v, |v|-1})$, where $N_{v,i}$ is a stratum of $\overline\Delta^{\bm \tau_{v,i}}$. As above, to $\bm Y \in \mathfrak{X}(T,(\bm\tau_e)_{e\in E(T)})$ we may associate a space
     \begin{equation}
         \mathcal{M}_T(\vv{\bm \gamma}, \bm \gamma_0, (\bm N_e)_{e \in E(T)}, (\bm N_v)_{v \in V_{\op{int}}(T)};\bm Y, \bm \tau, \bm s),
     \end{equation}
by analogy with~\eqref{eq-deep-boundary-strata} and taking into account Remark~\ref{remark: how-to-define-vertex-deep-strata}.

\begin{theorem}\label{theorem: transversality for MFTS}$\mbox{}$
For $(T,\bm\tau) \in \mathcal{T}_d^{\ell, \ell'}$ and a generic $\bm Y\in  \mathfrak{X}(T,\bm\tau)$, the spaces 
$$\mathcal{M}_T(\vv{\bm \gamma}, \bm \gamma_0; \bm Y, \bm \tau, \bm s) \quad\mbox{and}\quad \mathcal{M}_T(\vv{\bm \gamma}, \bm \gamma_0, (\bm N_e)_{e \in E(T)};\bm Y, \bm \tau, \bm s),$$ 
are smooth manifolds of dimensions
\begin{gather}\label{eq: tree-space-dim}
    \operatorname{dim}\mathcal{M}_T(\vv{\bm \gamma}, \bm \gamma_0; \bm Y, \bm \tau, \bm s)=\operatorname{ind}(\bm \gamma_0)-\operatorname{ind}(\bm \gamma_1)-\dots-\operatorname{ind}(\bm \gamma_d)-(n-2)\ell+|E_{\operatorname{int}}(T)|,\\
    \label{eq: tree-space-deep-strata-dim}
    \operatorname{dim}\mathcal{M}_T(\vv{\bm \gamma}, \bm \gamma_0, (\bm N_e)_{e \in E(T)};\bm Y, \bm \tau, \bm s)=\operatorname{ind}(\bm \gamma_0)-\operatorname{ind}(\bm \gamma_1)-\dots-\operatorname{ind}(\bm \gamma_d)\\
    \nonumber
    \qquad \qquad \qquad  -(n-1)\ell+|E_{\operatorname{int}}(T)|+\textstyle \sum_{e \in E(T)}r_e-1, 
\end{gather}
and the space $\mathcal{M}_T(\vv{\bm \gamma}, \bm \gamma_0, (\bm N_e)_{e \in E(T)}, (\bm N_v)_{v \in V_{\op{int}}(T)};\bm Y, \bm \tau, \bm s)$ is a smooth manifold of dimension 
\begin{gather} \label{eq: tree-space-deep-strata-dim-with-vert}
    \operatorname{dim}\mathcal{M}_T(\vv{\bm \gamma}, \bm \gamma_0, (\bm N_e)_{e \in E(T)}, (\bm N_v)_{v \in V_{\op{int}}(T)};\bm Y, \bm \tau, \bm s)=\operatorname{ind}(\bm \gamma_1)+\dots+\operatorname{ind}(\bm \gamma_d)-\operatorname{ind}(\bm \gamma')- \\
    \nonumber 
    \qquad \qquad \qquad (n-1)\ell+|E_{\op{int}}(T)|+ \textstyle \sum_{e \in E(T)}r_e-1-(n-1)\ell'. 
\end{gather}
\end{theorem}

\begin{proof}
    This is similar to the proof of Theorem~\ref{theorem: transversality-for-MFLS}.
    
    We can inductively show that $\tilde W_v(\vv{\bm \gamma};\mathfrak{X}(T,\bm\tau),\bm s)$ and $\tilde W_T( \vv{\bm \gamma}, \bm \gamma^0;\mathfrak{X}(T,\bm\tau),\bm s)$ are $C^{m-1}$-smooth Banach manifolds for an arbitrary large integer $m$ by observing that the maps involved in \eqref{eqn: MFTS-intermediate-space} and \eqref{eqn: MFTS-final-space} intersect transversely since the differentials of the maps of the form $\tilde E_0^s, \tilde E_0^t$ has all $C_{\bm\varepsilon}$-sections at smooth paths away from the critical points.

    Next, the projection map $\tilde W_T(\vv{\bm \gamma}, \bm \gamma^0;\mathfrak{X}(T,\bm\tau),\bm s) \to \mathfrak{X}(T,\bm\tau)$ has index 
    $$\operatorname{ind}(\bm \gamma_1)+\dots+\operatorname{ind}(\bm \gamma_d)-\operatorname{ind}(\bm \gamma')-(n-2)\ell+|E_{\operatorname{ind}}(T)|.$$
    The only additional step compared to that of Theorem~\ref{theorem: transversality-for-MFLS}(1)(2) is that the differential of concatenation $c^{\bm w}$ is an injective linear map at any point for any choice of weights $\bm w$. (1) then follows from the Sard-Smale theorem. 
\end{proof}

\section{Compactness}
\label{section: appendix-compactness}

\subsection{Construction of a consistent collection of perturbation data}

The goal of this subsection is to construct a consistent collection of perturbation spaces so that the boundaries of $\mathcal{M}(\bm \gamma, \bm \gamma'; \bm Y, \bm \tau, \bm s)$ and $\mathcal{M}_T(\vv{\bm \gamma}, \bm \gamma_0; \bm Y, \bm \tau, \bm s)$ are covered by products of moduli spaces for smaller $\ell$ and subtrees of $T$. The perturbation spaces $\mathfrak{X}_- \times (\mathfrak{X}_0)^{\ell-1} \times \mathfrak{X}_+$ from Appendix~\ref{appendix: transversality} are clearly not sufficient since our choices of $\bm Y\in \mathfrak{X}_- \times (\mathfrak{X}_0)^{\ell-1} \times \mathfrak{X}_+$ for various $\ell$, $\bm \tau$, and $\bm s$ are unrelated. The classical approach from \cite{seidel2008fukaya} requires making a choice of consistent perturbation data $\bm Y$ for all $\ell$ and all trees $T$.  Our construction will follow the approach of Abouzaid \cite[Lemma 7.2]{abouzaid2011plumbings} and Mescher \cite[Section 2]{Mescher2018} for the finite-dimensional case of Morse gradient trees.

The construction of a consistent collection of perturbation spaces is rather involved and will be carried out over the next few pages.  

The domains of our trajectories $\bm\Gamma$ are parametrized by $\mathcal{D}_\ell:=[0,+\infty)^\ell$, $\ell=0,1,2,\dots$, where $\ell+1$ is the number of switches and if $(\lambda_1,\dots,\lambda_\ell)\in \mathcal{D}_\ell$, then $\lambda_i$ is the length of the trajectory between the $(i-1)$st switch and the $i$th switch. 

Our basic building blocks are the spaces $\mathfrak{X}_{\pm}(k)$ and $\mathfrak{X}_0^i(k)$, $k=0,1,2,\dots$, $i=1,\dots, k$. (They are components of $\mathfrak{X}_{\op{sw}}(\ell)$ below.)  We set $\mathfrak{X}_{\pm}(0)=\mathfrak{X}_\pm$ and $\mathfrak{X}_0(0)=\mathfrak{X}_0$ and for $k>0$ we recursively define 
\begin{equation}
\mathfrak{X}_{-}(k)\coloneqq \left \{
\begin{array}{c}
Y \in C_{\bm \varepsilon}(\mathcal{D}_k \times(-\infty, 0], \mathfrak{X}) \, \, \mid   \, Y(\lambda, \cdot) \in \mathfrak{X}_{-} ~~ \forall  \lambda=(\lambda_1,\dots,\lambda_k)\in\mathcal{D}_k,\\
\lim _{\lambda_j' \rightarrow+\infty} Y|_{\lambda_j=\lambda_j'} \text { exists in } \mathfrak{X}_{-}(k-1) ~~ \forall j \in\{1,2, \dots, k\}
\end{array}
\right\};
\end{equation}
\begin{equation}
\mathfrak{X}_{+}(k)\coloneqq \left \{
\begin{array}{c}
Y \in C_{\bm \varepsilon}(\mathcal{D}_k \times [0,+\infty), \mathfrak{X}) \, \, \mid   \, Y(\lambda, \cdot) \in \mathfrak{X}_{+} ~~ \forall  \lambda=(\lambda_1,\dots,\lambda_k)\in\mathcal{D}_k,\\
\lim _{\lambda_j' \rightarrow+\infty} Y|_{\lambda_j=\lambda_j'} \text { exists in } \mathfrak{X}_{+}(k-1) ~~ \forall j \in\{1,2, \dots, k\}
\end{array}
\right\};
\end{equation}
\begin{equation}
\mathfrak{X}_0^i(k)\coloneqq \left \{
\begin{array}{c}
Y \in  C_{\bm \varepsilon}(\mathcal{D}_k  \times[0,+\infty), \mathfrak{X}) \mid  Y_{\lambda,i}\in \mathfrak{X}_0 
~~\forall \lambda=(\lambda_1,\dots,\lambda_k)\in\mathcal{D}_k,\\
\lim _{\lambda_j' \rightarrow+\infty} Y|_{\lambda_j=\lambda_j'} \text { exists in $\mathfrak{X}_0^{i-1}(k-1)$ or in $\mathfrak{X}_0^i(k-1)$} \\
 ~~ \forall j \in\{1, \dots, \hat i, \dots,k\} \text{~depending on whether $j<i$ or $j>i$},\\
\lim _{\lambda_i'\to +\infty} \operatorname{split}Y|_{\lambda_i=\lambda_i'} \text { exists in } \mathfrak{X}_{+}(k-1) \times \mathfrak{X}_{-}(k-1)
\end{array}
\right\},
\end{equation}
where $\hat{\cdot}$ means the term is omitted, and $Y_{\lambda,i}\in \mathfrak{X}_0$ is $Y(\lambda_1,\dots,\lambda_{i-1},\cdot,\lambda_{i+1},\dots,\lambda_k,\cdot)$, i.e., with all but the $i$th and last spots specified.

Next we define the {\em background perturbation space} for $\mathfrak{X}_\pm(k)$ for $k\geq 1$:
\begin{equation} \label{back pm}
    \mathfrak{X}_{ \pm}^{\op{back}}(k)\coloneqq  \left\{\bm Z=(Z_1, \dots, Z_k) \in \mathfrak{X}_{\pm}(k-1)^{\times k} \mid   \lim _{\lambda_i' \rightarrow \infty} Z_j|_{\lambda_i=\lambda_i'} =\lim _{\lambda_{j-1}' \rightarrow \infty}Z_i|_{\lambda_{j-1}=\lambda_{j-1}'} ~\forall j>i\right\}.
\end{equation}
Viewing $\lambda_j=+\infty$ as the faces of $\mathcal{D}_k$ at $+\infty$, there is a ``face map'' which assigns to $Y \in \mathfrak{X}_{\pm}(k)$ its {\em background perturbation data}\footnote{We keep Mescher's terminology of ``background perturbation data'', but perhaps it's better to call it ``boundary perturbation data'', where the boundary refers to the boundary at $+\infty$.}
\begin{gather}\label{eq-constructing-background}
F_{\pm,k}:\mathfrak{X}_{\pm}(k) \to \mathfrak{X}_{ \pm}^{\op{back} }(k), \quad
Y\mapsto  (\lim_{\lambda_1'\rightarrow \infty} Y|_{\lambda_1=\lambda_1'}, \dots, \lim_{\lambda_k' \rightarrow \infty} Y|_{\lambda_k=\lambda_k'}).   
\end{gather}
Similarly, we define the background perturbation spaces
\begin{equation}\label{back 0}
\mathfrak{X}_0^{\op{back},i}(k)\coloneqq \left\{
\begin{array}{c}
\bm Z'=(Z_1, \dots,\widehat{Z}_i,\dots Z_k, Z^{+}, Z^{-}) \in \\ 
\mathfrak{X}^{i-1}_0(k-1)^{\times i-1} \times \mathfrak{X}_0^{i}(k-1)^{\times k-i}\times \mathfrak{X}_{+}(k-1) \times \mathfrak{X}_{-}(k-1) \bigm| \\
\displaystyle\lim_{\lambda_a' \rightarrow+\infty}  Z_b|_{\lambda_a=\lambda_a'}=\lim _{\lambda_{b-1}' \rightarrow+\infty} Z_a|_{\lambda_{b-1}=\lambda_{b-1}'} \quad \forall a<b \in\{1, \dots,\hat i, \dots, k\},\\
\displaystyle \lim_{\lambda_i' \to +\infty} \op{split} Z_a|_{\lambda_i=\lambda_i'}=\lim _{\lambda_{a-1}' \rightarrow+\infty}(Z^{+}|_{\lambda_{a-1}=\lambda_{a-1}'},Z^{-}|_{\lambda_{a-1}=\lambda_{a-1}'})  \quad \forall a>i, \\
\displaystyle \lim_{\lambda_{i-1}' \to +\infty} \op{split} Z_a|_{\lambda_{i-1}=\lambda_{i-1}'}=\lim _{\lambda_a' \rightarrow+\infty}(Z^{+}|_{\lambda_a=\lambda_a'},Z^{-}|_{\lambda_a=\lambda_a'})  \quad \forall a<i
\end{array}
\right\},
\end{equation}
and there is a ``face map''
\begin{gather} \label{eq-constructing-background-part2}
    F_{0,k,i}:  \mathfrak{X}_0^i(k) \to \mathfrak{X}_0^{\op{back},i}(k),\\
    Y \mapsto (\lim_{\lambda_1' \rightarrow \infty}  Y|_{\lambda_1=\lambda_1'}, \dots,\widehat{\lim_{\lambda_i' \rightarrow \infty} Y|_{\lambda_i=\lambda_i'}},\dots, \lim_{\lambda_k' \rightarrow \infty} Y|_{\lambda_k=\lambda_k'}, \lim_{\lambda_i' \to \infty} \op{split}Y|_{\lambda_i=\lambda_i'}).\nonumber
\end{gather}
Given $\bm Z \in \mathfrak{X}^{\op{back}}_{\pm}(k)$ and $\bm Z' \in \mathfrak{X}^{\op{back},i}_{0}(k)$ we set $\mathfrak{X}_{ \pm}(k, \bm Z):=F_{\pm,k}^{-1}(\bm Z)$ and $\mathfrak{X}_{0}(k, \bm Z'):=F_{0,k,i}^{-1} (\bm Z')$.

We now assemble the above building blocks into spaces $\mathfrak{X}_{\op{sw}}(\ell)$, $\ell\geq 1$, of \emph{$\ell$-switching perturbation data} and $\mathfrak{X}^{\op{back}}_{\op{sw}}(\ell)$ of \emph{background $\ell$-switching perturbation data}: Given 
$\bm \tau=(\tau_1, \dots, \tau_\ell)$ and an ordering $\bm s$ on $[\ell]$, we set:
\begin{gather}\label{eqn: sw}
\mathfrak{X}_{\op{sw}}(\ell, \bm \tau, \bm s)=  \mathfrak{X}_-(\ell-1) \times \mathfrak{X}_0^1(\ell-1)\times\dots\times  \mathfrak{X}_0^{\ell-1}(\ell-1) \times \mathfrak{X}_+(\ell-1),\\
\mathfrak{X}_{\op{sw}}(\ell)=\prod_{\bm \tau, \bm s}  \mathfrak{X}_{\op{sw}}(\ell, \bm \tau, \bm s),\\
\label{eqn: sw2}
    \mathfrak{X}^{\op{back}}_{\op{sw}}(\ell, \bm \tau, \bm s)= \mathfrak{X}^{\op{back}}_-(\ell-1) \times \mathfrak{X}_0^{\op{back},1}(\ell-1)\times\dots\times  \mathfrak{X}_0^{\op{back},\ell-1}(\ell-1) \times \mathfrak{X}^{\op{back}}_+(\ell-1),\\
    \mathfrak{X}^{\op{back}}_{\op{sw}}(\ell)=\prod_{\bm \tau, \bm s} \mathfrak{X}^{\op{back}}_{\op{sw}}(\ell, \bm \tau, \bm s).
\end{gather}
When $\ell=1$ there is no background perturbation data. We remark that the right-hand sides of \eqref{eqn: sw} and \eqref{eqn: sw2} do not depend on $\bm\tau$ and $\bm s$.  

\begin{remark}
    An element of $\mathfrak{X}_{\op{sw}}(\ell, \bm \tau, \bm s)$ is to be viewed as an assignment of time-dependent vector fields
$$(-\infty,0]\to \mathfrak{X},~ [0,\lambda_1]\to \mathfrak{X},~ \dots ~,[0,\lambda_{\ell-1}]\to \mathfrak{X}, ~[0,+\infty)\to \mathfrak{X}$$
to each $(\lambda_1,\dots,\lambda_{\ell-1})\in \mathcal{D}_{\ell-1}$.
\end{remark}

Given $\bm Z_{\bm \tau, \bm s}=(\bm Z^-, \bm Z_1^0, \dots, \bm Z_{\ell-1}^0, \bm Z^+) \in \mathfrak{X}^{\op{back}}_{\op{sw}}(\ell, \bm \tau, \bm s)$ we define
the space of {\em $\ell$-switching data consistent with background data $\bm Z_{\bm \tau, \bm s}$}:
\begin{equation}
    \mathfrak{X}_{\op{sw}}(\ell, \bm \tau, \bm s, \bm Z_{\bm \tau, \bm s})=\mathfrak{X}_-(\ell-1, \bm Z^-) \times \mathfrak{X}_0^1(\ell-1, \bm Z^0_1) \times \dots \times \mathfrak{X}_0^{\ell-1}(\ell-1, \bm Z^0_{\ell-1}) \times \mathfrak{X}_+(\ell-1, \bm Z^+).
\end{equation}
For later use we also define the following:
\begin{gather}
    \label{perturbation-neg-semi-ray-back}
    \mathfrak{X}_-(\ell, \bm \tau, \bm s, \bm Z_{\bm \tau, \bm s})=\mathfrak{X}_-(\ell-1, \bm Z^-) \times \mathfrak{X}_0^1(\ell-1, \bm Z^0_1) \times \dots \times \mathfrak{X}_0^{\ell-1}(\ell-1, \bm Z^0_{\ell-1}),\\
    \label{perturbation-pos-semi-ray-back}
    \mathfrak{X}_+(\ell, \bm \tau, \bm s, \bm Z_{\bm \tau, \bm s})=\mathfrak{X}_0^1(\ell-1, \bm Z^0_1) \times \dots \times \mathfrak{X}_0^{\ell-1}(\ell-1, \bm Z^0_{\ell-1}) \times \mathfrak{X}_+(\ell-1, \bm Z^+),\\ 
    \label{perturbation-segment-back}
    \mathfrak{X}_0(\ell-1, \bm \tau, \bm s, \bm Z_{\bm \tau, \bm s})=\mathfrak{X}_0^1(\ell-1, \bm Z^0_1) \times \dots \times \mathfrak{X}_0^{\ell-1}(\ell-1, \bm Z^0_{\ell-1}).
\end{gather}
Similarly, for a given tuple $\bm Z=(\bm Z_{\bm \tau, \bm s})_{\bm \tau, \bm s} \in \mathfrak{X}^{\op{back}}_{\op{sw}}(\ell)$, we define 
\begin{equation}
    \mathfrak{X}_{\op{sw}}(\ell, \bm Z)=\prod_{\bm \tau, \bm s}\mathfrak{X}_{\op{sw}}(\ell, \bm \tau, \bm s, \bm Z_{\bm \tau, \bm s}).
\end{equation}

We say that $\bm Y^\ell= (\bm Y^\ell_{\bm \tau, \bm s}) \in \mathfrak{X}_{\op{sw}}(\ell)$\footnote{For a while we will make the dependency of $\bm Y$ and $\bm Z$ on $\ell$ explicit.} is \emph{symmetric}
if
\begin{equation}\label{eq-perturbation-symmetric}
    \bm Y^\ell_{\bm \tau, \bm s}= \bm Y^\ell_{\bm \tau, \bm s'} \text{ for any two orderings } \bm s, \bm s',
    \end{equation}
and \emph{commutative} if   
    \begin{equation}\label{eq-perturbation-switching}
    \bm Y^\ell_{\bm \tau, \bm s}|_{l_i=0}=\bm Y^\ell_{\bm \tau^{i,1}, \bm s^i}|_{l_i=0} =\bm Y^\ell_{\bm \tau^{i,2}, \bm s^i}|_{l_i=0}\text{ for all }i=1, \dots, \ell-1,
 \end{equation}
where $\bm s^i$ is obtained from $\bm s$ by transposing $s_{i}$ and $s_{i+1}$ and $\bm \tau^{i,1}$ and $\bm \tau^{i,2}$ are obtained from $\bm \tau$ by modifying $\tau_i$ and $\tau_{i+1}$ using the following rules:
\begin{gather}
    \tau_i^{i,1}=\tau_i^{i,2}=\tau_{i+1} \text{ and } \tau_{i+1}^{i,1}=\tau_{i+1}^{i,2}=\tau_i \text{ if } |\tau_i \cap \tau_{i+1}| \neq 1;\label{eqL data-commutativity-rule-1} \\
    \text{if }\tau_i=\{a, b\}\text{ and }\tau_{i+1}=\{b, c\},\text{ then } \tau_i^{i,1}=\{a,c\}, \tau_{i+1}^{i,1}=\{a,b\}, \tau_i^{i,2}=\{b,c\}, \tau_{i+1}^{i,2}=\{a,c\}. \label{eq: data-commutativity-rule-2}
\end{gather}
\eqref{eq: data-commutativity-rule-2} is related to the commutativity property of the switching map as explained in Proposition~\ref{prop: switching-properties}.

As in \eqref{eq: defn-MFLS-parameters} we define $\tilde W_\ell(\bm \gamma, \bm \gamma'; \mathfrak{X}_{\op{sw}}(\ell, \bm \tau, \bm s, \bm Z_{\bm \tau, \bm s}^\ell), \bm \tau, \bm s)$
and let $\mathcal{M}_{\ell}(\bm \gamma, \bm \gamma'; \bm Y^\ell_{\bm \tau, \bm s}, \bm \tau, \bm s)$ be the fiber over $\bm Y^\ell_{\bm \tau, \bm s} \in \mathfrak{X}_{\op{sw}}(\ell, \bm \tau,\bm s, \bm Z^\ell_{\bm \tau, \bm s})$ of the natural projection (note that we are slightly abusing notation here). We say that the perturbation data $\bm Y^\ell_{\bm \tau, \bm s}$ is {\em regular} if all the spaces $\mathcal{M}_\ell(\bm \gamma, \bm \gamma', \bm N; \bm Y^\ell_{\bm\tau,\bm s},\bm \tau,\bm s)$ are smooth manifolds for all choices of $\bm N$ as in Appendix~\ref{appendix: transversality}.
\begin{figure}[h]
    \centering
    \includegraphics[width=5cm]{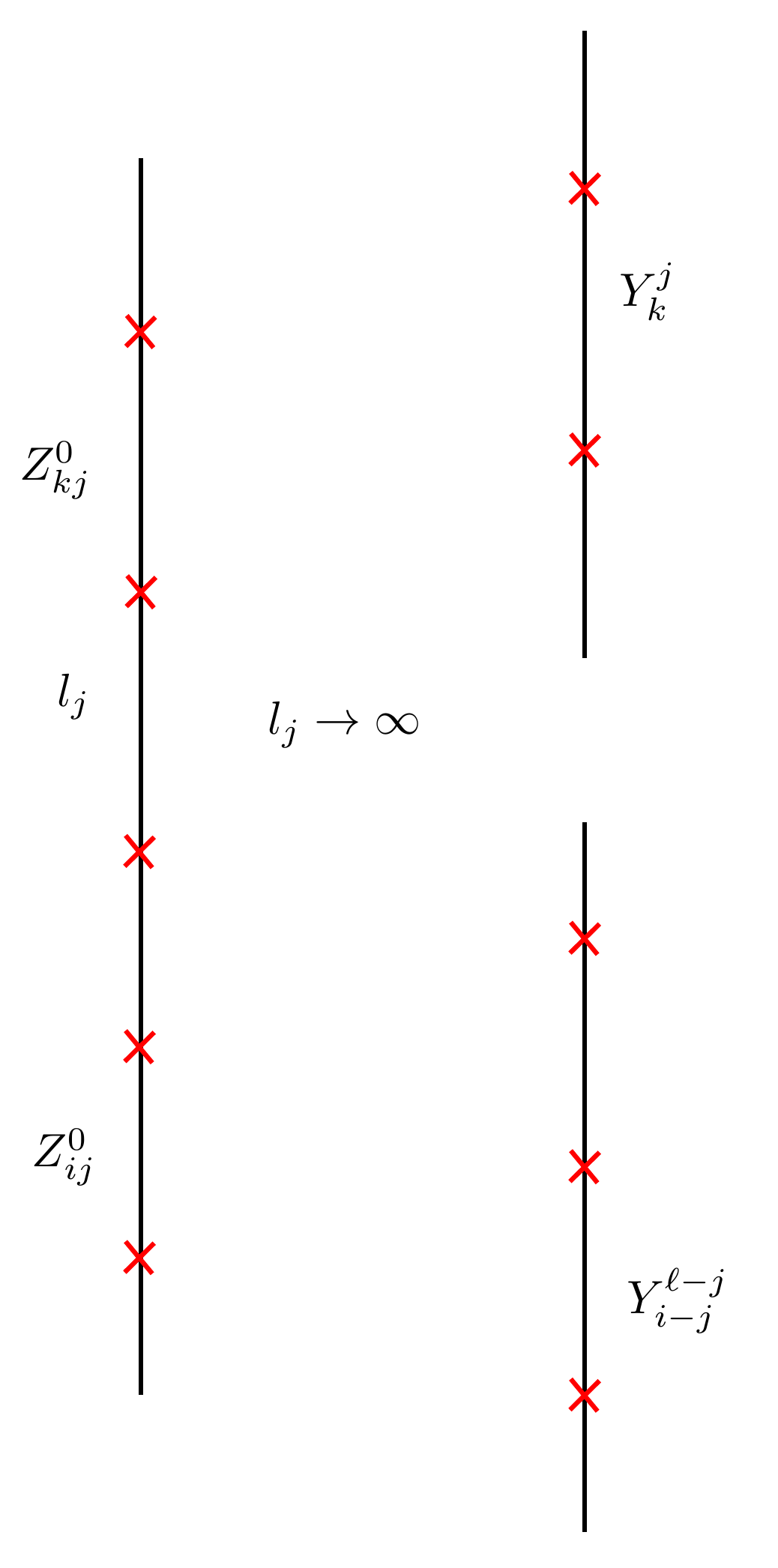}
    \caption{On the left we schematically show the domain of an MFLS with $\ell$ switchings, with switching markers denoted by crosses. When the $j$th length $l_j$ goes to $\infty$, we prescribe background data dictated by the data $\bm Y^j$ and $\bm Y^{\ell-j}$ picked for moduli spaces of MFLS with $j$ and $\ell-j$ switchings, respectively. Explicitly, for $i>j$ we have $Z^0_{ij}=Y^{\ell-j}_{i-j}$ and for $k<j$ we have $Z^0_{kj}=Y^j_k$.}
    \label{fig: data-consistency}
\end{figure}

\begin{definition}\label{defn-universal-data}
    A \emph{universal switching data} is a collection $\bm Y=(\bm Y^\ell)_{\ell \ge 1}$ of $\ell$-switching data $\bm Y^\ell\in \mathfrak{X}_{\op{sw}}(\ell)$ for all integers $\ell \ge 1$ such that the following hold:
    \begin{enumerate}
        \item For each $\ell$ the data $\bm Y^\ell$ is symmetric, commutative and regular.
        \item The data is \emph{consistent}, i.e., $\bm Y^\ell_{\bm \tau, \bm s} \in \mathfrak{X}_{\op{sw}}(\ell, \bm Z_{\bm \tau, \bm s}^{< \ell})$, where $\bm Z_{\bm \tau, \bm s}^{<\ell}$ is obtained from the collection $(\bm Y^1, \dots, \bm Y^{\ell-1}$) as follows:
        Let us write $\bm Z_{\bm \tau, \bm s}^{< \ell}=(\bm Z^-, \bm Z_1^0, \dots, \bm Z_{\ell-1}^0, \bm Z^+)$, where 
        $$\bm Z^{\pm}=(Z^{\pm}_1, \dots, Z^{\pm}_{\ell-1}) \quad \mbox{ and } \quad\bm Z^0_i=(Z^0_{i1}, \dots, Z^{0}_{i(i-1)}, Z_{ii}^+, Z_{ii}^-, Z^{0}_{i(i+1)}, \dots,  Z^0_{i(\ell-1)}).$$ Then the vector 
        $$(Z^-_{\ell'}, Z^0_{1\ell'}, \dots, Z^0_{(\ell'-1)\ell'}, Z^+_{\ell'\ell'}, Z^-_{\ell'\ell'}, Z^0_{(\ell'+1)\ell'}, \dots, Z^0_{(\ell-1)\ell'}, Z^+_{\ell'})$$
        is equal to $(\bm Y^{\ell'}, \bm Y^{\ell-\ell'})$ for each $\ell'$ such that $1 \le \ell' \le \ell-1$; see Figure~\ref{fig: data-consistency}.
    \end{enumerate}
\end{definition}

The following result is a generalization of Theorem~\ref{theorem: transversality-for-MFLS} and asserts the existence of regular $\ell$-switching data consistent with a given background data.

\begin{lemma}\label{lemma-regularity-with-background}
Given $\bm Z^\ell \in \mathfrak{X}^{\op{back}}_{\op{sw}}(\ell)$, for a generic $\bm Y^\ell \in \mathfrak{X}_{\op{sw}}(\ell, \bm Z^\ell)$:
    \be
    \item The spaces $\mathcal{M}_{\ell}(\bm \gamma, \bm \gamma'; \bm Y^\ell_{\bm \tau, \bm s}, \bm \tau, \bm s)$ are smooth manifolds with corners of dimension
    \begin{equation}
        \operatorname{dim}\mathcal{M}_{\ell}(\bm \gamma, \bm \gamma'; \bm Y^\ell_{\bm\tau, \bm s}, \bm \tau, \bm s)=\operatorname{ind}(\bm \gamma)-\operatorname{ind}(\bm \gamma')-(n-2)\ell-1.
    \end{equation}
    \item Given any collection of submanifolds $\bm N= (N_1, \dots, N_r)$ associated with a given tuple $\bm\tau$ and a partition of $(k_1, \dots, k_r)$ of $\ell$  as in Theorem~\ref{theorem: transversality-for-MFLS}, the space $\mathcal{M}_\ell(\bm \gamma, \bm \gamma', \bm N; \bm Y^\ell_{\bm\tau,\bm s},\bm \tau,\bm s)$ is a smooth manifold of dimension at most
    \begin{equation}\label{eq-dimension-deep-strata}
        \op{dim}\mathcal{M}_\ell(\bm \gamma, \bm \gamma', \bm N; \bm Y^\ell_{\bm \tau,\bm s}, \bm \tau, \bm s)=\op{ind}(\bm \gamma)-\op{ind}(\bm \gamma')-(n-1)\ell+r-1,
    \end{equation}
    where $k=\op{max}(k_1, \dots, k_r)$.
    \ee
\end{lemma}

\begin{proof}
    The proof is similar to that of Theorem~\ref{theorem: transversality-for-MFLS}.  The main difference is that, for each pair $(\bm \tau, \bm s)$, instead of looking for a residual subset of $\mathfrak{X}_{\op{sw}}(\ell, \bm \tau, \bm s)$, we now have to find a residual subset of the closed nonempty affine subset $\mathfrak{X}_{\op{sw}}(\ell, \bm \tau, \bm s, \bm Z_{\bm \tau, \bm s}^\ell)\subset \mathfrak{X}_{\op{sw}}(\ell, \bm \tau, \bm s)$. The proof of Theorem~\ref{theorem: transversality-for-MFLS} remains essentially unaffected; see \cite[Lemma 5.1]{Mescher2018} for more details.  
\end{proof}

\begin{theorem}\label{thm: existence of universal switching data}
There exists a universal switching data $\bm Y=(\bm Y^\ell)_{\ell \ge 1}$.
\end{theorem}

\begin{proof}
     This is proved by induction. For $\ell=1$, there is no background and we take a regular $\bm Y^1\in \mathfrak{X}_{\op{sw}}(1)=\prod_{\bm\tau,\bm s}\mathfrak{X}_-\times \mathfrak{X}_+$. For $\ell>1$, $\bm Z^{\ell}$ is obtained from pairs $(\bm Y^{\ell'}, \bm Y^{\ell-\ell'})$ for all $\ell'$ such that $1 \le \ell' \le \ell-1$, subject to the consistency condition of Definition~\ref{defn-universal-data}(2).  Lemma~\ref{lemma-regularity-with-background} guarantees the existence of a regular $\bm Z^{\ell}$.
\end{proof}

\subsection{Compactness for MFLS}

\begin{theorem}\label{thm: compactness-for-MFLS} 
Let $\bm Y=(\bm Y^\ell)_{\ell \ge 1}$ be a choice of universal switching data.
\begin{enumerate}
    \item If $\op{ind}(\bm \gamma')-\op{ind}(\bm \gamma)-(n-2)\ell-1=0$, then $\mathcal{M}_{\ell}(\bm \gamma, \bm \gamma'; \bm Y^{\ell}_{\bm \tau, \bm s}, \bm \tau, \bm s)$ consists of a finite number of points.
    \item If $\op{ind}(\bm \gamma')-\op{ind}(\bm \gamma)-(n-2)\ell-1=1$, then $\mathcal{M}_{\ell}(\bm \gamma, \bm \gamma'; \bm Y^{\ell}_{\bm \tau, \bm s}, \bm \tau, \bm s)$ admits a compactification $\overline{\mathcal{M}}_{\ell}(\bm \gamma, \bm \gamma'; \bm Y^{\ell}_{\bm \tau, \bm s}, \bm \tau, \bm s)$ whose boundary is covered by the following strata:
    \begin{gather}
        \mathcal{M}_{\ell'}(\bm \gamma, \bm \gamma''; \bm Y^{\ell'}_{\bm \tau_1, \bm s_1}, \bm \tau_1, \bm s_1) \times \mathcal{M}_{\ell-\ell'}(\bm \gamma'', \bm \gamma'; \bm Y^{\ell-\ell'}_{\bm \tau_2, \bm s_2}, \bm \tau_1, \bm s_2), \quad 0 \le \ell' \le \ell,\label{eq-compactness-breaking} \\
        \mathcal{M}_{\ell}(\bm \gamma, \bm \gamma', \Delta^{\tau_i, \tau_{i+1}}_{I_{s_i}, I_{s_{i+1}}}; \bm Y^{\ell}_{\bm \tau, \bm s}, \bm \tau, \bm s), \label{eq-compactness-two-switches}
    \end{gather}
    where $(\bm \tau_1, \bm \tau_2)=\bm \tau$ with $|\bm \tau_1|=\ell'$; $\bm s_1$ and $\bm s_2$ are permutations of $\{1, \dots, \ell'\}$ and $\{1, \dots, \ell-\ell'\}$, such that for $\bm s=(\bm s_1', \bm s_2')$ with $|\bm s_1'|=\ell'$ the ordering of $\bm s_1$ (resp.\ $\bm s_2$) coincides with that of $\bm s_1'$ (resp.\ $\bm s_2'$); and $\bm \gamma''$ is a critical point such that 
    $$\op{ind}(\bm \gamma)-\op{ind}(\bm \gamma'') -(n-2)\ell'-1=\op{ind}(\bm \gamma'')-\op{ind}(\bm \gamma') -(n-2)(\ell-\ell')-1=0.$$
    \end{enumerate}
\end{theorem}

\begin{proof}
    If a sequence $(\bm \Gamma^j)_{j \ge 1}$ of MFLS's in $\mathcal{M}_{\ell}(\bm \gamma, \bm \gamma'; \bm Y^{\ell}_{\bm \tau, \bm s}, \bm \tau, \bm s)$ does not converge to an MFLS in this space, then either one of the finite trajectories $\bm \Gamma^j_i$ converges to a pair of semi-infinite trajectories or a semi-infinite trajectory converges to a pair consisting of an infinite trajectory and a semi-infinite trajectory; this corresponds to $l_i \to \infty$ for some $i$ such that $1 \le i \le \ell$ and is described by~\eqref{eq-compactness-breaking}.  Indeed the convergence takes place by \cite[Theorem 3.5]{abbondandolo-majer2006}, since $X$ is a pseudogradient and $\bm Y=X$ on the regions of the MFLS that are far away from the switches. In the case of (1), such convergence does not occur by regularity and index considerations; in the case of (2) such convergence is possible. 

    The other possibility is that $l_i \to 0$ for at least one index $i$ satisfying $1 \le i \le \ell-1$. If there is only one such index, then the limit $\bm\Gamma^\infty$ is in \eqref{eq-compactness-two-switches} (see Figure~\ref{fig: coordinates}), unless $\theta_{s_i}=\theta_{s_{i+1}}$. The latter does not occur because such a stratum has dimension $-1$; hence by regularity the stratum must be empty. On the other hand, if more than one $l_i$ converges to $0$, then such $\bm \Gamma^\infty$ belongs to a manifold of type $\mathcal{M}_\ell(\bm \gamma, \bm \gamma', \bm N; \bm Y, \bm \tau, \bm s)$, which has negative dimension by ~\eqref{eq-dimension-deep-strata}; hence such $\bm \Gamma^\infty$ cannot be in the limit.
\end{proof}

\begin{remark}
    Observe that Theorem~\ref{thm: compactness-for-MFLS}(2) is not entirely accurate, because for moduli spaces of MFLS with $\ell$ switchings we used switching maps that depend on all $\ell$ time coordinates $\theta$ and even when the trajectory breaks those switching maps still depend on all $\ell$ coordinates, although we claim that it would depend on a fewer number of coordinates. To fix this issue one should also make switching maps depend on length coordinates $l_i$, in particular on distances towards nearest switchings along the trajectory. We leave this to the attentive reader.
\end{remark}

\begin{corollary} \label{cor: extend the moduli space}
If $\op{ind}(\bm \gamma')-\op{ind}(\bm \gamma)-(n-2)\ell-1=1$, then
$$\bigcup_{\bm \tau, \bm s} \mathcal{M}_{\ell}(\bm \gamma, \bm \gamma'; \bm Y^\ell_{\bm \tau, \bm s}, \bm \tau, \bm s)$$
admits a compactification with boundary consisting only of strata of the form~\eqref{eq-compactness-breaking}.
\end{corollary}

\begin{proof}
    The main observation here is that for any tuple $\bm \tau$ and any ordering $\bm s$, every point of 
    $$\mathcal{M}_{\ell}(\bm \gamma, \bm \gamma', \Delta^{\tau_i, \tau_{i+1}}_{I_{s_i}, I_{s_{i+1}}}; \bm Y^{\ell}_{\bm \tau, \bm s}, \bm \tau, \bm s)$$ 
    is also a point of 
    $$\mathcal{M}_{\ell}(\bm \gamma, \bm \gamma', \Delta^{\tau^{i,j}_i, \tau^{i,j}_{i+1}}_{I_{s^i_i}, I_{s^i_{i+1}}}; \bm Y^{\ell}_{\bm \tau^{i,j}, \bm s}, \bm \tau^{i,j}, \bm s^i),$$ 
    where $j=1,2$, depending on whether $\theta_{s_i}$ is greater than $\theta_{s_{i+1}}$; for the notation see \eqref{eq: data-commutativity-rule-2}. This holds, since the data $\bm Y$ is commutative and the switching maps satisfy the commutativity relations by Proposition~\ref{eq : commutativity-switching}, which allows us to switch $s_i$ and $s_{i+1}$. Note that by definition the switching maps associated to the remaining $\theta$-coordinates do not change. Therefore, we may identify these points of compactifications $\overline{\mathcal{M}}_{\ell}(\bm \gamma, \bm \gamma'; \bm Y^{\ell}_{\bm \tau, \bm s}, \bm \tau, \bm s)$ and make them into interior points of the union of such compactifications, while respecting the orientations. This last observation is clearly seen from our orientation conventions since essentially while switching the order of $s_i$ and $s_{i+1}$ in our convention~\eqref{eq: tree-orientation} on the orientations of these spaces we only change the order of two length coordinates $l_i$ and $l_{i+1}$.
\end{proof}

\begin{figure}[ht]
	\begin{overpic}[scale=.35]{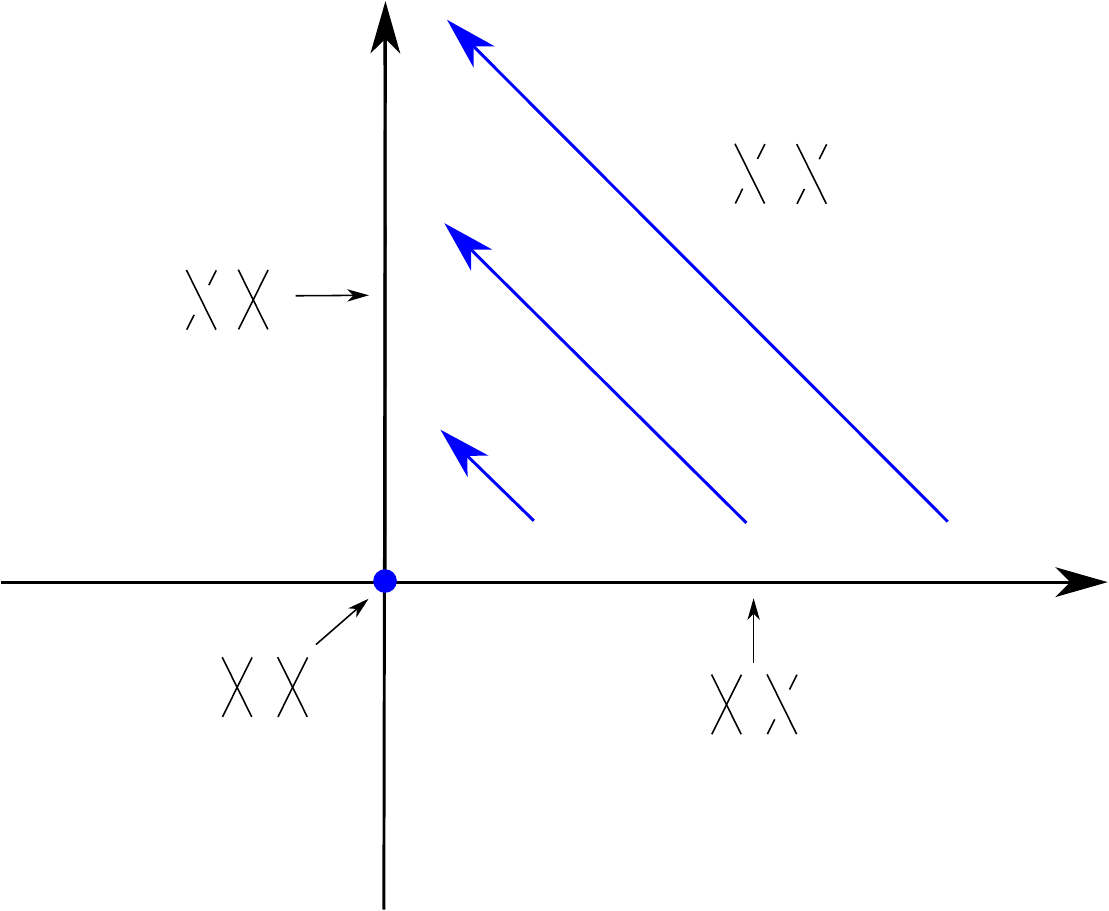}
    \put(70,53){\tiny \color{blue} $\bs Y^\ell_{\bm\tau,\bm s}$}
    \put(75,33){\tiny $\Delta_{I_{s_i}}^{\tau_i}$} \put(36,68){\tiny $\Delta_{I_{s_{i+1}}}^{\tau_{i+1}}$} \put(13, 34){\tiny $\Delta^{\tau_i, \tau_{i+1}}_{I_{s_i}, I_{s_{i+1}}}$}
	\end{overpic}
	\caption{Here we schematically depict portions of MFLS in $\mathcal{M}_{\ell}(\bm \gamma, \bm \gamma'; \bm Y^\ell_{\bm \tau, \bm s}, \bm \tau, \bm s)$ perturbed using $\bm Y^\ell_{\bm\tau,\bm s}$. The dot in the middle is a zero-length trajectory that passes through the deeper stratum $\Delta^{\tau_i, \tau_{i+1}}_{I_{s_i}, I_{s_{i+1}}}$.}
	\label{fig: coordinates}
\end{figure}
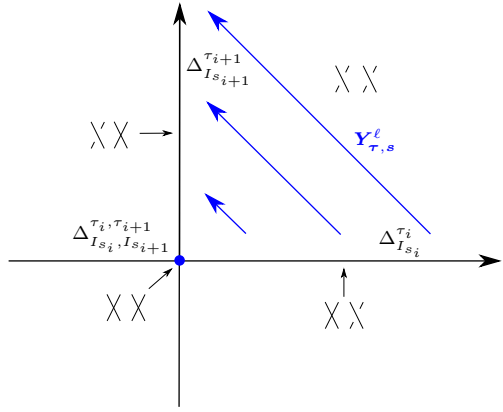

\subsection{Compactness for MFTS}

We introduce the various parameter spaces that ensure the compatibility on the boundary of moduli spaces of MFTS. Given $(T,\bm\tau) \in \mathcal{T}^{\ell, \ell'}_d$ and $\bm s\in S(T,\bm\tau)$ we define:
\begin{gather}\label{eqn: Xsw for trees}
    \mathfrak{X}_{\op{sw}} (T, \bm \tau, \bm s)=\prod_{i=1}^d\Big( \mathfrak{X}_-(|E_{\op{int}}(T)|+\ell) \times \prod_{j=1}^{\ell_{e_i}} \mathfrak{X}_0 (|E_{\op{int}}(T)|+\ell) \Big) \times \prod_{e \in E_{\op{int}}(T)} \prod_{j=0}^{\ell_{e}} \mathfrak{X}_0 (|E_{\op{int}}(T)|+\ell) \\
    \times \prod_{j=1}^{\ell_{e_0}} \mathfrak{X}_0 (|E_{\op{int}}(T)|+\ell) \times \mathfrak{X}_-(|E_{\op{int}}(T)|+\ell).\nonumber
\end{gather}
{\em For simplicity we have suppressed the superscripts $j$ in summands of the form $\mathfrak{X}_0^j(k)$; compare with \eqref{eqn: sw}.} The three groups of factors of $\mathfrak{X}_{\op{sw}} (T, \bm \tau, \bm s)$ correspond to the incoming edges $e_1,\dots, e_d$, the interior edges $E_{\op{int}}(T)$, and the outgoing edge $e_0$.  Observe that the right-hand side of \eqref{eqn: Xsw for trees} does not depend on $\ell'$.
We write $\bm\tau_e$ for the restriction of $\bm \tau$ to $e\in E(T)$ and $\bm s_e$ for the ordering of the switches on $e$ induced by $\bm s$.

The space $\mathfrak{X}_{\op{sw}}^{\op{back}}(T, \bm \tau, \bm s)$ is defined using \eqref{eqn: Xsw for trees} by adding superscripts ``back'' to each factor to indicate background perturbation spaces; these background perturbation spaces are defined in a manner analogous to \eqref{back pm} and \eqref{back 0}. Given an element $\bm Z=((\bm Z_{\tau_e, \bm s_{e}})_{e \in E(T)}) \in \mathfrak{X}_{\op{sw}}^{\op{back}}(T, \bm \tau, \bm s)$, we define the subspace $\mathfrak{X}_{\op{sw}}(T, \bm \tau, \bm s, \bm Z)$ of $\mathfrak{X}_{\op{sw}}(T,\bm \tau, \bm s)$ to be equal to
$$\prod_{i=1}^d \mathfrak{X}_-(\ell_{e_i}, \bm \tau_{e_i}, \bm s_{e_i}, \bm Z_{\tau_{e_i}, \bm s_{e_i}}) \times \prod_{e \in E_{\op{int}}(T)} \mathfrak{X}_0(\ell_{e}, \bm \tau_{e}, \bm s_{e}, \bm Z_{\tau_{e}, \bm s_{e}}) \times \mathfrak{X}_+(\ell_{e_0}, \bm \tau_{e_0}, \bm s_{e_0}, \bm Z_{\tau_{e_0}, \bm s_{e_0}}),$$
where the notation that we use here is analogous to \eqref{perturbation-neg-semi-ray-back}--\eqref{perturbation-segment-back}.

Given a collection of edges $F \subset E_{\op{int}}(T)$, there is a tree with switching data $(T/F,\bm \tau')$, obtained from $(T,\bm \tau)$ by contracting all the edges in $F$, following the prescription given in Section~\ref{subsection: the A infty structure}. Moreover, given $\bm\tau$ and $\bm s\in S(T,\bm\tau)$, there are induced $\bm\tau_{T/F}$ and $\bm s_{T/F}\in S(T/F,\bm\tau_{T/F})$. There exist forgetful maps
\begin{gather}
    \pi_F \colon \mathfrak{X}_{\op{sw}}(T,\bm \tau, \bm s) \to \mathfrak{X}_{\op{sw}}(T/F,\bm \tau_{T/F}, \bm s_{T/F}), \quad 
    \pi_F \colon \mathfrak{X}_{\op{sw}}^{\op{back}}(T, \bm \tau, \bm s) \to \mathfrak{X}_{\op{sw}}^{\op{back}}(T/F, \bm \tau_{T/F}, \bm s_{T/F}),
\end{gather}
which simply ``forget'' the coordinates corresponding to contracted edges; see \cite[Definition 4.13]{Mescher2018} for more details.

Given $d\geq 2$, a {\em $d$-perturbation data $\bm Y^d$} is a collection of $\bm Y_{T, \bm \tau, \bm s} \in \mathfrak{X}_{\op{sw}}(T, \bm \tau, \bm s)$ for all $\ell,\ell' \in \Z_{\ge 0}$, $(T,\bm\tau) \in \mathcal{T}_d^{\ell,\ell'}$, and $\bm s\in (T,\bm\tau)$. We say that $\bm Y^d$ is \emph{regular} if all the spaces 
$$\mathcal{M}_T(\vv{\bm \gamma}, \bm \gamma_0, (\bm N_e)_{e \in E(T)}, (\bm N_v)_{v \in V_{\op{int}}(T)};\bm Y_{T, \bm \tau, \bm s}, \bm\tau, \bm s)$$
are transversely cut out smooth manifolds; \emph{symmetric} if 
\begin{equation}\label{eq: MFTS0data-symmetric}
    \bm Y_{T, \bm \tau, \bm s}=\bm Y_{T, \bm \tau, \bm s'}
\end{equation}
for any $(T,\bm\tau) \in \mathcal{T}_d^{\ell, \ell'}$ and $\bm s,\bm s'\in S(T,\bm\tau)$;  and {\em commutative} if all the relations akin to~\eqref{eq-perturbation-switching} corresponding to ``collapsing'' two interior markers are satisfied.

Given $d\geq 2$, a {\em background $d$-perturbation data $\bm Z^d$} is a collection $\bm Z_{T, \bm \tau, \bm s} \in \mathfrak{X}^{\op{back}}_{\op{sw}}(T, \bm \tau, \bm s)$ for all $\ell,\ell' \in \Z_{\ge 0}$, $(T,\bm\tau) \in \mathcal{T}_{d}^{\ell, \ell'}$ with at least one finite-length segment (i.e., $|E_{\op{int}}(T)|+\ell>0$), and $\bm s\in S(T,\bm\tau)$. Here a {\em segment} is (the closure of) a component of an edge cut along the switching locus. The data $\bm Z^d$ is \emph{admissible} if $\pi_F(\bm Z_{T,\bm \tau,\bm s})=\bm Z_{T/F,\bm \tau_{T/F},\bm s_{T/F}}$ for all $(T,\bm\tau) \in \mathcal{T}_d^{\ell, \ell'}$, $\bm s\in S(T,\bm\tau)$, and $F \subset E_{\op{int}}(T)$.

Given a tuple $\bm Y^{<d}:=(\bm Y^2, \dots, \bm Y^{d-1})$, we may associate an admissible background $d$-perturbation data $\bm Z^{<d}_{T, \bm \tau, \bm s} \in \mathfrak{X}^{\op{back}}_{\op{sw}}(T, \bm \tau, \bm s)$ for each $(T,\bm\tau) \in \mathcal{T}_d^{\ell, \ell'}$ and each $\bm s$. By cutting the $i$th segment of an edge $e \in E(T)$ into two half-infinite segments, one splits $(T,\bm\tau)$ into $(T_1^{e,i},\bm \tau_{T_1}) \in \mathcal{T}_{k}^{\ell_1, \ell'_1}$ and $(T_2^{e,i},\bm\tau_{T_2}) \in \mathcal{T}_{d-k+1}^{\ell-\ell_1, \ell'-\ell'_1}$, and we simply use $\bm Y_{T^{e, i}_1, \bm \tau_{T_1}, \bm s_{T_1}}$ and $\bm Y_{T^{e,i}_2,\bm \tau_{T_2}, \bm s_{T_2}}$ to assign such background data; see the discussion prior to \cite[Definition 5.15]{Mescher2018} for more details. Here $\bm\tau_{T_1}$ is the restriction of $\bm \tau$ to $T_1^{e,i}$ and $\bm s_{T_1}$ is the ordering of the $\ell_1+\ell_1'$ switches of $T_1^{e,i}$ induced from $\bm s$.

\begin{definition} \label{defn: tree perturbation data} $\mbox{}$
\begin{enumerate}[(a)]
    \item A {\em tree perturbation data} is a family $\bm Y = (\bm Y^d)_{d \ge 2}$ of $d$-perturbation data $\bm Y^d$. The data $\bm Y$ is {\em regular (resp.\ symmetric, commutative)} if each $\bm Y^d$ is regular (resp.\ symmetric, commutative).
    \item A tree perturbation data $\bm Y = (\bm Y^d)_{d \ge 2}$ is {\em consistent} if for all $d \geq 2$, $(T,\bm \tau) \in \mathcal{T}^{\ell, \ell'}_d$, and $\bm s\in S(T,\bm\tau)$, $\bm Y_{T,\bm \tau,\bm s} \in \mathfrak{X}_{\op{sw}}(T, \bm Z^{<d}_{T, \bm \tau, \bm s})$, where $\bm Z^{<d}_{T, \bm \tau, \bm s} \in \mathfrak{X}^{\op{back}}_{\op{sw}}(T,\bm \tau,\bm s)$ is the background $d$-perturbation associated to $\bm Y^{<d}$ as described above.
    \item A tree perturbation data $\bm Y = (\bm Y^d)_{d \ge 2}$ is {\em universal} if it is symmetric, commutative, regular, consistent, and the following hold:
    \be[(1)]
    \item  for all $(T,\bm\tau)\in \mathcal{T}_d^{\ell, \ell'}$, $\bm s\in S(T,\bm\tau)$, and $F \subset E_{\op{int}}(T)$
    \begin{equation}
        \pi_F(\bm Y_{T,\bm\tau,\bm s})=\bm Y_{{T/F},\bm\tau_{T/F},\bm s_{T/F}};
    \end{equation} 
    \item for all $v \in V_{\op{int}}(T)$ and $i>0$, if $\bm\tau'$ is obtained from $\bm \tau$ by moving the last term of $\bm \tau_{e^v_i}$ to the first term of $\bm\tau_{v}$ (see Definitions~\ref{defn: metric ribbon tree} and \ref{defn: switching data tau} for the notation) and $\bm s'\in S(T,\bm\tau')$ is the induced ordering, then 
    \begin{equation}\label{eq: MFTS-data-vertex-compatibility-incoming-edge}
        \bm Y_{T, \bm \tau, \bm s}|_{ \{l^{e_i^v}_{\ell_{e_i^v}}=0\}}= \bm Y_{T,\bm \tau', \bm s'};
    \end{equation}
    \item
    for all $v \in V_{\op{int}}(T)$, if $\bm\tau'$ is obtained from $\bm\tau$ by moving the first term of $\bm\tau_{e^v_0}$ to the last term of $\bm\tau_v$ and $\bm s'\in S(T,\bm\tau')$ is the induced ordering, then
     \begin{equation}\label{eq: MFTS-data-vertex-compatibility-outgoing-edge}
        \bm Y_{T, \bm \tau, \bm s}|_{\{l^{e_0^v}_{\ell_{e_0^v}}=0\}}= \bm Y_{T, \bm s',\bm\tau'}.
    \end{equation}  
    \ee
\end{enumerate}
\end{definition}

\begin{theorem}\label{thm: MFTS-compactness}
A universal tree perturbation data $\bm Y = (\bm Y^d)_{d \ge 2}$ exists. Moreover, given $\vv{\bm \gamma}$, $\bm\gamma_0$, $\ell\in \Z_{\geq 0}$, and a binary tree $(T,\bm\tau) \in \mathcal{T}_d^{\ell}$ such that $\op{ind}(\vv{\bm \gamma},\bm\gamma_0, T,\bm s)=1$, each moduli space $\mathcal{M}_T(\vv{\bm \gamma}, \bm \gamma_0; \bm Y_{T, \bm \tau, \bm s},\bm \tau, \bm s)$ admits a compactification with boundary $\partial \mathcal{M}_T(\vv{\bm \gamma}, \bm \gamma_0; \bm Y_{T, \bm \tau, \bm s},\bm \tau, \bm s)$ covered by the following strata:
    \begin{gather}
    \label{eq: MFTS-boundary-breaking} 
    \bigcup_{e \in E(T), 0 \le i \le \ell_e,\bm\gamma_0'} \mathcal{M}_{T_1^{e, i}} (\bm \gamma_{j+1}, \dots, \bm \gamma_{j+d'}, \bm \gamma_0'; \bm Y_{T_1^{e, i}, \bm \tau_{T_1}, \bm s_{T_1}}, \bm \tau_{T_1}, \bm s_{T_1}) \times ~~~~~~~~~ \\
    ~~~~~~~~~\mathcal{M}_{T_2^{e,i}}(\bm \gamma_1, \dots, \bm \gamma_j, \bm \gamma_0', \bm \gamma_{j+d'+1}, \dots, \bm \gamma_{d}, \bm \gamma_0;\bm Y_{T_2^{e, i}, \bm \tau_{T_2},\bm s_{T_2}},\bm \tau_{T_2} , \bm s_{T_2} ),\nonumber \\
    \bigcup_{e \in E(T), 1 \le i \le \ell_e-1}\mathcal{M}_T(\vv{\bm \gamma}, \bm \gamma_0, \Delta^{\tau_i^e, \tau_{i+1}^e}_{t_{s^e_i}, t_{s^e_{i+1}}}; \bm Y_{T, \bm \tau, \bm s},\bm\tau, \bm s), \label{eq: MFTS-boundary-internal-segment}\\
    \bigcup_{e \in E_{\op{int}}(T) \op{ s.t. } \ell_e=0} \mathcal{M}_{T/e}(\vv{\bm \gamma}, \bm \gamma_0; \bm Y_{T/e,\bm\tau_{T/e}, \bm s_{T/e}}, \bm\tau_{T/e}, \bm s_{T/e}), \label{eq: MFTS-boundary-edge-contraction} \\
    \bigcup_{v \in V_{\op{int}}, 0 \le i \le |v|-1}\mathcal{M}_{T}(\vv{\bm \gamma}, \bm \gamma_0; \bm Y_{T,\bm\tau', \bm s'},\bm\tau',\bm s'),\label{eq: MFTS-boundary-switch-to-vertex} 
    \end{gather}
where $(T,\bm\tau',\bm s')$ in \eqref{eq: MFTS-boundary-switch-to-vertex} is given by Definition~\eqref{defn: tree perturbation data}(c)(2) or (3).
Moreover, the points of~\eqref{eq: MFTS-boundary-switch-to-vertex} are interior points of the compactification.
\end{theorem}

\begin{proof}
    The existence of a universal tree perturbation data is proved by an induction similar to the proof of Theorem~\ref{thm: compactness-for-MFLS}(1) and \cite[Lemma 5.16]{Mescher2018} and we leave details to the reader. 

    First, if one of the segments is ``stretched'', then the sequence of such trajectories approaches an element of \eqref{eq: MFTS-boundary-breaking} as in the proof of Theorem~\ref{thm: compactness-for-MFLS}(3). Note that if two or more such segments are ``stretched'', the limit belongs to a moduli space of negative dimension by the regularity and consistency of the data $\bm Y$, so this does not occur.

    Next suppose that the length of a finite-length segment $a$ converges to zero. By dimension reasons only one segment $a$ has length that goes to zero. If no endpoint of $a$ is in $V_{\op{int}}(T)$, then we have the boundary stratum \eqref{eq: MFTS-boundary-internal-segment} and the analysis of this case is completely analogous to that of \eqref{eq-compactness-two-switches} in Theorem~\ref{thm: compactness-for-MFLS}.  If both endpoints of $a$ are in $V_{\op{int}}(T)$, i.e., $a=e\in E(T)$, then we are in the case of \eqref{eq: MFTS-boundary-edge-contraction}.  The analysis of this is similar to the finite-dimensional case; see \cite[Theorem 5.17]{Mescher2018}.
    
    Finally assume that exactly one endpoint of $a$ is in $V_{\op{int}}(T)$.  Let $v$ be such a vertex and let $a\subset e_i^v$. If $i>0$ (resp.\ $i=0$), then the limit MFTS belongs to $\mathcal{M}_{T}(\vv{\bm \gamma}, \bm \gamma_0; \bm Y_{T,\bm\tau', \bm s'},\bm\tau',\bm s')$, where $(T,\bm\tau', \bm s')$ satisfies Definition~\eqref{defn: tree perturbation data}(c)(2) (resp.\ (c)(3)), and has dimension $0$ by Theorem~\ref{theorem: transversality for MFTS}. The switching point will pass through the vertex $v$ from the $i>0$ side to the $i=0$ side, or vice versa:  For example, if $i=0$, then the limit of the $t$-coordinate of the corresponding switching point belongs to the unique segment among the $|v|-1$ segments partitioning $[0,1]$ according to the weights at $v$; see Definition~\ref{def: MFTS}. Such a segment corresponds to a unique incoming edge $e^v_j$. Now, because $\bm Y$ satisfies \eqref{eq: MFTS-data-vertex-compatibility-incoming-edge} and \eqref{eq: MFTS-data-vertex-compatibility-outgoing-edge} and the weighted concatenation commutes with the switching map by Lemma~\ref{lemma: concatenation-commutativity}, this limit MFTS belongs to $\mathcal{M}_{T}(\vv{\bm \gamma}, \bm \gamma_0; \bm Y_{T,\bm \tau'',\bm s''},\bm \tau'', \bm s'')$ as in Definition~\eqref{defn: tree perturbation data}(c)(2) for $e_j^v$. This analysis also confirms the last claim assuming the appropriate gluing statement.
\end{proof}

\printbibliography

@article {Ziller-Katok-examples,
    AUTHOR = {Ziller, Wolfgang},
     TITLE = {Geometry of the {K}atok examples},
   JOURNAL = {Ergodic Theory Dynam. Systems},
  FJOURNAL = {Ergodic Theory and Dynamical Systems},
    VOLUME = {3},
      YEAR = {1983},
    NUMBER = {1},
     PAGES = {135--157},
      ISSN = {0143-3857,1469-4417},
   MRCLASS = {58E10 (53C22)},
  MRNUMBER = {743032},
MRREVIEWER = {Gudlaugur\ Thorbergsson},
       DOI = {10.1017/S0143385700001851},
       URL = {https://doi.org/10.1017/S0143385700001851},
}

@article {Katok-examples,
    AUTHOR = {Katok, A. B.},
     TITLE = {Ergodic perturbations of degenerate integrable {H}amiltonian
              systems},
   JOURNAL = {Izv. Akad. Nauk SSSR Ser. Mat.},
  FJOURNAL = {Izvestiya Akademii Nauk SSSR. Seriya Matematicheskaya},
    VOLUME = {37},
      YEAR = {1973},
     PAGES = {539--576},
      ISSN = {0373-2436},
   MRCLASS = {58F05 (28A65)},
  MRNUMBER = {331425},
MRREVIEWER = {K.\ Krzy\.zewski},
}

@article {Floer-unregularized,
    AUTHOR = {Floer, Andreas},
     TITLE = {The unregularized gradient flow of the symplectic action},
   JOURNAL = {Comm. Pure Appl. Math.},
  FJOURNAL = {Communications on Pure and Applied Mathematics},
    VOLUME = {41},
      YEAR = {1988},
    NUMBER = {6},
     PAGES = {775--813},
      ISSN = {0010-3640,1097-0312},
   MRCLASS = {58F05},
  MRNUMBER = {948771},
MRREVIEWER = {Holger\ Kantz},
       DOI = {10.1002/cpa.3160410603},
       URL = {https://doi.org/10.1002/cpa.3160410603},
}

@book {FOOO,
    AUTHOR = {Fukaya, Kenji and Oh, Yong-Geun and Ohta, Hiroshi and Ono,
              Kaoru},
     TITLE = {Lagrangian intersection {F}loer theory: anomaly and
              obstruction. {P}art {I}},
    SERIES = {AMS/IP Studies in Advanced Mathematics},
    VOLUME = {46.1},
 PUBLISHER = {American Mathematical Society, Providence, RI; International
              Press, Somerville, MA},
      YEAR = {2009},
     PAGES = {xii+396},
      ISBN = {978-0-8218-4836-4},
   MRCLASS = {53D40 (53D12 53D37)},
  MRNUMBER = {2553465},
MRREVIEWER = {Michael\ J.\ Usher},
       DOI = {10.1090/amsip/046.1},
       URL = {https://doi.org/10.1090/amsip/046.1},
}

@book {FOOO2,
    AUTHOR = {Fukaya, Kenji and Oh, Yong-Geun and Ohta, Hiroshi and Ono,
              Kaoru},
     TITLE = {Lagrangian intersection {F}loer theory: anomaly and
              obstruction. {P}art {II}},
    SERIES = {AMS/IP Studies in Advanced Mathematics},
    VOLUME = {46.2},
 PUBLISHER = {American Mathematical Society, Providence, RI; International
              Press, Somerville, MA},
      YEAR = {2009},
     PAGES = {i--xii and 397--805},
      ISBN = {978-0-8218-4837-1},
   MRCLASS = {53D40 (53D12)},
  MRNUMBER = {2548482},
MRREVIEWER = {Michael\ J.\ Usher},
       DOI = {10.1090/amsip/046.2},
       URL = {https://doi.org/10.1090/amsip/046.2},
}

@article {Par,
    AUTHOR = {Pardon, John},
     TITLE = {An algebraic approach to virtual fundamental cycles on moduli
              spaces of pseudo-holomorphic curves},
   JOURNAL = {Geom. Topol.},
  FJOURNAL = {Geometry \& Topology},
    VOLUME = {20},
      YEAR = {2016},
    NUMBER = {2},
     PAGES = {779--1034},
      ISSN = {1465-3060,1364-0380},
   MRCLASS = {53D35 (37J10 53D37 53D40 53D42 53D45 54B40 57R17)},
  MRNUMBER = {3493097},
MRREVIEWER = {Sonja\ Hohloch},
       DOI = {10.2140/gt.2016.20.779},
       URL = {https://doi.org/10.2140/gt.2016.20.779},
}

@article {HWZ2,
    AUTHOR = {Hofer, H. and Wysocki, K. and Zehnder, E.},
     TITLE = {Applications of polyfold theory {I}: {T}he polyfolds of
              {G}romov-{W}itten theory},
   JOURNAL = {Mem. Amer. Math. Soc.},
  FJOURNAL = {Memoirs of the American Mathematical Society},
    VOLUME = {248},
      YEAR = {2017},
    NUMBER = {1179},
     PAGES = {v+218},
      ISSN = {0065-9266,1947-6221},
      ISBN = {978-1-4704-2203-5; 978-1-4704-4060-2},
   MRCLASS = {53D45 (14N35 57R17)},
  MRNUMBER = {3683060},
MRREVIEWER = {William\ Liu},
       DOI = {10.1090/memo/1179},
       URL = {https://doi.org/10.1090/memo/1179},
}

@article {MW,
    AUTHOR = {McDuff, Dusa and Wehrheim, Katrin},
     TITLE = {Smooth {K}uranishi atlases with isotropy},
   JOURNAL = {Geom. Topol.},
  FJOURNAL = {Geometry \& Topology},
    VOLUME = {21},
      YEAR = {2017},
    NUMBER = {5},
     PAGES = {2725--2809},
      ISSN = {1465-3060,1364-0380},
   MRCLASS = {53D35 (53D45 54B15 57R17 57R95)},
  MRNUMBER = {3687107},
MRREVIEWER = {Michael\ J.\ Usher},
       DOI = {10.2140/gt.2017.21.2725},
       URL = {https://doi.org/10.2140/gt.2017.21.2725},
}

@article {McD,
    AUTHOR = {McDuff, Dusa},
     TITLE = {Constructing the virtual fundamental class of a {K}uranishi
              atlas},
   JOURNAL = {Algebr. Geom. Topol.},
  FJOURNAL = {Algebraic \& Geometric Topology},
    VOLUME = {19},
      YEAR = {2019},
    NUMBER = {1},
     PAGES = {151--238},
      ISSN = {1472-2747,1472-2739},
   MRCLASS = {53D35 (18B30 53D45 57R17 57R95)},
  MRNUMBER = {3910580},
MRREVIEWER = {Eduardo\ A.\ Gonzalez},
       DOI = {10.2140/agt.2019.19.151},
       URL = {https://doi.org/10.2140/agt.2019.19.151},
}

@book{FOOO3,
    AUTHOR = {Fukaya, Kenji and Oh, Yong-Geun and Ohta, Hiroshi and Ono,
              Kaoru},
     TITLE = {Kuranishi structures and virtual fundamental chains},
    SERIES = {Springer Monographs in Mathematics},
 PUBLISHER = {Springer, Singapore},
      YEAR = {[2020] \copyright 2020},
     PAGES = {xv+638},
      ISBN = {978-981-15-5562-6; 978-981-15-5561-9},
   MRCLASS = {53D45 (32M99 53D37 53D40)},
  MRNUMBER = {4179586},
MRREVIEWER = {Hsian-Hua\ Tseng},
       DOI = {10.1007/978-981-15-5562-6},
       URL = {https://doi.org/10.1007/978-981-15-5562-6},
}

@misc{DW,
    title={Counting embedded curves in symplectic $6$-manifolds},
    author={Aleksander Doan and Thomas Walpuski},
    year={2019},
    eprint={1910.12338},
    archivePrefix={arXiv},
    primaryClass={math.SG}
}

@misc{ES,
    title={Ghost bubble censorship},
    author={Tobias Ekholm and Vivek Shende},
    year={2022},
    eprint={2212.05835},
    archivePrefix={arXiv},
    primaryClass={math.SG}
}

@book{seidel2008fukaya,
  title={Fukaya categories and Picard-Lefschetz theory},
  author={Seidel, Paul},
  volume={10},
  year={2008},
  publisher={European Mathematical Society}
}

@article{colin2020applications,
    author={Vincent Colin and Ko Honda and Yin Tian},
    title={Applications of higher\hyp{}dimensional Heegaard Floer homology to contact topology},
    journal ={J. Topol.},
    volume={17},
    year={2024},
    pages = {e12349}
}

@misc{KY,
    title={Heegaard Floer symplectic homology and {V}iterbo's isomorphism theorem in the context of multiple identical particles},
    author={Roman Krutowski and Tianyu Yuan},
    year={2023},
    eprint={2311.17031},
    archivePrefix={arXiv},
    primaryClass={math.SG}
    }

@article {fukaya1997zero,
    AUTHOR = {Fukaya, Kenji and Oh, Yong-Geun},
     TITLE = {Zero-loop open strings in the cotangent bundle and {M}orse
              homotopy},
   JOURNAL = {Asian J. Math.},
  FJOURNAL = {Asian Journal of Mathematics},
    VOLUME = {1},
      YEAR = {1997},
    NUMBER = {1},
     PAGES = {96--180},
      ISSN = {1093-6106},
   MRCLASS = {58E05 (58D29 58F09)},
  MRNUMBER = {1480992},
MRREVIEWER = {Joa Weber},
       DOI = {10.4310/AJM.1997.v1.n1.a5},
       URL = {https://doi.org/10.4310/AJM.1997.v1.n1.a5},
}

@article {abbondandolo2010floer,
    AUTHOR = {Abbondandolo, Alberto and Schwarz, Matthias},
     TITLE = {Floer homology of cotangent bundles and the loop product},
   JOURNAL = {Geom. Topol.},
  FJOURNAL = {Geometry \& Topology},
    VOLUME = {14},
      YEAR = {2010},
    NUMBER = {3},
     PAGES = {1569--1722},
      ISSN = {1465-3060},
   MRCLASS = {53D40 (55N45 55P50 57R58)},
  MRNUMBER = {2679580},
MRREVIEWER = {Janko Latschev},
       DOI = {10.2140/gt.2010.14.1569},
       URL = {https://doi.org/10.2140/gt.2010.14.1569},
}

@article {abouzaid2012wrapped,
    AUTHOR = {Abouzaid, Mohammed},
     TITLE = {On the wrapped {F}ukaya category and based loops},
   JOURNAL = {J. Symplectic Geom.},
  FJOURNAL = {The Journal of Symplectic Geometry},
    VOLUME = {10},
      YEAR = {2012},
    NUMBER = {1},
     PAGES = {27--79},
      ISSN = {1527-5256},
   MRCLASS = {53D37 (53D12)},
  MRNUMBER = {2904032},
MRREVIEWER = {Janko Latschev},
       URL = {http://projecteuclid.org/euclid.jsg/1332853049},
}

@misc{ekholm2021skeins,
      title={Skeins on branes}, 
      author={Tobias Ekholm and Vivek Shende},
      year={2019},
      eprint={1901.08027},
      archivePrefix={arXiv},
      primaryClass={math.SG}
}

@unpublished{FO3lagrangian,
      title={Lagrangian surgery and holomorphic discs}, 
      author={Fukaya, Kenji and Oh, Yong-Geun and Ohta, Hiroshi and Ono,
              Kaoru},
      note={Chapter 10 of the preprint version of \cite{FOOO2}}
}

@article {morton2021dahas,
    AUTHOR = {Morton, Hugh and Samuelson, Peter},
     TITLE = {D{AHA}s and skein theory},
   JOURNAL = {Comm. Math. Phys.},
  FJOURNAL = {Communications in Mathematical Physics},
    VOLUME = {385},
      YEAR = {2021},
    NUMBER = {3},
     PAGES = {1655--1693},
      ISSN = {0010-3616},
   MRCLASS = {57K14 (17B37)},
  MRNUMBER = {4283999},
       DOI = {10.1007/s00220-021-04052-8},
       URL = {https://doi.org/10.1007/s00220-021-04052-8},
}

@article {abouzaid2010geometric,
    AUTHOR = {Abouzaid, Mohammed},
     TITLE = {A geometric criterion for generating the {F}ukaya category},
   JOURNAL = {Publ. Math. Inst. Hautes \'{E}tudes Sci.},
  FJOURNAL = {Publications Math\'{e}matiques. Institut de Hautes \'{E}tudes
              Scientifiques},
    NUMBER = {112},
      YEAR = {2010},
     PAGES = {191--240},
      ISSN = {0073-8301},
   MRCLASS = {53D37},
  MRNUMBER = {2737980},
MRREVIEWER = {Timothy Perutz},
       DOI = {10.1007/s10240-010-0028-5},
       URL = {https://doi.org/10.1007/s10240-010-0028-5},
}

@book {milnor2016morse,
    AUTHOR = {Milnor, John},
     TITLE = {Morse theory},
    SERIES = {Annals of Mathematics Studies, No. 51},
      NOTE = {Based on lecture notes by M. Spivak and R. Wells},
 PUBLISHER = {Princeton University Press, Princeton, N.J.},
      YEAR = {1963},
     PAGES = {vi+153},
   MRCLASS = {57.50 (53.72)},
  MRNUMBER = {0163331},
MRREVIEWER = {H. I. Levine},
}

@unpublished{CHT,
    title={Towards a definition of Khovanov homology for links in fibered $3$-manifolds},
    author={Vincent Colin and Ko Honda and Yin Tian},
    note={in preparation}
}

@misc{chas-sullivan1999,
    title={String topology},
    author={Chas, Moira and Sullivan, Dennis},
     archivePrefix={arXiv},
  eprint={math/9911159v1},
  year={1999}
}

@article{honda2022higher,
    title={Higher-dimensional Heegaard Floer homology and Hecke algebras},
    author={Honda, Ko and Tian, Yin and Yuan, Tianyu},
    journal={J. Eur. Math. Soc. (JEMS)},
    year={to appear}
}

@article {AS2009,
    AUTHOR = {Abbondandolo, Alberto and Schwarz, Matthias},
     TITLE = {A smooth pseudo-gradient for the {L}agrangian action
              functional},
   JOURNAL = {Adv. Nonlinear Stud.},
  FJOURNAL = {Advanced Nonlinear Studies},
    VOLUME = {9},
      YEAR = {2009},
    NUMBER = {4},
     PAGES = {597--623},
      ISSN = {1536-1365,2169-0375},
   MRCLASS = {37J45 (37D15 47J30 58E05)},
  MRNUMBER = {2560122},
MRREVIEWER = {Alessandro\ Portaluri},
       DOI = {10.1515/ans-2009-0402},
       URL = {https://doi.org/10.1515/ans-2009-0402},
}

@book {McDuff-Salamon-2012,
    AUTHOR = {McDuff, Dusa and Salamon, Dietmar},
     TITLE = {{$J$}-holomorphic curves and symplectic topology},
    SERIES = {American Mathematical Society Colloquium Publications},
    VOLUME = {52},
   EDITION = {Second},
 PUBLISHER = {American Mathematical Society, Providence, RI},
      YEAR = {2012},
     PAGES = {xiv+726},
      ISBN = {978-0-8218-8746-2},
   MRCLASS = {53D45 (32Q65 53D35)},
  MRNUMBER = {2954391},
MRREVIEWER = {Mark\ Alan\ Branson},
}

@article {ACF2016,
    AUTHOR = {Albers, Peter and Cieliebak, Kai and Frauenfelder, Urs},
     TITLE = {Symplectic {T}ate homology},
   JOURNAL = {Proc. Lond. Math. Soc.},
  FJOURNAL = {Proceedings of the London Mathematical Society. Third Series},
    VOLUME = {112},
      YEAR = {2016},
    NUMBER = {1},
     PAGES = {169--205},
      ISSN = {0024-6115,1460-244X},
   MRCLASS = {57R19 (55U15 57R17)},
  MRNUMBER = {3458149},
MRREVIEWER = {Ferit\ \"{O}zt\"{u}rk},
       DOI = {10.1112/plms/pdv065},
       URL = {https://doi.org/10.1112/plms/pdv065},
}

@article {abouzaid2009tropical,
    AUTHOR = {Abouzaid, Mohammed},
     TITLE = {Morse homology, tropical geometry, and homological mirror
              symmetry for toric varieties},
   JOURNAL = {Selecta Math. (N.S.)},
  FJOURNAL = {Selecta Mathematica. New Series},
    VOLUME = {15},
      YEAR = {2009},
    NUMBER = {2},
     PAGES = {189--270},
      ISSN = {1022-1824,1420-9020},
   MRCLASS = {53D37 (14J33 14M25 14T05 53D40)},
  MRNUMBER = {2529936},
MRREVIEWER = {Kwokwai\ Chan},
       DOI = {10.1007/s00029-009-0492-2},
       URL = {https://doi.org/10.1007/s00029-009-0492-2},
}

@article {siegel2021,
    AUTHOR = {Siegel, Kyler},
     TITLE = {Squared {D}ehn twists and deformed symplectic invariants},
   JOURNAL = {J. Symplectic Geom.},
  FJOURNAL = {The Journal of Symplectic Geometry},
    VOLUME = {19},
      YEAR = {2021},
    NUMBER = {5},
     PAGES = {1189--1280},
      ISSN = {1527-5256,1540-2347},
   MRCLASS = {53D12 (53D37)},
  MRNUMBER = {4436690},
}

@article{wendl2008,
  title={Lectures on Symplectic Field Theory},
  author={Wendl, Chris},
  journal={Lecture notes, Humboldt-Universit{\"a}t zu Berlin},
  year={2008}
}

@book {evans1998,
    AUTHOR = {Evans, Lawrence C.},
     TITLE = {Partial differential equations},
    SERIES = {Graduate Studies in Mathematics},
    VOLUME = {19},
 PUBLISHER = {American Mathematical Society, Providence, RI},
      YEAR = {1998},
     PAGES = {xviii+662},
      ISBN = {0-8218-0772-2},
   MRCLASS = {35-01},
  MRNUMBER = {1625845},
MRREVIEWER = {Luigi\ Rodino},
       DOI = {10.1090/gsm/019},
       URL = {https://doi.org/10.1090/gsm/019},
}

@book{klingenberg2012lectures,
  title={Lectures on closed geodesics},
  author={Klingenberg, Wilhelm},
  volume={230},
  year={2012},
  publisher={Springer Science \& Business Media}
}

@article {irie2018,
    AUTHOR = {Irie, Kei},
     TITLE = {A chain level {B}atalin-{V}ilkovisky structure in string
              topology via de {R}ham chains},
   JOURNAL = {Int. Math. Res. Not. IMRN},
  FJOURNAL = {International Mathematics Research Notices. IMRN},
      YEAR = {2018},
    NUMBER = {15},
     PAGES = {4602--4674},
      ISSN = {1073-7928,1687-0247},
   MRCLASS = {55P48 (16E45 18D50 58A10)},
  MRNUMBER = {3842380},
MRREVIEWER = {Sinan\ Yalin},
       DOI = {10.1093/imrn/rnx023},
       URL = {https://doi.org/10.1093/imrn/rnx023},
}

@article {cohenjones2002,
    AUTHOR = {Cohen, Ralph L. and Jones, John D. S.},
     TITLE = {A homotopy theoretic realization of string topology},
   JOURNAL = {Math. Ann.},
  FJOURNAL = {Mathematische Annalen},
    VOLUME = {324},
      YEAR = {2002},
    NUMBER = {4},
     PAGES = {773--798},
      ISSN = {0025-5831,1432-1807},
   MRCLASS = {55P43 (55P35 55U35)},
  MRNUMBER = {1942249},
MRREVIEWER = {Alexander\ A.\ Voronov},
       DOI = {10.1007/s00208-002-0362-0},
       URL = {https://doi.org/10.1007/s00208-002-0362-0},
}

@book {mescher2018,
    AUTHOR = {Mescher, Stephan},
     TITLE = {Perturbed gradient flow trees and {$A_\infty$}-algebra
              structures in {M}orse cohomology},
    SERIES = {Atlantis Studies in Dynamical Systems},
    VOLUME = {6},
 PUBLISHER = {Atlantis Press, [Paris]; Springer, Cham},
      YEAR = {2018},
     PAGES = {xxv+171},
      ISBN = {978-3-319-76583-9; 978-3-319-76584-6},
   MRCLASS = {58E05 (55P43 57R70)},
  MRNUMBER = {3791518},
       DOI = {10.1007/978-3-319-76584-6},
       URL = {https://doi.org/10.1007/978-3-319-76584-6},
}

@incollection {abbondandolo-majer2006,
    AUTHOR = {Abbondandolo, Alberto and Majer, Pietro},
     TITLE = {Lectures on the {M}orse complex for infinite-dimensional
              manifolds},
 BOOKTITLE = {Morse theoretic methods in nonlinear analysis and in
              symplectic topology},
    SERIES = {NATO Sci. Ser. II Math. Phys. Chem.},
    VOLUME = {217},
     PAGES = {1--74},
 PUBLISHER = {Springer, Dordrecht},
      YEAR = {2006},
      ISBN = {978-1-4020-4273-7; 1-4020-4273-6},
   MRCLASS = {58E05},
  MRNUMBER = {2276948},
MRREVIEWER = {David\ E.\ Hurtubise},
       DOI = {10.1007/1-4020-4266-3\_01},
       URL = {https://doi.org/10.1007/1-4020-4266-3_01},
}

@article {robles2007,
    AUTHOR = {Robles, Colleen},
     TITLE = {Geodesics in {R}anders spaces of constant curvature},
   JOURNAL = {Trans. Amer. Math. Soc.},
  FJOURNAL = {Transactions of the American Mathematical Society},
    VOLUME = {359},
      YEAR = {2007},
    NUMBER = {4},
     PAGES = {1633--1651},
      ISSN = {0002-9947,1088-6850},
   MRCLASS = {53C60 (53C22)},
  MRNUMBER = {2272144},
MRREVIEWER = {Xinyue\ Cheng},
       DOI = {10.1090/S0002-9947-06-04051-7},
       URL = {https://doi.org/10.1090/S0002-9947-06-04051-7},
}

@article {lu2015,
    AUTHOR = {Lu, Guangcun},
     TITLE = {Methods of infinite dimensional {M}orse theory for geodesics
              on {F}insler manifolds},
   JOURNAL = {Nonlinear Anal.},
  FJOURNAL = {Nonlinear Analysis. Theory, Methods \& Applications. An
              International Multidisciplinary Journal},
    VOLUME = {113},
      YEAR = {2015},
     PAGES = {230--282},
      ISSN = {0362-546X,1873-5215},
   MRCLASS = {58E05 (53C20 53C22 58B20)},
  MRNUMBER = {3281856},
MRREVIEWER = {Liviu\ Octavian\ Popescu},
       DOI = {10.1016/j.na.2014.09.016},
       URL = {https://doi.org/10.1016/j.na.2014.09.016},
}

@article {gongxue2020,
    AUTHOR = {Gong, Wenmin and Xue, Jinxin},
     TITLE = {Floer homology in the cotangent bundle of a closed {F}insler
              manifold and noncontractible periodic orbits},
   JOURNAL = {Nonlinearity},
  FJOURNAL = {Nonlinearity},
    VOLUME = {33},
      YEAR = {2020},
    NUMBER = {12},
     PAGES = {6297--6348},
      ISSN = {0951-7715,1361-6544},
   MRCLASS = {53D40 (37J39 70H12)},
  MRNUMBER = {4164680},
MRREVIEWER = {Alexander\ Fel\cprime shtyn},
       DOI = {10.1088/1361-6544/abb190},
       URL = {https://doi.org/10.1088/1361-6544/abb190},
}

@book {schwarz93,
    AUTHOR = {Schwarz, Matthias},
     TITLE = {Morse homology},
    SERIES = {Progress in Mathematics},
    VOLUME = {111},
 PUBLISHER = {Birkh\"auser Verlag, Basel},
      YEAR = {1993},
     PAGES = {x+235},
      ISBN = {3-7643-2904-1},
   MRCLASS = {58E05 (55N35 57R70)},
  MRNUMBER = {1239174},
MRREVIEWER = {Daniel\ M.\ Burns, Jr.},
       DOI = {10.1007/978-3-0348-8577-5},
       URL = {https://doi.org/10.1007/978-3-0348-8577-5},
}

@article {abouzaid2011plumbings,
    AUTHOR = {Abouzaid, Mohammed},
     TITLE = {A topological model for the {F}ukaya categories of plumbings},
   JOURNAL = {J. Differential Geom.},
  FJOURNAL = {Journal of Differential Geometry},
    VOLUME = {87},
      YEAR = {2011},
    NUMBER = {1},
     PAGES = {1--80},
      ISSN = {0022-040X,1945-743X},
   MRCLASS = {53D37 (53D40)},
  MRNUMBER = {2786590},
MRREVIEWER = {Timothy\ Perutz},
       URL = {http://projecteuclid.org/euclid.jdg/1303219772},
}

@misc{mazuir2021I,
    title={Higher algebra of $A_\infty$ and $\Omega BAs$-algebras in Morse theory I},
    author={Thibaut Mazuir},
    year={2021},
    eprint={2102.06654},
    archivePrefix={arXiv},
    primaryClass={math.SG}
}

@misc{mazuir2021II,
    title={Higher algebra of $A_\infty$ and $\Omega BAs$-algebras in Morse theory II},
    author={Thibaut Mazuir},
    year={2021},
    eprint={2102.08996},
    archivePrefix={arXiv},
    primaryClass={math.SG}
}

@misc{ganatra2013thesis,
    title={Symplectic cohomology and duality for the wrapped Fukaya category},
    author={Sheel Ganatra},
    year={2013},
    eprint={1304.7312},
    archivePrefix={arXiv},
    primaryClass={math.SG}
}

@misc{godin2007higher,
    title={Higher string topology operations},
    author={Godin, Véronique},
    year={2007},
    eprint={0711.4859},
    archivePrefix={arXiv},
    primaryClass={math.AT}
}

@article {tamanoi2010tqft,
    AUTHOR = {Tamanoi, Hirotaka},
     TITLE = {Loop coproducts in string topology and triviality of higher
              genus {TQFT} operations},
   JOURNAL = {J. Pure Appl. Algebra},
  FJOURNAL = {Journal of Pure and Applied Algebra},
    VOLUME = {214},
      YEAR = {2010},
    NUMBER = {5},
     PAGES = {605--615},
      ISSN = {0022-4049,1873-1376},
   MRCLASS = {55P50 (55N45 55P35 57R56)},
  MRNUMBER = {2577666},
MRREVIEWER = {Dev\ Prakash\ Sinha},
       DOI = {10.1016/j.jpaa.2009.07.011},
       URL = {https://doi.org/10.1016/j.jpaa.2009.07.011},
}

@article {ganatra2023cyclic,
    AUTHOR = {Ganatra, Sheel},
     TITLE = {Cyclic homology, {$S^1$}-equivariant {F}loer cohomology and
              {C}alabi-{Y}au structures},
   JOURNAL = {Geom. Topol.},
  FJOURNAL = {Geometry \& Topology},
    VOLUME = {27},
      YEAR = {2023},
    NUMBER = {9},
     PAGES = {3461--3584},
      ISSN = {1465-3060,1364-0380},
   MRCLASS = {53D37 (19D55 53D12)},
  MRNUMBER = {4674834},
       DOI = {10.2140/gt.2023.27.3461},
       URL = {https://doi.org/10.2140/gt.2023.27.3461},
}

@article {FO97zero-loop,
    AUTHOR = {Fukaya, Kenji and Oh, Yong-Geun},
     TITLE = {Zero-loop open strings in the cotangent bundle and {M}orse
              homotopy},
   JOURNAL = {Asian J. Math.},
  FJOURNAL = {Asian Journal of Mathematics},
    VOLUME = {1},
      YEAR = {1997},
    NUMBER = {1},
     PAGES = {96--180},
      ISSN = {1093-6106,1945-0036},
   MRCLASS = {58E05 (58D29 58F09)},
  MRNUMBER = {1480992},
MRREVIEWER = {Joa\ Weber},
       DOI = {10.4310/AJM.1997.v1.n1.a5},
       URL = {https://doi.org/10.4310/AJM.1997.v1.n1.a5},
}

@incollection {sullivan2004open-closed,
    AUTHOR = {Sullivan, Dennis},
     TITLE = {Open and closed string field theory interpreted in classical
              algebraic topology},
 BOOKTITLE = {Topology, geometry and quantum field theory},
    SERIES = {London Math. Soc. Lecture Note Ser.},
    VOLUME = {308},
     PAGES = {344--357},
 PUBLISHER = {Cambridge Univ. Press, Cambridge},
      YEAR = {2004},
      ISBN = {0-521-54049-6},
   MRCLASS = {81T45 (55P48 81T30)},
  MRNUMBER = {2079379},
MRREVIEWER = {Riccardo\ Longoni},
       DOI = {10.1017/CBO9780511526398.014},
       URL = {https://doi.org/10.1017/CBO9780511526398.014},
}

@article {GH2009loop-product,
    AUTHOR = {Goresky, Mark and Hingston, Nancy},
     TITLE = {Loop products and closed geodesics},
   JOURNAL = {Duke Math. J.},
  FJOURNAL = {Duke Mathematical Journal},
    VOLUME = {150},
      YEAR = {2009},
    NUMBER = {1},
     PAGES = {117--209},
      ISSN = {0012-7094,1547-7398},
   MRCLASS = {58E10 (55N45 55P50 58E05)},
  MRNUMBER = {2560110},
MRREVIEWER = {John\ F.\ Oprea},
       DOI = {10.1215/00127094-2009-049},
       URL = {https://doi.org/10.1215/00127094-2009-049},
}

@article {HW2023product-coproduct,
    AUTHOR = {Hingston, Nancy and Wahl, Nathalie},
     TITLE = {Product and coproduct in string topology},
   JOURNAL = {Ann. Sci. \'Ec. Norm. Sup\'er. (4)},
  FJOURNAL = {Annales Scientifiques de l'\'Ecole Normale Sup\'erieure.
              Quatri\`eme S\'erie},
    VOLUME = {56},
      YEAR = {2023},
    NUMBER = {5},
     PAGES = {1381--1447},
      ISSN = {0012-9593,1873-2151},
   MRCLASS = {55P50 (55P35 57N65)},
  MRNUMBER = {4706499},
}

@incollection {CG2004polarized,
    AUTHOR = {Cohen, Ralph L. and Godin, V\'eronique},
     TITLE = {A polarized view of string topology},
 BOOKTITLE = {Topology, geometry and quantum field theory},
    SERIES = {London Math. Soc. Lecture Note Ser.},
    VOLUME = {308},
     PAGES = {127--154},
 PUBLISHER = {Cambridge Univ. Press, Cambridge},
      YEAR = {2004},
      ISBN = {0-521-54049-6},
   MRCLASS = {55P35 (57R56)},
  MRNUMBER = {2079373},
MRREVIEWER = {Hossein\ Abbaspour},
       DOI = {10.1017/CBO9780511526398.008},
       URL = {https://doi.org/10.1017/CBO9780511526398.008},
}

@article {chataur2005,
    AUTHOR = {Chataur, David},
     TITLE = {A bordism approach to string topology},
   JOURNAL = {Int. Math. Res. Not.},
  FJOURNAL = {International Mathematics Research Notices},
      YEAR = {2005},
    NUMBER = {46},
     PAGES = {2829--2875},
      ISSN = {1073-7928,1687-0247},
   MRCLASS = {55P48 (55N20 55N22 55P35 55P43)},
  MRNUMBER = {2180465},
MRREVIEWER = {Hossein\ Abbaspour},
       DOI = {10.1155/IMRN.2005.2829},
       URL = {https://doi.org/10.1155/IMRN.2005.2829},
}

@article {laudenbach2011,
    AUTHOR = {Laudenbach, Fran\c{c}ois},
     TITLE = {A note on the {C}has-{S}ullivan product},
   JOURNAL = {Enseign. Math. (2)},
  FJOURNAL = {L'Enseignement Math\'ematique. Revue Internationale. 2e
              S\'erie},
    VOLUME = {57},
      YEAR = {2011},
    NUMBER = {1-2},
     PAGES = {3--21},
      ISSN = {0013-8584},
   MRCLASS = {55P50 (55N45 55T05)},
  MRNUMBER = {2850582},
MRREVIEWER = {David\ Chataur},
       DOI = {10.4171/LEM/57-1-1},
       URL = {https://doi.org/10.4171/LEM/57-1-1},
}

@article {CHO2023coproduct,
    AUTHOR = {Cieliebak, Kai and Hingston, Nancy and Oancea, Alexandru},
     TITLE = {Loop coproduct in {M}orse and {F}loer homology},
   JOURNAL = {J. Fixed Point Theory Appl.},
  FJOURNAL = {Journal of Fixed Point Theory and Applications},
    VOLUME = {25},
      YEAR = {2023},
    NUMBER = {2},
     PAGES = {Paper No. 59, 84},
      ISSN = {1661-7738,1661-7746},
   MRCLASS = {57R58 (53D40)},
  MRNUMBER = {4597665},
MRREVIEWER = {Jun\ Zhang},
       DOI = {10.1007/s11784-023-01061-z},
       URL = {https://doi.org/10.1007/s11784-023-01061-z},
}

@article {BBJ2018,
    AUTHOR = {Ben-Zvi, David and Brochier, Adrien and Jordan, David},
     TITLE = {Integrating quantum groups over surfaces},
   JOURNAL = {J. Topol.},
  FJOURNAL = {Journal of Topology},
    VOLUME = {11},
      YEAR = {2018},
    NUMBER = {4},
     PAGES = {874--917},
      ISSN = {1753-8416,1753-8424},
   MRCLASS = {14D24 (14D23 16T25 18D10 57R56 57T05)},
  MRNUMBER = {3847209},
MRREVIEWER = {Jason\ Stuart\ Hanson},
       DOI = {10.1112/topo.12072},
       URL = {https://doi.org/10.1112/topo.12072},
}

@article {GJS2023,
    AUTHOR = {Gunningham, Sam and Jordan, David and Safronov, Pavel},
     TITLE = {The finiteness conjecture for skein modules},
   JOURNAL = {Invent. Math.},
  FJOURNAL = {Inventiones Mathematicae},
    VOLUME = {232},
      YEAR = {2023},
    NUMBER = {1},
     PAGES = {301--363},
      ISSN = {0020-9910,1432-1297},
   MRCLASS = {57K31},
  MRNUMBER = {4557403},
MRREVIEWER = {Mohamed\ Elhamdadi},
       DOI = {10.1007/s00222-022-01167-0},
       URL = {https://doi.org/10.1007/s00222-022-01167-0},
}

@misc{WvdB2015curvature,
    title={The curvature problem for formal and infinitesimal deformations},
    author={Lowen, Wendy and Van den Bergh, Michel},
    year={2015},
    eprint={1505.03698},
    archivePrefix={arXiv},
    primaryClass={math.KT}
}

@misc{vfKreeke2023deformations,
    title={$A_\infty$-deformations and their derived categories},
    author={van de Kreeke, Jasper},
    year={2023},
    eprint={2308.08026},
    archivePrefix={arXiv},
    primaryClass={math.KT}
}

@misc{MSS2025orbifold,
    title={Orbifold Floer theory for global quotients and Hamiltonian dynamics.},
    author={Mak, Cheuk Yu and Seyfaddini, Sobhan and Smith, Ivan},
    year={2025},
    eprint={2502.11290},
    archivePrefix={arXiv},
    primaryClass={math.SG}
}

@misc{HH25openDM,
    title={An open-closed Deligne-Mumford field theory associated to a Lagrangian submanifold.},
    author={Hirschi, Amanda, and Kai Hugtenburg},
    year={2025},
    eprint={2501.04687},
    archivePrefix={arXiv},
    primaryClass={math.SG}
}

@incollection {Keller2006dgc,
    AUTHOR = {Keller, Bernhard},
     TITLE = {On differential graded categories},
 BOOKTITLE = {International {C}ongress of {M}athematicians. {V}ol. {II}},
     PAGES = {151--190},
 PUBLISHER = {Eur. Math. Soc., Z\"urich},
      YEAR = {2006},
      ISBN = {978-3-03719-022-7},
   MRCLASS = {18E30 (14A22 16D90)},
  MRNUMBER = {2275593},
MRREVIEWER = {Volodymyr\ V.\ Lyubashenko},
}

@article {Hersovich2017spectral,
    AUTHOR = {Herscovich, Estanislao},
     TITLE = {Spectral sequences associated to deformations},
   JOURNAL = {J. Homotopy Relat. Struct.},
  FJOURNAL = {Journal of Homotopy and Related Structures},
    VOLUME = {12},
      YEAR = {2017},
    NUMBER = {3},
     PAGES = {513--548},
      ISSN = {2193-8407,1512-2891},
   MRCLASS = {18G40 (14D15 16E45 16S80 16W70 18G55 55T05)},
  MRNUMBER = {3691296},
MRREVIEWER = {Eduardo\ do Nascimento Marcos},
       DOI = {10.1007/s40062-016-0137-z},
       URL = {https://doi.org/10.1007/s40062-016-0137-z},
}

\end{document}